\documentclass[12pt,reqno]{amsart}

\usepackage[colorlinks=true,citecolor=blue,linkcolor=blue]{hyperref}
\usepackage{amsmath,amsthm,amssymb,amsfonts,mathrsfs,latexsym}
\usepackage{amsaddr,csquotes}
\usepackage{tikz}
\usepackage{graphicx}

\theoremstyle{definition}
\newtheorem{theorem}{Theorem}
\newtheorem{definition}[theorem]{Definition}
\newtheorem{corollary}[theorem]{Corollary}
\newtheorem{lemma}[theorem]{Lemma}
\newtheorem{proposition}[theorem]{Proposition}

\newtheorem{remark}[theorem]{Remark}

\newcommand{\bbbr}{\mathbb R}

\newcommand{\bbbc}{\mathbb C}
\newcommand{\bbbz}{\mathbb Z}
\newcommand{\bbbn}{\mathbb N}
\newcommand{\Lps}{\dot{L}^p_s(\bbbr^n)}
\newcommand{\Fspq}{\dot{F}^s_{p,q}(\bbbr^n)}

\newcommand{\Lloc}{L^1_{loc}(\bbbr^n)}
\newcommand{\FT}{\mathcal{F}}
\newcommand{\iFT}{\mathcal{F}^{-1}}

\newcommand{\Lebes}{\mathcal{L}}
\newcommand{\unitsph}{\mathbb{S}^{n-1}}
\newcommand{\HLmax}{\mathcal{M}}
\newcommand{\PFSmax}{\mathcal{P}}
\newcommand{\Pint}{\mathscr{P}}
\newcommand{\Sw}{\mathcal{S}}
\newcommand{\Ccinfty}{\mathcal{C}^{\infty}_0}
\newcommand{\Remain}{\mathcal{R}}
\newcommand{\Lift}{\mathcal{I}}
\newcommand{\Diff}{\mathcal{D}}
\newcommand{\borel}{\mathcal{B}}
\newcommand{\LPG}{\mathcal{G}}
\newcommand{\LPg}{\mathfrak{g}}
\newcommand{\Tfunct}{\mathfrak{T}}
\newcommand{\num}{\boldsymbol{n}}
\newcommand{\den}{\boldsymbol{d}}
\newcommand{\functrep}{\mathfrak{F}}
\newcommand{\multi}{\boldsymbol{m}}
\newcommand{\bfarIm}{b^{(f.m.)}}
\newcommand{\bnearIm}{b^{(n.m.)}}
\newcommand{\Rez}{\operatorname{Re}}

\def\esssup{\mathop{\rm ess\,sup\,}}
\def\dist{{\rm dist\,}}
\def\mvint_#1{\mathchoice
	{\mathop{\vrule width 6pt height 3 pt depth -2.5pt
			\kern -8pt \intop}\nolimits_{\kern -3pt #1}}%
	{\mathop{\vrule width 5pt height 3 pt depth -2.6pt
			\kern -6pt \intop}\nolimits_{#1}}%
	{\mathop{\vrule width 5pt height 3 pt depth -2.6pt
			\kern -6pt \intop}\nolimits_{#1}}%
	{\mathop{\vrule width 5pt height 3 pt depth -2.6pt
			\kern -6pt \intop}\nolimits_{#1}}}

\begin{document}
	
	\title[A weak inequality in Sobolev spaces]
	{A weak inequality in fractional homogeneous Sobolev spaces}
	
	\author{Lifeng Wang}
	\address[Lifeng Wang]{Department of Mathematics\\University of Pittsburgh\\Pittsburgh, PA 15260, USA}
	\email{lifeng.wang.1987@outlook.com}
	\subjclass[2020]{42B35.}
	\keywords{Fractional Sobolev spaces, generalized Littlewood-Paley functions, Whitney decomposition, homogeneous Triebel-Lizorkin spaces, Peetre-Fefferman-Stein maximal functions.}

\begin{abstract}
For $s\in\bbbr$ and $0<q<\infty$, we denote
\begin{equation*}
\mathcal{D}_{s,q}f(x):=\big(\int_{\mathbb{R}^n}\frac{|f(x)-f(y)|^q}{|x-y|^{n+sq}}dy\big)^{\frac{1}{q}}.
\end{equation*}
In this paper, we prove the following inequality
\begin{equation*}
\|\mathcal{D}_{s,q}f\|_{L^{p,\infty}(\mathbb{R}^n)}\lesssim\|f\|_{\dot{L}^p_s(\mathbb{R}^n)},
\end{equation*}
where $\|\cdot\|_{L^{p,\infty}(\mathbb{R}^n)}$ is the weak $L^p$ quasinorm and $\|\cdot\|_{\dot{L}^p_s(\mathbb{R}^n)}$ is the homogeneous Sobolev norm, and parameters satisfy the condition that $n\geq2$, $1<p<q$, $2\leq q<\infty$, and $0<s=n(\frac{1}{p}-\frac{1}{q})<1$. Furthermore, we prove the estimate
\begin{equation*}
\|\mathfrak{g}_{s,q}(f)\|_{L^p(\mathbb{R}^n)}\lesssim\|f\|_{\dot{F}^s_{p,q}(\mathbb{R}^n)}
\end{equation*}
when $0<p,q<\infty$, $-1<s<1$, $\|\cdot\|_{\dot{F}^s_{p,q}(\mathbb{R}^n)}$ denotes the homogeneous Triebel-Lizorkin quasinorm and the Littlewood-Paley-Poisson function $\mathfrak{g}_{s,q}(f)(\cdot)$ is a generalization of the classical Little-wood-Paley $g$-function. Moreover, we prove the weak type $(p,p)$ boundedness of the $\mathcal{G}_{\lambda,q}$-function and the $\mathcal{R}_{s,q}$-function, where the $\mathcal{G}_{\lambda,q}$-function is a generalization of the well-known classical Littlewood-Paley $g_{\lambda}^*$-function. In addition, we prove that when $0<p<q<\infty$ and $-\infty<s\leq n(\frac{1}{p}-\frac{1}{q})$, we have
\begin{equation*}
\|\mathcal{D}_{s,q}f\|_{L^{p}(\mathbb{R}^n)}=\infty.
\end{equation*}
And when $0<p,q<\infty$ and $-\infty<s\leq0$, we have
\begin{equation*}
\|\mathcal{D}_{s,q}f\|_{L^{p,\infty}(\mathbb{R}^n)}=\infty.
\end{equation*}
\end{abstract}

\maketitle
\tableofcontents
	
\section{Introduction}\label{Introduction}
We denote the $n$-dimensional Fourier transform of an integrable function $f\in L^1(\bbbr^n)$ by
$$\FT_n f(\xi)=\int_{\bbbr^n}f(x)e^{-2\pi ix\cdot\xi}dx,$$
and the $n$-dimensional inverse Fourier transform is denoted by
$$\iFT_n f(\xi)=\int_{\bbbr^n}f(x)e^{2\pi ix\cdot\xi}dx,$$
where for $x\in\bbbr^n$ and $\xi\in\bbbr^n$, $x\cdot\xi$ is the inner product. We denote $\bbbn$ is the set of all the positive integers, and $\bbbn_0:=\bbbn\cup\{0\}$, and $\Lloc$ is the space of all locally integrable functions defined on $\bbbr^n$, and a function $f(x)$ is in $\Lloc$ if and only if $\int_K|f(x)|dx<\infty$ for every compact set $K\subseteq\bbbr^n$. We denote $\Sw(\bbbr^n)$ is the space of Schwartz functions defined on $\bbbr^n$, and the space of continuous linear functionals on $\Sw(\bbbr^n)$, denoted by $\Sw'(\bbbr^n)$, is called the space of tempered distributions on $\bbbr^n$. Let $g\in\Sw(\bbbr^n)$ and $N\in\bbbn_0$, and we define
\begin{equation}\label{eq1.89}
\rho_N(g):=\sum_{\substack{|\alpha|\leq N\\|\beta|\leq N}}\sup_{x\in\bbbr^n}|x^{\alpha}\partial^{\beta}g(x)|.
\end{equation}
According to \cite[Proposition 2.3.4. (b)]{14classical}, a linear functional $f$ defined on $\Sw(\bbbr^n)$ is in $\Sw'(\bbbr^n)$ if and only if there exist a positive finite constant $C_1$ and a number $N\in\bbbn_0$, whose values are determined by $f$, so that the following estimate (\ref{eq1.90}) is true for all $g\in\Sw(\bbbr^n)$,
\begin{equation}\label{eq1.90}
|<f,g>|\leq C_1\cdot\rho_N(g).
\end{equation}
Let $f\in\Sw'(\bbbr^n)$. We say a function $F(x)$ defined on $\bbbr^n$ is a function representative of $f$ in the sense of $\Sw'(\bbbr^n)$ if and only if we can write
\begin{equation}\label{eq1.91}
<f,g>=\!\!\int_{\bbbr^n}F(x)\!\cdot\!g(x)dx\quad\text{for all $g\!\in\!\Sw(\bbbr^n)$}.
\end{equation}
We denote the collection of all the function representatives of $f\in\Sw'(\bbbr^n)$ in the sense of $\Sw'(\bbbr^n)$ is $\functrep(f)$. Then $F_1(x)\in\functrep(f)\cap\Lloc$ and $F_2(x)\in\functrep(f)\cap\Lloc$ imply the following equation (\ref{eq1.92}) is true for all $g\in\Sw(\bbbr^n)$,
\begin{equation}\label{eq1.92}
\int_{\bbbr^n}F_1(x)\!\cdot\!g(x)dx=\!\!\int_{\bbbr^n}F_2(x)\!\cdot\!g(x)dx,
\end{equation}
hence by Lemma \ref{lemma15}, $F_1(x)=F_2(x)$ for almost every $x\in\bbbr^n$. If $f_1$ and $f_2$ are continuous linear functionals in $\Sw'(\bbbr^n)$, then we say $f_1=f_2$ in the sense of $\Sw'(\bbbr^n)$ if the following equation (\ref{eq1.99}) is true for all $g\in\Sw(\bbbr^n)$,
\begin{equation}\label{eq1.99}
<f_1,g>=<f_2,g>.
\end{equation} 
Furthermore, if either $f_1$ or $f_2$ has at least one function representative in the sense of $\Sw'(\bbbr^n)$, then $f_1=f_2$ in the sense of $\Sw'(\bbbr^n)$ means
\begin{equation}\label{eq1.100}
\functrep(f_1)=\functrep(f_2).
\end{equation}
Moreover, if $f_1\in\Sw'(\bbbr^n)$ and $f_2\in\Sw'(\bbbr^n)$ share a common function representative in the sense of $\Sw'(\bbbr^n)$, i.e. $\functrep(f_1)\cap\functrep(f_2)\neq\emptyset$, then equation (\ref{eq1.99}) is true for all $g\in\Sw(\bbbr^n)$ and $f_1=f_2$ in the sense of $\Sw'(\bbbr^n)$, and we have $\functrep(f_1)=\functrep(f_2)$. Given $f\in\Sw'(\bbbr^n)$ and $F(x)\in\functrep(f)$, because we can modify the value of the function $F(x)$ on an arbitrary set of Lebesgue measure zero so that equation (\ref{eq1.91}) is true for the modified function and for all $g\!\in\!\Sw(\bbbr^n)$, then the modified function also belongs to $\functrep(f)$. Therefore, if $f\in\Sw'(\bbbr^n)$ has one function representative in the sense of $\Sw'(\bbbr^n)$, then $f\in\Sw'(\bbbr^n)$ has infinitely many function representatives in the sense of $\Sw'(\bbbr^n)$, and two distinct locally integrable function representatives in the sense of $\Sw'(\bbbr^n)$ are equal for almost every $x\in\bbbr^n$. In this case, we can consider $f\in\Sw'(\bbbr^n)$ is a symbol that represents an infinite collection $\functrep(f)$ of function representatives in the sense of $\Sw'(\bbbr^n)$. The space $\Sw_0(\bbbr^n)$ is a subspace of $\Sw(\bbbr^n)$ that inherits the same topology as $\Sw(\bbbr^n)$, and a Schwartz function $\phi\in\Sw(\bbbr^n)$ is in $\Sw_0(\bbbr^n)$ if and only if $\phi$ satisfies the condition
	\begin{equation}\label{Sw0.condition.1}
		\int_{\bbbr^n}x^{\gamma}\phi(x)dx=0\qquad\text{for all multi-indices $\gamma$},
	\end{equation}
	or equivalently, the condition
	\begin{equation}\label{Sw0.condition.2}
		(\partial^{\gamma}\iFT_n\phi)(0)=0\qquad\text{for all multi-indices $\gamma$}.
	\end{equation}
Furthermore, if $\omega$ and $\omega'$ are multi-indices, then we can deduce from Taylor's formula for the Schwartz function $\partial^{\omega}\iFT_n\phi(\xi)$ at $\xi=0$ with integral form of the remainder the following estimate
\begin{align}
&|\partial^{\omega}\iFT_n\phi(\xi)|\!
\leq\!\!\!\int_0^1\!\!\frac{(1-t)^{M-1}}{(M-1)!}\!\cdot\!\bigg|\sum_{|\omega'|=M}\!\!\!\xi^{\omega'}
(\partial^{\omega+\omega'}\iFT_n\phi)(t\xi)\bigg|dt\nonumber\\
&\lesssim\sum_{|\omega'|=M}\sup_{x\in\bbbr^n}|(\partial^{\omega+\omega'}\iFT_n\phi)(x)|\cdot|\xi|^M,
\label{Sw0.condition.3}
\end{align}
where $M$ is an arbitrarily large positive integer and the implicit constant depends on $n$ and $M$. By \cite[Proposition 1.1.3]{14modern}, the dual space of $\Sw_0(\bbbr^n)$ under the topology inherited from $\Sw(\bbbr^n)$ is the space of tempered distributions modulo polynomials. The space of tempered distributions modulo polynomials is the space of equivalence classes of tempered distributions, where two tempered distributions $u$ and $v$ are in the same equivalence class if and only if their difference $u-v$ is another tempered distribution whose function representative in the sense of $\Sw'(\bbbr^n)$ is a polynomial on $\bbbr^n$. Because the topology of $\Sw(\bbbr^n)$ is the only topology we will consider in this paper when it comes to the matters of $\Sw(\bbbr^n)$ and $\Sw_0(\bbbr^n)$, we denote the space of tempered distributions modulo polynomials by $\Sw_0'(\bbbr^n)$. A linear functional $f$ defined on $\Sw_0(\bbbr^n)$ is in $\Sw_0'(\bbbr^n)$ if and only if there exist a positive finite constant $C_2$ and a number $N\in\bbbn_0$, whose values are determined by $f$, so that the following estimate (\ref{eq1.93}) is true for all $\phi\in\Sw_0(\bbbr^n)$,
\begin{equation}\label{eq1.93}
|<f,\phi>|\leq C_2\cdot\rho_N(\phi).
\end{equation}
Let $f\in\Sw_0'(\bbbr^n)$. We say a function $F(x)$ defined on $\bbbr^n$ is a function representative of $f$ in the sense of $\Sw_0'(\bbbr^n)$ if and only if we can write
\begin{equation}\label{eq1.94}
<f,\phi>=\!\!\int_{\bbbr^n}F(x)\!\cdot\!\phi(x)dx\quad\text{for all $\phi\!\in\!\Sw_0(\bbbr^n)$}.
\end{equation}
We denote the collection of all the function representatives of $f\in\Sw_0'(\bbbr^n)$ in the sense of $\Sw_0'(\bbbr^n)$ is $\functrep_0(f)$. Then $F_1(x)\in\functrep_0(f)$ and $F_2(x)\in\functrep_0(f)$ imply the following equation (\ref{eq1.95}) is true for all $\phi\in\Sw_0(\bbbr^n)$,
\begin{equation}\label{eq1.95}
\int_{\bbbr^n}F_1(x)\!\cdot\!\phi(x)dx=\!\!\int_{\bbbr^n}F_2(x)\!\cdot\!\phi(x)dx.
\end{equation}
Given $f\in\Sw'(\bbbr^n)\subseteq\Sw_0'(\bbbr^n)$, we see that if $f$ has a function representative $F(x)$ in the sense of $\Sw'(\bbbr^n)$, then $F(x)$ is also a function representative of $f$ in the sense of $\Sw_0'(\bbbr^n)$, thus $\functrep(f)\subseteq\functrep_0(f)$. Because Schwartz functions in $\Sw_0(\bbbr^n)$ satisfy the condition (\ref{Sw0.condition.1}), then for the same $f\in\Sw'(\bbbr^n)$, its function representative in the sense of $\Sw_0'(\bbbr^n)$ and its function representative in the sense of $\Sw'(\bbbr^n)$ are generally different, with the possible difference of a polynomial. If $f\in\Sw_0'(\bbbr^n)$ and $F(x)$ is a function in $\functrep_0(f)$ and $P(x)$ is an arbitrary polynomial defined on $\bbbr^n$, then the function $F(x)+P(x)$ also belongs to $\functrep_0(f)$. Therefore, if $f\in\Sw_0'(\bbbr^n)$ has one function representative in the sense of $\Sw_0'(\bbbr^n)$, then $f\in\Sw_0'(\bbbr^n)$ has infinitely many function representatives in the sense of $\Sw_0'(\bbbr^n)$. In this case, we can consider $f\in\Sw_0'(\bbbr^n)$ is a symbol that represents an infinite collection $\functrep_0(f)$ of function representatives in the sense of $\Sw_0'(\bbbr^n)$. If $f_1$ and $f_2$ are continuous linear functionals in $\Sw_0'(\bbbr^n)$, then we say $f_1=f_2$ in the sense of $\Sw_0'(\bbbr^n)$ if the following equation (\ref{eq1.96}) is true for all $\phi\in\Sw_0(\bbbr^n)$,
\begin{equation}\label{eq1.96}
<f_1,\phi>=<f_2,\phi>.
\end{equation} 
Furthermore, if either $f_1$ or $f_2$ has at least one function representative in the sense of $\Sw_0'(\bbbr^n)$, then $f_1=f_2$ in the sense of $\Sw_0'(\bbbr^n)$ means that both $f_1\in\Sw_0'(\bbbr^n)$ and $f_2\in\Sw_0'(\bbbr^n)$ have infinitely many function representatives in the sense of $\Sw_0'(\bbbr^n)$, and we have
\begin{equation}\label{eq1.97}
\functrep_0(f_1)=\functrep_0(f_2).
\end{equation}
Moreover, if $f_1\in\Sw_0'(\bbbr^n)$ and $f_2\in\Sw_0'(\bbbr^n)$ share a common function representative in the sense of $\Sw_0'(\bbbr^n)$, i.e. $\functrep_0(f_1)\cap\functrep_0(f_2)\neq\emptyset$, then equation (\ref{eq1.96}) is true for all $\phi\in\Sw_0(\bbbr^n)$ and $f_1=f_2$ in the sense of $\Sw_0'(\bbbr^n)$, and we have $\functrep_0(f_1)=\functrep_0(f_2)$. By the definition of dual spaces, we have $\Sw'(\bbbr^n)\subseteq\Sw_0'(\bbbr^n)$. In addition, for every $f\in\Sw_0'(\bbbr^n)$, we can use the Hahn-Banach theorem, i.e. \cite[3.3 Theorem]{Rudin.Funct.Anal}, to obtain an extension of $f$ to a tempered distribution defined on all of $\Sw(\bbbr^n)$, and we still denote this extended tempered distribution by the same symbol $f$. In this paper, we call the extended tempered distribution $f\in\Sw'(\bbbr^n)$ obtained by the Hahn-Banach theorem, i.e. \cite[3.3 Theorem]{Rudin.Funct.Anal}, is the HB-extension of the continuous linear functional $f\in\Sw_0'(\bbbr^n)$. Another important subspace of $\Sw(\bbbr^n)$ is the space $\Sw_{00}(\bbbr^n)$ defined below,
\begin{equation}\label{eq1.62}
\Sw_{00}(\bbbr^n)\!=\!\{g\!\in\!\Sw(\bbbr^n)\!:\!\partial^{\gamma}g(0)\!=\!0\quad\text{for all multi-indices $\gamma$}\}.
\end{equation}
	If $\omega$, $\omega'$ are multi-indices and $g\in\Sw_{00}(\bbbr^n)$, then we can deduce from Taylor's formula for the Schwartz function $\partial^{\omega}g(\xi)$ at $\xi=0$ with integral form of the remainder the following estimate
	\begin{equation}\label{eq1.63}
		|\partial^{\omega}g(\xi)|
		\lesssim\sum_{|\omega'|=M}\sup_{x\in\bbbr^n}|(\partial^{\omega+\omega'}g)(x)|\cdot|\xi|^M,
	\end{equation}
where $M$ is an arbitrarily large positive integer and the implicit constant depends on $n$ and $M$. The dual space of $\Sw_{00}(\bbbr^n)$ under the topology inherited from $\Sw(\bbbr^n)$ is the space $\Sw_{00}'(\bbbr^n)$ of continuous linear functionals defined on the subspace $\Sw_{00}(\bbbr^n)$. A linear functional $f$ defined on $\Sw_{00}(\bbbr^n)$ is in $\Sw_{00}'(\bbbr^n)$ if and only if there exist a positive finite constant $C_3$ and a number $N\in\bbbn_0$, whose values are determined by $f$, so that the following estimate (\ref{eq1.107}) is true for all $g\in\Sw_{00}(\bbbr^n)$,
\begin{equation}\label{eq1.107}
|<f,g>|\leq C_3\cdot\rho_N(g).
\end{equation}
Let $f\in\Sw_{00}'(\bbbr^n)$. We say a function $F(x)$ defined on $\bbbr^n$ is a function representative of $f$ in the sense of $\Sw_{00}'(\bbbr^n)$ if and only if we can write
\begin{equation}\label{eq1.108}
<f,g>=\!\!\int_{\bbbr^n}F(x)\!\cdot\!g(x)dx\quad\text{for all $g\!\in\!\Sw_{00}(\bbbr^n)$}.
\end{equation}
We denote the collection of all the function representatives of $f\in\Sw_{00}'(\bbbr^n)$ in the sense of $\Sw_{00}'(\bbbr^n)$ is $\functrep_{00}(f)$. Then $F_1(x)\in\functrep_{00}(f)$ and $F_2(x)\in\functrep_{00}(f)$ imply the following equation (\ref{eq1.109}) is true for all $g\in\Sw_{00}(\bbbr^n)$,
\begin{equation}\label{eq1.109}
\int_{\bbbr^n}F_1(x)\!\cdot\!g(x)dx=\!\!\int_{\bbbr^n}F_2(x)\!\cdot\!g(x)dx.
\end{equation}
Every $f\in\Sw_{00}'(\bbbr^n)$ can be extended to a tempered distribution defined on all of $\Sw(\bbbr^n)$ by the Hahn-Banach theorem and denoted by the same symbol $f$, and we call this extended tempered distribution is the HB-extension of $f\in\Sw_{00}'(\bbbr^n)$. By the definition of dual spaces, we have $\Sw'(\bbbr^n)\subseteq\Sw_{00}'(\bbbr^n)$, hence $\Sw'(\bbbr^n)\subseteq\Sw_0'(\bbbr^n)\cap\Sw_{00}'(\bbbr^n)$. From the definitions of $\Sw_0(\bbbr^n)$ and $\Sw_{00}(\bbbr^n)$, we see that a Schwartz function $\phi$ is in $\Sw_0(\bbbr^n)$ if and only if $\iFT_n\phi$ is in $\Sw_{00}(\bbbr^n)$. We have the following propositions related to $\Sw(\bbbr^n)$, $\Sw_0(\bbbr^n)$, and $\Sw_{00}(\bbbr^n)$.
\begin{proposition}\label{proposition2}
If $F(x)\in\Lloc$ satisfies the condition that
\begin{equation}\label{eq1.101}
\int_{\bbbr^n}\!\!\!\!\!F(x)\!\cdot\!\phi(x)dx\!=\!0\text{ for all $\phi\!\in\!\Sw(\bbbr^n)$ with $spt.\iFT_n\phi\!\subseteq\!\bbbr^n\!\setminus\!\{0\}$,}
\end{equation}
and if there exists $f\in\Sw'(\bbbr^n)$ so that $F(x)\in\functrep(f)$, then $F(x)$ equals a polynomial function for almost every $x\in\bbbr^n$. Therefore, if equation (\ref{eq1.101}) holds true for all $\phi\in\Sw_{0}(\bbbr^n)$, then $F(x)$ equals a polynomial function for almost every $x\in\bbbr^n$.
\end{proposition}
\begin{proof}[Proof of Proposition \ref{proposition2}]
Let $f\in\Sw'(\bbbr^n)$ and $F(x)\in\functrep(f)$, then we have
\begin{equation}\label{eq1.102}
<f,g>=\int_{\bbbr^n}F(x)\cdot g(x)dx\quad\text{for all $g\in\Sw(\bbbr^n)$.}
\end{equation}
Combining conditions (\ref{eq1.101}) and (\ref{eq1.102}), we deduce that for all $\phi\!\in\!\Sw(\bbbr^n)$ with $spt.\iFT_n\phi\!\subseteq\!\bbbr^n\!\setminus\!\{0\}$,
\begin{equation}\label{eq1.103}
<\FT_n f,\iFT_n\phi>=<f,\phi>=\int_{\bbbr^n}F(x)\!\cdot\!\phi(x)dx=0.
\end{equation}
Because the Fourier transform is a homeomorphism from $\Sw(\bbbr^n)$ onto $\Sw(\bbbr^n)$, we have obtained $<\FT_n f,g>=0$ for all $g\!\in\!\Sw(\bbbr^n)$ with $spt.g\!\subseteq\!\bbbr^n\!\setminus\!\{0\}$, thus $\FT_n f$ is supported at the origin of $\bbbr^n$, hence by \cite[Corollary 2.4.2.]{14classical}, there exists a polynomial function $P(x)$ of $x\in\bbbr^n$ so that $P(x)\in\functrep(f)\cap\Lloc$. Since $F(x)\in\functrep(f)\cap\Lloc$, Lemma \ref{lemma15} implies the functions $F(x)$ and $P(x)$ are equal for almost every $x\in\bbbr^n$. If equation (\ref{eq1.101}) is true for all $\phi\in\Sw_{0}(\bbbr^n)$, then we can use the fact that $\phi\in\Sw_0(\bbbr^n)$ if and only if $\iFT_n\phi\in\Sw_{00}(\bbbr^n)$ to deduce the same conclusion for the function $F(x)$.
\end{proof}
We consider the integral (\ref{eq1.104}) for $g\in\Sw(\bbbr^n)$,
\begin{equation}\label{eq1.104}
\int_{\bbbr^n}F(x)\!\cdot\!g(x)dx.
\end{equation}
If the function $F(x)$ is in $L^{p_1}(\bbbr^n)+L^{p_2}(\bbbr^n)$ for $1\leq p_1,p_2\leq\infty$, then we can write $F(x)=F_1(x)+F_2(x)$ where $F_1(x)\in L^{p_1}(\bbbr^n)$ and $F_2(x)\in L^{p_2}(\bbbr^n)$, and hence
\begin{equation*}
(\ref{eq1.104})=\int_{\bbbr^n}F_1(x)\!\cdot\!g(x)dx+\int_{\bbbr^n}F_2(x)\!\cdot\!g(x)dx.
\end{equation*}
Therefore, H\"{o}lder's inequality implies the integral (\ref{eq1.104}) defines a tempered distribution denoted by the symbol $f$ and $F(x)\in\functrep(f)$. Similarly, if $F(x)$ is a function defined on $\bbbr^n$ that satisfies the condition
\begin{equation}\label{eq1.105}
|F(x)|\leq C_4\cdot(1+|x|)^{C_5}
\end{equation}
for some real numbers $C_4\in(0,\infty)$ and $C_5\in(-\infty,\infty)$, then the integral (\ref{eq1.104}) defines a tempered distribution and $F(x)$ is a function representative of this tempered distribution in the sense of $\Sw'(\bbbr^n)$. In particular, let $f\in\Sw_0'(\bbbr^n)$, $F_1(x)\in\functrep_0(f)\cap L^{p_1}(\bbbr^n)$, and $F_2(x)\in\functrep_0(f)\cap L^{p_2}(\bbbr^n)$ for some $1\leq p_1,p_2\leq\infty$, then the function $F_1(x)-F_2(x)$ is in $L^{p_1}(\bbbr^n)+L^{p_2}(\bbbr^n)\subseteq\Lloc$ and there exists $\tilde{f}\in\Sw'(\bbbr^n)$ so that $F_1(x)-F_2(x)\in\functrep(\tilde{f})$. According to Proposition \ref{proposition2}, the difference $F_1(x)-F_2(x)$ equals a polynomial for almost every $x\in\bbbr^n$. The proof of Proposition \ref{proposition1} below can be found in section \ref{proof.of.proposition1}.
\begin{proposition}\label{proposition1}
Assume that $\multi(\xi)$ is a smooth function on $\bbbr^n\setminus\{0\}$, and for every multi-index $\alpha=(\alpha_1,\cdots,\alpha_n)$, there exist nonnegative real numbers $L_1^{\alpha}$ and $L_2^{\alpha}$ such that
\begin{equation}\label{eq1.64}
|\partial^{\alpha}\multi(\xi)|\lesssim\max\{|\xi|^{-L_1^{\alpha}},|\xi|^{L_2^{\alpha}}\}\quad
\text{for $\xi\in\bbbr^n\setminus\{0\}$},
\end{equation}
where the values of $L_1^{\alpha}$ and $L_2^{\alpha}$ depend on $\alpha$ and are independent of $\xi$, and the implicit constant in (\ref{eq1.64}) depends on fixed parameters. Then for every function $g\in\Sw_{00}(\bbbr^n)$, the product $\multi(\xi)\cdot g(\xi)$ is a function in $\Sw_{00}(\bbbr^n)$. And we have the following conclusions.\\
(i) If $f\in\Sw_{00}'(\bbbr^n)$, then the product $\multi(\xi)\cdot f$ of the function $\multi(\xi)$ and the continuous linear functional $f\in\Sw_{00}'(\bbbr^n)$ is a well-defined continuous linear functional on the subspace $\Sw_{00}(\bbbr^n)$ in the topology of $\Sw(\bbbr^n)$, and its action on functions in $\Sw_{00}(\bbbr^n)$ is given by
\begin{equation}\label{eq1.65}
<\multi(\xi)\cdot f,g>=<f,\multi\cdot g>\quad\text{for $g\in\Sw_{00}(\bbbr^n)$}.
\end{equation}
Therefore we can apply Hahn-Banach theorem to obtain an extended tempered distribution defined on all of $\Sw(\bbbr^n)$ and written in the same symbol as $\multi(\xi)\cdot f$. The restriction of this extended tempered distribution on the subspace $\Sw_{00}(\bbbr^n)$ is given by equation (\ref{eq1.65}). Furthermore, the distributional inverse Fourier transform of the extended tempered distribution $\multi(\xi)\cdot f$ is well-defined and its action on a general Schwartz function $\varphi\in\Sw(\bbbr^n)$ is given by
\begin{equation}\label{eq1.66}
<\iFT_n[\multi(\xi)\cdot f],\varphi>=<\multi(\xi)\cdot f,\iFT_n\varphi>.
\end{equation}
Moreover, if $\varphi$ belongs to the subspace $\Sw_0(\bbbr^n)$, then $\iFT_n\varphi$ belongs to the subspace $\Sw_{00}(\bbbr^n)$ and we can continue from (\ref{eq1.66}) as follows,
\begin{equation}\label{eq1.67}
<\iFT_n[\multi(\xi)\cdot f],\varphi>=<\multi(\xi)\cdot f,\iFT_n\varphi>
=<f,\multi\cdot\iFT_n\varphi>.
\end{equation}
(ii) If $f\in\Sw_0'(\bbbr^n)$ and $\multi(\xi)$ is also a function in $L^1(\bbbr^n)$, then formula (\ref{eq1.110}) is true for all $g\in L^1(\bbbr^n)$ and $x\in\bbbr^n$,
\begin{equation}\label{eq1.110}
\iFT_n[\multi(\xi)\FT_n g(\xi)](x)=(\iFT_n\multi)*g(x).
\end{equation}
When $\varphi\in\Sw_0(\bbbr^n)$, we have the function $\multi(\xi)\cdot\FT_n\varphi(\xi)$ is in the subspace $\Sw_{00}(\bbbr^n)$, thus for every $x\in\bbbr^n$,
\begin{equation}\label{eq1.111}
\iFT_n[\multi(\xi)\FT_n\varphi(\xi)](x)=(\iFT_n\multi)*\varphi(x)\in\Sw_0(\bbbr^n).
\end{equation}
The convolution $f*\FT_n\multi$ of the function $\FT_n\multi\in L^{\infty}(\bbbr^n)$ and the continuous linear functional $f\in\Sw_0'(\bbbr^n)$ is a well-defined continuous linear functional in $\Sw_0'(\bbbr^n)$, and its action on functions $\varphi\in\Sw_0(\bbbr^n)$ is given by
\begin{equation}\label{eq1.112}
<f*\FT_n\multi,\varphi>=<f,(\iFT_n\multi)*\varphi>.
\end{equation}
Therefore we can apply Hahn-Banach theorem to obtain an extended tempered distribution defined on all of $\Sw(\bbbr^n)$ and written in the same symbol as $f*\FT_n\multi$. The restriction of this extended tempered distribution on the subspace $\Sw_{0}(\bbbr^n)$ is given by equation (\ref{eq1.112}). Furthermore, the distributional inverse Fourier transform of the extended tempered distribution $f*\FT_n\multi$ is well-defined and its action on a general Schwartz function $g\in\Sw(\bbbr^n)$ is given by
\begin{equation}\label{eq1.113}
<\iFT_n[f*\FT_n\multi],g>=<f*\FT_n\multi,\iFT_n g>.
\end{equation}
Moreover, if $g$ is a function in the subspace $\Sw_{00}(\bbbr^n)$, then $\iFT_n g$ is a function in the subspace $\Sw_{0}(\bbbr^n)$ and we can combine (\ref{eq1.111}), (\ref{eq1.112}), and (\ref{eq1.113}) to obtain the following equation,
\begin{align}
&<\iFT_n[f*\FT_n\multi],g>=<f,\iFT_n\multi*\iFT_n g>\nonumber\\
&=<f,\iFT_n[\multi(\xi)\cdot g(\xi)]>.\label{eq1.114}
\end{align}
In particular, the convolution $f*g$ of $f\in\Sw_{0}'(\bbbr^n)$ and a function $g\in\Sw(\bbbr^n)$ is a continuous linear functional in $\Sw_0'(\bbbr^n)$ defined by the equation
\begin{equation}\label{eq1.124}
<f*g,\varphi>=<f,\tilde{g}*\varphi>\text{ for $\varphi\!\in\!\Sw_0(\bbbr^n)$},
\end{equation}
where $\tilde{g}(x)=g(-x)$. And $f*g\in\Sw_{0}'(\bbbr^n)$ has the following smooth function representative (\ref{eq1.123}) in the sense of $\Sw_0'(\bbbr^n)$,
\begin{equation}\label{eq1.123}
x\in\bbbr^n\longmapsto<f,g(x-\cdot)>,
\end{equation}
where $f\!\in\!\Sw'(\bbbr^n)$ appearing in (\ref{eq1.123}) is the HB-extension of $f\!\in\!\Sw_{0}'(\bbbr^n)$. In addition, for every multi-index $\alpha$, there exist positive finite constants $C_{\alpha}$ and $K_{\alpha}$ such that
\begin{equation}\label{eq1.125}
|\partial^{\alpha}(f*g)(x)|\leq C_{\alpha}\cdot(1+|x|)^{K_{\alpha}}.
\end{equation}
Therefore, for every $h\in\Sw(\bbbr^n)$, the following equation (\ref{eq1.115}) holds true in the sense of $\Sw_0'(\bbbr^n)$,
\begin{equation}\label{eq1.115}
(f*\FT_n\multi)*h=(f*h)*\FT_n\multi.
\end{equation}
(iii) If $f\in\Sw_{00}'(\bbbr^n)$ and $\multi(\xi)$ is also a function in $L^1(\bbbr^n)$, then formula (\ref{eq1.116}) is true in the sense of $\Sw_{0}'(\bbbr^n)$,
\begin{equation}\label{eq1.116}
\FT_n[\multi(\xi)\cdot f]=\FT_n f*\FT_n\multi,
\end{equation}
where $f$ appearing on the left side of (\ref{eq1.116}) is the given continuous linear functional in $\Sw_{00}'(\bbbr^n)$, and $\FT_n f$ appearing on the right side of (\ref{eq1.116}) is the distributional Fourier transform of the HB-extension of $f\in\Sw_{00}'(\bbbr^n)$.\\ 
(iv) If $f\in\Sw_{0}'(\bbbr^n)$ and $\varphi\in\Sw_0(\bbbr^n)$, then the convolution $f*\varphi$ is a well-defined tempered distribution in $\Sw'(\bbbr^n)$ and has the following smooth function representative (\ref{eq1.119}) in the sense of $\Sw'(\bbbr^n)$,
\begin{equation}\label{eq1.119}
x\in\bbbr^n\longmapsto<f,\varphi(x-\cdot)>.
\end{equation}
The action of $f*\varphi\in\Sw'(\bbbr^n)$ on a function $g\in\Sw(\bbbr^n)$ is given by
\begin{equation}\label{eq1.122}
<f*\varphi,g>=<f,\tilde{\varphi}*g>,
\end{equation}
where $\tilde{\varphi}(x)=\varphi(-x)$. Furthermore, for every multi-index $\alpha$, there exist positive finite constants $C_{\alpha}$ and $K_{\alpha}$ such that
\begin{equation}\label{eq1.120}
|\partial^{\alpha}(f*\varphi)(x)|\leq C_{\alpha}\cdot(1+|x|)^{K_{\alpha}}.
\end{equation}
Moreover, the following formula (\ref{eq1.126}) is true in the sense of $\Sw'(\bbbr^n)$,
\begin{equation}\label{eq1.126}
\FT_n[f*\varphi]=\FT_n f\cdot\FT_n\varphi,
\end{equation}
where $\FT_n[f*\varphi]$ on the left side of (\ref{eq1.126}) is the distributional Fourier transform of $f*\varphi\in\Sw'(\bbbr^n)$, and $\FT_n f$ on the right side of (\ref{eq1.126}) is the distributional Fourier transform of the HB-extension of $f\in\Sw_{0}'(\bbbr^n)$. The statement of Proposition \ref{proposition1} ends here.
\end{proposition}
We consider Proposition \ref{proposition1} (ii) and (iv) are partial generalizations of \cite[Theorem 2.3.20.]{14classical}. From \cite[Theorem 2.3.20.]{14classical} and Proposition \ref{proposition1} (ii), we see that given $f\in\Sw_{0}'(\bbbr^n)$ and $g\in\Sw(\bbbr^n)$, the convolution $f*g$ is a well-defined continuous linear functional in $\Sw_{0}'(\bbbr^n)$. And we can also obtain the tempered distribution $\overline{f}*g\in\Sw'(\bbbr^n)\subseteq\Sw_{0}'(\bbbr^n)$ by convoluting the HB-extension $\overline{f}\in\Sw'(\bbbr^n)$ of $f\in\Sw_{0}'(\bbbr^n)$ with $g\in\Sw(\bbbr^n)$. Furthermore, $f*g=\overline{f}*g$ in the sense of $\Sw_{0}'(\bbbr^n)$.\\

We review the definition of homogeneous Triebel-Lizorkin spaces. Through out the whole paper, we fix the Schwartz function $\psi\in\Sw_0(\bbbr^n)$ and the positive finite constant $C_{\psi}$ whose value is solely determined by $\psi$, that satisfy the following conditions,
	\begin{align}
		&spt.\FT_n\psi\subseteq\{\xi\in\bbbr^n:\frac{1}{2}\leq|\xi|<2\},\label{eq1-7}\\
		&\sum_{j\in\bbbz}\FT_n\psi(2^{-j}\xi)=1\quad\text{if}\quad\xi\neq0,\label{eq1-8}\\
		&|\FT_n\psi(\xi)|\geq C_{\psi}>0\quad\text{if}\quad\frac{3}{5}\leq|\xi|<\frac{5}{3}.\label{eq1.82}
	\end{align}
For every $j\in\bbbz$ and $x\in\bbbr^n$, we denote $\psi_{2^{-j}}(x)=2^{jn}\psi(2^j x)$, then $\psi_{2^{-j}}\in\Sw_{0}(\bbbr^n)$. When $f\in\Sw_{0}'(\bbbr^n)$, Proposition \ref{proposition1} (iv) implies that $f*\psi_{2^{-j}}\in\Sw'(\bbbr^n)$ has the smooth function representative (\ref{eq1.117}) in the sense of $\Sw'(\bbbr^n)$, 
\begin{equation}\label{eq1.117}
F_j:x\in\bbbr^n\longmapsto<f,\psi_{2^{-j}}(x-\cdot)>.
\end{equation}
And we have the following decomposition
\begin{equation}\label{eq1-9}
f=\sum_{j\in\bbbz}f*\psi_{2^{-j}},
\end{equation}
where the series in (\ref{eq1-9}) converges to $f\in\Sw_0'(\bbbr^n)$ in the sense of $\Sw_0'(\bbbr^n)$ (cf. \cite[section 1.1.1]{14modern}). 
\begin{definition}\label{definition1}
For $0<p<\infty$, $0<q<\infty$ and $s\in\bbbr$, the homogeneous Triebel-Lizorkin space $\Fspq$ is a subspace of $\Sw_0'(\bbbr^n)$, and an element $f\in\Sw_0'(\bbbr^n)$ is in $\Fspq$ if and only if the quasinorm defined below is finite,
\begin{align}
&\|f\|_{\Fspq}:=\big(\int_{\bbbr^n}\big(\sum_{k\in\bbbz}2^{ksq}|f*\psi_{2^{-k}}(x)|^q\big)^{\frac{p}{q}}dx\big)^{\frac{1}{p}}\nonumber\\
&=\big(\int_{\bbbr^n}\big(\sum_{k\in\bbbz}2^{ksq}|F_k(x)|^q\big)^{\frac{p}{q}}dx\big)^{\frac{1}{p}},
\label{eq1.127}
\end{align}
where for each $k\in\bbbz$, $F_k(x)$ is the function given in (\ref{eq1.117}).
\end{definition}
For an appropriate function $f$ so that the expressions in (\ref{eq1-10}) and (\ref{eq1-11}) make sense, $s\in\bbbr$, and $0<q<\infty$, we use the following notations
	\begin{equation}\label{eq1-10}
		\Diff_{s,q}f(x):=\big(\int_{\bbbr^n}\frac{|f(x+y)-f(x)|^q}{|y|^{n+sq}}dy\big)^{\frac{1}{q}},
	\end{equation}
	and
	\begin{equation}\label{eq1-11}
		\Lift_s f(x):=\iFT_n(|\xi|^s\FT_n f(\xi))(x).
	\end{equation}
	In \cite{Wang2023}, the following known inequality is proven.
	\begin{theorem}[cf. Corollary 1.7 (i) of \cite{Wang2023}]\label{theorem1}
		If $0<p,q<\infty$, $\max\{0,n(\frac{1}{p}-\frac{1}{q})\}<s<1$, and $f\in\Fspq$ has a function representative, then we have
		\begin{equation}\label{eq1-12}
			\|\Diff_{s,q}f\|_{L^p(\bbbr^n)}\lesssim\|f\|_{\Fspq}.
		\end{equation}
	\end{theorem}
Many other authors also studied the counterpart of this inequality for the inhomogeneous Triebel-Lizorkin space $F^s_{p,q}(\bbbr^n)$. In \cite[section 2.5.10]{1983functionspaces}, H. Triebel gave an equivalence characterization theorem of the inhomogeneous Triebel-Lizorkin space $F^s_{p,q}(\bbbr^n)$ in terms of iterated difference, where the conditions for the parameters are $0<p<\infty$, $0<q\leq\infty$, $\frac{n}{\min\{p,q\}}<s<M$ and $M$ is the iteration number. In \cite[Theorem 1 page 102]{Stein1961}, E. M. Stein gave the equivalence characterization
$$\|(\!\int_{\bbbr^n}\!\!\!\!|h|^{-2\alpha}|f(\cdot\!+\!h)\!-\!f(\cdot)|^2\frac{dh}{|h|^n})^{\frac{1}{2}}\|_{L^p(\bbbr^n)}\!+\!\|f\|_{L^p(\bbbr^n)}\!\sim\!\|f\|_{L^p_{\alpha}(\bbbr^n)}$$ where the restrictions $0<\alpha<1$, $1<p<\infty$ and $\frac{2n}{n+2\alpha}<p<\infty$ were considered essentially sharp. And for $1<p<\infty$, $\alpha\in\bbbr$ and $f\in\Sw'(\bbbr^n)$, the inhomogeneous Sobolev norm of $f$ (cf. \cite[section 1.3.1]{14modern}) is defined to be
	\begin{equation}\label{Wang2023-eq45}
		\|f\|_{L_{\alpha}^p(\bbbr^n)}:=\|\iFT_n((1+|\xi|^2)^{\alpha/2}\FT_n f)\|_{L^p(\bbbr^n)}.
	\end{equation}
	Since the inhomogeneous Triebel-Lizorkin space satisfies $L^p_{\alpha}(\bbbr^n)\sim F^{\alpha}_{p,2}(\bbbr^n)$ if $1<p<\infty$, we consider E. M. Stein's result is an improvement of H. Triebel's. Furthermore, in \cite[Theorem 1 page 393]{seeger1989note}, A. Seeger provided an equivalence characterization for the homogeneous anisotropic space
	$$\|f\|_{\dot{F}^{\alpha}_{p,q}(\bbbr^n)}\sim\|S^{\alpha}_{q,r,m}f\|_{L^p(\bbbr^n)}$$
	where $0<p<\infty$, $0<q\leq\infty$, $m>\alpha/a_0$, $r\geq 1$ with
	$$\alpha>\max\{0,\nu(\frac{1}{p}-\frac{1}{r}),\nu(\frac{1}{q}-\frac{1}{r})\},$$
	and
	$$S^{\alpha}_{q,r,m}f(x)=(\int_0^{\infty}[\mvint_{\varrho(h)\leq t}
	|(\varDelta^m_{h}f)(x)|^r dh]^{q/r}\frac{dt}{t^{1+\alpha q}})^{1/q}.$$
	If we consider the isotropic case of A. Seeger's result, in which $\varrho(h)$ above can be deemed as $|h|$ and $a_0$ can be deemed as $1$, then by letting $r=q$ and changing the order of integration we can obtain
	\begin{align*}
		&\|f\|_{\dot{F}^{\alpha}_{p,q}(\bbbr^n)}
		\sim\|(\int_0^{\infty}\int_{|h|\leq t}t^{-1-n-q\alpha}\cdot|(\varDelta^m_h f)(\cdot)|^q dhdt)^{\frac{1}{q}}\|_{L^p(\bbbr^n)}\\
		&\sim\|(\int_{\bbbr^n}|h|^{-q\alpha}\cdot|(\varDelta^m_h f)(\cdot)|^q\frac{dh}{|h|^n})^{\frac{1}{q}}\|_{L^p(\bbbr^n)},
	\end{align*}
	for $0<p<\infty$, $1\leq q\leq\infty$ and $\max\{0,\nu(\frac{1}{p}-\frac{1}{q})\}<\alpha<m$. Recently in \cite[Theorem 1.2 page 693]{Prats19} M. Prats proves an equivalence characterization theorem of the inhomogeneous norm $\|f\|_{F^s_{p,q}(\Omega)}$ in terms of the sum of $\|f\|_{W^{k,p}(\Omega)}$ and
	\begin{equation}\label{Wang2023-eq1}
		\sum_{|\alpha|=k}\big(\int_{\Omega}\big(\int_{\Omega}
		\frac{|D^{\alpha}f(x)-D^{\alpha}f(y)|^q}{|x-y|^{\sigma q+d}}
		dy\big)^{\frac{p}{q}}dx\big)^{\frac{1}{p}}
	\end{equation}
	when parameters satisfy $1\leq p<\infty,1\leq q\leq\infty,s=k+\sigma$, $\max\{0,d(\frac{1}{p}-\frac{1}{q})\}<\sigma<1$ and $\Omega$ is a uniform domain in $\bbbr^d$. Furthermore, M. Prats also shows under the same conditions on parameters, the equivalence relation stands if (\ref{Wang2023-eq1}) is replaced by
	\begin{equation}\label{Wang2023-eq2}
		\sum_{|\alpha|=k}\big(\int_{\Omega}\big(\int_{\mathbf{Sh}(x)}
		\frac{|D^{\alpha}f(x)-D^{\alpha}f(y)|^q}{|x-y|^{\sigma q+d}}
		dy\big)^{\frac{p}{q}}dx\big)^{\frac{1}{p}}
	\end{equation}
	where $\mathbf{Sh}(x):=\{y\in\Omega:|y-x|\leq c_{\Omega}\delta(x)\}$ is the Carleson box centered at $x$, $\delta(x)=\dist(x,\partial\Omega)$ and $c_{\Omega}>1$ is a constant. Moreover, when $1\leq q\leq p<\infty$, the set $\mathbf{Sh}(x)$ in (\ref{Wang2023-eq2}) can be improved and replaced by the Whitney ball $B(x,\rho\delta(x))$ for $0<\rho<1$. Now to relate our main result with the known inequality (\ref{eq1-12}), we review the definition of homogeneous Sobolev spaces below, written in the language of the theory of functional and function representatives.
\begin{definition}\label{definition2}
Let $s\in\bbbr$ and $1<p<\infty$. The homogeneous Sobolev space $\Lps$ is a subspace of $\Sw_0'(\bbbr^n)$, and an element $f\in\Sw_0'(\bbbr^n)$ is in $\Lps$ if and only if the tempered distribution $\Lift_s f=\iFT_n(|\xi|^s\FT_n f)$ has a function representative in the sense of $\Sw_0'(\bbbr^n)$ and this function representative is in $L^p(\bbbr^n)$. And the homogeneous Sobolev norm of $f$ is given by
\begin{equation}\label{eq1-13}
\|f\|_{\Lps}:=\big(\int_{\bbbr^n}|\Lift_s f|^p dx\big)^{\frac{1}{p}}.
\end{equation}
\end{definition}
In the theory of functional and function representatives, the integral of a continuous linear functional $f\in\Sw_0'(\bbbr^n)$ is justified when $f$ has a function representative in the sense of $\Sw_0'(\bbbr^n)$. And when integrating a continuous linear functional $f\in\Sw_0'(\bbbr^n)$, we integrate its appropriately specified function representative in the sense of $\Sw_0'(\bbbr^n)$. As illustrated by the argument given at the end of the proof of Lemma \ref{lemma11}, integrals of distinct function representatives in the same sense may be different. As demonstrated by the proof of Lemma \ref{lemma13}, when it is difficult to find a function representative in the sense of $\Sw'(\bbbr^n)$, the specification of a function representative in the sense of $\Sw_0'(\bbbr^n)$ is necessary. Moreover, the same principle also applies to the supremum and derivative of a continuous linear functional $f\in\Sw_0'(\bbbr^n)$. Integrals, supremums, and derivatives included in this paper are all considered this way. For example, when fixing a multi-index $\gamma=(\gamma_1,\gamma_2,\cdots,\gamma_n)$, we can deduce from the definition of Schwartz functions in $\Sw(\bbbr^n)$ that we can find a positive finite constant $C_6$ and a number $N\in\bbbn_0$ so that the following estimate (\ref{eq1.98}) is true for all $g\in\Sw(\bbbr^n)$,
\begin{equation}\label{eq1.98}
\bigg|\int_{\bbbr^n}x^{\gamma}\cdot g(x)dx\bigg|\leq C_6\cdot\rho_N(g),
\end{equation}
where the values of $C_6$ and $N$ are determined by $\gamma$ and $n$, and $\rho_N(g)$ is defined in (\ref{eq1.89}). Thus the integral in (\ref{eq1.98}) defines a tempered distribution denoted by the symbol $x^{\gamma}$, and the polynomial function $x^{\gamma}$ is in $\functrep(x^{\gamma})\subseteq\functrep_0(x^{\gamma})$. Furthermore, equations (\ref{Sw0.condition.1}) and (\ref{eq1.94}) indicate that both of the following equations are true,
\begin{equation}\label{eq1.83}
<x^{\gamma},\phi>=\int_{\bbbr^n}x^{\gamma}\phi(x)dx=0=\int_{\bbbr^n}0\cdot\phi(x)dx
\end{equation}
for all $\phi\in\Sw_0(\bbbr^n)$, and
\begin{equation}\label{eq1.84}
<x^{\gamma},g>=\int_{\bbbr^n}x^{\gamma}\cdot g(x)dx
\end{equation}
for all $g\in\Sw(\bbbr^n)$. Equation (\ref{eq1.83}) tells that when considering $x^{\gamma}$ as an element of $\Sw_0'(\bbbr^n)$, the function representative of the symbol $x^{\gamma}$ in the sense of $\Sw_0'(\bbbr^n)$ can be specified to be the function
\begin{equation*}
y\in\bbbr^n\longmapsto x^{\gamma}(y)=0,
\end{equation*}
thus $\partial_1 x^{\gamma}=0$. Equation (\ref{eq1.84}) tells that when considering $x^{\gamma}$ as an element of $\Sw'(\bbbr^n)$, the function representative of the symbol $x^{\gamma}$ in the sense of $\Sw'(\bbbr^n)$ can be specified to be the function
\begin{equation*}
y\in\bbbr^n\longmapsto x^{\gamma}(y)=y^{\gamma},
\end{equation*}
thus $\partial_1 x^{\gamma}=\gamma_1\cdot y_1^{\gamma_1-1}y_2^{\gamma_2}\cdots y_n^{\gamma_n}$. For every multi-index $\beta$, since $\phi\in\Sw_{0}(\bbbr^n)$ implies $\partial^{\beta}\phi\in\Sw_{0}(\bbbr^n)$, we can define the derivative of a continuous linear functional $f\in\Sw_{0}'(\bbbr^n)$ by the following equation
\begin{equation}\label{eq1.128}
<\partial^{\beta}f,\phi>=(-1)^{|\beta|}<f,\partial^{\beta}\phi>.
\end{equation}
And (\ref{eq1.93}) indicates $\partial^{\beta}f\in\Sw_{0}'(\bbbr^n)$. Under the stronger condition that $f\in\Sw_{0}'(\bbbr^n)$ has a sufficiently smooth function representative $F(x)\in\functrep_{0}(f)$ and the function $F(x)$ satisfies the condition that for every multi-index $\alpha$, there exist positive finite constants $C_{\alpha}$ and $K_{\alpha}$ so that
\begin{equation}\label{eq1.129}
|\partial^{\alpha}F(x)|\leq C_{\alpha}\cdot(1+|x|)^{K_{\alpha}},
\end{equation}
then we can use (\ref{eq1.128}) and integration by parts to deduce the following equation for all $\phi\in\Sw_{0}(\bbbr^n)$,
\begin{equation}\label{eq1.130}
<\!\partial^{\beta}f,\phi\!>=\!(-1)^{|\beta|}\!\!\!\int_{\bbbr^n}\!\!\!\!\!F(x)\!\cdot\!
\partial^{\beta}\phi(x)dx\!=\!\!\!\int_{\bbbr^n}\!\!\!\!\!\partial^{\beta}F(x)\!\cdot\!\phi(x)dx,
\end{equation}
thus we obtain the function $\partial^{\beta}F(x)\in\functrep_{0}(\partial^{\beta}f)$. In a similar way, we can define the derivative of a continuous linear functional in $\Sw_{00}'(\bbbr^n)$. For every multi-index $\beta$, since $g\in\Sw_{00}(\bbbr^n)$ implies $\partial^{\beta}g\in\Sw_{00}(\bbbr^n)$, we define the derivative of a continuous linear functional $f\in\Sw_{00}'(\bbbr^n)$ by the following equation
\begin{equation}\label{eq1.131}
<\partial^{\beta}f,g>=(-1)^{|\beta|}<f,\partial^{\beta}g>.
\end{equation}
And (\ref{eq1.107}) indicates $\partial^{\beta}f\in\Sw_{00}'(\bbbr^n)$. Under the stronger condition that $f\in\Sw_{00}'(\bbbr^n)$ has a sufficiently smooth function representative $F(x)\in\functrep_{00}(f)$ and the function $F(x)$ satisfies condition (\ref{eq1.129}), then we deduce from (\ref{eq1.131}) and integration by parts that the following equation is true for all $g\in\Sw_{00}(\bbbr^n)$,
\begin{equation}\label{eq1.132}
<\!\partial^{\beta}f,g\!>=\!(-1)^{|\beta|}\!\!\!\int_{\bbbr^n}\!\!\!\!\!F(x)\!\cdot\!
\partial^{\beta}g(x)dx\!=\!\!\!\int_{\bbbr^n}\!\!\!\!\!\partial^{\beta}F(x)\!\cdot\!g(x)dx,
\end{equation}
thus we have the function $\partial^{\beta}F(x)\in\functrep_{00}(\partial^{\beta}f)$. Therefore condition (\ref{eq1.129}) implies the derivative of a function representative is also a function representative of the derivative of a functional. According to \cite[Definition 2.3.6]{14classical}, this conclusion holds true when $F(x)\in\functrep(f)$ for some $f\in\Sw'(\bbbr^n)$ is sufficiently smooth and satisfies condition (\ref{eq1.129}).
\begin{remark}\label{remark3}
We remark that if $f\in\Sw_0'(\bbbr^n)$, then by the Hahn-Banach theorem, i.e. \cite[3.3 Theorem]{Rudin.Funct.Anal}, we have that $f$ extends to a continuous linear functional on all of $\Sw(\bbbr^n)$ and we denote this extended tempered distribution by the same symbol $f$. Thus the distributional Fourier transform $\FT_n f\in\Sw'(\bbbr^n)\subseteq\Sw_{00}'(\bbbr^n)$ is well-defined for $f\in\Sw_0'(\bbbr^n)$. Because the function $\multi(\xi)=|\xi|^s$ for $s\in\bbbr$ satisfies condition (\ref{eq1.64}), the theory in Proposition \ref{proposition1} (i) tells us that the continuous linear functional $|\xi|^s\cdot\FT_n f$, initially defined on the subspace $\Sw_{00}(\bbbr^n)$ by the following equation
\begin{equation}\label{eq1-58}
<|\xi|^s\cdot\FT_n f,g>=<\FT_n f,|\xi|^s\cdot g>
\end{equation}
for $g\in\Sw_{00}(\bbbr^n)$, can extend to a tempered distribution defined on all of $\Sw(\bbbr^n)$ by the Hahn-Banach theorem, i.e. \cite[3.3 Theorem]{Rudin.Funct.Anal}. And hence $\Lift_s f\!=\!\iFT_n(|\xi|^s\!\cdot\!\FT_n f)\!\in\!\Sw'(\bbbr^n)$ is the distributional inverse Fourier transform of the extended tempered distribution $|\xi|^s\cdot\FT_n f$. Since $\phi\in\Sw_0(\bbbr^n)$ whenever $\iFT_n\phi\in\Sw_{00}(\bbbr^n)$, the action of $\Lift_s f$ on a Schwartz function $\phi\in\Sw_0(\bbbr^n)$ is given by
\begin{align}
&<\Lift_s f,\phi>=<\iFT_n(|\xi|^s\cdot\FT_n f),\phi>=<|\xi|^s\cdot\FT_n f,\iFT_n\phi>\nonumber\\
&=<\FT_n f,|\xi|^s\cdot\iFT_n\phi>=<f,\FT_n[|\xi|^s\cdot\iFT_n\phi(\xi)]>,\label{eq1-59}
\end{align}
where in the last equation of (\ref{eq1-59}), we use \cite[Definition 2.3.7.]{14classical} for the distributional Fourier transform of the extended tempered distribution $f$. Since $\FT_n[|\xi|^s\cdot\iFT_n\phi(\xi)](x)\in\Sw_0(\bbbr^n)$ by Proposition \ref{proposition1}, the action of the extended tempered distribution $f$ on $\FT_n[|\xi|^s\cdot\iFT_n\phi(\xi)]$ coincides with the action of the original continuous linear functional $f\in\Sw_0'(\bbbr^n)$ on the function $\FT_n[|\xi|^s\cdot\iFT_n\phi(\xi)]$.
\end{remark}
We can use Definition \ref{definition1} and the formula
	\begin{equation}\label{eq1-14}
		(\sum_{k\in\bbbz}a_k)^{\theta}\leq\sum_{k\in\bbbz}a_k^{\theta}	
	\end{equation}
for a sequence $\{a_k\}_{k\in\bbbz}$ of nonnegative real numbers and $0\leq\theta\leq 1$ to deduce that $\|f\|_{\Fspq}\leq\|f\|_{\dot{F}^s_{p,2}(\bbbr^n)}$ for any $s\in\bbbr$ when $2\leq q<\infty$. And we can combine this observation with \cite[Theorem 1 (i) in section 5.2.3]{1983functionspaces}, Lemma \ref{lemma12} (i), and the condition $f\in\Lps$ to obtain
\begin{equation}\label{eq1-15}
\|f\|_{\dot{F}^s_{p,2}(\bbbr^n)}\sim\|\Lift_s f\|_{\dot{F}^0_{p,2}(\bbbr^n)}\lesssim
\|\Lift_s  f\|_{L^p(\bbbr^n)}=\|f\|_{\Lps}
\end{equation}
for $s\in\bbbr$ and $1<p<\infty$. Furthermore, we deduce from Theorem \ref{theorem1} the following inequality
	\begin{equation}\label{eq1-16}
		\|\Diff_{s,q}f\|_{L^p(\bbbr^n)}\lesssim\|f\|_{\Lps},
	\end{equation}
where $1<p<q$, $2\leq q<\infty$ and $n(\frac{1}{p}-\frac{1}{q})<s<1$ and $f\in\Lps$ has a function representative. In \cite[Theorem 1']{fefferman1970}, C. Fefferman stated that the following inequality
	\begin{equation}\label{eq1-17}
		\|\Diff_{s,2}f\|_{L^{p,\infty}(\bbbr^n)}\lesssim\|f\|_{\Lps}
	\end{equation}
is true when $1<p<2$ and $0<s=n(\frac{1}{p}-\frac{1}{2})<1$, and $\|\cdot\|_{L^{p,\infty}(\bbbr^n)}$ denotes the weak $L^p$ quasinorm. We can consider that inequality (\ref{eq1-17}) is the case of inequality (\ref{eq1-16}) under the critical condition $s=n(\frac{1}{p}-\frac{1}{2})$ when $q=2$. In this paper, we review the definition of Sobolev space $W^{1,p}(\bbbr^n)$, and then state and prove the generalization of inequality (\ref{eq1-17}) below.
\begin{definition}\label{definition3}
For $1\leq p\leq\infty$, we define the Sobolev space $W^{1,p}(\bbbr^n)$ is the space of functions $f$ defined on $\bbbr^n$, all of whose first-order derivatives exist in the classical sense, so that the following quasinorm is finite,
\begin{equation}\label{eq1.85}
\|f\|_{W^{1,p}(\bbbr^n)}:=\|f\|_{L^{p}(\bbbr^n)}+\|\nabla_n f\|_{L^{p}(\bbbr^n)}<\infty,
\end{equation}
where $\nabla_n f(x)$ is the $n$-dimensional gradient vector of the function $f(x)$ with respect to $x\in\bbbr^n$.
\end{definition}
\begin{theorem}\label{theorem2}
If $n\geq2$, $1<p<q$, $2\leq q<\infty$, and $0<s=n(\frac{1}{p}-\frac{1}{q})<1$, and if $f\in\Lps$ has a real-valued function representative $f(x)\in\functrep_0(f)\cap W^{1,p_0}(\bbbr^n)$ for some $1\leq p_0\leq\infty$, then we have
\begin{equation}\label{eq1-18}
\|\Diff_{s,q}f\|_{L^{p,\infty}(\bbbr^n)}\lesssim\|f\|_{\dot{L}^p_s(\bbbr^n)}.
\end{equation}
\end{theorem}
The proof of Theorem \ref{theorem2} is given in section \ref{proof.of.theorem2}. In Theorem \ref{theorem2}, the condition that $f\in\Lps$ has a real-valued function representative $f(x)\in\functrep_0(f)\cap W^{1,p_0}(\bbbr^n)$ for some $1\leq p_0\leq\infty$, is necessary to guarantee the subharmonicity of $$|\nabla_{n+1}\Pint(\varGamma(\frac{-s}{2})\Lift_s f;x,t)|^q$$ when $q\neq 2$, where $\Pint(\varGamma(\frac{-s}{2})\Lift_s f;x,t)$ is the Poisson integral of the tempered distribution $\varGamma(\frac{-s}{2})\Lift_s f$ and $\varGamma(\frac{-s}{2})$ is the evaluation of the meromorphic extension of the gamma function at $\frac{-s}{2}$. The definitions of the gamma function and the Poisson integral are given in section \ref{Preliminaries}. For more details about this requirement, see Lemma \ref{lemma9}, Remark \ref{remark1}, Lemma \ref{lemma10}, Lemma \ref{lemma11}, and inequality (\ref{eq6.22}) in the proof of Theorem \ref{theorem2} where we use this subharmonicity.\\
	
In order to prove Theorem \ref{theorem2}, we investigate the boundedness property of the generalized Littlewood-Paley $\LPg_{s,q}$-function given by
	\begin{equation}\label{eq1-19}
		\LPg_{s,q}(f)(x)=\big(\int_0^{\infty}t^{q-sq-1}\big|\nabla_{n+1}\Pint(f;x,t)\big|^q dt\big)^{\frac{1}{q}},
	\end{equation}
and we use Lemma \ref{lemma16} to prove the following theorem. The proof of Theorem \ref{theorem3} is given in section \ref{proof.of.theorem3}.
\begin{theorem}\label{theorem3}
Let $0<p,q<\infty$, and $-1<s<1$. If $f\in\Fspq$ has a function representative $f(x)\in\functrep_0(f)\cap L^{p_0}(\bbbr^n)$ for some $p_0\in[1,\infty]$, then we have
\begin{equation}\label{eq1-20}
\|\LPg_{s,q}(f)\|_{L^p(\bbbr^n)}\lesssim\|f\|_{\Fspq},
\end{equation}
where $\LPg_{s,q}(f)$ is the generalized Littlewood-Paley function defined in (\ref{eq1-19}). Recall the following expression
\begin{equation}\label{eq1.106}
\big|\nabla_{n+1}\Pint(f;x,t)\big|^q=\bigg(\sum_{k=1}^{n+1}
|\partial_k\Pint(f;x,t)|^2\bigg)^\frac{q}{2}.
\end{equation}
When $1\leq p_0<\infty$, each term $\partial_k\Pint(f;x,t)=f*\partial_k[P_t(\cdot)](x)$ in (\ref{eq1.106}) is a function and $f$ appearing in it is the function representative $f(x)\in\functrep_0(f)\cap L^{p_0}(\bbbr^n)$. When $p_0=\infty$, each term $\partial_k\Pint(f;x,t)$ in (\ref{eq1.106}) is the tempered distribution defined by the integral
\begin{equation}\label{eq1.121}
\int_{\bbbr^n}\partial_k\Pint(f;x,t)\!\cdot\!g(x)dx\quad\text{for $g\in\Sw(\bbbr^n)$}.
\end{equation}
\end{theorem}
With a more general nonzero integrable kernel that satisfies some additional conditions, replacing derivatives of the Poisson kernel in (\ref{eq1-19}), D. Fan and F. Zhao have proven the inequality (\ref{eq1-20}) for $1<p,q<\infty$ and $s=0$ in their recent development \cite[Theorem 2.1]{FZ2021}. When the kernel $\Phi$ is a compactly supported function in $L^q(\bbbr^n)$ for some $q>1$ and satisfies $\int_{\bbbr^n}\Phi(x)dx=0$, H. Al-Qassem, L. Cheng and Y. Pan have proven in \cite[Theorem 1.1]{AQCP2018} the estimate
	\begin{equation}\label{weak.eq13}
		\|\mathcal{S}_{\Phi}^{(\lambda)}(f)\|_{L^p(\bbbr^n)}\lesssim\|\Phi\|_{L^q(\bbbr^n)}\cdot
		\|f\|_{\dot{F}^0_{p,\lambda}(\bbbr^n)},
	\end{equation}
	where $\mathcal{S}_{\Phi}^{(\lambda)}(f)(x)$ is their generalized Littlewood-Paley function given by
	\begin{equation}\label{weak.eq14}
		\mathcal{S}_{\Phi}^{(\lambda)}(f)(x)=\big(\int_0^{\infty}
		|\Phi_t*f(x)|^{\lambda}\frac{dt}{t}\big)^{\frac{1}{\lambda}}\quad\text{for}\quad\lambda>1,
	\end{equation}
	and $1/\big(1+\frac{1}{\lambda}-\max\{\frac{1}{\lambda},\frac{1}{q}\}\big)<p<\infty$. In the same paper \cite[Theorem 1.2 and Theorem 1.3]{AQCP2018}, the authors also have shown if additionally the kernel $\Phi$ has a radial majorant in $L^q(\bbbr^n)$ for $1<q<\lambda$, then the operator $\mathcal{S}_{\Phi}^{(\lambda)}(f)$ given in (\ref{weak.eq14}) is a bounded operator on $L^p(\bbbr^n)$ for $(\lambda nq)/(\lambda nq-\lambda+q)<p<\infty$, and in the case when $q=1$, the lower bound for the range of $p$ becomes $(\lambda n)/(\lambda n-1)$. Imposing more conditions on the kernel in \cite[Theorem 1.4]{AQCP2018}, the inequality (\ref{weak.eq13}) is proven for $1<\lambda,p<\infty$. In \cite{AQCP2021}, H. Al-Qassem, L. Cheng and Y. Pan generalize their results in \cite{AQCP2018} to the case of the product generalized Littlewood-Paley function defined on $\bbbr^n\times\bbbr^m$ and the homogeneous Triebel-Lizorkin space $\dot{F}^{\vec{s}}_{p,q}(\bbbr^n\times\bbbr^m)$. Indeed, early in the year 2018, F. Liu and Q. Xue published their results about the boundedness properties of the multiple Littlewood–Paley operator on the product space $L^2(\bbbr^m\times\bbbr^n)$ in \cite{LX2018}. The boundedness of the multiple Littlewood–Paley function with a rough kernel on $L^2(\bbbr^m\times\bbbr^n)$ was also studied in \cite{LW2014} by F. Liu and H. Wu. The $L^p(\bbbr^n\times\bbbr^m)$ $(1<p<\infty)$ boundedness property of the multiple Littlewood–Paley operator, where the kernel of the multiple Littlewood–Paley operator is the product of a function defined on $\bbbr^n$ and a function defined on $\bbbr^m$, was studied in \cite[(2.1)]{FS1982} by R. Fefferman and E. M. Stein. And the $L^p(\bbbr^n)$-boundedness property of the operator $\mathcal{S}_{\Phi}^{(2)}(f)$ defined by (\ref{weak.eq14}) has been studied in \cite{Sato2008} by S. Sato. The proof of \cite[Lemma 4]{Sato2008} has been used extensively by H. Al-Qassem, L. Cheng and Y. Pan in their paper \cite{AQCP2018} to achieve the boundedness from $\dot{F}^0_{p,\lambda}(\bbbr^n)$ into $L^p(\bbbr^n)$. In \cite{FS2002}, D. Fan and S. Sato defined their generalized Littlewood-Paley function on $\bbbr^{n+1}$ for a Schwartz function $f$ defined on $\bbbr^{n+1}$, and they have
	\begin{equation}\label{eq1-33}
		S_{\psi,\gamma}(f)(x,z)=\bigg(\int_0^{\infty}|\psi_t\sharp f(x,z)|^2\frac{dt}{t}\bigg)^{\frac{1}{2}},
	\end{equation}
	where $x\in\bbbr^n,z\in\bbbr$ and
	\begin{equation}\label{eq1-34}
		\psi_t\sharp f(x,z)=\int_{\bbbr^n}b(y)\psi_t(y)f(x-y,z-\gamma(|y|))dy,
	\end{equation}
	and $b(y)$ is a measurable function on $\bbbr^n$ and $\gamma(\cdot)$ is a continuous function defined on the interval $[0,\infty)$, and they have proven that the operator defined in (\ref{eq1-33}) is bounded on $L^p(\bbbr^{n+1})$ for certain ranges of $p$. In \cite{DFP2000}, Y. Ding, D. Fan, and Y. Pan studied the $L^p(\bbbr^n)$-boundedness property of the Littlewood-Paley operator $\mathcal{S}_{\Phi}^{(2)}(f)$ given in (\ref{weak.eq14}) when the function $\Phi$ satisfies the rough kernel condition which says there exist $\alpha>0$ and $\beta>0$ such that 
	\begin{equation}\label{eq1-35}
		|\FT_n\Phi(\xi)|\lesssim\min\{|\xi|^{\alpha},|\xi|^{-\beta}\}\qquad\text{for all $0\neq\xi\in\bbbr^n$}.
	\end{equation}
	Another well-known result about the boundedness property of operator $\mathcal{S}_{\Phi}^{(2)}(f)$ given in (\ref{weak.eq14}) is by A. Benedek, A. P. Calder\'{o}n and R. Panzone in \cite{BCP1962}, which says the operator $\mathcal{S}_{\Phi}^{(2)}(f)$ is bounded on $L^p(\bbbr^n)$ for all $p\in(1,\infty)$ if the kernel $\Phi$ satisfies the following conditions that $\FT_n\Phi(0)=0$ and
	\begin{align*}
		&|\Phi(x)|\lesssim(1+|x|)^{-n-\varepsilon}\qquad&\text{for some $\varepsilon>0$},\\
		&\int_{\bbbr^n}|\Phi(x-y)-\Phi(x)|dx\lesssim|y|^{\varepsilon}\qquad&\text{for some $\varepsilon>0$}.
	\end{align*}
	The interested reader can refer to \cite{HLYY2019,TW2014,BM2022,MY2008,JYYZ2022,Cheng2007,Lerner2019,Bourgain1989,Javier2013,HouWu2019,GM2022,GM2023} for more studies on the classical Littlewood-Paley square function and its variations.\\
	
	Comparing Theorem \ref{theorem3} with previous results about the boundedness properties of generalized Littlewood-Paley functions with different kernels, we find that Theorem \ref{theorem3} is outstanding due to the following reasons. First, the generalized Littlewood-Paley $\LPg_{s,q}$-function defined in (\ref{eq1-19}) uses derivatives of the Poisson kernel $P_t(x)$ as its kernel function and $|\nabla_{n+1}\Pint(f;x,t)|^q$ inherits the subharmonicity, as shown in Lemma \ref{lemma10}. The subharmonicity is an important property in harmonic analysis, and plays a crucial role in the proof of our main Theorem \ref{theorem2}. Second, the derivatives of the Poisson kernel satisfy the H\"{o}rmander type condition as given in Lemma \ref{lemma4}. This good property is used repeatedly in the proof of Lemma \ref{lemma7} and hence in the proofs of Theorem \ref{theorem4} and Theorem \ref{theorem5}. In many cases, the H\"{o}rmander type condition is a desirable condition. Third, the derivatives of the Poisson kernel frequently appear in the process of the pointwise convergence to a function $f(x)$ by its Poisson integral $\Pint(f;x,t)=f*P_t(x)$ as $t\rightarrow0$. The almost everywhere pointwise convergence property of the Poisson integral is a well-established result given in \cite[chapter I, Theorem 1.25]{SteinWeissFourier}. As shown in (\ref{eq1.60}) and (\ref{eq1.61}), the derivatives of the Poisson kernel $P_t(x)$ have their integrals on $\bbbr^n$ equal to zero, while at the same time the Poisson kernel has its integral on $\bbbr^n$ equal to $1$ and hence the family of functions $\{P_t(\cdot)\}_{t>0}$ constitutes an approximate identity as $t\rightarrow0$. Fourth, with all the above advantages inherited from the Poisson kernel, Theorem \ref{theorem3} achieves advantageous ranges for the involved parameters $p$, $q$, $s$. In particular, the applicability of Theorem \ref{theorem3} to negative values of $s$ has a significant place in (\ref{eq4.25}) in the proof of Theorem \ref{theorem4}. Because of these reasons, we name the generalized Littlewood-Paley $\LPg_{s,q}$-function defined in (\ref{eq1-19}) is the Littlewood-Paley-Poisson function $\LPg_{s,q}(f)(x)$.\\
	
An immediate consequence of Theorem \ref{theorem3} is the following Corollary \ref{corollary1}, whose proof can be found in section \ref{proofs.of.corollaries.1&2}.
\begin{corollary}\label{corollary1}
Let $0<p<\infty$ and $2\leq q<\infty$. If $f\in\dot{F}^0_{p,2}(\bbbr^n)$ has a function representative $f(x)\in\functrep_0(f)\cap L^{p_0}(\bbbr^n)$ for some $1\leq p_0\leq\infty$, then we have
\begin{equation}\label{eq1-21}
\|\LPg_{0,q}(f)\|_{L^p(\bbbr^n)}\lesssim\|f\|_{\dot{F}^{0}_{p,2}(\bbbr^n)},
\end{equation}
where $\LPg_{0,q}(f)$ is given by
\begin{equation}\label{eq1-22}
\LPg_{0,q}(f)(x)=\big(\int_0^{\infty}t^{q-1}\big|\nabla_{n+1}\Pint(f;x,t)\big|^q dt\big)^{\frac{1}{q}}.	
\end{equation}
\end{corollary}
When $q=2$ and $s=0$, we recover from (\ref{eq1-19}) the classical Littlewood-Paley $g$-function introduced in \cite[section 1 of chapter IV]{Stein1971} by E. M. Stein, and Corollary \ref{corollary1} and its consequential estimates (\ref{eq1-23}), (\ref{eq1-24}) turn out to be partial generalizations of \cite[Theorem 1 in section 1 of chapter IV]{Stein1971}. The special case of (\ref{eq1-25}) when $p=q=2$ can be obtained by using routine calculations and Plancherel's identity, and it has been used to prove the weak type $(p,p)$ boundedness of the classical Littlewood-Paley $g_{\lambda}^*$-function under the critical condition $\lambda=\frac{2}{p}$ in \cite[Theorem 1]{fefferman1970}. To prove the more challenging case of (\ref{eq1-25}) when $p=q\neq 2$ and the more general Theorem \ref{theorem3}, where the absence of Plancherel's identity poses a major difficulty, we use the illuminating method introduced in \cite[section 2.6.4]{1992functionspaces} by H. Triebel. Indeed, H. Triebel provided a harmonic characterization of the inhomogeneous Triebel-Lizorkin space $F^s_{p,q}(\bbbr^n)$ in \cite[section 2.6.4]{1992functionspaces} with a rough condition $$k>\frac{n}{\min\{p,q\}}+\max\{s,n(\frac{1}{p}-1),0\},$$ where $k$ is the order of derivatives of the Poisson integral $\Pint(f;x,t)$.\\
	
We also study the weak type $(p,p)$ boundedness of the following operators. Let $0<s,q,\lambda<\infty$ and $f$ be an appropriate function, then we define the generalized Littlewood-Paley $\LPG_{\lambda,q}$-function is given by
	\begin{equation}\label{eq1-26}
		\LPG_{\lambda,q}(f)(x)\!:=\!\big(\!\!\int_0^{\infty}\!\!\!\!\int_{\bbbr^n}\!\frac{t^{\lambda n+q-n-1}}{(t+|x-y|)^{\lambda n}}\!\cdot\!|\nabla_{n+1}\Pint(f;y,t)|^q dydt\big)^{\frac{1}{q}},
	\end{equation}
	and the $\Remain_{s,q}$-function is given by
	\begin{equation}\label{eq1-27}
		\Remain_{s,q}(f)(x)\!:=\!\big(\!\!\int_{\bbbr^n}\!\!\!|x-y|^{-n-sq}\!\cdot\![\!\int_0^{2|x-y|}\!\!\!\!\!\!\!
		|\partial_{n+1}\Pint(f;y,t)|t^s dt]^q dy\big)^{\frac{1}{q}}.
	\end{equation}
We refer to section \ref{Preliminaries} for the definitions of derivatives of the Poisson integral $\Pint(f;y,t)$. As another consequence of Theorem \ref{theorem3}, the boundedness properties of the above operators are given below. The proof of Corollary \ref{corollary2} is also given in section \ref{proofs.of.corollaries.1&2}.
\begin{corollary}\label{corollary2}
For $2\leq q<\infty$, $1<\lambda<\infty$ and $0<s<1$, the operators $\LPG_{\lambda,q}$ and $\Remain_{s,q}$ defined in (\ref{eq1-26}) and (\ref{eq1-27}) are bounded from $L^q(\bbbr^n)$ to $L^q(\bbbr^n)$.
\end{corollary}
We now propose the weak type $(p,p)$ boundedness theorems for the generalized Littlewood-Paley $\LPG_{\lambda,q}$-function and the $\Remain_{s,q}$-function.
\begin{theorem}\label{theorem4}
For $1<p<q$, $2\leq q<\infty$, $\lambda=\frac{q}{p}>1$, and $0<n(\frac{1}{p}-\frac{1}{q})<1$, the operator $\LPG_{\lambda,q}$ defined in (\ref{eq1-26}) is of weak type $(p,p)$ for functions in $L^p(\bbbr^n)$ and satisfies the inequality
\begin{equation}\label{eq1-31}
\Lebes^n(\{x\in\bbbr^n:\LPG_{\lambda,q}(f)(x)>\alpha\})\lesssim\alpha^{-p}\|f\|_{L^p(\bbbr^n)}^p,
\end{equation}
where the constant only depends on some fixed parameters, $\alpha$ is a positive number and $f$ is a function in $L^p(\bbbr^n)$.
\end{theorem}
The proof of Theorem \ref{theorem4} can be found in section \ref{proof.of.theorem4}. We notice that when $q=2$ in (\ref{eq1-26}), the function $\LPG_{\lambda,2}(f)(x)$ is the classical Littlewood-Paley $g_{\lambda}^*$-function, which was introduced and studied in \cite{Stein1961-2,Stein1971,fefferman1970}, and Theorem \ref{theorem4} is a generalization of \cite[Theorem 1]{fefferman1970}. Comparing Theorem \ref{theorem4} with \cite[Theorem 1]{fefferman1970}, we find that the extra condition $0<n(\frac{1}{p}-\frac{1}{q})<1$ in Theorem \ref{theorem4} is due to the requirement of Theorem \ref{theorem3}, which has been used to overcome the absence of Plancherel's identity in the case $q\neq 2$. The weak type $(1,1)$ boundedness of the classical Littlewood-Paley $g_{\lambda}^*$-function is given in \cite[Lemma 2]{Stein1961-2}. In \cite{SXY2014}, S. Shi, Q. Xue, and K. Yabuta study the multilinear generalization of the classical Littlewood-Paley $g_{\lambda}^*$-function given by
	\begin{equation}\label{eq1-36}
		g_{\lambda}^*(\vec{f})(x)=\big(\!\!\int_0^{\infty}\!\!\!\!\int_{\bbbr^n}\!\!\big(\frac{t}{t+|x-z|}\big)^{\lambda n} \cdot|G_t(\vec{f})(z)|^2\frac{dzdt}{t^{n+1}}\big)^{\frac{1}{2}},
	\end{equation}
	where $x\in\bbbr^n$, $\vec{f}=(f_1,f_2,\cdots,f_m)\in\Sw(\bbbr^n)\times\Sw(\bbbr^n)\times\cdots\times\Sw(\bbbr^n)$, and
	\begin{equation*}
		G_t(\vec{f})(z)=\frac{1}{t^{mn}}\int_{(\bbbr^n)^m}K\bigg(\frac{z-y_1}{t},\cdots,\frac{z-y_m}{t}\bigg)
		\prod_{i=1}^m f_i(y_i)dy_i,
	\end{equation*}
	and the function $K$ defined on $(\bbbr^n)^m$ is a multilinear standard kernel. These authors prove in \cite[Theorem 1.1 and Theorem 1.2]{SXY2014} that the multilinear Littlewood-Paley $g_{\lambda}^*$-function given in (\ref{eq1-36}) is bounded from $L^1\times\cdots\times L^1$ to $L^{\frac{1}{m},\infty}$, and also possesses the strong type weighted estimate from $L^{p_1}(\omega_1)\times\cdots\times L^{p_m}(\omega_m)$ to $L^p(\gamma_{\vec{\omega}})$ if each $p_i>1$ and the weak type weighted estimate from $L^{p_1}(\omega_1)\times\cdots\times L^{p_m}(\omega_m)$ to $L^{p,\infty}(\gamma_{\vec{\omega}})$ if at least one $p_i=1$, where $$\gamma_{\vec{\omega}}=\prod_{i=1}^m\omega_i^{\frac{p}{p_i}}$$ and each $\omega_i$ is a nonnegative weight function on $\bbbr^n$, and $p_1,\cdots,p_m\in[1,\infty)$ and $p$ satisfy the condition
	\begin{equation}\label{eq1-37}
		\frac{1}{p}=\frac{1}{p_1}+\cdots+\frac{1}{p_m}.
	\end{equation}
	In \cite[Theorem 1.3]{HX2019}, S. He and Q. Xue prove that the multilinear Littlewood-Paley $g_{\lambda}^*$-function given in (\ref{eq1-36}) is bounded from $H^{p_1}(\bbbr^n)\times\cdots\times H^{p_m}(\bbbr^n)$ to $L^p(\bbbr^n)$, where $p,p_1,\cdots,p_m$ satisfy the condition (\ref{eq1-37}) and
	\begin{equation*}
		\frac{nm}{nm+\gamma}<p_i\leq1\qquad\text{for $i=1,\cdots,m$},
	\end{equation*}
	and $\gamma$ is a positive number determined by the multilinear standard kernel $K$, and $H^{p_i}(\bbbr^n)$ $(i=1,\cdots,m)$ are Hardy spaces. The interested reader can refer to \cite{HLYY2019,MNY2010,LuTao2017,TW2014,MY2008,JYYZ2022,CLYY2017,CWYZ2020,MW1974,CX2021,HYY2016,HYY2018,CHIFRRSY,CLX2017} for more studies on the classical Littlewood-Paley $g_{\lambda}^*$-function.\\
	
In order to complete the proof of Theorem \ref{theorem2}, we will need the following Theorem \ref{theorem5}. From Step 6 of the proof of Theorem \ref{theorem2} given in section \ref{proof.of.theorem2}, we see that the condition $s=n(\frac{1}{p}-\frac{1}{q})$ in Theorem \ref{theorem2} is due to the weak type $(p,p)$ boundedness of the $\Remain_{s,q}$-function. The proof of Theorem \ref{theorem5} is in section \ref{proof.of.theorem5}.
\begin{theorem}\label{theorem5}
For $1<p<q$, $2\leq q<\infty$, and $0<s=n(\frac{1}{p}-\frac{1}{q})<1$, the operator $\Remain_{s,q}$ defined in (\ref{eq1-27}) is of weak type $(p,p)$ for functions in $L^p(\bbbr^n)$ and satisfies the inequality
\begin{equation}\label{eq1-32}
\Lebes^n(\{x\in\bbbr^n:\Remain_{s,q}(f)(x)>\alpha\})\lesssim\alpha^{-p}\|f\|_{L^p(\bbbr^n)}^p,
\end{equation}
where the constant only depends on some fixed parameters, $\alpha$ is a positive number and $f$ is a function in $L^p(\bbbr^n)$.
\end{theorem}
We comment that conclusions obtained so far can be used to deduce inequality (\ref{eq1-16}) under certain conditions, and both of the inequalities (\ref{eq1-12}) and (\ref{eq1-16}) are untrue when $s\leq\max\{0,n(\frac{1}{p}-\frac{1}{q})\}$. The proofs of Corollary \ref{corollary3} and Corollary \ref{corollary4} can be found in section \ref{proofs.of.corollaries.3&4}.
\begin{corollary}\label{corollary3}
If $n\geq2$, $1<p<q$, $2\leq q<\infty$, $n(\frac{1}{p}-\frac{1}{q})<s<1$, and $f\in\Sw_0(\bbbr^n)$ is a real-valued function, then we have
\begin{equation}\label{eq1.87}
\|\Diff_{s,q}f\|_{L^p(\bbbr^n)}\lesssim\|f\|_{\Lps}.
\end{equation}
\end{corollary}
\begin{corollary}\label{corollary4}
Assume that $f$ is an arbitrary nonzero function in $\Sw_0(\bbbr^n)$. Then we have
\begin{equation}\label{eq1.88}
\|\Diff_{s,q}f\|_{L^p(\bbbr^n)}=\infty
\end{equation}
when $0<p<q<\infty$ and $-\infty<s\leq n(\frac{1}{p}-\frac{1}{q})$, and both of the inequalities (\ref{eq1-12}) and (\ref{eq1-16}) cannot be true. We also have
\begin{equation}\label{eq1.133}
\|\Diff_{s,q}f\|_{L^{p,\infty}(\bbbr^n)}=\infty
\end{equation}
when $0<p,q<\infty$ and $-\infty<s\leq0$, and all
of the inequalities (\ref{eq1-12}), (\ref{eq1-16}), and (\ref{eq1-18}) cannot be true.
\end{corollary}
\paragraph*{\textbf{Notation}} 
In this paper, unless otherwise stated, the dimension number $n$ of the ambient space $\bbbr^n$ is greater than or equal to $1$. And we will flexibly use the following notations for inequalities. Let $X,Y$ be two expressions, then ``$X\lesssim Y$'' means $X$ is dominated by a constant multiple of $Y$ and the constant is determined by some fixed parameters. In particular, ``$X\lesssim 1$'' means $X$ is bounded from above by a positive finite constant whose value is determined by some fixed parameters. If $X\lesssim Y$ and $Y\lesssim X$, then we consider $X$ and $Y$ are equivalent and write $X\sim Y$. Furthermore, when we want to emphasize the existence of a constant and the dependence of this constant on some fixed parameters, we usually write it like ``$X\leq C\cdot Y$ where $C=C(n,p)$'' and it means $X$ is dominated by the product of the constant $C$ and the expression $Y$, and the value of the constant $C$ only depends on the parameters $n$ and $p$. Moreover, unless otherwise stated, we will number constants in different sections independently. For example, we may have constants $C_1=C_1(n,p)$, $C_2=C_2(n,p,q)$, $C_3=C_3(q,s)$ in section $1$ and we may also have constants $C_1=C_1(n,p)$, $C_2=C_2(p,q,s)$, $C_3=C_3(s,\lambda)$ in section $2$. However, the constant $C_1$ of section $1$ may be different from the constant $C_1$ of section $2$ even though they depend on the same parameters $n$ and $p$. And the value of $C_2$ in section $1$ depends only on parameters $n,p,q$ while the value of $C_2$ in section $2$ depends only on parameters $p,q,s$. In this way, a change in the numbering of constants in one section will not affect the numbering of constants in other sections of the paper. Other useful notations in this paper include the following. If $E$ is a subset of $\bbbr^n$, then $\Ccinfty(E)$ represents the set of all the smooth functions whose support sets are compact sets contained in $E$. For two nonempty sets $A$ and $B$ in $\bbbr^n$, we use $\Lebes^n(A)$ to denote the $n$-dimensional Lebesgue measure of $A$ and $\dist(A,B)=\inf_{x\in A,y\in B}|x-y|$. We use $B^n(x,r)$ to denote the ball in $\bbbr^n$ centered at $x\in\bbbr^n$ with radius $r>0$, then the unit ball in $\bbbr^n$ is given as $B^n(0,1)$ and its volume is $\nu_n=\Lebes^n(B^n(0,1))$ and also $\Lebes^n(B^n(x,r))=\nu_n r^n$. The unit sphere in $\bbbr^n$ is written as $\unitsph$ and its surface area is denoted by $\omega_{n-1}$. If a ball in $\bbbr^{n+1}$ is centered at $(x,t)\in\bbbr^{n+1}$ and has radius $r$, we denote such a ball by $B^{n+1}((x,t),r)$, then the unit ball in $\bbbr^{n+1}$ is given as $B^{n+1}((0,0),1)$ and its volume is $\nu_{n+1}=\Lebes^{n+1}(B^{n+1}((0,0),1))$ and also $\Lebes^{n+1}(B^{n+1}((x,t),r))=\nu_{n+1}r^{n+1}$. A cone in $\bbbr^{n+1}_+$, whose vertex is the point $(x,0)\in\bbbr^{n+1}$ for some $x\in\bbbr^n$ and whose aperture is the positive number $a>0$, is denoted by $\gamma_a(x):=\{(z,t)\in\bbbr^{n+1}_+:|z-x|<at\}$. If the set $I$ represents a cube in $\bbbr^n$ with sides parallel to the coordinate axes, then we denote the center of cube $I$ by $c(I)$ and the side length of cube $I$ by $l(I)$. In addition, given a sequence $\{f_k(x)\}_{k\in\bbbz}$ of functions defined on $\bbbr^n$ and $0<p,q\leq\infty$, we use the notation $\|\{f_k\}_{k\in\bbbz}\|_{L^p(l^q)}$ to denote the $\|\cdot\|_{L^p(\bbbr^n)}$-quasinorm of the $\|\cdot\|_{l^q}$-quasinorm of the sequence $\{f_k(x)\}_{k\in\bbbz}$. And ``$\mvint_{}$'' denotes the mean value integral.\\

\section{Preliminaries}\label{Preliminaries}
In order to write this paper in a more self-included fashion, we also cite useful results directly from the literature and the proofs of cited results can be found from their respective source. Recall the definition of the gamma function below
\begin{equation}\label{eq1-1}
\varGamma(z)\!:=\!\!\!\int_0^{\infty}\!\!\!\!\!t^{z-1}e^{-t}dt\quad\text{for a complex number $z$ with $\operatorname{Re}z\!>\!0$,}
\end{equation}
and the gamma function has a meromorphic extension for all $z\in\bbbc$ (cf. \cite[Appendix A.5]{14classical}). We denote the Poisson kernel by $P(x)=C_0(1+|x|^2)^{-\frac{n+1}{2}}$ for $x\in\bbbr^n$, and we also designate here that throughout the whole paper, we denote
	\begin{equation}\label{eq1-2}
		C_0:=\frac{\varGamma(\frac{n+1}{2})}{\pi^{\frac{n+1}{2}}}\quad\text{(cf. \cite[section 2.1.2]{14classical}).}
	\end{equation}
Furthermore, we use the notation 
\begin{equation}\label{eq2.379}
P_t(x)\!=\!\frac{1}{t^n}P(\frac{x}{t})\!=\!\frac{C_0 t}{(t^2\!+\!|x|^2)^{\frac{n+1}{2}}}
\text{ for $x\!\in\!\bbbr^n$ and $t\!>\!0$,}
\end{equation}
then it is a well-known result that 
	\begin{equation}\label{eq1-3}
		\FT_n P_t(\xi)=\FT_n P(t\xi)=e^{-2\pi t|\xi|}.
	\end{equation}
	For $1\leq k\leq n$, we have
	\begin{equation}\label{eq1-3-1}
		\partial_k[P_t(x)]=\frac{\partial}{\partial x_k}[P_t(x)]=\frac{-(n+1)C_0tx_k}{(t^2+|x|^2)^{\frac{n+3}{2}}},
	\end{equation}
	and we can use integration by parts with respect to $x_k$ and (\ref{eq1-3}) to find that the Fourier transform $\FT_n[\partial_k[P_t(x)]](\xi)$ is given by the formula
	\begin{equation}\label{eq2-81}
		\FT_n[\partial_k[P_t(x)]](\xi)=2\pi i\xi_k\cdot e^{-2\pi t|\xi|}.
	\end{equation}
	Furthermore, we have
	\begin{equation}\label{eq1.60}
		\int_{\bbbr^n}\partial_k[P_t(x)]dx=\FT_n[\partial_k[P_t(x)]](0)=0.
	\end{equation}
	When $k=n+1$, we have
	\begin{equation}\label{eq1-3-2}
		\partial_{n+1}[P_t(x)]=\frac{\partial}{\partial t}[P_t(x)]=
		\frac{C_0}{(t^2+|x|^2)^{\frac{n+1}{2}}}-\frac{C_0(n+1)t^2}{(t^2+|x|^2)^{\frac{n+3}{2}}},
	\end{equation}
	and we can use the dominated convergence theorem to exchange the differentiation with respect to $t\in(0,\infty)$ and the Fourier transform with respect to $x\in\bbbr^n$ and then we can obtain
\begin{equation}\label{eq2-82}
\FT_n[\partial_{n+1}[P_t(x)]](\xi)=\partial_{n+1}[\FT_n P_t](\xi)=-2\pi|\xi|\cdot e^{-2\pi t|\xi|}.
\end{equation}
Moreover, we have
\begin{equation}\label{eq1.61}
\int_{\bbbr^n}\partial_{n+1}[P_t(x)]dx=\FT_n[\partial_{n+1}[P_t(x)]](0)=0.
\end{equation}
Given an appropriate function $f$ defined on $\bbbr^n$, the Poisson integral of this function is a function defined on $\bbbr^{n+1}_+:=\{(x,t):x\in\bbbr^n,t>0\}$ and is denoted by 
	\begin{equation}\label{eq1-4}
		\Pint(f;x,t):=f*P_t(x)=\int_{\bbbr^n}f(y)\cdot\frac{C_0 t}{(t^2+|x-y|^2)^{\frac{n+1}{2}}}dy.
	\end{equation}
When $f\in L^p(\bbbr^n)$ for $1\leq p\leq\infty$, as demonstrated in the proof of Lemma \ref{lemma10}, we can use the dominated convergence theorem and some direct calculations to obtain
\begin{align}
&\partial_k\Pint(f;x,t)\!\!:=\!\!\frac{\partial}{\partial x_k}(\!\Pint(f;x,t))
\!\!=\!\!f*\partial_k[P_t(\cdot)](x)\nonumber\\
&=\int_{\bbbr^n}\!\!\!\!f(y)\!\cdot\!
\frac{-(n+1)C_0 t(x_k-y_k)}{(t^2+|x-y|^2)^{\frac{n+3}{2}}}dy,\label{eq1-5}
\end{align}
where $x_k$ and $y_k$ are the $k$-th coordinates of $x$ and $y$, respectively, and $1\leq k\leq n$. We also have
\begin{align}
&\partial_{n+1}\Pint(f;x,t):=\frac{\partial}{\partial t}(\Pint(f;x,t))
=f*\partial_{n+1}[P_t(\cdot)](x)\nonumber\\
&=\int_{\bbbr^n}f(y)\cdot\big(\frac{C_0}{(t^2+|x-y|^2)^{\frac{n+1}{2}}}-
\frac{C_0(n+1)t^2}{(t^2+|x-y|^2)^{\frac{n+3}{2}}}\big)dy.\label{eq1-6}
\end{align}
Furthermore, $\nabla_n\Pint(f;x,t)$ represents the $n$-dimensional gradient vector of the Poisson integral $\Pint(f;x,t)$ with respect to $x\in\bbbr^n$, and $\nabla_{n+1}\Pint(f;x,t)$ represents the $(n+1)$-dimensional gradient vector of the Poisson integral $\Pint(f;x,t)$ with respect to both $x\in\bbbr^n$ and $t>0$. 
\begin{lemma}\label{lemma15}
(i) If $F(x)\in\Lloc$ satisfies the condition that
\begin{equation}\label{eq2.206}
\int_{\bbbr^n}F(x)\!\cdot\!g(x)dx=0\quad\text{for all $g\in\Ccinfty(\bbbr^n)$},
\end{equation}
then $F(x)=0$ for almost every $x\in\bbbr^n$. Since $\Ccinfty(\bbbr^n)\subseteq\Sw(\bbbr^n)$, if $F(x)\in\Lloc$ satisfies the equation (\ref{eq2.206}) for all $g\in\Sw(\bbbr^n)$, then $F(x)=0$ for almost every $x\in\bbbr^n$.\\
(ii) If $\{F_t(x)\}_{t>0}\subseteq\Lloc$ is a family of uniformly integrable functions so that
\begin{equation}\label{eq2.207}
\lim_{t\rightarrow0}\int_{\bbbr^n}F_t(x)\!\cdot\!g(x)dx=0\quad\text{for all $g\in\Sw(\bbbr^n)$},
\end{equation}
and if there exists a subset $E\subseteq\bbbr^n$ so that for every $x\in E$ and uniformly for all $t\in(0,\infty)$, we have
\begin{equation}\label{eq2.208}
F_t(x)=\lim_{\delta\rightarrow0}\mvint_{B^n(x,\delta)}F_t(y)dy,
\end{equation}
then
\begin{equation}\label{eq2.209}
\lim_{t\rightarrow0}F_t(x)=0\quad\text{for every $x\in E\subseteq\bbbr^n$}.
\end{equation}
In particular, if $\{F_t(x)\}_{t>0}$ is a family of equicontinuous and uniformly integrable functions on $\bbbr^n$ that also satisfies condition (\ref{eq2.207}), then
\begin{equation}\label{eq2.210}
\lim_{t\rightarrow0}F_t(x)=0\quad\text{for every $x\in\bbbr^n$}.
\end{equation}
\end{lemma}
\begin{proof}[Proof of Lemma \ref{lemma15}]
Step 1: We prove part (i) by proving that condition (\ref{eq2.206}) implies $\int_K F(x)dx=0$ for every compact set $K\subseteq\bbbr^n$. Given a compact set $K\subseteq\bbbr^n$ and $\varepsilon>0$, we denote its $\varepsilon$-neighborhood is
\begin{equation}\label{eq2.211}
K_{\varepsilon}:=\bigcup_{x\in K}B^n(x,\varepsilon)=\{y\in\bbbr^n:\dist(\{y\},K)<\varepsilon\},
\end{equation}
and we can find a function $\varphi_{\varepsilon}(x)\in\Ccinfty(\bbbr^n)\subseteq\Sw(\bbbr^n)$ so that all of the following conditions are satisfied,
\begin{equation}\label{eq2.212}
0\leq\varphi_{\varepsilon}(x)\leq1\text{ for all $x\in\bbbr^n$,}
\end{equation}
\begin{equation}\label{eq2.213}
\text{$\varphi_{\varepsilon}(x)=1$ if $x\in K$, and $spt.\varphi_{\varepsilon}\subseteq K_{\varepsilon}$.}
\end{equation}
Then condition (\ref{eq2.206}) implies for all $\varepsilon>0$, we have
\begin{equation}\label{eq2.214}
\int_{\bbbr^n}F(x)\!\cdot\!\varphi_{\varepsilon}(x)dx=\int_{K_{\varepsilon}}F(x)\!\cdot\!\varphi_{\varepsilon}(x)dx=0.
\end{equation}
Thus we obtain the following estimate,
\begin{align}
&\big|\int_K F(x)dx\big|=\big|\int_K F(x)dx-
\int_{K_{\varepsilon}}F(x)\!\cdot\!\varphi_{\varepsilon}(x)dx\big|\nonumber\\
&=\!\big|\!\!\int_K\!\!\!F(x)dx\!-\!\!\int_{K}\!\!\!F(x)\!\cdot\!\varphi_{\varepsilon}(x)dx\!-\!\!
\int_{K_{\varepsilon}\setminus K}\!\!\!\!\!\!\!\!\!\!F(x)
\!\cdot\!\varphi_{\varepsilon}(x)dx\big|\nonumber\\
&\leq\int_{K_{\varepsilon}\setminus K}\!\!\!\!\!\!\!|F(x)|\!\cdot\!\varphi_{\varepsilon}(x)dx
\leq\int_{K_{\varepsilon}\setminus K}\!\!\!\!\!\!\!|F(x)|dx.\label{eq2.215}
\end{align}
Since $F(x)\in\Lloc$ and since $\lim_{\varepsilon\rightarrow0}\Lebes^n(K_{\varepsilon}\setminus K)=0$, we have
\begin{equation*}
\lim_{\varepsilon\rightarrow0}\int_{K_{\varepsilon}\setminus K}\!\!\!\!\!\!\!|F(x)|dx=0.
\end{equation*}
Sending $\varepsilon\rightarrow0$ in (\ref{eq2.215}) yields $\int_K F(x)dx=0$. Invoking the condition $F(x)\in\Lloc$ and the Lebesgue differentiation theorem, we obtain for almost every $x\in\bbbr^n$,
\begin{equation*}
F(x)=\lim_{\delta\rightarrow0}\mvint_{B^n(x,\delta)}F(y)dy=0.
\end{equation*}
Step 2: We prove part (ii). Let $x\in E\subseteq\bbbr^n$ and $\rho>0$. By condition (\ref{eq2.208}), there exists a small number $\delta_0>0$ so that uniformly for all $t\in(0,\infty)$, we have
\begin{equation}\label{eq2.216}
\big|F_t(x)-\mvint_{B^n(x,\delta_0)}F_t(y)dy\big|<\frac{1}{3}\rho.
\end{equation}
Since $\{F_t(x)\}_{t>0}$ is uniformly integrable on $\bbbr^n$, there exists a small number $\varepsilon_0>0$ so that uniformly for all $t\in(0,\infty)$, we have
\begin{equation}\label{eq2.217}
\frac{1}{\Lebes^n(B^n(x,\delta_0))}\int_{\delta_0\leq|y|<\delta_0+\varepsilon_0}|F_t(y)|dy<
\frac{1}{3}\rho.
\end{equation}
Now we pick a function $\varphi(x)\in\Ccinfty(\bbbr^n)\subseteq\Sw(\bbbr^n)$ satisfying the following conditions that $0\leq\varphi(y)\leq1$ for all $y\in\bbbr^n$, $\varphi(y)=1$ if $|y-x|<\delta_0$, and $\varphi(y)=0$ if $|y-x|\geq\delta_0+\varepsilon_0$. By condition (\ref{eq2.207}), we have
\begin{equation*}
\lim_{t\rightarrow0}\int_{\bbbr^n}\!\!\!\!F_t(y)\!\cdot\!\varphi(y)dy
=\lim_{t\rightarrow0}\int_{B^n(x,\delta_0+\varepsilon_0)}\!\!\!\!\!\!\!\!F_t(y)\!\cdot\!\varphi(y)dy=0,
\end{equation*}
hence there exists a positive number $t_0>0$ such that $0<t<t_0$ implies
\begin{equation}\label{eq2.218}
\frac{1}{\Lebes^n(B^n(x,\delta_0))}\!\cdot\!\big|\int_{B^n(x,\delta_0+\varepsilon_0)}
\!\!\!\!\!\!\!\!F_t(y)\!\cdot\!\varphi(y)dy\big|<\frac{1}{3}\rho.
\end{equation}
Combining (\ref{eq2.216}), (\ref{eq2.217}), and (\ref{eq2.218}) all together yields that for every $x\in E$, $|F_t(x)|<\rho$ whenever $0<t<t_0$, then (\ref{eq2.209}) is a consequence of the arbitrariness of $\rho>0$. If $\{F_t(x)\}_{t>0}$ is a family of equicontinuous functions on $\bbbr^n$, then condition (\ref{eq2.208}) is satisfied for every $x\in E=\bbbr^n$. The proof of Lemma \ref{lemma15} is complete.
\end{proof}
The following lemma is the famous Whitney decomposition lemma of open sets in $\bbbr^n$.
\begin{lemma}[cf. Appendix J of \cite{14classical}]\label{lemma1}
Let $\Omega$ be an open nonempty proper subset of $\bbbr^n$. Then there exists a family of closed dyadic cubes $\{Q_j\}_j$ (called the Whitney cubes of $\Omega$) such that\\
(a) $\bigcup_j Q_j=\Omega$ and the $Q_j$'s have disjoint interiors.\\
(b) $\sqrt{n}l(Q_j)\leq\dist(Q_j,\Omega^c)\leq 4\sqrt{n}l(Q_j)$. Thus $10\sqrt{n}Q_j$ meets $\Omega^c$.\\
(c) If the boundaries of two cubes $Q_j$ and $Q_k$ touch, then $$\frac{1}{4}\leq\frac{l(Q_j)}{l(Q_k)}\leq 4.$$
(d) For a given $Q_j$, there exist at most $12^n\!-\!4^n$ cubes $Q_k$ that touch it.\\
(e) Let $0<\varepsilon<\frac{1}{4}$. If $Q_j^*$ has the same center as $Q_j$ and $l(Q_j^*)=(1+\varepsilon)l(Q_j)$ then 
$$\chi_{\Omega}=\sum_j\chi_{Q_j^*}\leq 12^n-4^n+1.$$
\end{lemma}
In Lemma \ref{lemma2} below we show that if two Whitney cubes produced by Lemma \ref{lemma1} are sufficiently far away from each other, then the distance between one arbitrary point in one cube and another arbitrary point in the other cube is comparable to the distance between the two centers of these cubes.
	\begin{lemma}\label{lemma2}
		Let $\Omega$ be an open nonempty proper subset of $\bbbr^n$ and the collection $\{Q_j\}_j$ of Whitney cubes of $\Omega$ satisfies the conclusions of Lemma \ref{lemma1}. Then for two cubes $Q_j$ and $Q_m$ in this collection we further have the following properties.\\
		(1) If $|c(Q_m)-c(Q_j)|>11(\sqrt{n}+1)l(Q_m)$, and if $y\in Q_m$, $z\in Q_j$ are arbitrary points in the respective cube, then we have
		\begin{equation}\label{eq.lemma2-1}
			|y-c(Q_j)|\geq\frac{\sqrt{n}+1}{2}l(Q_j),
		\end{equation}
		and
		\begin{equation}\label{eq.lemma2-2}
			C_1|y-c(Q_j)|\leq|y-z|\leq C_2|y-c(Q_j)|,
		\end{equation}
		where $C_1=1-\frac{\sqrt{n}}{\sqrt{n}+1}$, $C_2=1+\frac{\sqrt{n}}{\sqrt{n}+1}$ and $0<C_1<1<C_2$. Furthermore, we have
		\begin{equation}\label{eq.lemma2-3}
			C_3|c(Q_m)-c(Q_j)|\leq|y-z|\leq C_4|c(Q_m)-c(Q_j)|,
		\end{equation}
		where $C_3=(1-\frac{\sqrt{n}}{\sqrt{n}+1})(1-\frac{\sqrt{n}}{22(\sqrt{n}+1)})$, $C_4=(1+\frac{\sqrt{n}}{\sqrt{n}+1})(1+\frac{\sqrt{n}}{22(\sqrt{n}+1)})$, and $0<C_3<1<C_4$.\\
		(2) If $|c(Q_m)-c(Q_j)|\leq 11(\sqrt{n}+1)l(Q_m)$, then the cube $Q_j$ is contained in the $n$-dimensional ball that has the same center with $Q_m$ and has radius $11(\sqrt{n}+1)^2 l(Q_m)$.
	\end{lemma}
	\begin{proof}[Proof of Lemma \ref{lemma2}]
		First we prove Lemma \ref{lemma2} (1). If $|c(Q_m)-c(Q_j)|>11(\sqrt{n}+1)l(Q_m)$ and $l(Q_j)<20l(Q_m)$, then for $y\in Q_m$ we have
		\begin{align}
			&|y-c(Q_j)|\geq |c(Q_m)-c(Q_j)|-|c(Q_m)-y|\nonumber\\
			&>11(\sqrt{n}+1)l(Q_m)-\frac{\sqrt{n}}{2}l(Q_m)\nonumber\\
			&>(\frac{21}{2}\sqrt{n}+11)\cdot\frac{l(Q_j)}{20}>\frac{\sqrt{n}+1}{2}\cdot l(Q_j).\label{eq17}
		\end{align}
		If $|c(Q_m)-c(Q_j)|>11(\sqrt{n}+1)l(Q_m)$ and $l(Q_j)\geq 20l(Q_m)$, then we pick $w\in\Omega^c$ and $y'\in Q_m$ such that
		\begin{equation}\label{eq18}
			\dist(Q_m,\Omega^c)\leq|w-y'|<\frac{3}{2}\dist(Q_m,\Omega^c)\leq 6\sqrt{n}l(Q_m),
		\end{equation}
		by Lemma \ref{lemma1} (b) and the definition of $\dist(Q_m,\Omega^c)$. Since both $y$ and $y'$ belong to $Q_m$, we have
		\begin{equation}\label{eq19}
			|w-y|\leq|w-y'|+|y'-y|<7\sqrt{n}l(Q_m).
		\end{equation}
		Because the ball centered at $c(Q_j)$ and inscribed in the cube $Q_j$ has radius $\frac{1}{2}l(Q_j)$, we can obtain
		\begin{align}
			&|y-c(Q_j)|\geq|w-c(Q_j)|-|w-y|\nonumber\\
			&>\dist(Q_j,\Omega^c)+\frac{1}{2}l(Q_j)-7\sqrt{n}l(Q_m)\nonumber\\
			&\geq(\sqrt{n}+\frac{1}{2})l(Q_j)-\frac{7\sqrt{n}}{20}l(Q_j)>\frac{\sqrt{n}+1}{2}\cdot l(Q_j).\label{eq20}
		\end{align}
		Therefore (\ref{eq.lemma2-1}) is proved. If $z\in Q_j$, then we use (\ref{eq.lemma2-1}) to obtain
		\begin{align}
			&|y-z|\leq|y-c(Q_j)|+|c(Q_j)-z|\nonumber\\
			&\leq|y-c(Q_j)|+\frac{\sqrt{n}}{2}l(Q_j)\nonumber\\
			&\leq(1+\frac{\sqrt{n}}{1+\sqrt{n}})|y-c(Q_j)|,\nonumber
		\end{align}
		and
		\begin{align}
			&|y-z|\geq|y-c(Q_j)|-|c(Q_j)-z|\nonumber\\
			&\geq|y-c(Q_j)|-\frac{\sqrt{n}}{2}l(Q_j)\nonumber\\
			&\geq(1-\frac{\sqrt{n}}{1+\sqrt{n}})|y-c(Q_j)|.\nonumber
		\end{align}
		Thus (\ref{eq.lemma2-2}) is proved. When $|c(Q_m)-c(Q_j)|>11(\sqrt{n}+1)l(Q_m)$ and $y\in Q_m$, we have
		\begin{align}
			&|y-c(Q_j)|\leq|y-c(Q_m)|+|c(Q_m)-c(Q_j)|\nonumber\\
			&\leq\frac{\sqrt{n}}{2}l(Q_m)+|c(Q_m)-c(Q_j)|\nonumber\\
			&<(1+\frac{\sqrt{n}}{22(\sqrt{n}+1)})|c(Q_m)-c(Q_j)|,\label{eq21}
		\end{align}
		and
		\begin{align}
			&|y-c(Q_j)|\geq|c(Q_m)-c(Q_j)|-|y-c(Q_m)|\nonumber\\
			&\geq|c(Q_m)-c(Q_j)|-\frac{\sqrt{n}}{2}l(Q_m)\nonumber\\
			&>(1-\frac{\sqrt{n}}{22(\sqrt{n}+1)})|c(Q_m)-c(Q_j)|.\label{eq22}
		\end{align}
		Applying (\ref{eq21}) to the right side of (\ref{eq.lemma2-2}) and applying (\ref{eq22}) to the left side of (\ref{eq.lemma2-2}) yields the desired inequality (\ref{eq.lemma2-3}).\\
		
		\begin{figure}[!htbp]
			\begin{tikzpicture}
				\begin{scope}[xshift=4.25cm,yshift=-1.8cm]
					\node[right] at (0,0) {$c(Q_m)$};
					\fill (0,0) circle[radius=0.05];
					\draw[line width=1] (2,2) -- node[right] {$l(Q_m)$} (2,-2);
					\draw[line width=1] (2,-2) -- (-2,-2);
					\draw[line width=1] (-2,-2) -- (-2,2);
					\draw[line width=1] (-2,2) -- (2,2);
					\draw[dotted] (0,0) -- (0,2);
					\draw[domain=0:360,smooth,variable=\t] plot ({2*cos(\t)},{2*sin(\t)});
				\end{scope}
				\begin{scope}
					\node[below] at (0,0) {$c(Q_j)$};
					\fill (0,0) circle[radius=0.05];
					\draw[line width=1] (-1,-1) -- node[left] {$l(Q_j)$} (-1,1);
					\draw[line width=1] (-1,1) -- (1,1);
					\draw[line width=1] (1,1) -- (1,-1);
					\draw[line width=1] (1,-1) -- (-1,-1);
					\draw[dotted] (0,0) -- (0,1);
					\draw[domain=0:360,smooth,variable=\t] plot ({1*cos(\t)},{1*sin(\t)});
				\end{scope}
				\draw[dotted] (0,0) -- (4.25,-1.8);
			\end{tikzpicture}
			\caption{Lemma \ref{lemma2} (2). If cubes $Q_j$ and $Q_m$ have disjoint interiors, then the two $n$-dimensional balls inscribed in $Q_j$ and $Q_m$ are disjoint balls, and thus the sum of radii of these two balls is less than or equal to the distance between the centers of these two balls.}
			\label{figure2}
		\end{figure}
		Now we prove Lemma \ref{lemma2} (2). If $|c(Q_m)-c(Q_j)|\leq 11(\sqrt{n}+1)l(Q_m)$, then the conclusion is trivial if $Q_m=Q_j$. If $Q_m\neq Q_j$, then they must have disjoint interiors and hence the two $n$-dimensional balls inscribed in $Q_m$ and $Q_j$, i.e. the balls $B^n(c(Q_m),\frac{1}{2}l(Q_m))$ and $B^n(c(Q_j),\frac{1}{2}l(Q_j))$, must be disjoint balls and we have
		\begin{equation}
			\frac{1}{2}l(Q_m)+\frac{1}{2}l(Q_j)\leq|c(Q_m)-c(Q_j)|\leq 11(\sqrt{n}+1)l(Q_m),
		\end{equation}
		and thus
		\begin{equation}\label{eq23}
			l(Q_j)<22(\sqrt{n}+1)l(Q_m).
		\end{equation}
		For $z\in Q_j$, we can obtain the following inequality
		\begin{align*}
			&|z-c(Q_m)|\leq|z-c(Q_j)|+|c(Q_j)-c(Q_m)|\\
			&\leq\frac{\sqrt{n}}{2}l(Q_j)+11(\sqrt{n}+1)l(Q_m)\\
			&<11(\sqrt{n}+1)^2 l(Q_m),
		\end{align*}
		therefore the cube $Q_j$ is contained in the ball 
		\begin{equation*}
			B^n(c(Q_m),11(\sqrt{n}+1)^2 l(Q_m)).\qedhere
		\end{equation*}
	\end{proof}
	According to Lemma \ref{lemma2}, for two Whitney cubes $Q_j$ and $Q_m$ of the set $\Omega$, we define the following relations:
	\begin{align}
		\text{``$Q_j$ near $Q_m$''}&\text{ means }
		|c(Q_m)-c(Q_j)|\leq 11(\sqrt{n}+1)l(Q_m),\label{near}\\
		\text{``$Q_j$ far from $Q_m$''}&\text{ means }
		|c(Q_m)-c(Q_j)|>11(\sqrt{n}+1)l(Q_m),\label{far}
	\end{align}
	where $c(Q_j)$ and $c(Q_m)$ represent the centers of cubes $Q_j$ and $Q_m$, respectively, and $l(Q_m)$ is the side length of cube $Q_m$. Lemma \ref{lemma3} below is cited from \cite{fefferman1970} and it has been used frequently through out the proofs of our main results so we also include a detailed proof for this lemma. For a locally integrable function $f$, we denote the uncentered Hardy-Littlewood maximal function by
	\begin{equation*}
		M(f)(x)=\sup_{x\in I}\mvint_I|f(y)|dy,
	\end{equation*}
	where the supremum is taken over all the cubes $I\subseteq\bbbr^n$ containing the point $x\in\bbbr^n$.
\begin{lemma}[cf. \cite{fefferman1970}]\label{lemma3}
For $1\leq p<\infty$, let $f\in L^p(\bbbr^n)$ be a function and let $\alpha$ be a real number satisfying $0<\alpha<\esssup_{x\in\bbbr^n}|f(x)|$. Then the set $\Omega:=\{x\in\bbbr^n:M(|f|^p)(x)>\alpha^p\}$ is an open nonempty proper subset of $\bbbr^n$, and there exist a collection of cubes $\{I_j\}_j$ and two functions, the ``good'' function $g$ and the ``bad'' function $b$, satisfying the following conditions.\\
(i) The $I_j$'s are pairwise disjoint, $\Omega=\bigcup_j I_j$ and $$\sum_j\Lebes^n(I_j)\leq C_5\alpha^{-p}\|f\|_{L^p(\bbbr^n)}^p.$$
(ii) For almost every $x\notin\Omega$, $|f(x)|\leq\alpha$.\\
(iii) For any one of the cubes $\{I_j\}_j$, $\mvint_{I_j}|f(y)|^p dy\leq C_6\alpha^p$.\\
(iv) For any cube $I_j$ of the collection, if $\tilde{I_j}$ denotes the cube concentric with $I_j$ but with twice the side length of $I_j$. Then every point of $\bbbr^n$ lies in at most $[\nu_n(40n)^n]+1$ cubes of the collection $\{\tilde{I_j}\}_j$.\\
(v) $f=g+b$ and $b$ is supported in $\Omega$.\\
(vi) We have $\mvint_{\!I_j}\!|b(y)|^p dy\!\!\leq\!\!C_7\alpha^p$ and $\int_{\!I_j}\!|b(y)|dy\!\!\leq\!\!C_8\alpha\Lebes^n(I_j)$ for each cube $I_j$.\\
(vii) $\int_{I_j}b(y)dy=0$ for each cube $I_j$.\\
(viii) For almost every $x\!\in\!\bbbr^n$, $|g(x)|\!\leq\!C_9\alpha$ and $\|g\|_{L^p(\bbbr^n)}\!\leq\!\|f\|_{L^p(\bbbr^n)}$.\\
The dependence of the above constants on parameters is given as follows: $C_5=C_5(n)$, $C_6=C_6(n)$, $C_7=C_7(n,p)$, $C_8=C_8(n,p)$ and $C_9=C_9(n,p)$.
\end{lemma}
\begin{proof}[Proof of Lemma \ref{lemma3}]
First we prove $\Omega=\{x\in\bbbr^n:M(|f|^p)(x)>\alpha^p\}$ is an open nonempty proper subset of $\bbbr^n$. If for all $x\in\bbbr^n$, $M(|f|^p)(x)\leq\alpha^p$, then $f\in L^p(\bbbr^n)$ for $1\leq p<\infty$ implies $|f(x)|^p\in L^1(\bbbr^n)$ is locally integrable on $\bbbr^n$. By Lebesgue differentiation theorem, we have 
\begin{equation}\label{eq1.86}
|f(x)|^p\!=\!\lim_{\substack{x=c(I)\\l(I)\rightarrow 0}}\mvint_I\!|f(y)|^p dy\!\leq\!\alpha^p
\quad\text{for almost every $x\!\in\!\bbbr^n$.} 
\end{equation}
Thus we have $\esssup_{x\in\bbbr^n}|f(x)|\leq\alpha<\esssup_{x\in\bbbr^n}|f(x)|$ and this is a contradiction, hence $\Omega$ is nonempty. If $x\in\Omega$ and the uncentered Hardy-Littlewood maximal function satisfies $M(|f|^p)(x)>\alpha^p$, then there exists a cube $I_x\ni x$ such that $\mvint_{I_x}|f(y)|^p dy>\alpha^p$. And for every $z\in I_x$, we have $M(|f|^p)(z)\geq\mvint_{I_x}|f(y)|^p dy>\alpha^p$ and thus the interior of the cube $I_x$ is included in $\Omega$, hence the set $\Omega$ is an open set. If $\Omega$ is all of $\bbbr^n$ and $\Lebes^n(\Omega)=\infty$, because the uncentered Hardy-Littlewood maximal function $M$ is a bounded operator from $L^1(\bbbr^n)$ into $L^{1,\infty}(\bbbr^n)$, then by \cite[Theorem 2.1.6]{14classical} we have
		\begin{equation}\label{eq24}
			\alpha^p\Lebes^n(\Omega)\leq C_5\| |f|^p \|_{L^1(\bbbr^n)}=C_5\|f\|_{L^p(\bbbr^n)}^p<\infty,
		\end{equation}
where $C_5=C_5(n)$. This is a contradiction, therefore $\Omega$ must be a proper subset of $\bbbr^n$. Now because $\Omega$ satisfies all the conditions of Lemma \ref{lemma1}, we apply Lemma \ref{lemma1} to $\Omega$ and obtain a collection of cubes $\{I_j\}_j$ having disjoint interiors, and $\Omega=\bigcup_j I_j$. By (\ref{eq24}), we have
		\begin{equation}\label{eq25}
			\sum_j\Lebes^n(I_j)=\Lebes^n(\Omega)\leq C_5\alpha^{-p}\|f\|_{L^p(\bbbr^n)}^p,
		\end{equation}
and Lemma \ref{lemma3} (i) is proved. To prove Lemma \ref{lemma3} (ii), since $|f|^p\in L^1(\bbbr^n)$ and almost every point in $\bbbr^n$ is a Lebesgue point of $|f|^p$, thus Lemma \ref{lemma3} (ii) is a consequence of (\ref{eq1.86}). To prove Lemma \ref{lemma3} (iii), let $I_j$ be a cube in the collection, by Lemma \ref{lemma1} we have
		\begin{equation}\label{eq26}
			\sqrt{n}l(I_j)\leq\dist(I_j,\Omega^c)\leq 4\sqrt{n}l(I_j),
		\end{equation}
		thus if $\bar{I_j}$ is the cube concentric with $I_j$ but with the side length $l(\bar{I_j})=10\sqrt{n}l(I_j)$ then the set $\bar{I_j}\bigcap\Omega^c$ is nonempty. Otherwise, $\bar{I_j}$ is contained in $\Omega$ and hence
		\begin{equation}\label{eq27}
			\dist(I_j,\Omega^c)\geq\dist(I_j,\bar{I_j}^c)\geq \frac{1}{2}l(\bar{I_j})-\frac{1}{2}l(I_j)>4\sqrt{n}l(I_j),
		\end{equation}
		and (\ref{eq27}) poses a contradiction to (\ref{eq26}). And if $w$ is a point in the set $\bar{I_j}\bigcap\Omega^c$, then we have
		\begin{equation}\label{eq28}
			\frac{\int_{I_j}|f(y)|^p dy}{(10\sqrt{n})^n\Lebes^n(I_j)}\leq
			\mvint_{\bar{I_j}}|f(y)|^p dy\leq M(|f|^p)(w)\leq\alpha^p.
		\end{equation}
		This inequality proves Lemma \ref{lemma3} (iii) with $C_6=(10\sqrt{n})^n$. To prove Lemma \ref{lemma3} (iv), let $x\in\bbbr^n$ and $x\in\tilde{I_j}$ for some $j$, then $|x-c(I_j)|\leq\sqrt{n}l(I_j)$. This tells $x$ must belong to $\Omega$, otherwise $$\sqrt{n}l(I_j)\leq \dist(I_j,\Omega^c)<|c(I_j)-x|\leq\sqrt{n}l(I_j),$$ and we obtain a contradiction. Therefore $x\in\Omega$ and $\dist(\{x\},\Omega^c)>0$. Let $y\in\Omega^c$, $z\in I_j$ satisfy
		\begin{equation}\label{eq29}
			\sqrt{n}l(I_j)\leq\dist(I_j,\Omega^c)\leq|y-z|<\frac{3}{2}\dist(I_j,\Omega^c)\leq 6\sqrt{n}l(I_j).
		\end{equation}
		Then $|y-c(I_j)|\leq|y-z|+|z-c(I_j)|<6\sqrt{n}l(I_j)+\frac{1}{2}\sqrt{n}l(I_j)$, and we have
		\begin{equation}\label{eq30}
			\dist(\{x\},\Omega^c)\leq|x-y|\leq|x-c(I_j)|+|y-c(I_j)|<8\sqrt{n}l(I_j).
		\end{equation}
		Furthermore, if $w\in\Omega^c$ satisfies
		\begin{equation}\label{eq31}
			\dist(\{x\},\Omega^c)\leq|w-x|<\frac{3}{2}\dist(\{x\},\Omega^c),
		\end{equation}
		then because the ball centered at $c(I_j)$ with radius $\frac{1}{2}l(I_j)$ is inscribed in the cube $I_j$, we have
		\begin{align}
			&|w-x|\geq|w-c(I_j)|-|x-c(I_j)|\nonumber\\
			&\geq\dist(I_j,\Omega^c)+\frac{1}{2}l(I_j)-\sqrt{n}l(I_j)\geq\frac{1}{2}l(I_j).\label{eq32}
		\end{align}
		Inequalities (\ref{eq30}), (\ref{eq31}), (\ref{eq32}) tell us that the condition $x\in\tilde{I_j}$ implies
		\begin{equation}\label{eq33}
			\frac{\dist(\{x\},\Omega^c)}{8\sqrt{n}}\leq l(I_j)\leq 3\cdot\dist(\{x\},\Omega^c).
		\end{equation}
		Hence the cube $I_j$ is contained in a ball centered at $x$ with radius $5\sqrt{n}\cdot \dist(\{x\},\Omega^c)$ and the side length $l(I_j)$ is at least $\frac{\dist(\{x\},\Omega^c)}{8\sqrt{n}}$. There can be at most $[\nu_n(40n)^n]+1$ such cubes that satisfy this condition, and Lemma \ref{lemma3} (iv) is proved. Now we define
		\begin{align}
			g(x)=
			\begin{cases}
				\mvint_{I_j}f(y)dy &\text{ if }x\in I_j\\
				f(x)               &\text{ if }x\notin\Omega
			\end{cases}
			&\quad\text{ and }
			b(x)=f(x)-g(x),\label{eq34}
		\end{align}
		then $b(x)$ is supported in $\Omega$ and Lemma \ref{lemma3} (vii) is obvious by definition. To prove Lemma \ref{lemma3} (vi), we use Lemma \ref{lemma3} (iii) and Jensen's inequality for $1\leq p<\infty$ to obtain
		\begin{align}
			&\mvint_{I_j}|b(y)|^p dy\leq \mvint_{I_j}2^{p-1}(|f(y)|^p+|g(y)|^p)dy\nonumber\\
			&\leq 2^{p-1}(C_6\alpha^p+(\mvint_{I_j}|f(y)|dy)^p)\nonumber\\
			&\leq 2^{p-1}(C_6\alpha^p+\mvint_{I_j}|f(y)|^p dy)\leq C_7\alpha^p,\nonumber
		\end{align}
		where $C_7=2^p C_6=2^p(10\sqrt{n})^n$. And the estimate for $\int_{I_j}|b(y)|dy$ is obtained by H\"{o}lder's inequality. To prove Lemma \ref{lemma3} (viii), when $x\notin\Omega$, we have $|g(x)|=|f(x)|\leq\alpha$ by (\ref{eq34}) and Lemma \ref{lemma3} (ii). When $x\in I_j$ for some $j$, we use (\ref{eq34}), Jensen's inequality for $1\leq p<\infty$ and Lemma \ref{lemma3} (iii) to obtain
		\begin{equation}\label{eq35}
			|g(x)|\leq\mvint_{I_j}|f(y)|dy\leq(\mvint_{I_j}|f(y)|^p dy)^{\frac{1}{p}}\leq C_9\alpha,
		\end{equation}
		where $C_9=(10\sqrt{n})^{\frac{n}{p}}$. Furthermore, we have
		\begin{align}
			&\|g\|_{L^p(\bbbr^n)}^p=\int_{\Omega^c}|f(x)|^p dx+\sum_j\int_{I_j}|g(x)|^p dx\nonumber\\
			&=\int_{\Omega^c}|f(x)|^p dx+\sum_j\Lebes^n(I_j)(\mvint_{I_j}|f(y)|dy)^p\nonumber\\
			&\leq\int_{\Omega^c}|f(x)|^p dx+\sum_j\Lebes^n(I_j)\mvint_{I_j}|f(y)|^p dy\nonumber\\
			&=\int_{\Omega^c}|f(x)|^p dx+\int_{\bigcup_j I_j}|f(y)|^p dy=\|f\|_{L^p(\bbbr^n)}^p.\nonumber
		\end{align}
		The proof of Lemma \ref{lemma3} is now complete.
	\end{proof}
	The next Lemma \ref{lemma4} shows that derivatives of the Poisson kernel $P_t(x)$ satisfy the H\"{o}rmander type condition.
	\begin{lemma}\label{lemma4}
		Let $\beta\in\bbbr$ and $\beta>1$. Then for $1\leq k\leq n$, the following quantities can be estimated from above by a positive finite constant whose value only depends on $n$ and $\beta$, 
		\begin{align}
			&\sup_{x\in\bbbr^n}\int_0^{\infty}\int_{|y|>\beta|x|}\bigg|\frac{t(y_k-x_k)}{(t^2+|y-x|^2)^{\frac{n+3}{2}}}-
			\frac{ty_k}{(t^2+|y|^2)^{\frac{n+3}{2}}}\bigg|dydt,\label{eq.lemma4-1}\\
			&\sup_{x\in\bbbr^n}\int_0^{\infty}\int_{|y|>\beta|x|}\bigg|\frac{1}{(t^2+|y-x|^2)^{\frac{n+1}{2}}}-
			\frac{1}{(t^2+|y|^2)^{\frac{n+1}{2}}}\bigg|dydt,\label{eq.lemma4-2}\\
			&\sup_{x\in\bbbr^n}\int_0^{\infty}\int_{|y|>\beta|x|}\bigg|\frac{t^2}{(t^2+|y-x|^2)^{\frac{n+3}{2}}}-
			\frac{t^2}{(t^2+|y|^2)^{\frac{n+3}{2}}}\bigg|dydt,\label{eq.lemma4-3}
		\end{align}
		where $x_k$ and $y_k$ are the $k$-th coordinates of $x,y\in\bbbr^n$, respectively.
	\end{lemma}
	\begin{proof}[Proof of Lemma \ref{lemma4}]
		When $|y|>\beta|x|$, $\beta>1$ and $0\leq r\leq 1$, we have
		\begin{align}
			|y-rx|&\geq|y|-r|x|\geq|y|-|x|>(1-\frac{1}{\beta})|y|,\nonumber\\
			|y-rx|&\leq|y|+r|x|\leq|y|+|x|<(1+\frac{1}{\beta})|y|,\nonumber
		\end{align}
		hence we obtain
		\begin{equation}\label{eq46}
			(1-\frac{1}{\beta})|y|<|y-rx|<(1+\frac{1}{\beta})|y|.
		\end{equation}
		If we let $h_1(r)=(t^2+|y-rx|^2)^{-\frac{n+3}{2}}$ and $h_2(r)=(t^2+|y-rx|^2)^{-\frac{n+1}{2}}$ and let $h_1'(r)$ and $h_2'(r)$ denote their respective derivatives, then by the fundamental theorem of calculus, we have the following estimates,
		\begin{align}
			&\bigg|\frac{1}{(t^2+|y-x|^2)^{\frac{n+3}{2}}}-\frac{1}{(t^2+|y|^2)^{\frac{n+3}{2}}}\bigg|\nonumber\\
			&\leq\int_0^1|h_1'(r)|dr\leq\int_0^1\frac{(n+3)|y-rx|\cdot|x|}{(t^2+|y-rx|^2)^{\frac{n+5}{2}}}dr
			\nonumber\\
			&\leq\frac{C_{10}|y|\cdot|x|}{(t^2+|y|^2)^{\frac{n+5}{2}}},\label{eq47}
		\end{align}
		and
		\begin{align}
			&\bigg|\frac{1}{(t^2+|y-x|^2)^{\frac{n+1}{2}}}-\frac{1}{(t^2+|y|^2)^{\frac{n+1}{2}}}\bigg|\nonumber\\
			&\leq\int_0^1|h_2'(r)|dr\leq\int_0^1\frac{(n+1)|y-rx|\cdot|x|}{(t^2+|y-rx|^2)^{\frac{n+3}{2}}}dr
			\nonumber\\
			&\leq\frac{C_{11}|y|\cdot|x|}{(t^2+|y|^2)^{\frac{n+3}{2}}},\label{eq48}
		\end{align}
		where $C_{10}=C_{10}(n,\beta)$ and $C_{11}=C_{11}(n,\beta)$. To prove (\ref{eq.lemma4-1}) for $1\leq k\leq n$, we use (\ref{eq46}) and (\ref{eq47}) and calculate as follows,
		\begin{align}
			&\int_0^{\infty}\int_{|y|>\beta|x|}\bigg|\frac{t(y_k-x_k)}{(t^2+|y-x|^2)^{\frac{n+3}{2}}}-
			\frac{ty_k}{(t^2+|y|^2)^{\frac{n+3}{2}}}\bigg|dydt\nonumber\\
			&\leq\int_0^{\infty}\int_{|y|>\beta|x|}\bigg|\frac{t(y_k-x_k)}{(t^2+|y-x|^2)^{\frac{n+3}{2}}}-
			\frac{ty_k}{(t^2+|y-x|^2)^{\frac{n+3}{2}}}\bigg|\nonumber\\
			&\quad+\bigg|\frac{ty_k}{(t^2+|y-x|^2)^{\frac{n+3}{2}}}-\frac{ty_k}{(t^2+|y|^2)^{\frac{n+3}{2}}}\bigg|dydt\nonumber\\
			&\leq C_{12}\int_0^{\infty}\int_{|y|>\beta|x|}\frac{t|x|}{(t^2+|y-x|^2)^{\frac{n+3}{2}}}+
			\frac{t|y|^2\cdot|x|}{(t^2+|y|^2)^{\frac{n+5}{2}}}dydt\nonumber\\
			&\leq C_{13}\int_0^{\infty}\int_{|y|>\beta|x|}\frac{t|x|}{(t^2+|y|^2)^{\frac{n+3}{2}}}dydt,
			\label{eq49}
		\end{align}
		where $C_{12}=C_{12}(n,\beta)$ and $C_{13}=C_{13}(n,\beta)$. We write the double integral in the last line of (\ref{eq49}) as below
		\begin{equation}\label{eq50}
			\int_0^{|x|}\int_{|y|>\beta|x|}\frac{t|x|}{(t^2+|y|^2)^{\frac{n+3}{2}}}dydt+
			\int_{|x|}^{\infty}\int_{|y|>\beta|x|}\frac{t|x|}{(t^2+|y|^2)^{\frac{n+3}{2}}}dydt.
		\end{equation}
		For the first term in (\ref{eq50}), we apply the inequality $(t^2+|y|^2)^{\frac{n+3}{2}}\geq|y|^{n+3}$ and obtain
		\begin{align}
			&\int_0^{|x|}\int_{|y|>\beta|x|}\frac{t|x|}{(t^2+|y|^2)^{\frac{n+3}{2}}}dydt\nonumber\\
			&\leq\int_0^{|x|}\int_{|y|>\beta|x|}\frac{t|x|}{|y|^{n+3}}dydt\nonumber\\
			&=C_{14}|x|^2\cdot|x|\cdot|x|^{-3}\lesssim 1,\label{eq51}
		\end{align}
		where $C_{14}=C_{14}(n,\beta)$. For the second term in (\ref{eq50}), we consider $(t^2+|y|^2)^{\frac{n+3}{2}}$ is the $(n+3)$-fold product of the term $(t^2+|y|^2)^{\frac{1}{2}}$ and in this product, we apply the inequality $(t^2+|y|^2)^{\frac{1}{2}}\geq t$ twice, apply the inequality $(t^2+|y|^2)^{\frac{1}{2}}\geq(2t|y|)^{\frac{1}{2}}$ once and apply the inequality $(t^2+|y|^2)^{\frac{1}{2}}\geq|y|$ for $n$ times and obtain
		\begin{align}
			&\int_{|x|}^{\infty}\int_{|y|>\beta|x|}\frac{t|x|}{(t^2+|y|^2)^{\frac{n+3}{2}}}dydt\nonumber\\
			&\leq\int_{|x|}^{\infty}\int_{|y|>\beta|x|}\frac{t|x|}{t^2\cdot(2t|y|)^{\frac{1}{2}}\cdot|y|^n}dydt
			\nonumber\\
			&=C_{15}|x|\cdot|x|^{-\frac{1}{2}}\cdot|x|^{-\frac{1}{2}}\lesssim 1,\label{eq52}
		\end{align}
		where $C_{15}=C_{15}(n,\beta)$. Combining (\ref{eq49}), (\ref{eq50}), (\ref{eq51}) and (\ref{eq52}) yields (\ref{eq.lemma4-1}) can be estimated from above by a positive finite constant whose value only depends on $n$ and $\beta$. To prove (\ref{eq.lemma4-2}), we use (\ref{eq48}) and estimate as below
		\begin{align}
			&\int_0^{\infty}\int_{|y|>\beta|x|}\bigg|\frac{1}{(t^2+|y-x|^2)^{\frac{n+1}{2}}}-
			\frac{1}{(t^2+|y|^2)^{\frac{n+1}{2}}}\bigg|dydt\nonumber\\
			&\leq C_{11}\int_0^{\infty}\int_{|y|>\beta|x|}\frac{|y|\cdot|x|}{(t^2+|y|^2)^{\frac{n+3}{2}}}dydt
			\nonumber\\
			&=C_{11}\int_0^{|x|}\int_{|y|>\beta|x|}\frac{|y|\cdot|x|}{(t^2+|y|^2)^{\frac{n+3}{2}}}dydt\nonumber\\
			&\quad+C_{11}\int_{|x|}^{\infty}\int_{|y|>\beta|x|}\frac{|y|\cdot|x|}{(t^2+|y|^2)^{\frac{n+3}{2}}}dydt.\label{eq53}
		\end{align}
		For the first term of (\ref{eq53}), we use the inequality $(t^2+|y|^2)^{\frac{n+3}{2}}\geq|y|^{n+3}$ and obtain
		\begin{align}
			&\int_0^{|x|}\int_{|y|>\beta|x|}\frac{|y|\cdot|x|}{(t^2+|y|^2)^{\frac{n+3}{2}}}dydt\nonumber\\
			&\leq\int_0^{|x|}\int_{|y|>\beta|x|}\frac{|x|}{|y|^{n+2}}dydt\nonumber\\
			&=C_{16}|x|\cdot|x|\cdot|x|^{-2}\lesssim 1,\label{eq54}
		\end{align}
		where $C_{16}=C_{16}(n,\beta)$. For the second term of (\ref{eq53}), we consider $(t^2+|y|^2)^{\frac{n+3}{2}}$ is the $(n+3)$-fold product of the term $(t^2+|y|^2)^{\frac{1}{2}}$ and in this product, we apply the inequality $(t^2+|y|^2)^{\frac{1}{2}}\geq t$ once, apply the inequality $(t^2+|y|^2)^{\frac{1}{2}}\geq(2t|y|)^{\frac{1}{2}}$ once and apply the inequality $(t^2+|y|^2)^{\frac{1}{2}}\geq|y|$ for $(n+1)$ times and obtain
		\begin{align}
			&\int_{|x|}^{\infty}\int_{|y|>\beta|x|}\frac{|y|\cdot|x|}{(t^2+|y|^2)^{\frac{n+3}{2}}}dydt\nonumber\\
			&\leq\int_{|x|}^{\infty}\int_{|y|>\beta|x|}\frac{|y|\cdot|x|}{t\cdot(2t|y|)^{\frac{1}{2}}\cdot|y|^{n+1}}dydt\nonumber\\
			&=C_{17}|x|\cdot|x|^{-\frac{1}{2}}\cdot|x|^{-\frac{1}{2}}\lesssim 1,\label{eq55}
		\end{align}
		where $C_{17}=C_{17}(n,\beta)$. Combining (\ref{eq53}), (\ref{eq54}) and (\ref{eq55}) yields (\ref{eq.lemma4-2}) can be estimated from above by a positive finite constant whose value only depends on $n$ and $\beta$. To prove (\ref{eq.lemma4-3}), we use (\ref{eq47}) and estimate as below
		\begin{align}
			&\int_0^{\infty}\int_{|y|>\beta|x|}\bigg|\frac{t^2}{(t^2+|y-x|^2)^{\frac{n+3}{2}}}-
			\frac{t^2}{(t^2+|y|^2)^{\frac{n+3}{2}}}\bigg|dydt\nonumber\\
			&\leq C_{10}\int_0^{\infty}\int_{|y|>\beta|x|}\frac{t^2|y|\cdot|x|}{(t^2+|y|^2)^{\frac{n+5}{2}}}dydt
			\nonumber\\
			&=C_{10}\int_0^{|x|}\int_{|y|>\beta|x|}\frac{t^2|y|\cdot|x|}{(t^2+|y|^2)^{\frac{n+5}{2}}}dydt
			\nonumber\\
			&\quad+C_{10}\int_{|x|}^{\infty}\int_{|y|>\beta|x|}\frac{t^2|y|\cdot|x|}{(t^2+|y|^2)^{\frac{n+5}{2}}}dydt.\label{eq56}
		\end{align}
		To estimate the first term of (\ref{eq56}), we apply the inequality $(t^2+|y|^2)^{\frac{n+5}{2}}\geq|y|^{n+5}$ and obtain
		\begin{align}
			&\int_0^{|x|}\int_{|y|>\beta|x|}\frac{t^2|y|\cdot|x|}{(t^2+|y|^2)^{\frac{n+5}{2}}}dydt\nonumber\\
			&\leq\int_0^{|x|}\int_{|y|>\beta|x|}\frac{t^2|x|}{|y|^{n+4}}dydt\nonumber\\
			&=C_{18}|x|\cdot|x|^3\cdot|x|^{-4}\lesssim 1,\label{eq57}
		\end{align}
		where $C_{18}=C_{18}(n,\beta)$. For the second term of (\ref{eq56}), we consider $(t^2+|y|^2)^{\frac{n+5}{2}}$ is the $(n+5)$-fold product of the term $(t^2+|y|^2)^{\frac{1}{2}}$ and in this product, we apply the inequality $(t^2+|y|^2)^{\frac{1}{2}}\geq t$ for three times, apply the inequality $(t^2+|y|^2)^{\frac{1}{2}}\geq(2t|y|)^{\frac{1}{2}}$ once and apply the inequality $(t^2+|y|^2)^{\frac{1}{2}}\geq|y|$ for $(n+1)$ times and obtain
		\begin{align}
			&\int_{|x|}^{\infty}\int_{|y|>\beta|x|}\frac{t^2|y|\cdot|x|}{(t^2+|y|^2)^{\frac{n+5}{2}}}dydt
			\nonumber\\
			&\leq\int_{|x|}^{\infty}\int_{|y|>\beta|x|}\frac{t^2|y|\cdot|x|}{t^3\cdot(2t|y|)^{\frac{1}{2}}\cdot|y|^{n+1}}dydt\nonumber\\
			&=C_{19}|x|\cdot|x|^{-\frac{1}{2}}\cdot|x|^{-\frac{1}{2}}\lesssim 1,\label{eq58}
		\end{align}
		where $C_{19}=C_{19}(n,\beta)$. Combining (\ref{eq56}), (\ref{eq57}) and (\ref{eq58}) yields (\ref{eq.lemma4-3}) can be estimated from above by a positive finite constant whose value only depends on $n$ and $\beta$.
	\end{proof}
	\begin{lemma}\label{lemma6}
		Let $f\in L^p(\bbbr^n)$ for $1\leq p<\infty$, and assume $0<\alpha<\esssup_{x\in\bbbr^n}|f(x)|$ and $\Omega=\{x\in\bbbr^n:M(|f|^p)(x)>\alpha^p\}$ are given by Lemma \ref{lemma3}. And let $\{I_j\}_j$ be the associated Whitney cubes and $f=g+b$ as in Lemma \ref{lemma3}. Let $\chi_{I_j}$ denote the characteristic function of the cube $I_j$ and $b_j\!=\!b\!\cdot\!\chi_{I_j}$. Then the following estimates are true.\\
		(1) For $1\leq k\leq n+1$ and $y\in\Omega^c$, we have
		\begin{equation}\label{eq2-1}
			\big|\partial_k\Pint(\sum_j b_j;y,t)\big|\lesssim\frac{\alpha}{t},
		\end{equation}
		where the constant depends on $n,p$.\\
		(2) For $1\leq k\leq n+1$, $I_m$ is a Whitney cube and $y\in I_m$, we have
		\begin{equation}\label{eq2-2}
			\big|\partial_k\Pint(\sum_{\substack{j\\I_j\text{ far from }I_m}} b_j;y,t)\big|\lesssim\frac{\alpha}{t},
		\end{equation}
		where the constant depends on $n,p$.
	\end{lemma}
	\begin{proof}[Proof of Lemma \ref{lemma6}](1) If $1\leq k\leq n$, $y\in\Omega^c$ and $z\in I_j$, by Lemma \ref{lemma1} (b) and the fact that the $n$-dimensional ball centered at $c(I_j)$ and inscribed in the Whitney cube $I_j$ has radius $\frac{1}{2}l(I_j)$, we can obtain
		\begin{equation}\label{eq2-3}
			|y-c(I_j)|\geq\dist(I_j,\Omega^c)+\frac{1}{2}l(I_j)>\frac{\sqrt{n}+1}{2}\cdot l(I_j),
		\end{equation}
		and thus
		\begin{equation}\label{eq2-4}
			(1-\frac{\sqrt{n}}{\sqrt{n}+1})\cdot|y-c(I_j)|\leq|y-z|
			\leq(1+\frac{\sqrt{n}}{\sqrt{n}+1})\cdot|y-c(I_j)|.
		\end{equation}
		Furthermore, we also have
		\begin{align}
			&\sum_j\Lebes^n(I_j)\cdot\frac{t|y-c(I_j)|}{(t^2+|y-c(I_j)|^2)^{\frac{n+3}{2}}}\nonumber\\
			&=\sum_j\int_{I_j}\frac{t|y-c(I_j)|}{(t^2+|y-c(I_j)|^2)^{\frac{n+3}{2}}}dz\nonumber\\
			&\lesssim\sum_j\int_{I_j}\frac{t|y-z|}{(t^2+|y-z|^2)^{\frac{n+3}{2}}}dz\nonumber\\
			&\lesssim\int_{\bbbr^n}\frac{t|y-z|}{(t^2+|y-z|^2)^{\frac{n+3}{2}}}dz\lesssim t^{-1},\label{eq2-5}
		\end{align}
		where we use the fact that $\{I_j\}_j$ have disjoint interiors and $\bigcup_j I_j=\Omega\subseteq\bbbr^n$, and the fact that the integral $$\int_{\bbbr^n}\frac{|z|}{(1+|z|^2)^{\frac{n+3}{2}}}dz$$ converges to a positive finite constant whose value only depends on the dimension $n$. Applying (\ref{eq1-5}), (\ref{eq2-4}), (\ref{eq2-5}), and Lemma \ref{lemma3} (vi) yields
		\begin{align}
			&\big|\partial_k\Pint(\sum_j b_j;y,t)\big|\nonumber\\
			&\lesssim\sum_j\int_{I_j}|b(z)|\cdot\frac{t|y-z|}{(t^2+|y-z|^2)^{\frac{n+3}{2}}}dz\nonumber\\
			&\lesssim\sum_j\int_{I_j}|b(z)|dz\cdot\frac{t|y-c(I_j)|}{(t^2+|y-c(I_j)|^2)^{\frac{n+3}{2}}}\nonumber\\
			&\lesssim\sum_j\alpha\Lebes^n(I_j)\cdot\frac{t|y-c(I_j)|}{(t^2+|y-c(I_j)|^2)^{\frac{n+3}{2}}}
			\lesssim\frac{\alpha}{t},\label{eq2-6}
		\end{align}
		where the constants in (\ref{eq2-6}) depend on $n,p$. If $k=n+1$, $y\in\Omega^c$ and $z\in I_j$, then we still use (\ref{eq2-3}) and (\ref{eq2-4}) to obtain
		\begin{align}
			&\sum_j\Lebes^n(I_j)\cdot\big(\frac{1}{(t^2+|y-c(I_j)|^2)^{\frac{n+1}{2}}}
			+\frac{t^2}{(t^2+|y-c(I_j)|^2)^{\frac{n+3}{2}}}\big)\nonumber\\
			&=\sum_j\int_{I_j}\frac{1}{(t^2+|y-c(I_j)|^2)^{\frac{n+1}{2}}}
			+\frac{t^2}{(t^2+|y-c(I_j)|^2)^{\frac{n+3}{2}}}dz\nonumber\\
			&\lesssim\sum_j\int_{I_j}\frac{1}{(t^2+|y-z|^2)^{\frac{n+1}{2}}}
			+\frac{t^2}{(t^2+|y-z|^2)^{\frac{n+3}{2}}}dz\nonumber\\
			&\lesssim\int_{\bbbr^n}\frac{1}{(t^2+|y-z|^2)^{\frac{n+1}{2}}}
			+\frac{t^2}{(t^2+|y-z|^2)^{\frac{n+3}{2}}}dz\lesssim t^{-1},\label{eq2-7}
		\end{align}
		where we use the fact that the Whitney cubes $\{I_j\}_j$ have disjoint interiors and $\bigcup_j I_j=\Omega\subseteq\bbbr^n$, and the fact that both integrals 
		$$\int_{\bbbr^n}\frac{1}{(1+|z|^2)^{\frac{n+1}{2}}}dz\quad\text{and}\quad \int_{\bbbr^n}\frac{1}{(1+|z|^2)^{\frac{n+3}{2}}}dz$$ converge to positive finite constants whose values only depend on the dimension $n$. Applying (\ref{eq1-6}), (\ref{eq2-4}), (\ref{eq2-7}), and Lemma \ref{lemma3} (vi) yields
\begin{align}
&\big|\partial_{n+1}\Pint(\sum_j b_j;y,t)\big|\nonumber\\
&\lesssim\!\!\sum_j\!\!\int_{I_j}\!\!\!|b(z)|\!\cdot\!\big(\frac{1}{(t^2\!+\!|y\!-\!z|^2)^{\frac{n+1}{2}}}\!+\!\frac{t^2}{(t^2\!+\!|y\!-\!z|^2)^{\frac{n+3}{2}}}\big)dz\nonumber\\
&\lesssim\!\!\sum_j\!\!\int_{I_j}\!\!\!|b(z)|dz\!\cdot\!\big(\frac{1}{(t^2\!+\!|y\!-\!c(I_j)|^2)^{\frac{n+1}{2}}}\!+\!\frac{t^2}{(t^2\!+\!|y\!-\!c(I_j)|^2)^{\frac{n+3}{2}}}\big)\nonumber\\
&\lesssim\!\!\sum_j\!\alpha\Lebes^n(I_j)\!\cdot\!\big(\frac{1}{(t^2\!+\!|y\!-\!c(I_j)|^2)^{\frac{n+1}{2}}}\!+\!\frac{t^2}{(t^2\!+\!|y\!-\!c(I_j)|^2)^{\frac{n+3}{2}}}\big)\nonumber\\
&\lesssim\frac{\alpha}{t},\label{eq2-8}
\end{align}
where the constants in (\ref{eq2-8}) depend on $n,p$. The proof of Lemma \ref{lemma6} (1) is now complete.\\
		
		(2) Assume $1\leq k\leq n$, $I_m$ is a Whitney cube and $y\in I_m$, and if $I_j$ is another Whitney cube satisfying the relation (\ref{far}) between the two cubes $I_j$ and $I_m$, then we can use inequality (\ref{eq.lemma2-3}) of Lemma \ref{lemma2} for $y\in I_m$ and $z\in I_j$ to obtain the following estimate
		\begin{align}
			&\sum_{\substack{j\\I_j\text{ far from }I_m}}\Lebes^n(I_j)\cdot
			\frac{t|c(I_m)-c(I_j)|}{(t^2+|c(I_m)-c(I_j)|^2)^{\frac{n+3}{2}}}\nonumber\\
			&=\sum_{\substack{j\\I_j\text{ far from }I_m}}\int_{I_j}
			\frac{t|c(I_m)-c(I_j)|}{(t^2+|c(I_m)-c(I_j)|^2)^{\frac{n+3}{2}}}dz\nonumber\\
			&\lesssim\sum_{\substack{j\\I_j\text{ far from }I_m}}\int_{I_j}
			\frac{t|y-z|}{(t^2+|y-z|^2)^{\frac{n+3}{2}}}dz\nonumber\\
			&\lesssim\int_{\bbbr^n}\frac{t|y-z|}{(t^2+|y-z|^2)^{\frac{n+3}{2}}}dz\lesssim t^{-1},\label{eq2-9}
		\end{align}
		where the constants in (\ref{eq2-9}) depend on $n$. Applying (\ref{eq1-5}), inequality (\ref{eq.lemma2-3}) of Lemma \ref{lemma2}, Lemma \ref{lemma3} (vi), and (\ref{eq2-9}) yields the following
		\begin{align}
			&\big|\partial_k\Pint(\sum_{\substack{j\\I_j\text{ far from }I_m}} b_j;y,t)\big|\nonumber\\
			&\lesssim\sum_{\substack{j\\I_j\text{ far from }I_m}}\int_{I_j}|b(z)|\cdot
			\frac{t|y-z|}{(t^2+|y-z|^2)^{\frac{n+3}{2}}}dz\nonumber\\
			&\lesssim\sum_{\substack{j\\I_j\text{ far from }I_m}}\int_{I_j}|b(z)|dz\cdot
			\frac{t|c(I_m)-c(I_j)|}{(t^2+|c(I_m)-c(I_j)|^2)^{\frac{n+3}{2}}}\nonumber\\
			&\lesssim\sum_{\substack{j\\I_j\text{ far from }I_m}}\alpha\Lebes^n(I_j)\cdot
			\frac{t|c(I_m)-c(I_j)|}{(t^2+|c(I_m)-c(I_j)|^2)^{\frac{n+3}{2}}}\lesssim\frac{\alpha}{t},\label{eq2-10}
		\end{align}
		where the constants in (\ref{eq2-10}) depend on $n,p$. Assume $k=n+1$, $I_m$ is a Whitney cube and $y\in I_m$, and if $I_j$ is another Whitney cube satisfying the relation (\ref{far}) between the two cubes $I_j$ and $I_m$, then we can use inequality (\ref{eq.lemma2-3}) of Lemma \ref{lemma2} for $y\in I_m$ and $z\in I_j$ to obtain the following estimate
		\begin{align}
			&\sum_{\substack{j\\I_j\text{ far from }I_m}}
			\Lebes^n(I_j)\big(\frac{1}{(t^2+|c(I_m)-c(I_j)|^2)^{\frac{n+1}{2}}}\nonumber\\
			&\quad+\frac{t^2}{(t^2+|c(I_m)-c(I_j)|^2)^{\frac{n+3}{2}}}\big)\nonumber\\
			&=\sum_{\substack{j\\I_j\text{ far from }I_m}}
			\int_{I_j}\frac{1}{(t^2+|c(I_m)-c(I_j)|^2)^{\frac{n+1}{2}}}\nonumber\\
			&\quad+\frac{t^2}{(t^2+|c(I_m)-c(I_j)|^2)^{\frac{n+3}{2}}}dz\nonumber\\
			&\lesssim\sum_{\substack{j\\I_j\text{ far from }I_m}}\int_{I_j}\frac{1}{(t^2+|y-z|^2)^{\frac{n+1}{2}}}
			+\frac{t^2}{(t^2+|y-z|^2)^{\frac{n+3}{2}}}dz\nonumber\\
			&\lesssim\int_{\bbbr^n}\frac{1}{(t^2+|y-z|^2)^{\frac{n+1}{2}}}
			+\frac{t^2}{(t^2+|y-z|^2)^{\frac{n+3}{2}}}dz\lesssim t^{-1},\label{eq2-11}
		\end{align}
		where the constants in (\ref{eq2-11}) depend on the dimension $n$. Applying (\ref{eq1-6}), inequality (\ref{eq.lemma2-3}) of Lemma \ref{lemma2}, Lemma \ref{lemma3} (vi), and (\ref{eq2-11}) yields
		\begin{align}
			&\big|\partial_{n+1}\Pint(\sum_{\substack{j\\I_j\text{ far from }I_m}} b_j;y,t)\big|\nonumber\\
			&\lesssim\!\!\sum_{\substack{j\\I_j\text{ far from }I_m}}\!\!\!\!\!\!\!\int_{I_j}\!\!|b(z)|
			\big\{\frac{1}{(t^2+|y-z|^2)^{\frac{n+1}{2}}}+\frac{t^2}{(t^2+|y-z|^2)^{\frac{n+3}{2}}}\big\}dz\nonumber\\
			&\lesssim\sum_{\substack{j\\I_j\text{ far from }I_m}}\int_{I_j}|b(z)|dz\cdot
			\big\{\frac{1}{(t^2+|c(I_m)-c(I_j)|^2)^{\frac{n+1}{2}}}\nonumber\\
			&\quad+\frac{t^2}{(t^2+|c(I_m)-c(I_j)|^2)^{\frac{n+3}{2}}}\big\}\nonumber\\
			&\lesssim\sum_{\substack{j\\I_j\text{ far from }I_m}}\alpha\Lebes^n(I_j)\cdot
			\big\{\frac{1}{(t^2+|c(I_m)-c(I_j)|^2)^{\frac{n+1}{2}}}\nonumber\\
			&\quad+\frac{t^2}{(t^2+|c(I_m)-c(I_j)|^2)^{\frac{n+3}{2}}}\big\}\lesssim\frac{\alpha}{t},\label{eq2-12}
		\end{align}
		where the constants in (\ref{eq2-12}) depend on $n,p$. The proof of Lemma \ref{lemma6} (2) is now complete.
	\end{proof}
\begin{lemma}\label{lemma7}
Let $f\in L^p(\bbbr^n)$ for $1\leq p<\infty$, and assume $0<\alpha<\esssup_{x\in\bbbr^n}|f(x)|$ and $\Omega=\{x\in\bbbr^n:M(|f|^p)(x)>\alpha^p\}$ are given by Lemma \ref{lemma3}. And let $\{I_j\}_j$ be the associated Whitney cubes and $f=g+b$ as in Lemma \ref{lemma3}. Let $\chi_{I_j}$ denote the characteristic function of the cube $I_j$ and $b_j=b\cdot\chi_{I_j}$. Then for $1\leq k\leq n+1$, we have the following estimates,
\begin{align}
\int_0^{\infty}\!\!\!\!\int_{\Omega^c}\!\!\big|\partial_k\Pint(\sum_j b_j;y,t)\big|dydt
&\!\lesssim\!\sum_j\int_{I_j}|b(z)|dz,\label{eq2-13}\\
\sum_m\!\!\int_0^{\infty}\!\!\!\!\int_{I_m}\!\!\big|\partial_k\Pint(\sum_{\substack{j\\I_j\text{ far from }I_m}} b_j;y,t)\big|dydt
&\!\lesssim\!\sum_j\int_{I_j}|b(z)|dz,\label{eq2-14}
\end{align}
and the constants in (\ref{eq2-13}) and (\ref{eq2-14}) only depend on $n$.
\end{lemma}
\begin{proof}[Proof of Lemma \ref{lemma7}]
Notice that by Lemma \ref{lemma3} (vi) and (i), we have
\begin{equation*}
\sum_j\int_{I_j}|b(z)|dz\lesssim\alpha^{1-p}\|f\|_{L^p(\bbbr^n)}^p<\infty,
\end{equation*}
with the implicit constant depending on $n$ and $p$. First we prove (\ref{eq2-13}) and (\ref{eq2-14}) when $1\leq k\leq n$. Due to Lemma \ref{lemma3} (vii) and (\ref{eq1-5}), $\int_{I_j}b(z)dz=0$ and we can write $\partial_k\Pint(b_j;y,t)$ identically as
		\begin{equation}\label{eq2-15}
			(n+1)C_0\!\!\int_{I_j}\!\!b(z)\big\{\frac{t(y_k-c(I_j)_k)}{(t^2+|y-c(I_j)|^2)^{\frac{n+3}{2}}}
			-\frac{t(y_k-z_k)}{(t^2+|y-z|^2)^{\frac{n+3}{2}}}\big\}dz,
		\end{equation}
		where $1\leq k\leq n$, and $y_k$, $z_k$, $c(I_j)_k$ are the $k$-th coordinates of $y$, $z$, $c(I_j)$ in $\bbbr^n$, respectively. Since the Whitney cubes have disjoint interiors, then by applying (\ref{eq2-15}) and by exchanging the order of integration, we can rewrite the left side of (\ref{eq2-13}) equivalently as below
		\begin{align}
			&\int_0^{\infty}\int_{\Omega^c}\big|\sum_j\partial_k\Pint(b_j;y,t)\big|dydt\nonumber\\
			&\lesssim\sum_j\int_{I_j}|b(z)|\int_0^{\infty}\int_{\Omega^c}
			\big|\frac{t(y_k-c(I_j)_k)}{(t^2+|y-c(I_j)|^2)^{\frac{n+3}{2}}}\nonumber\\
			&\quad-\frac{t(y_k-z_k)}{(t^2+|y-z|^2)^{\frac{n+3}{2}}}\big|dydtdz,\label{eq2-16}
		\end{align}
		where the constant depends on $n$. If we use the change of variable $w=y-c(I_j)$ for $y\in\Omega^c$ and if $z\in I_j$, then from Lemma \ref{lemma1} (b) and the fact that the $n$-dimensional ball centered at $c(I_j)$ and inscribed in $I_j$ has radius $\frac{1}{2}l(I_j)$, we can deduce the inequality $|z-c(I_j)|\leq\frac{\sqrt{n}}{2}l(I_j)$. Furthermore, we can obtain
		\begin{equation}\label{eq2-17}
			|w|\geq\dist(I_j,\Omega^c)+\frac{1}{2}l(I_j)\geq(\sqrt{n}+\frac{1}{2})l(I_j)>2|z-c(I_j)|.
		\end{equation}
		Now we can invoke Lemma \ref{lemma4} (\ref{eq.lemma4-1}) with $\beta=2$ and obtain the following estimate for the inside double integral of (\ref{eq2-16}) as below
		\begin{align}
			&\int_0^{\infty}\int_{\Omega^c}\big|\frac{t(y_k-c(I_j)_k)}{(t^2+|y-c(I_j)|^2)^{\frac{n+3}{2}}} -\frac{t(y_k-z_k)}{(t^2+|y-z|^2)^{\frac{n+3}{2}}}\big|dydt\nonumber\\
			&\leq\!\!\!\int_0^{\infty}\!\!\!\!\int_{|w|>2|z-c(I_j)|}\!\big|\frac{t w_k}{(t^2+|w|^2)^{\frac{n+3}{2}}} \!-\!\frac{t(w_k-z_k+c(I_j)_k)}{(t^2+|w-z+c(I_j)|^2)^{\frac{n+3}{2}}}\big|dwdt\nonumber\\
			&\lesssim1,\label{eq2-18}
		\end{align}
		where the constant in (\ref{eq2-18}) only depends on $n$. Inserting (\ref{eq2-18}) into (\ref{eq2-16}) yields the estimate (\ref{eq2-13}) for $1\leq k\leq n$. To prove (\ref{eq2-14}) for $1\leq k\leq n$, we can still use (\ref{eq2-15}) and exchange the order of integration and the summation over $j$ to estimate the left side of (\ref{eq2-14}) from above by the following
		\begin{align}
			&\sum_m\sum_{\substack{j\\I_j\text{ far from }I_m}}\int_0^{\infty}\int_{I_m}\big|\partial_k\Pint( b_j;y,t)\big|dydt\nonumber\\
			&\lesssim\sum_m\sum_{\substack{j\\I_j\text{ far from }I_m}}\int_{I_j}|b(z)|\int_0^{\infty}\int_{I_m}
			\big|\frac{t(y_k-c(I_j)_k)}{(t^2+|y-c(I_j)|^2)^{\frac{n+3}{2}}}\nonumber\\
			&\quad-\frac{t(y_k-z_k)}{(t^2+|y-z|^2)^{\frac{n+3}{2}}}\big|dydtdz,\label{eq2-19}
		\end{align}
		and the constant depends on $n$. Notice that after exchanging the order of the summation over $m$ and the summation over $j$ in the above expression (\ref{eq2-19}), we can rewrite (\ref{eq2-19}) identically as the expression below
		\begin{equation}\label{eq2-20}
			\sum_j\!\!\int_{I_j}\!\!|b(z)|\!\!\int_0^{\infty}\!\!\!\!\int_{U_j}\!\!\big|\frac{t(y_k-c(I_j)_k)}{(t^2+|y-c(I_j)|^2)^{\frac{n+3}{2}}}-\frac{t(y_k-z_k)}{(t^2+|y-z|^2)^{\frac{n+3}{2}}}\big|dydtdz,
		\end{equation}
		where the set $U_j$ is the union of all the Whitney cube $I_m$'s that satisfy the condition ``$I_j$ far from $I_m$'' for some fixed $j$. By the definition (\ref{far}), when $y\in U_j$, $y$ belongs to some Whitney cube $I_m$ and the centers of $I_m$ and $I_j$ satisfy the condition $|c(I_m)-c(I_j)|>11(\sqrt{n}+1)l(I_m)$. By Lemma \ref{lemma2} (1), we can further deduce that $|y-c(I_j)|\geq\frac{\sqrt{n}+1}{2}\cdot l(I_j)$. Moreover, if $z\in I_j$ for this $j$ then $|z-c(I_j)|\leq\frac{\sqrt{n}}{2}\cdot l(I_j)$. From the above analysis, we can conclude that
		\begin{equation}\label{eq2-21}
			|y-c(I_j)|\geq(1+\frac{1}{\sqrt{n}})|z-c(I_j)|
		\end{equation}
		if $y\in U_j$ and $z\in I_j$ for some fixed $j$. Therefore applying the change of variable $w=y-c(I_j)$, the above conclusion (\ref{eq2-21}) and Lemma \ref{lemma4} (\ref{eq.lemma4-1}) with $\beta=1+\frac{1}{\sqrt{n}}$, we can estimate the inside double integral of (\ref{eq2-20}) as follows
		\begin{align}
			&\int_0^{\infty}\!\!\!\!\int_{U_j}\big|\frac{t(y_k-c(I_j)_k)}{(t^2+|y-c(I_j)|^2)^{\frac{n+3}{2}}}-\frac{t(y_k-z_k)}{(t^2+|y-z|^2)^{\frac{n+3}{2}}}\big|dydt\nonumber\\
			&\lesssim\int_0^{\infty}\!\!\!\!\int_{\bar{U}_j}\big|\frac{tw_k}{(t^2+|w|^2)^{\frac{n+3}{2}}}-\frac{t[w_k-(z_k-c(I_j)_k)]}{(t^2+|w-(z-c(I_j))|^2)^{\frac{n+3}{2}}}\big|dwdt\nonumber\\
			&\lesssim 1,\label{eq2-22}
		\end{align}
		where $\bar{U}_j=\{w\in\bbbr^n:|w|\geq(1+\frac{1}{\sqrt{n}})\cdot|z-c(I_j)|\}$ and the constants in (\ref{eq2-22}) depend on $n$. Inserting (\ref{eq2-22}) into (\ref{eq2-20}) and combining the result with (\ref{eq2-19}) yield the estimate (\ref{eq2-14}) for $1\leq k\leq n$.\\
		
		Second, we prove (\ref{eq2-13}) and (\ref{eq2-14}) when $k=n+1$. Since $\int_{I_j}b(z)dz=0$, we can use formula (\ref{eq1-6}) to rewrite $\partial_{n+1}\Pint(b_j;y,t)$ equivalently as the sum of the following two terms,
		\begin{equation}\label{eq2-23}
			\int_{I_j}b(z)\cdot\big(\frac{C_0}{(t^2+|y-z|^2)^{\frac{n+1}{2}}}
			-\frac{C_0}{(t^2+|y-c(I_j)|^2)^{\frac{n+1}{2}}}\big)dz
		\end{equation}
		and
		\begin{equation}\label{eq2-24}
			\int_{I_j}b(z)\cdot\big(\frac{(n+1)C_0 t^2}{(t^2+|y-c(I_j)|^2)^{\frac{n+3}{2}}}
			-\frac{(n+1)C_0 t^2}{(t^2+|y-z|^2)^{\frac{n+3}{2}}}\big)dz.
		\end{equation}
		Inserting (\ref{eq2-23}) and (\ref{eq2-24}) into the left side of (\ref{eq2-13}) yields that the term $$\int_0^{\infty}\int_{\Omega^c}\big|\partial_{n+1}\Pint(\sum_j b_j;y,t)\big|dydt$$
		can be estimated from above by a constant multiple of the sum of the following two terms,
		\begin{equation}\label{eq2-25}
			\sum_j\!\!\int_{I_j}\!\!|b(z)|\!\!\int_0^{\infty}\!\!\!\!\int_{\Omega^c}\!\!\big|
			\frac{1}{(t^2+|y-z|^2)^{\frac{n+1}{2}}}-\frac{1}{(t^2+|y-c(I_j)|^2)^{\frac{n+1}{2}}}\big|dydtdz
		\end{equation}
		and
		\begin{equation}\label{eq2-26}
			\sum_j\!\!\int_{I_j}\!\!|b(z)|\!\!\int_0^{\infty}\!\!\!\!\int_{\Omega^c}\!\!\big|
			\frac{t^2}{(t^2+|y-c(I_j)|^2)^{\frac{n+3}{2}}}-\frac{t^2}{(t^2+|y-z|^2)^{\frac{n+3}{2}}}\big|dydtdz,
		\end{equation}
		and the constant depends on $n$. If we use the change of variable $w=y-c(I_j)$ for $y\in\Omega^c$ and if $z\in I_j$, then with the same reason we still have the inequality (\ref{eq2-17}). Applying Lemma \ref{lemma4} (\ref{eq.lemma4-2}) with $\beta=2$, we can estimate the inside double integral of (\ref{eq2-25}) as follows
		\begin{align}
			&\int_0^{\infty}\int_{\Omega^c}\big|\frac{1}{(t^2+|y-z|^2)^{\frac{n+1}{2}}}
			-\frac{1}{(t^2+|y-c(I_j)|^2)^{\frac{n+1}{2}}}\big|dydt\nonumber\\
			&\lesssim\!\!\int_0^{\infty}\!\!\!\!\int_{|w|>2|z-c(I_j)|}\!\!\!\!\!\!\!\!\!\!\!\!\!\!\!\!\!\!\!\!|(t^2+|w-(z-c(I_j))|^2)^{-\frac{n+1}{2}}\!-\!(t^2+|w|^2)^{-\frac{n+1}{2}}|dwdt\nonumber\\
			&\lesssim 1,\label{eq2-27}
		\end{align}
		where the constants in (\ref{eq2-27}) only depend on $n$. Applying Lemma \ref{lemma4} (\ref{eq.lemma4-3}) with $\beta=2$, we can estimate the inside double integral of (\ref{eq2-26}) as follows
		\begin{align}
			&\int_0^{\infty}\int_{\Omega^c}\big|\frac{t^2}{(t^2+|y-c(I_j)|^2)^{\frac{n+3}{2}}}
			-\frac{t^2}{(t^2+|y-z|^2)^{\frac{n+3}{2}}}\big|dydt\nonumber\\
			&\lesssim\!\!\int_0^{\infty}\!\!\!\!\int_{|w|>2|z-c(I_j)|}\!\!\!\!\!\!\!\!\!\!\!\!\!\!\!\!\!\!\!\!|t^2(t^2+|w|^2)^{-\frac{n+3}{2}}
			\!-\!t^2(t^2+|w-(z-c(I_j))|^2)^{-\frac{n+3}{2}}|dwdt\nonumber\\
			&\lesssim 1,\label{eq2-28}
		\end{align}
		and the constants in (\ref{eq2-28}) depend on $n$. Inserting (\ref{eq2-27}), (\ref{eq2-28}) into (\ref{eq2-25}), (\ref{eq2-26}), respectively, and combining the results with the left side of (\ref{eq2-13}) yield the estimate (\ref{eq2-13}) for $k=n+1$. To prove estimate (\ref{eq2-14}) when $k=n+1$, we still rewrite $\partial_{n+1}\Pint(b_j;y,t)$ equivalently as the sum of (\ref{eq2-23}) and (\ref{eq2-24}), and estimate the term $$\sum_m\int_0^{\infty}\int_{I_m}\big|\partial_{n+1}\Pint(\sum_{\substack{j\\I_j\text{ far from }I_m}} b_j;y,t)\big|dydt$$ from above by a constant multiple of the sum of the following two terms,
		\begin{align}
			&\sum_m\sum_{\substack{j\\I_j\text{ far from }I_m}}\int_{I_j}|b(z)|\nonumber\\
			&\cdot\int_0^{\infty}\!\!\!\!\int_{I_m}\big|\frac{1}{(t^2+|y-z|^2)^{\frac{n+1}{2}}}
			-\frac{1}{(t^2+|y-c(I_j)|^2)^{\frac{n+1}{2}}}\big|dydtdz\label{eq2-29}
		\end{align}
		and
		\begin{align}
			&\sum_m\sum_{\substack{j\\I_j\text{ far from }I_m}}\int_{I_j}|b(z)|\nonumber\\
			&\cdot\int_0^{\infty}\!\!\!\!\int_{I_m}\big|\frac{t^2}{(t^2+|y-c(I_j)|^2)^{\frac{n+3}{2}}}
			-\frac{t^2}{(t^2+|y-z|^2)^{\frac{n+3}{2}}}\big|dydtdz,\label{eq2-30}
		\end{align}
and the constant depends on $n$. Notice that after exchanging the order of the summation over $m$ and the summation over $j$ in the two expressions above, we can rewrite (\ref{eq2-29}) and (\ref{eq2-30}) identically as the expressions below,
\begin{equation}\label{eq2-31}
\sum_j\!\!\int_{I_j}\!\!|b(z)|\!\!\int_0^{\infty}\!\!\!\!\int_{U_j}\!\!\big|\frac{1}{(t^2\!+\!|y\!-\!z|^2)^{\frac{n+1}{2}}}\!-\!\frac{1}{(t^2\!+\!|y\!-\!c(I_j)|^2)^{\frac{n+1}{2}}}\big|dydtdz
\end{equation}
and
\begin{equation}\label{eq2-32}
\sum_j\!\!\int_{I_j}\!\!|b(z)|\!\!\int_0^{\infty}\!\!\!\!\int_{U_j}\!\!\big|\frac{t^2}{(t^2\!+\!|y\!-\!c(I_j)|^2)^{\frac{n+3}{2}}}\!-\!\frac{t^2}{(t^2\!+\!|y\!-\!z|^2)^{\frac{n+3}{2}}}\big|dydtdz,
\end{equation}
where $U_j$ is the union of all the Whitney cube $I_m$'s that satisfy the condition ``$I_j$ far from $I_m$'' for some fixed $j$. By the definition (\ref{far}), when $y\in U_j$, $y$ belongs to some Whitney cube $I_m$ and the centers of $I_m$ and $I_j$ satisfy the condition $|c(I_m)-c(I_j)|>11(\sqrt{n}+1)l(I_m)$. By inequality (\ref{eq.lemma2-1}) of Lemma \ref{lemma2}, we can deduce that $|y-c(I_j)|\geq\frac{\sqrt{n}+1}{2}\cdot l(I_j)$. Moreover, if $z\in I_j$ for this $j$ then $|z-c(I_j)|\leq\frac{\sqrt{n}}{2}\cdot l(I_j)$. From the above analysis, we can conclude that (\ref{eq2-21}) is still true for the variables $y$ and $z$ appearing in (\ref{eq2-31}) and (\ref{eq2-32}). Recall that $$\bar{U}_j=\{w\in\bbbr^n:|w|\geq(1+\frac{1}{\sqrt{n}})|z-c(I_j)|\}.$$
Applying the change of variable $w=y-c(I_j)$ and Lemma \ref{lemma4} (\ref{eq.lemma4-2}) with $\beta=1+\frac{1}{\sqrt{n}}$, we can estimate the inside double integral of (\ref{eq2-31}) as below
		\begin{align}
			&\int_0^{\infty}\int_{U_j}\big|\frac{1}{(t^2+|y-z|^2)^{\frac{n+1}{2}}}
			-\frac{1}{(t^2+|y-c(I_j)|^2)^{\frac{n+1}{2}}}\big|dydt\nonumber\\
			&\lesssim\!\!\int_0^{\infty}\!\!\!\int_{\bar{U}_j}\!\big|\frac{1}{(t^2+|w-(z-c(I_j))|^2)^{\frac{n+1}{2}}}
			-\frac{1}{(t^2+|w|^2)^{\frac{n+1}{2}}}\big|dwdt\nonumber\\
			&\lesssim 1,\label{eq2-33}
		\end{align}
		where the constants in (\ref{eq2-33}) depend on $n$. Applying the change of variable $w=y-c(I_j)$ and Lemma \ref{lemma4} (\ref{eq.lemma4-3}) with $\beta=1+\frac{1}{\sqrt{n}}$, we can estimate the inside double integral of (\ref{eq2-32}) as below
		\begin{align}
			&\int_0^{\infty}\int_{U_j}\big|\frac{t^2}{(t^2+|y-c(I_j)|^2)^{\frac{n+3}{2}}}
			-\frac{t^2}{(t^2+|y-z|^2)^{\frac{n+3}{2}}}\big|dydt\nonumber\\
			&\lesssim\!\!\int_0^{\infty}\!\!\!\int_{\bar{U}_j}\!\big|\frac{t^2}{(t^2+|w|^2)^{\frac{n+3}{2}}}
			\!-\!\frac{t^2}{(t^2+|w-(z-c(I_j))|^2)^{\frac{n+3}{2}}}\big|dwdt\nonumber\\
			&\lesssim 1,\label{eq2-34}
		\end{align}
		and the constants in (\ref{eq2-34}) depend on $n$. Inserting (\ref{eq2-33}), (\ref{eq2-34}) into (\ref{eq2-31}), (\ref{eq2-32}), respectively, and combining the results with (\ref{eq2-29}), (\ref{eq2-30}) yield the estimate (\ref{eq2-14}) for $k=n+1$. Now we finish the proof of Lemma \ref{lemma7}.
	\end{proof}
We rephrase part of \cite[Theorem 6.1.2.]{14classical} with the language from the theory of functional and function representatives.
\begin{lemma}\label{lemma12}
Assume that $1<p<\infty$, and we have the following conclusions.\\
(i) If $f\in\Sw_0'(\bbbr^n)$ has a function representative in the sense of $\Sw_0'(\bbbr^n)$ and this function representative is in $L^p(\bbbr^n)$, then $f\in\dot{F}^0_{p,2}(\bbbr^n)$ and we have
\begin{equation}\label{eq2.170}
\|f\|_{\dot{F}^0_{p,2}(\bbbr^n)}\lesssim\|f\|_{L^p(\bbbr^n)}.
\end{equation}
(ii) If $f\in\dot{F}^0_{p,2}(\bbbr^n)$, then $f$ has a function representative in the sense of $\Sw_0'(\bbbr^n)$ and this function representative is in $L^p(\bbbr^n)$, and we have
\begin{equation}\label{eq2.171}
\|f\|_{L^p(\bbbr^n)}\lesssim\|f\|_{\dot{F}^0_{p,2}(\bbbr^n)}.
\end{equation}
The implicit constants in (\ref{eq2.170}) and (\ref{eq2.171}) depend on $n$, $p$, and $\psi$, where $\psi\in\Sw_0(\bbbr^n)$ is the fixed Schwartz function satisfying conditions (\ref{eq1-7}), (\ref{eq1-8}), and (\ref{eq1.82}). Therefore, when $1<p<\infty$, the condition $f\in\dot{F}^0_{p,2}(\bbbr^n)$ implies
\begin{equation}\label{eq2.219}
\|f\|_{\dot{F}^0_{p,2}(\bbbr^n)}\sim\|f\|_{L^p(\bbbr^n)}.
\end{equation}
\end{lemma}
\begin{proof}[Proof of Lemma \ref{lemma12}]
Step 1: The proof is very much alike to \cite[the proof of Theorem 6.1.2.]{14classical}, therefore we will only sketch it. To prove Lemma \ref{lemma12} (i), we assume the function $F(x)\in\functrep_{0}(f)\cap L^p(\bbbr^n)$. Because the function (\ref{eq2.220}) is in $\Sw_{0}(\bbbr^n)$ for every $j\in\bbbz$ and $x\in\bbbr^n$,
\begin{equation}\label{eq2.220}
y\in\bbbr^n\longmapsto\psi_{2^{-j}}(x-y),
\end{equation}
and we recall the function $F_j(x)$ defined in (\ref{eq1.117}), then we have
\begin{equation}\label{eq2.221}
F_j(x)\!=<\!f,\psi_{2^{-j}}(x\!-\!\cdot)\!>=\!\!\!\int_{\bbbr^n}\!\!\!\!\!F(y)\!\cdot\!\psi_{2^{-j}}
(x\!-\!y)dy\!=\!F\!*\!\psi_{2^{-j}}(x),
\end{equation}
and $F_j(x)\in\functrep(f*\psi_{2^{-j}})\cap L^p(\bbbr^n)$. Furthermore, Definition \ref{definition1} implies
\begin{align}
&\|f\|_{\dot{F}^0_{p,2}(\bbbr^n)}=\big(\int_{\bbbr^n}\big(\sum_{j\in\bbbz}|f*\psi_{2^{-j}}(x)|^2\big)
^{\frac{p}{2}}dx\big)^{\frac{1}{p}}\nonumber\\
&=\big(\int_{\bbbr^n}\big(\sum_{j\in\bbbz}|F_j(x)|^2\big)^{\frac{p}{2}}dx\big)^{\frac{1}{p}}.
\label{eq2.226}
\end{align}
We first prove Lemma \ref{lemma12} (i) when $p=2$. Let $g\in\Sw(\bbbr^n)$. Then the function $g*\psi_{2^{-j}}(x)$ is in $\Sw(\bbbr^n)$ and Plancherel's identity is applicable. Since the fixed function $\psi\in\Sw_0(\bbbr^n)$ with conditions (\ref{eq1-7}), (\ref{eq1-8}), and (\ref{eq1.82}), also satisfies \cite[(6.1.3)]{14classical}, we can use the argument given for \cite[(6.1.9), (6.1.10), (6.1.11), and (6.1.12)]{14classical} to deduce that
\begin{equation}\label{eq2.222}
\|g\|_{\dot{F}^0_{2,2}(\bbbr^n)}^2=\int_{\bbbr^n}\sum_{j\in\bbbz}|g*\psi_{2^{-j}}(x)|^2 dx\lesssim
\|g\|_{L^2(\bbbr^n)}^2,
\end{equation}
and the implicit constant in (\ref{eq2.222}) depends on $n,\psi$. Since the function $F(x)\in\functrep_{0}(f)\cap L^2(\bbbr^n)$ when $p=2$, we can find a sequence $\{g_u\}_{u\in\bbbn}\subseteq\Sw(\bbbr^n)$ so that $g_u\rightarrow F$ in $L^2(\bbbr^n)$ as $u\rightarrow\infty$. Then \cite[Theorem 1.2.10. (Minkowski's inequality)]{14classical} and (\ref{eq2.221}) imply for each fixed $j\in\bbbz$,
\begin{equation}\label{eq2.223}
\lim_{u\rightarrow\infty}\|F_j-g_u*\psi_{2^{-j}}\|_{L^2(\bbbr^n)}=0.
\end{equation}
Now we fix $N\in\bbbn_0$ and use the triangle inequality and (\ref{eq2.222}) to obtain that for all $u\in\bbbn$,
\begin{align}
&\int_{\bbbr^n}\sum_{|j|\leq N}|F_j(x)|^2 dx\nonumber\\
&\lesssim\!\!\!\int_{\bbbr^n}\sum_{|j|\leq N}\!|F_j(x)\!-\!g_u\!*\!\psi_{2^{-j}}(x)|^2
\!+\!\!\!\sum_{|j|\leq N}\!|g_u\!*\!\psi_{2^{-j}}(x)|^2 dx\nonumber\\
&\lesssim\!\!\!\sum_{|j|\leq N}\!\|F_j\!-\!g_u\!*\!\psi_{2^{-j}}\|_{L^2(\bbbr^n)}^2
\!+\!\!\!\int_{\bbbr^n}\sum_{j\in\bbbz}\!|g_u\!*\!\psi_{2^{-j}}(x)|^2 dx\nonumber\\
&\lesssim\sum_{|j|\leq N}\|F_j-g_u*\psi_{2^{-j}}\|_{L^2(\bbbr^n)}^2+\|g_u\|_{L^2(\bbbr^n)}^2,
\label{eq2.224}
\end{align}
where the implicit constants in (\ref{eq2.224}) depend on $n,\psi$ and do not depend on $u\in\bbbn$ and $N\in\bbbn_0$. First sending $u\rightarrow\infty$ and then sending $N\rightarrow\infty$ in the estimate (\ref{eq2.224}) yields
\begin{equation}\label{eq2.225}
\|f\|_{\dot{F}^0_{2,2}(\bbbr^n)}^2\!=\!\!\!\int_{\bbbr^n}\sum_{j\in\bbbz}\!|F_j(x)|^2 dx
\!\lesssim\!\|F\|_{L^2(\bbbr^n)}^2\!=\!\|f\|_{L^2(\bbbr^n)}^2
\end{equation}
under the condition that $f\in\Sw_0'(\bbbr^n)$ has a function representative $F(x)\in\functrep_{0}(f)\cap L^2(\bbbr^n)$. The proof of Lemma \ref{lemma12} (i) when $p=2$ is complete. For the proof of Lemma \ref{lemma12} (i) when $p\neq2$, we define the operator $\vec{T}$ below,
\begin{equation*}
\vec{T}(F)(x)\!\!=\!\!\big\{\!\!\!\int_{\bbbr^n}\!\!\!\!\!F(y)\!\cdot\!\psi_{2^{-j}}(x\!-\!y)dy
\big\}_{j\in\bbbz}\!\!=\!\!\big\{F\!*\!\psi_{2^{-j}}(x)\big\}_{j\in\bbbz},
\end{equation*}
and the rest of the proof will be the same as the proof for \cite[(6.1.4)]{14classical}, therefore we omit it.\\

Step 2: We prove Lemma \ref{lemma12} (ii). When $1<p<\infty$ and $f\in\dot{F}^0_{p,2}(\bbbr^n)\subseteq\Sw_{0}'(\bbbr^n)$, Proposition \ref{proposition1} (iv) implies $f\!*\!\psi_{2^{-j}}\in\Sw'(\bbbr^n)$ and the function $F_j(x)\in\functrep(f\!*\!\psi_{2^{-j}})$ for every $j\in\bbbz$, where $F_j(x)$ is defined in (\ref{eq1.117}). Associated to an arbitrary finite subset $\bbbz'$ of the set $\bbbz$ of all integers, we define the finite sum $S_{\bbbz'}(f)$ of tempered distributions by
\begin{equation}\label{eq2.227}
S_{\bbbz'}(f):=\sum_{j\in\bbbz'}f\!*\!\psi_{2^{-j}}.
\end{equation}
Let $g\in\Sw(\bbbr^n)$, then $g\!*\!\psi_{2^{-j}}\in\Sw_{0}(\bbbr^n)$ for all $j\in\bbbz$. We deduce from conditions (\ref{eq1-7}) and (\ref{eq1-8}) that
\begin{align}
&<f\!*\!\psi_{2^{-j}},g>=\sum_{l=-1}^1<f\!*\!\psi_{2^{-j}},g\!*\!\psi_{2^{-j-l}}>\nonumber\\
&=\sum_{l=-1}^1\int_{\bbbr^n}F_j(x)\cdot g\!*\!\psi_{2^{-j-l}}(x)dx.\label{eq2.228}
\end{align}
We use Cauchy-Schwartz inequality, H\"{o}lder's inequality with $\frac{1}{p}+\frac{1}{p'}=1$, and (\ref{eq2.226}) to obtain the following estimate, 
\begin{align}
&|<S_{\bbbz'}(f),g>|\leq\sum_{j\in\bbbz'}|<f\!*\!\psi_{2^{-j}},g>|\nonumber\\
&\leq\sum_{l=-1}^1\int_{\bbbr^n}\sum_{j\in\bbbz'}|F_j(x)|\cdot|g\!*\!\psi_{2^{-j-l}}(x)|dx\nonumber\\
&\leq\!\!\!\sum_{l=-1}^1\!\!\big(\!\!\int_{\bbbr^n}\!\!\!\big(\!\sum_{j\in\bbbz'}\!|F_j(x)|^2
\big)^{\!\frac{p}{2}}dx\!\big)^{\!\frac{1}{p}}\!\cdot\!\big(\!\!\int_{\bbbr^n}\!\!\!\big(\!
\sum_{j\in\bbbz'}\!|g\!*\!\psi_{2^{-j-l}}(x)|^2\big)^{\!\frac{p'}{2}}\!dx\big)^{\!\frac{1}{p'}}
\nonumber\\
&\leq\!\!\!\sum_{l=-1}^1\|f\|_{\dot{F}^0_{p,2}(\bbbr^n)}\!\cdot\!\big(\!\!\int_{\bbbr^n}\!\!\!\big(\!
\sum_{j\in\bbbz'}\!|g\!*\!\psi_{2^{-j-l}}(x)|^2\big)^{\!\frac{p'}{2}}\!dx\big)^{\!\frac{1}{p'}}.
\label{eq2.229}
\end{align}
Since $g\in\Sw(\bbbr^n)\subseteq L^{p'}(\bbbr^n)$ and $1<p'<\infty$, Lemma \ref{lemma12} (i) implies
\begin{equation}\label{eq2.230}
\big(\!\!\int_{\bbbr^n}\!\!\!\big(\!\sum_{j\in\bbbz}\!|g\!*\!\psi_{2^{-j}}(x)|^2\big)^{\!\frac{p'}{2}}
\!dx\big)^{\!\frac{1}{p'}}\lesssim\|g\|_{L^{p'}(\bbbr^n)}<\infty,
\end{equation}  
and the implicit constant in (\ref{eq2.230}) depends on $n,p,\psi$. Setting $\bbbz'=\{j\in\bbbz:N<|j|\leq M\}$ in (\ref{eq2.229}) for $N\in\bbbn_0$, $M\in\bbbn_0$, and $N<M$ yields the sequence
\begin{equation*}
\{<\sum_{|j|\leq N}f\!*\!\psi_{2^{-j}},g>\}_{N\in\bbbn_0}
\end{equation*}
is a Cauchy sequence of scalars and the limit in (\ref{eq2.231}) exists for every $g\in\Sw(\bbbr^n)$ and defines a linear functional on $\Sw(\bbbr^n)$,
\begin{equation}\label{eq2.231}
<S(f),g>:=\lim_{N\rightarrow\infty}<\sum_{|j|\leq N}f\!*\!\psi_{2^{-j}},g>.
\end{equation}
Replacing $\bbbz'$ in (\ref{eq2.229}) by the finite subset $\bbbz''=\{j\in\bbbz:|j|\leq N\}$ and using (\ref{eq2.230}) and (\ref{eq2.231}) yields that for all $g\in\Sw(\bbbr^n)$,
\begin{equation}\label{eq2.232}
|<S(f),g>|\lesssim\|f\|_{\dot{F}^0_{p,2}(\bbbr^n)}\cdot\|g\|_{L^{p'}(\bbbr^n)}.
\end{equation}
If $h$ is a function in $L^{p'}(\bbbr^n)$ and $\{g_u\}_{u\in\bbbn}\subseteq\Sw(\bbbr^n)$ is a sequence converging to $h$ in $L^{p'}(\bbbr^n)$, then the sequence $\{g_u\}_{u\in\bbbn}$ is Cauchy in the $\|\cdot\|_{L^{p'}(\bbbr^n)}$-norm. Applying (\ref{eq2.232}) to the function $(g_u-g_v)\in\Sw(\bbbr^n)$ yields that $\{<S(f),g_u>\}_{u\in\bbbn}$ is a Cauchy sequence of scalars, and we can define
\begin{equation}\label{eq2.233}
<S(f),h>:=\lim_{u\rightarrow\infty}<S(f),g_u>,
\end{equation}
and hence
\begin{equation}\label{eq2.234}
|\!<\!S(f),h\!>\!|\lesssim\|f\|_{\dot{F}^0_{p,2}(\bbbr^n)}\cdot\|h\|_{L^{p'}(\bbbr^n)},
\end{equation}
where the implicit constant in (\ref{eq2.234}) depends on $n,p,\psi$. Therefore the linear functional $S(f)$ defined in (\ref{eq2.231}) extends to a continuous linear functional on $L^{p'}(\bbbr^n)$, and by the Riesz representation theorem, there exists a function $F(x)\in L^p(\bbbr^n)$ such that the following equation (\ref{eq2.235}) is true for all $h\in L^{p'}(\bbbr^n)$,
\begin{equation}\label{eq2.235}
<S(f),h>=\int_{\bbbr^n}F(x)\cdot h(x)dx,
\end{equation} 
and
\begin{equation}\label{eq2.236}
\|F\|_{L^p(\bbbr^n)}\lesssim\|f\|_{\dot{F}^0_{p,2}(\bbbr^n)}.
\end{equation}
According to (\ref{eq1-9}), the sequence $\{\sum_{|j|\leq N}f\!*\!\psi_{2^{-j}}\}_{N\in\bbbn_0}$ of partial sums converges to $f\in\Sw_0'(\bbbr^n)$ in the sense of $\Sw_0'(\bbbr^n)$. Therefore for all $\phi\in\Sw_{0}(\bbbr^n)\subseteq\Sw(\bbbr^n)\subseteq L^{p'}(\bbbr^n)$, we deduce from (\ref{eq2.231}) and (\ref{eq2.235}) that
\begin{align}
&<f,\phi>=\lim_{N\rightarrow\infty}<\sum_{|j|\leq N}f\!*\!\psi_{2^{-j}},\phi>\nonumber\\
&=<S(f),\phi>=\int_{\bbbr^n}F(x)\cdot\phi(x)dx,\label{eq2.237}
\end{align}
thus the function $F(x)\in\functrep_{0}(f)$. The inequality (\ref{eq2.171}) is a consequence of (\ref{eq2.236}). The proof of Lemma \ref{lemma12} is complete.
\end{proof}
\begin{lemma}\label{lemma16}
For every $f\in\Sw_{0}'(\bbbr^n)$, $1\leq k\leq n$, and $0<t<\infty$, all of the following equations (\ref{eq2-35}), (\ref{eq2-36}), and (\ref{eq2.172}) are true in the sense of $\Sw_0'(\bbbr^n)$,
\begin{align}
f*\partial_{n+1}[P_t(\cdot)]&=\iFT_n[-2\pi|\xi| e^{-2\pi t|\xi|}\FT_n f],\label{eq2-35}\\
f*\partial_{k}[P_t(\cdot)]&=\iFT_n[2\pi i\xi_k e^{-2\pi t|\xi|}\FT_n f],\label{eq2-36}\\
f*P_t&=\iFT_n[e^{-2\pi t|\xi|}\FT_n f],\label{eq2.172}
\end{align}
where both sides of the equations (\ref{eq2-35}), (\ref{eq2-36}), and (\ref{eq2.172}) are tempered distributions, and $f$ appearing on the left sides of (\ref{eq2-35}), (\ref{eq2-36}), (\ref{eq2.172}) is the given continuous linear functional in $\Sw_{0}'(\bbbr^n)$, and $\FT_n f$ appearing on the right sides of (\ref{eq2-35}), (\ref{eq2-36}), (\ref{eq2.172}) is the distributional Fourier transform of the HB-extension of $f\in\Sw_{0}'(\bbbr^n)$. Furthermore, if $f\in\Sw_{0}'(\bbbr^n)$ has a function representative $f(x)\in\functrep_{0}(f)\cap L^{p_0}(\bbbr^n)$ for some $1\leq p_0\leq\infty$, then we have for every $1\leq k\leq n+1$,
\begin{align}
f*\partial_{k}[P_t(\cdot)](x)&\in\functrep_{0}(f*\partial_{k}[P_t(\cdot)])\cap L^{p_0}(\bbbr^n),
\label{eq2.238}\\
f*P_t(x)&\in\functrep_{0}(f*P_t)\cap L^{p_0}(\bbbr^n).\label{eq2.239}
\end{align}
Moreover, we deduce from formulae (\ref{eq1-4}), (\ref{eq1-5}), and (\ref{eq1-6}) that the functions $\Pint(f;x,t)$ and $\partial_k\Pint(f;x,t)$ are all in $L^{p_0}(\bbbr^n)$, thus the following integrals in (\ref{eq2.240}) and (\ref{eq2.241}) define tempered distributions denoted by the symbols $\Pint(f;x,t)$ and $\partial_k\Pint(f;x,t)$ respectively,
\begin{align}
&\int_{\bbbr^n}\Pint(f;x,t)\cdot g(x)dx,\label{eq2.240}\\
&\int_{\bbbr^n}\partial_k\Pint(f;x,t)\cdot g(x)dx,\label{eq2.241}
\end{align}
where $g\in\Sw(\bbbr^n)$ and $1\leq k\leq n+1$, and hence we have
\begin{align}
\Pint(f;x,t)&\in\functrep(\Pint(f;x,t)),\label{eq2.242}\\
\partial_k\Pint(f;x,t)&\in\functrep(\partial_k\Pint(f;x,t)).\label{eq2.243}
\end{align}
In addition, the equations (\ref{eq2.244}) and (\ref{eq2.245}) below hold true in the sense of $\Sw_{0}'(\bbbr^n)$ for all $1\leq k\leq n+1$,
\begin{align}
\Pint(f;x,t)&=f*P_t,\label{eq2.244}\\
\partial_k\Pint(f;x,t)&=f*\partial_{k}[P_t(\cdot)],\label{eq2.245}
\end{align}
and both sides of equations (\ref{eq2.244}) and (\ref{eq2.245}) are tempered distributions.
\end{lemma}
\begin{proof}[Proof of Lemma \ref{lemma16}]
Step 1: We prove equations (\ref{eq2-35}), (\ref{eq2-36}), and (\ref{eq2.172}). Let $f\in\Sw_{0}'(\bbbr^n)$. For $1\leq k\leq n$ and $0<t<\infty$, we denote
\begin{equation}\label{eq2.246}
\multi_{n+1}(\xi)=-2\pi|\xi| e^{-2\pi t|\xi|},
\end{equation}
\begin{equation}\label{eq2.247}
\multi_{k}(\xi)=-2\pi i\xi_k e^{-2\pi t|\xi|},\quad\multi_{0}(\xi)=e^{-2\pi t|\xi|}.
\end{equation}
We recall (\ref{eq1-3}), (\ref{eq1-3-1}), (\ref{eq2-81}), (\ref{eq1-3-2}), (\ref{eq2-82}) and obtain
\begin{equation}\label{eq2.248}
\FT_n\multi_{0}(x)\!=\!P_t(x),\quad\FT_n\multi_{k}(x)\!=\!\partial_{k}[P_t(x)]\quad\text{for $1\!\leq\!k\!\leq\!n\!+\!1$}.
\end{equation}
And we also have that for all $0\leq k\leq n+1$, the function $\multi_{k}(\xi)\in L^1(\bbbr^n)$ satisfies condition (\ref{eq1.64}), and derivatives of $\multi_{k}(\xi)$ are not continuous at $\xi=0$. Therefore by Proposition \ref{proposition1} (ii), all the left sides of the equations (\ref{eq2-35}), (\ref{eq2-36}), (\ref{eq2.172}) are well-defined continuous linear functionals in $\Sw_{0}'(\bbbr^n)$ and extend to tempered distributions defined on all of $\Sw(\bbbr^n)$ by the Hahn-Banach theorem. Furthermore, equations (\ref{eq1.111}) and (\ref{eq1.112}) imply for all $\phi\in\Sw_{0}(\bbbr^n)$ and $0\leq k\leq n+1$,
\begin{align}
&<f*\FT_n\multi_{k},\phi>=<f,(\iFT_n\multi_{k})*\phi>\nonumber\\
&=<f,\iFT_n[\multi_k(\xi)\FT_n\phi(\xi)]>.\label{eq2.249}
\end{align}
The distributional Fourier transform $\FT_n f$ of the HB-extension of $f\in\Sw_{0}'(\bbbr^n)$ is a tempered distribution in $\Sw'(\bbbr^n)$, and we can write the right sides of (\ref{eq2-35}), (\ref{eq2-36}), (\ref{eq2.172}) equivalently as follows,
\begin{align}
\iFT_n[-2\pi|\xi| e^{-2\pi t|\xi|}\FT_n f]&=\iFT_n[\multi_{n+1}(\xi)\FT_n f],\label{eq2.250}\\
\iFT_n[2\pi i\xi_k e^{-2\pi t|\xi|}\FT_n f]&=-\iFT_n[\multi_{k}(\xi)\FT_n f],\label{eq2.251}\\
\iFT_n[e^{-2\pi t|\xi|}\FT_n f]&=\iFT_n[\multi_{0}(\xi)\FT_n f],\label{eq2.252}
\end{align}
where $1\leq k\leq n$ in (\ref{eq2.251}). Proposition \ref{proposition1} (i) indicates that for every $0\leq k\leq n+1$, the linear functional 
\begin{equation}\label{eq2.196}
\multi_{k}(\xi)\cdot\FT_n f,
\end{equation}
which is initially defined on $\Sw_{00}(\bbbr^n)$ by the following equation
\begin{equation}\label{eq1-55}
<\!\multi_{k}(\xi)\!\cdot\!\FT_n f,g\!>=<\!\FT_n f,\multi_{k}\!\cdot\!g\!>\quad\text{for $g\!\in\!\Sw_{00}(\bbbr^n)$},
\end{equation}
is a continuous linear functional on $\Sw_{00}(\bbbr^n)$ in the topology of $\Sw(\bbbr^n)$. According to the Hahn-Banach theorem, i.e. \cite[3.3 Theorem]{Rudin.Funct.Anal}, this linear functional defined in (\ref{eq2.196}) and (\ref{eq1-55}) has a continuous linear extension on all of $\Sw(\bbbr^n)$ in the topology of $\Sw(\bbbr^n)$, and we still denote the extended tempered distribution by (\ref{eq2.196}). The action of the extended tempered distribution (\ref{eq2.196}) on functions in $\Sw_{00}(\bbbr^n)$ is given by (\ref{eq1-55}). And the distributional inverse Fourier transform of the extended tempered distribution (\ref{eq2.196}), as given in \cite[Definition 2.3.7.]{14classical}, is well-defined by the following equation
\begin{equation}\label{eq1-56}
<\!\iFT_n[\multi_{k}(\xi)\!\cdot\!\FT_n f],\phi\!>=<\!\multi_{k}(\xi)\!\cdot\!\FT_n f,\iFT_n\phi\!>
\end{equation}
for functions $\phi\in\Sw(\bbbr^n)$. Furthermore, if $\phi\in\Sw_0(\bbbr^n)$ then $\iFT_n\phi\in\Sw_{00}(\bbbr^n)$, and we can deduce from equation (\ref{eq1.67}) that for all $0\!\leq\!k\!\leq\!n+1$,
\begin{align}
&<\iFT_n[\multi_{k}(\xi)\!\cdot\!\FT_n f],\phi>\nonumber\\
&=<\FT_n f,\multi_{k}\cdot\iFT_n\phi>\nonumber\\
&=<f,\FT_n[\multi_{k}(\xi)\iFT_n\phi(\xi)]>,\nonumber\\
&=<f,\iFT_n[\multi_{k}(-\xi)\FT_n\phi(\xi)]>,\label{eq1-57}
\end{align}
where the penultimate line of (\ref{eq1-57}) is due to the definition of distributional Fourier transforms. Since $\multi_{k}(-\xi)\cdot\FT_n\phi(\xi)\in\Sw_{00}(\bbbr^n)$ and $\iFT_n[\multi_{k}(-\xi)\FT_n\phi(\xi)]\in\Sw_{0}(\bbbr^n)$, the last line of (\ref{eq1-57}) is the action of the HB-extension of $f\in\Sw_{0}'(\bbbr^n)$ on the function $\iFT_n[\multi_{k}(-\xi)\FT_n\phi(\xi)]$ and coincides with the action of the original continuous linear functional $f\in\Sw_{0}'(\bbbr^n)$ on the function $\iFT_n[\multi_{k}(-\xi)\FT_n\phi(\xi)]$. Combining (\ref{eq2.249}), (\ref{eq2.250}), (\ref{eq2.251}), (\ref{eq2.252}), and (\ref{eq1-57}) all together with the fact (\ref{eq2.253}) proves equations (\ref{eq2-35}), (\ref{eq2-36}), and (\ref{eq2.172}) in the sense of $\Sw_0'(\bbbr^n)$,
\begin{equation}\label{eq2.253}
\multi_{k}(-\xi)=
\begin{cases}
\multi_{k}(\xi)&\quad\text{if $k\in\{0,n+1\}$},\\
-\multi_{k}(\xi)&\quad\text{if $1\leq k\leq n$}.
\end{cases}
\end{equation}

Step 2: If $f\in\Sw_{0}'(\bbbr^n)$ has a function representative $f(x)\in\functrep_{0}(f)\cap L^{p_0}(\bbbr^n)$ for some $1\leq p_0\leq\infty$ and $\phi\in\Sw_{0}(\bbbr^n)$, then we deduce from Proposition \ref{proposition1} (ii), (\ref{eq2.246}), (\ref{eq2.247}), and (\ref{eq2.248}) that for every $0\leq k\leq n+1$, the function $(\iFT_n\multi_{k})*\phi$ is in $\Sw_{0}(\bbbr^n)$, thus we invoke H\"{o}lder's inequality with $\frac{1}{p_0}+\frac{1}{p_0'}=1$ and Young's inequality to justify the exchange of the order of integration and obtain
\begin{align}
&<f*\FT_n\multi_{k},\phi>=<f,(\iFT_n\multi_{k})*\phi>\nonumber\\
&=\!\!\!\int_{\bbbr^n}\!\!\!\!f(x)\!\cdot\!(\iFT_n\multi_{k})\!*\!\phi(x)dx
\!=\!\!\!\int_{\bbbr^n}\!\!\!\!f\!*\!\FT_n\multi_{k}(x)\!\cdot\!\phi(x)dx.\label{eq2.254}
\end{align}
Therefore (\ref{eq2.238}) and (\ref{eq2.239}) are proven, and $f$ appearing on the left sides of (\ref{eq2.238}) and (\ref{eq2.239}) is the function representative $f(x)\in\functrep_{0}(f)\cap L^{p_0}(\bbbr^n)$. Formulae (\ref{eq1-4}), (\ref{eq1-5}), (\ref{eq1-6}), and Young's inequality imply the functions $\Pint(f;x,t)$ and $\partial_k\Pint(f;x,t)$ for $1\leq k\leq n+1$ are all in $L^{p_0}(\bbbr^n)$, thus an application of H\"{o}lder's inequality with $\frac{1}{p_0}+\frac{1}{p_0'}=1$ and the estimate (\ref{eq2.255}) indicate the integrals in (\ref{eq2.240}) and (\ref{eq2.241}) define tempered distributions denoted by the symbols $\Pint(f;x,t)$ and $\partial_k\Pint(f;x,t)$ respectively, and (\ref{eq2.242}) and (\ref{eq2.243}) are proven,
\begin{equation}\label{eq2.255}
\|g\|_{L^{p_0'}(\bbbr^n)}\lesssim\rho_N(g)\quad\text{for $g\in\Sw(\bbbr^n)$},
\end{equation}
where $\rho_N(g)$ is defined in (\ref{eq1.89}), and the implicit constant and the number $N\in\bbbn_0$ in (\ref{eq2.255}) depend on $n,p_0$ and do not depend on $g$. Combining (\ref{eq1-4}), (\ref{eq2.239}), and (\ref{eq2.242}) yields that
\begin{align}
&\Pint(f;x,t)=f*P_t(x)\in\functrep_{0}(f*P_t)\cap\functrep(\Pint(f;x,t))\nonumber\\
&\subseteq\functrep_{0}(f*P_t)\cap\functrep_{0}(\Pint(f;x,t)),\label{eq2.256}
\end{align}
therefore the function $\Pint(f;x,t)$ is a common function representative of tempered distributions $f*P_t$ and $\Pint(f;x,t)$ in the sense of $\Sw_{0}'(\bbbr^n)$, and equation (\ref{eq2.244}) is proven. Similarly, we use (\ref{eq1-5}), (\ref{eq1-6}), (\ref{eq2.238}), and (\ref{eq2.243}) to deduce that for $1\leq k\leq n+1$,
\begin{align}
&\partial_k\Pint(f;x,t)=f*\partial_{k}[P_t(\cdot)](x)\nonumber\\
&\in\functrep_{0}(f*\partial_{k}[P_t(\cdot)])\cap
\functrep(\partial_k\Pint(f;x,t))\nonumber\\
&\subseteq\functrep_{0}(f*\partial_{k}[P_t(\cdot)])\cap\functrep_{0}(\partial_k\Pint(f;x,t)),
\label{eq2.257}
\end{align}
and the function $\partial_k\Pint(f;x,t)$ is a common function representative of tempered distributions $f*\partial_{k}[P_t(\cdot)]$ and $\partial_k\Pint(f;x,t)$ in the sense of $\Sw_{0}'(\bbbr^n)$, thus equation (\ref{eq2.245}) is proven. The proof of Lemma \ref{lemma16} is now complete.
\end{proof}
\begin{lemma}\label{lemma8}
Assume that $f\in\Sw_{0}'(\bbbr^n)$ and $\multi(\xi)\in L^1(\bbbr^n)$ is a smooth function on $\bbbr^n\setminus\{0\}$ that satisfies condition (\ref{eq1.64}). For every $j\in\bbbz$, we denote $f_j:=f*\psi_{2^{-j}}\in\Sw'(\bbbr^n)$ then the following equation (\ref{eq2.258}) is true in the sense of $\Sw_{0}'(\bbbr^n)$,
\begin{equation}\label{eq2.258}
f_j*\FT_n\multi=(f*\psi_{2^{-j}})*\FT_n\multi=(f*\FT_n\multi)*\psi_{2^{-j}},
\end{equation}
therefore we have
\begin{equation}\label{eq2.259}
f*\FT_n\multi=\sum_{j\in\bbbz}f_j*\FT_n\multi,
\end{equation}
and the series in (\ref{eq2.259}) converges to $f*\FT_n\multi\in\Sw_{0}'(\bbbr^n)$ in the sense of $\Sw_{0}'(\bbbr^n)$. In particular, setting $\multi(\xi)=-2\pi|\xi|\cdot e^{-2\pi t|\xi|}$ yields the following equation (\ref{eq2.260}) is true in the sense of $\Sw_{0}'(\bbbr^n)$ for all $j\in\bbbz$ and $t\in(0,\infty)$,
\begin{equation}\label{eq2.260}
f_j*\partial_{n+1}[P_t(\cdot)]=(f*\partial_{n+1}[P_t(\cdot)])*\psi_{2^{-j}},
\end{equation}
and we also have
\begin{equation}\label{eq2.261}
f*\partial_{n+1}[P_t(\cdot)]=\sum_{j\in\bbbz}f_j*\partial_{n+1}[P_t(\cdot)],
\end{equation}
where the series in (\ref{eq2.261}) converges to $f*\partial_{n+1}[P_t(\cdot)]\in\Sw_{0}'(\bbbr^n)$ in the sense of $\Sw_{0}'(\bbbr^n)$.
\end{lemma}
\begin{proof}[Proof of Lemma \ref{lemma8}]
Let $f\in\Sw_{0}'(\bbbr^n)$. Since $\psi\in\Sw_0(\bbbr^n)$ satisfies the conditions (\ref{eq1-7}), (\ref{eq1-8}), and (\ref{eq1.82}), we deduce from Proposition \ref{proposition1} (iv) that for every $j\in\bbbz$, $\psi_{2^{-j}}(x)=2^{jn}\psi(2^j x)\in\Sw_{0}(\bbbr^n)$ and $f_j=f*\psi_{2^{-j}}\in\Sw'(\bbbr^n)\subseteq\Sw_{0}'(\bbbr^n)$. When $\phi\in\Sw_{0}(\bbbr^n)$ and $\tilde{\psi}_{2^{-j}}(x)=2^{jn}\psi(-2^j x)$, we have $\iFT_n\multi*\phi\in\Sw_{0}(\bbbr^n)\subseteq\Sw(\bbbr^n)$ and we use (\ref{eq1.111}), (\ref{eq1.112}), and (\ref{eq1.122}) to obtain 
\begin{align}
&<f_j*\FT_n\multi,\phi>=<(f*\psi_{2^{-j}})*\FT_n\multi,\phi>\nonumber\\
&=<f*\psi_{2^{-j}},\iFT_n\multi*\phi>=<f,\tilde{\psi}_{2^{-j}}*\iFT_n\multi*\phi>\nonumber\\
&=<f,\iFT_n[\FT_n\psi(-2^{-j}\xi)\cdot\multi(\xi)\cdot\FT_n\phi(\xi)]>.\label{eq2.262}
\end{align}
Since $\tilde{\psi}_{2^{-j}}*\phi\in\Sw_{0}(\bbbr^n)$, Proposition \ref{proposition1} (ii) implies $f*\FT_n\multi$ is a well-defined continuous linear functional in $\Sw_{0}'(\bbbr^n)$ and we use (\ref{eq1.111}), (\ref{eq1.112}), and (\ref{eq1.122}) to obtain
\begin{align}
&<(f*\FT_n\multi)*\psi_{2^{-j}},\phi>=<f*\FT_n\multi,\tilde{\psi}_{2^{-j}}*\phi>\nonumber\\
&=<f,\iFT_n\multi*(\tilde{\psi}_{2^{-j}}*\phi)>\nonumber\\
&=<f,\iFT_n[\multi(\xi)\cdot\FT_n\psi(-2^{-j}\xi)\cdot\FT_n\phi(\xi)]>.\label{eq2.263}
\end{align}
Equations (\ref{eq2.262}) and (\ref{eq2.263}) prove (\ref{eq2.258}) for every $j\in\bbbz$. Since $f*\FT_n\multi\in\Sw_{0}'(\bbbr^n)$, equation (\ref{eq2.259}) is a consequence of (\ref{eq2.258}) and (\ref{eq1-9}) if we replace $f\in\Sw_{0}'(\bbbr^n)$ in (\ref{eq1-9}) by $f*\FT_n\multi\in\Sw_{0}'(\bbbr^n)$. By setting $\multi(\xi)=-2\pi|\xi|\cdot e^{-2\pi t|\xi|}$, we have $\FT_n\multi(x)=\partial_{n+1}[P_t(x)]$ for all $t\in(0,\infty)$. The proof of Lemma \ref{lemma8} is now complete.
\end{proof}
The following lemma is cited from \cite[section 2.4.4]{14modern} and is also introduced in \cite[Theorem 4.14, Chapter VI]{SteinWeissFourier}. It lays the foundation for the justification of subharmonicity needed in this paper.
	\begin{lemma}[cf. Lemma 2.4.9 of \cite{14modern}]\label{lemma9}
		Let $u_j$ be real-valued harmonic functions on $\bbbr^{n+1}$ satisfying the system of equations
		\begin{align}
			&\sum_{j=1}^{n+1}\frac{\partial u_j}{\partial x_j}=0,\label{eq2-50}\\
			&\frac{\partial u_j}{\partial x_k}-\frac{\partial u_k}{\partial x_j}=0,\quad\text{for }k,j\in\{1,\cdots,n+1\} \text{ and }k\neq j,\label{eq2-51}
		\end{align}
		and let $F=(u_1,\cdots,u_{n+1})$. Then the function
		\begin{equation*}
			|F|^q=(\sum_{j=1}^{n+1}|u_j|^2)^{\frac{q}{2}}
		\end{equation*}
		is subharmonic when $q\geq\frac{n-1}{n}$, i.e. it satisfies $\Delta(|F|^q)\geq 0$, on $\bbbr^{n+1}_+$.
	\end{lemma}
	\begin{remark}\label{remark1}
		From the proof of Lemma \ref{lemma9} given in \cite{14modern}, we know the condition that $u_j$ is a real-valued function on $\bbbr^{n+1}$ for every $1\leq j\leq n+1$, is a necessary condition. This is because the proof relies on the diagonalization of the $(n+1)\times(n+1)$ matrix
		\begin{equation}\label{eq2-52}
			A:=\bigg\{\frac{\partial u_k}{\partial x_j}\bigg\}^{n+1}_{k,j=1},
		\end{equation}
		and the condition (\ref{eq2-51}) shows that the matrix $A$ is real symmetric and thus is a Hermitian matrix if $u_j$'s ($1\leq j\leq n+1$) are all real-valued functions on $\bbbr^{n+1}$, hence the diagonalization of a Hermitian matrix is justified. However, if $u_j$'s ($1\leq j\leq n+1$) are complex-valued functions, then (\ref{eq2-51}) only tells us that the matrix $A$ defined in (\ref{eq2-52}) is a complex symmetric matrix and may not be a Hermitian matrix, since a Hermitian matrix is a complex conjugate symmetric matrix.
	\end{remark}
	\begin{lemma}\label{lemma10}
If $n\geq2$, $1\leq p\leq\infty$, and $f\in L^p(\bbbr^n)$ is a real-valued function, and if $F(x,t)\!=\!\nabla_{n+1}\Pint(f;x,t)$, then $|F|^q\!=\!|\nabla_{n+1}\Pint(f;x,t)|^q$ is subharmonic on $\bbbr^{n+1}_+$ when $q\geq\frac{n-1}{n}$.
	\end{lemma}
	\begin{proof}[Proof of Lemma \ref{lemma10}]
		If $f$ is real-valued, then the component functions of $F(x,t)$ are all real-valued functions on $\bbbr^{n+1}_+$. To apply Lemma \ref{lemma9}, we need to check the component functions of $F(x,t)$ satisfy the system of equations (\ref{eq2-50}) and (\ref{eq2-51}). Now we prove formulae (\ref{eq1-5}) and (\ref{eq1-6}) when $f$ is a function in $L^p(\bbbr^n)$ for some $1\leq p\leq\infty$. We will only prove formula (\ref{eq1-5}) when $k=1$. The proof for the other cases of (\ref{eq1-5}) when $k=2,\cdots,n$ will be similar. Let $\vec{e}_1$ be the unit elementary vector in $\bbbr^n$ whose first coordinate is $1$ and all the other coordinates are $0$. For $x=(x_1,x_2,\cdots,x_n)\in\bbbr^n$ and $y=(y_1,y_2,\cdots,y_n)\in\bbbr^n$, we denote $x'=(x_2,\cdots,x_n)\in\bbbr^{n-1}$ and $y'=(y_2,\cdots,y_n)\in\bbbr^{n-1}$. Let $h\neq0$ be a real number, then we have
\begin{equation}\label{eq2-151}
\partial_1\Pint(f;x,t)\!=\!\lim_{h\rightarrow0}\!\int_{\bbbr^n}\!\!\!\!\frac{f(y)}{h}\!\cdot\!
\big[P_t(x\!+\!h\vec{e}_1\!-\!y)\!-\!P_t(x\!-\!y)\big]dy
\end{equation}
		for every $x\in\bbbr^n$ and every $0<t<\infty$. Denote
		\begin{align}
			&E^{(1)}:=\{y\in\bbbr^n:t<|x_1-y_1|\},\label{eq2-152}\\
			&E^{(2)}:=\{y\in\bbbr^n:t\geq|x_1-y_1|\},\label{eq2-153}
		\end{align}
		and the characteristic functions associated with these two sets are denoted as $\chi_{E^{(1)}}$ and $\chi_{E^{(2)}}$, respectively. We claim that the dominating function for the integrand on the right side of (\ref{eq2-151}) is given by
		\begin{equation}\label{eq2-154}
			\Upsilon_1(y)=\frac{\chi_{E^{(1)}}(y)\cdot|f(y)|}{(t^2+\frac{1}{4}|x-y|^2)^{\frac{n+1}{2}}}
			+\frac{\chi_{E^{(2)}}(y)\cdot|f(y)|}{(t^2+|x'-y'|^2)^{\frac{n+1}{2}}}
		\end{equation} 
		uniformly for $h\in\bbbr$ with $0<|h|<\frac{t}{2}$. We use the mean value theorem and formula (\ref{eq1-3-1}) and find a real number $h^*$ between $0$ and $h$ so that
		\begin{align}
			&\bigg|\frac{f(y)}{h}\cdot\big[P_t(x+h\vec{e}_1-y)-P_t(x-y)\big]\bigg|\nonumber\\
			&=\big|f(y)\cdot\partial_1[P_t(x+h^*\vec{e}_1-y)]\big|\nonumber\\
			&\lesssim\frac{|f(y)|}{(t^2+|x'-y'|^2+(h^*+x_1-y_1)^2)^{\frac{n+1}{2}}}\nonumber\\
			&\lesssim\frac{|f(y)|}{(t^2+|x'-y'|^2+g(|h^*|))^{\frac{n+1}{2}}},\label{eq2-155}
		\end{align}
		where the implicit constant depends on $n$, and $g(|h^*|)$ is the evaluation of the nonnegative quadratic function
		\begin{equation}\label{eq2-156}
			g(u)=u^2-2|x_1-y_1|u+|x_1-y_1|^2
		\end{equation}
		at $u=|h^*|$. When $y\in E^{(1)}$, we have $0<|h^*|<|h|<\frac{t}{2}<\frac{1}{2}|x_1-y_1|$. On the closed and bounded interval $[0,\frac{1}{2}|x_1-y_1|]$, the nonnegative quadratic function $g(u)$ defined in (\ref{eq2-156}) is decreasing and achieves a nonnegative minimum value at $u=\frac{1}{2}|x_1-y_1|$, and this nonnegative minimum value is $\frac{1}{4}|x_1-y_1|^2$. Recall that
		$$|x'-y'|^2=\sum_{j=2}^n(x_j-y_j)^2,$$
		and then we have
		\begin{align}
			&(\ref{eq2-155})\leq\frac{|f(y)|}{(t^2+|x'-y'|^2+\frac{1}{4}|x_1-y_1|^2)^{\frac{n+1}{2}}}\nonumber\\
			&\leq\frac{|f(y)|}{(t^2+\frac{1}{4}|x'-y'|^2+\frac{1}{4}|x_1-y_1|^2)^{\frac{n+1}{2}}}=\Upsilon_1(y),
			\label{eq2-157}
		\end{align}
		if $y\in E^{(1)}$. By using the nonnegativity of the function $g(u)$, we also have
		\begin{equation}\label{eq2-158}
			(\ref{eq2-155})\leq\frac{|f(y)|}{(t^2+|x'-y'|^2)^{\frac{n+1}{2}}}=\Upsilon_1(y)
		\end{equation}
		if $y\in E^{(2)}$. Inequalities (\ref{eq2-155}), (\ref{eq2-157}), (\ref{eq2-158}) together prove the claim (\ref{eq2-154}). It is not hard to see that we can use H\"{o}lder's inequality with $\frac{1}{p}+\frac{1}{p'}=1$ to obtain
		\begin{align}
			&\int_{E^{(1)}}\frac{|f(y)|}{(t^2+\frac{1}{4}|x-y|^2)^{\frac{n+1}{2}}}dy\nonumber\\
			&\leq\|f\|_{L^p(\bbbr^n)}\cdot\|\frac{1}{(t^2+\frac{1}{4}|y|^2)^{\frac{n+1}{2}}}\|_{L^{p'}(\bbbr^n)}<\infty,
			\label{eq2-159}
		\end{align}
		where $1\leq p\leq\infty$. Recall the definition of the set $E^{(2)}$ given in (\ref{eq2-153}) and also that $dy'=dy_2\cdots dy_n$, then we have
		$$E^{(2)}=[x_1-t,x_1+t]\times\bbbr^{n-1},$$
		and we use H\"{o}lder's inequality with $\frac{1}{p}+\frac{1}{p'}=1$ to obtain
		\begin{align}
			&\int_{E^{(2)}}\frac{|f(y)|}{(t^2+|x'-y'|^2)^{\frac{n+1}{2}}}dy\nonumber\\
			&\leq\|f\|_{L^p(E^{(2)})}\cdot\bigg[\int_{E^{(2)}}\frac{1}{(t^2+|x'-y'|^2)^{\frac{(n+1)p'}{2}}}dy
			\bigg]^{\frac{1}{p'}}\nonumber\\
			&=\|f\|_{L^p(E^{(2)})}\cdot\bigg[2t\int_{\bbbr^{n-1}}\frac{1}{(t^2+|y'|^2)^{\frac{(n+1)p'}{2}}}dy'
			\bigg]^{\frac{1}{p'}}\nonumber\\
			&\leq\|f\|_{L^p(\bbbr^n)}\cdot(2t)^{\frac{1}{p'}}\cdot\bigg[\int_{\bbbr^{n-1}}
			\frac{1}{(t^2+|y'|^2)^{\frac{(n+1)p'}{2}}}dy'\bigg]^{\frac{1}{p'}}<\infty,\label{eq2-160}
		\end{align}
		where $1<p\leq\infty$ in (\ref{eq2-160}). For the case where $p=1$, we have
		\begin{align}
			&\int_{E^{(2)}}\frac{|f(y)|}{(t^2+|x'-y'|^2)^{\frac{n+1}{2}}}dy\nonumber\\
			&\leq\|f\|_{L^1(E^{(2)})}\cdot\esssup_{y\in E^{(2)}}\frac{1}{(t^2+|x'-y'|^2)^{\frac{n+1}{2}}}\nonumber\\
			&=\|f\|_{L^1(E^{(2)})}\cdot\esssup_{y'\in\bbbr^{n-1}}\frac{1}{(t^2+|y'|^2)^{\frac{n+1}{2}}}\leq
			\frac{\|f\|_{L^1(\bbbr^n)}}{t^{n+1}}<\infty.\label{eq2-161}
		\end{align}
		Inequalities (\ref{eq2-159}), (\ref{eq2-160}), (\ref{eq2-161}) show that the dominating function $\Upsilon_1(y)$ is in $L^1(\bbbr^n)$ whenever $f$ is a function in $L^p(\bbbr^n)$ for some $1\leq p\leq\infty$. Hence by invoking the dominated convergence theorem, we can bring the limit into the integral sign on the right side of (\ref{eq2-151}) and obtain the formula (\ref{eq1-5}) when $k=1$. Now we prove formula (\ref{eq1-6}). Let $h\in\bbbr$, then we have
		\begin{equation}\label{eq2-162}
			\partial_{n+1}\Pint(f;x,t)=\lim_{h\rightarrow0}\int_{\bbbr^n}\frac{f(x-y)}{h}\cdot\big[P_{t+h}(y)-P_t(y)\big]
			dy
		\end{equation}
		for every $x\in\bbbr^n$ and every $0<t<\infty$. We claim that the dominating function for the integrand on the right side of (\ref{eq2-162}) is given by
		\begin{equation}\label{eq2-163}
			\Upsilon_2(y)=\frac{|f(x-y)|}{(\frac{1}{4}t^2+|y|^2)^{\frac{n+1}{2}}}
		\end{equation} 
		uniformly for $h\in\bbbr$ with $0<|h|<\frac{t}{2}$. We use the mean value theorem and formula (\ref{eq1-3-2}) and find a real number $h^*$ between $0$ and $h$ so that
		\begin{align}
			&\bigg|\frac{f(x-y)}{h}\cdot\big[P_{t+h}(y)-P_t(y)\big]\bigg|\nonumber\\
			&=\big|f(x-y)\cdot\partial_{n+1}[P_{t+h^*}(y)]\big|\nonumber\\
			&\lesssim\frac{|f(x-y)|}{((t+h^*)^2+|y|^2)^{\frac{n+1}{2}}},\label{eq2-164}
		\end{align}
		where the implicit constant in (\ref{eq2-164}) depends on $n$. When $0<|h^*|<|h|<\frac{t}{2}$, we have $\frac{1}{2}t<t+h^*<\frac{3}{2}t$ and hence (\ref{eq2-164}) can be dominated by the function $\Upsilon_2(y)$ for every $y\in\bbbr^n$. And a simple application of H\"{o}lder's inequality with $\frac{1}{p}+\frac{1}{p'}=1$ shows that the dominating function $\Upsilon_2(y)$ is in $L^1(\bbbr^n)$ whenever $f$ is a function in $L^p(\bbbr^n)$ for some $1\leq p\leq\infty$. Hence by invoking the dominated convergence theorem, we can bring the limit into the integral sign on the right side of (\ref{eq2-162}) and obtain the formula (\ref{eq1-6}).\\
		
If $k,j\in\{1,\cdots,n\}$ and $k\neq j$, we can use the same method as above but with a small number of modifications, and the dominated convergence theorem, and the assumption that $f\in L^p(\bbbr^n)$ for some $p$ with $1\leq p\leq\infty$, and the formulae (\ref{eq1-5}) and (\ref{eq1-6}) to obtain
		\begin{align}
			&\partial_k^2\Pint(f;x,t)=f*\partial_k^2[P_t(\cdot)](x)\nonumber\\
			&\!=\!\!\int_{\bbbr^n}\!\!\!\!f(y)\!\cdot\!\big[\frac{-(n+1)C_0 t}{(t^2+|x-y|^2)^{\frac{n+3}{2}}}\!+\!
			\frac{(n+3)(n+1)C_0 t(x_k-y_k)^2}{(t^2+|x-y|^2)^{\frac{n+5}{2}}}\big]dy\label{eq2.167}
		\end{align}
		for $1\leq k\leq n$, and
		\begin{align}
			&\partial_{n+1}^2\Pint(f;x,t)=f*\partial_{n+1}^2[P_t(\cdot)](x)\nonumber\\
			&\!=\!\!\int_{\bbbr^n}\!\!\!\!f(y)\!\cdot\!\big[\frac{-3(n+1)C_0 t}{(t^2+|x-y|^2)^{\frac{n+3}{2}}}\!+\!
			\frac{(n+3)(n+1)C_0 t^3}{(t^2+|x-y|^2)^{\frac{n+5}{2}}}\big]dy,\label{eq2.168}
		\end{align}
therefore we can continue the above calculations with the dominated convergence theorem and obtain when $n\geq2$,
\begin{equation*}
\sum_{k=1}^{n+1}\partial_k^2\Pint(f;x,t)=\Delta\Pint(f;x,t)=0
\end{equation*}
and the Poisson integral $\Pint(f;x,t)$ is a smooth and harmonic function defined on $\bbbr^{n+1}_+$\! whenever $f\in L^p(\bbbr^n)$ for some $p\in[1,\infty]$, thus condition (\ref{eq2-50}) is satisfied and
\begin{equation}\label{eq2.372}
\partial^{\alpha}\Pint(f;x,t)=f*\partial^{\alpha}[P_t(\cdot)](x)
\end{equation}
for every multi-index $\alpha=(\alpha_1,\cdots,\alpha_n,\alpha_{n+1})$ and for all $(x,t)\in\bbbr^{n+1}_+$. In particular, we have
		\begin{align}
			&\partial_j\partial_k\Pint(f;x,t)\nonumber\\
			&=\int_{\bbbr^n}f(y)\cdot\frac{(n+1)(n+3)C_0 t(y_k-x_k)(y_j-x_j)}{(t^2+|x-y|^2)^{\frac{n+5}{2}}}dy,
			\label{eq2-54}
		\end{align}
and similarly, $\partial_k\partial_j\Pint(f;x,t)$ has the same expression as given in (\ref{eq2-54}). Therefore we have $\partial_j\partial_k\Pint(f;x,t)=\partial_k\partial_j\Pint(f;x,t)$ for $k,j\in\{1,\cdots,n\}$ and $k\neq j$. If $1\leq k\leq n$, we use the same method, the dominated convergence theorem, the assumption that $f\in L^p(\bbbr^n)$ with $1\leq p\leq\infty$, and the formula (\ref{eq1-6}) to obtain
\begin{align}
&\partial_k\partial_{n+1}\Pint(f;x,t)\nonumber\\
&\!=\!\!\!\!\int_{\bbbr^n}\!\!\!\!\!\!f(y)\!\cdot\!\!
\bigg[\!\frac{(n\!+\!1)(n\!+\!3)C_0 t^2(x_k\!-\!y_k)}{(t^2\!+\!|x\!-\!y|^2)^{\frac{n+5}{2}}}
\!-\!\frac{(n\!+\!1)C_0(x_k\!-\!y_k)}{(t^2\!+\!|x\!-\!y|^2)^{\frac{n+3}{2}}}\!\bigg]\!dy,\label{eq2-55}
\end{align}
and similarly, $\partial_{n+1}\partial_k\Pint(f;x,t)$ has the same expression as given in (\ref{eq2-55}). Hence we also have $\partial_k\partial_{n+1}\Pint(f;x,t)=\partial_{n+1}\partial_k\Pint(f;x,t)$ for $1\leq k\leq n$. By using the dominated convergence theorem and some direct calculations, we have that all the component functions of $F(x,t)$ are also harmonic on $\bbbr^{n+1}_+$\!\!.\! By now we have shown $F(x,t)=\nabla_{n+1}\Pint(f;x,t)$ satisfies all the conditions of Lemma \ref{lemma9} and thus $|\nabla_{n+1}\Pint(f;x,t)|^q$ is subharmonic on $\bbbr^{n+1}_+$ when $q\geq\frac{n-1}{n}$.
\end{proof}
\begin{lemma}\label{lemma11}
Let $0<s<1<p<\infty$. Assume that $f\in\Lps$ has a real-valued function representative $f(x)\in\functrep_{0}(f)\cap W^{1,p_0}(\bbbr^n)$ for some $1\leq p_0\leq\infty$, then $\varGamma(\frac{-s}{2})\Lift_s f\in\Sw'(\bbbr^n)$ has a real-valued function representative $F(x)\in\functrep_{0}(\varGamma(\frac{-s}{2})\Lift_s f)\cap L^p(\bbbr^n)$, where $\Lift_s f=\iFT_n(|\xi|^s\FT_n f)$ is the tempered distribution considered in Remark \ref{remark3} and $\varGamma(\frac{-s}{2})$ is the evaluation of the meromorphic extension of the gamma function at $\frac{-s}{2}$.
\end{lemma}
\begin{proof}[Proof of Lemma \ref{lemma11}]
From Definition \ref{definition2}, we know that $f$ is an element of $\Lps$ if and only if $\Lift_s f$ has a function representative (in the sense of $\Sw_0'(\bbbr^n)$) in $L^p(\bbbr^n)$. Now we prove that the product $\varGamma(\frac{-s}{2})\Lift_s f$ has a real-valued function representative $F(x)$ in the sense of $\Sw_0'(\bbbr^n)$ and $F(x)\in L^p(\bbbr^n)$. The existence of $F(x)$ is provided by Definition \ref{definition2}, and we have
\begin{equation*}
\|F\|_{L^p(\bbbr^n)}=|\varGamma(\frac{-s}{2})|\cdot(\int_{\bbbr^n}|\Lift_s f|^p dx)^{\frac{1}{p}}<\infty.
\end{equation*}
We prove $F(x)$ is a real-valued function for almost every $x\in\bbbr^n$. Let $\phi\in\Sw_0(\bbbr^n)$ satisfy conditions (\ref{Sw0.condition.1}), (\ref{Sw0.condition.2}), (\ref{Sw0.condition.3}). We recall the distribution $u_z$ defined in \cite[Definition 2.4.5]{14classical}. If $z\in\bbbc$ is a complex number with $\Rez z>-n$, then
\begin{equation}\label{eq2-57}
<u_z,g>=\int_{\bbbr^n}\frac{\pi^{\frac{z+n}{2}}}{\varGamma(\frac{z+n}{2})}\cdot|x|^z g(x)dx.
\end{equation}
And from \cite[7. of Examples 2.3.5.]{14classical}, we see that $u_z$ defined in (\ref{eq2-57}) is a tempered distribution on $\bbbr^n$ whenever $\Rez z>-n$. Indeed, if $g\in\Sw(\bbbr^n)$, then
\begin{align}
&|<u_z,g>|\nonumber\\
&\lesssim\int_{|x|<1}|x|^{\Rez z}\cdot|g(x)|dx\nonumber\\
&\quad+\int_{|x|\geq1}\frac{|x|^{\Rez z}}{(1+|x|)^M}\cdot(1+|x|)^M|g(x)|dx\nonumber\\
&\lesssim\int_{|x|<1}|x|^{\Rez z}dx\cdot\|g\|_{L^{\infty}(\bbbr^n)}\nonumber\\
&\quad+\int_{|x|\geq1}\frac{|x|^{\Rez z}}{(1+|x|)^M}dx\cdot\|(1+|\cdot|)^M|g(\cdot)|\|_{L^{\infty}(\bbbr^n)},\label{eq2-115}
\end{align}
and both integrals on the right end of (\ref{eq2-115}) converge since we can choose $M$ to be a sufficiently large positive integer so that $M>n+\Rez z>0$, and $\|(1+|\cdot|)^M|g(\cdot)|\|_{L^{\infty}(\bbbr^n)}$ can be estimated from above by a finite sum of Schwartz seminorms of $g$. We also deduce from (\ref{eq2-57}) that when $\Rez z>-n$, the locally integrable function (\ref{eq2.173}) belongs to the collection $\functrep(u_z)$, 
\begin{equation}\label{eq2.173}
x\in\bbbr^n\longmapsto\frac{\pi^{\frac{z+n}{2}}}{\varGamma(\frac{z+n}{2})}\cdot|x|^z. 
\end{equation}
Furthermore, given an arbitrary nonnegative integer $N$, \cite[expression (2.4.8)]{14classical} yields an analytic extension of $<u_z,g>$ for $g\in\Sw(\bbbr^n)$ to the case where $z\!\in\!\bbbc$ satisfies $\Rez z>-N-n-1$, and we have
		\begin{align}
			&<u_z,g>\nonumber\\
			&=\int_{|x|\geq 1}\frac{\pi^{\frac{z+n}{2}}}{\varGamma(\frac{z+n}{2})}\cdot|x|^z g(x)dx\nonumber\\
			&\quad+\sum_{|\alpha|\leq N}\frac{\pi^{\frac{z+n}{2}}}{\varGamma(\frac{z+n}{2})}\cdot
			\frac{\frac{1}{\alpha!}\int_{\unitsph}\theta^{\alpha}d\theta}{|\alpha|+z+n}\cdot
			(\partial^{\alpha}g)(0)\nonumber\\
			&\quad+\int_{|x|<1}\frac{\pi^{\frac{z+n}{2}}}{\varGamma(\frac{z+n}{2})}\cdot
			\bigg\{g(x)-\sum_{|\alpha|\leq N}\frac{(\partial^{\alpha}g)(0)}{\alpha!}\cdot x^{\alpha}\bigg\}
			\cdot|x|^z dx,\label{eq2-58}
		\end{align}
where $\alpha!=\alpha_1!\cdots\alpha_n!$ for the multi-index $\alpha=(\alpha_1,\cdots,\alpha_n)$ and $\varGamma(\frac{z+n}{2})$ in the expression (\ref{eq2-58}) is the evaluation of the meromorphic extension of the gamma function at $\frac{z+n}{2}$. We can use Taylor's formula for the Schwartz function $g(x)$ at $x=0$ with integral form of the remainder to deduce that the expression
$$g(x)-\sum_{|\alpha|\leq N}\frac{(\partial^{\alpha}g)(0)}{\alpha!}\cdot x^{\alpha}$$
in the last term on the right side of (\ref{eq2-58}), can be bounded from above by the product of $|x|^{N+1}$ and a finite sum of Schwartz seminorms of $g$, and thus the term $<u_z,g>$ as given in (\ref{eq2-58}) can be estimated from above by a finite sum of Schwartz seminorms of $g$ and hence defines a tempered distribution on $\bbbr^n$ if $\Rez z>-N-n-1$. By \cite[Theorem 2.4.6]{14classical}, we have the distributional identity $\FT_n u_z=u_{-n-z}$ for all $z\in\bbbc$. Now we apply this identity to the case where $z=s\in(0,1)$ then we have that the integral in (\ref{eq2.197}) defines a tempered distribution denoted by the symbol $|\xi|^s$, and the locally integrable function $|\xi|^s$ is in $\functrep(|\xi|^s)$,
\begin{equation}\label{eq2.197}
\int_{\bbbr^n}|\xi|^s\cdot g(\xi)d\xi\quad\text{for $g\in\Sw(\bbbr^n)$}.
\end{equation}
Furthermore, the locally integrable function
\begin{equation}\label{eq2-116}
\xi\in\bbbr^n\longmapsto\frac{\pi^{\frac{s+n}{2}}}{\varGamma(\frac{s+n}{2})}\cdot|\xi|^s 
\end{equation}
is a common function representative in
\begin{equation}\label{eq2.198}
\functrep(\frac{\pi^{\frac{s+n}{2}}}{\varGamma(\frac{s+n}{2})}\cdot|\xi|^s)\bigcap\functrep(u_s),
\end{equation}
therefore the following equation (\ref{eq2-59}) holds true in the sense of $\Sw'(\bbbr^n)$,
\begin{equation}\label{eq2-59}
\frac{\pi^{\frac{s+n}{2}}}{\varGamma(\frac{s+n}{2})}\cdot\FT_n(|\xi|^s)=\FT_n(u_s)=u_{-n-s},
\end{equation}
and $\FT_n(|\xi|^s)$ is the tempered distribution $\varGamma(\frac{s+n}{2})\pi^{-\frac{s+n}{2}}u_{-n-s}$, which can be obtained from (\ref{eq2-58}) by choosing $N=0$, since $-n>-n-s>-n-1$. For an arbitrary Schwartz function $g\in\Sw(\bbbr^n)$, we apply \cite[Proposition 2.3.22. (11)]{14classical} to obtain
\begin{align}
&<\FT_n(|\xi|^s)*\phi,g>=<\iFT_n[\FT_n(|\xi|^s)*\phi],\FT_n g>\nonumber\\
&=<|\xi|^s\cdot\iFT_n\phi,\FT_n g>=<|\xi|^s,\FT_n g\cdot\iFT_n\phi>\nonumber\\
&=\int_{\bbbr^n}|\xi|^s\!\cdot\!\FT_n g(\xi)\!\cdot\!\iFT_n\phi(\xi)d\xi,\label{eq2.199}
\end{align}
where the last equation in (\ref{eq2.199}) is because $\FT_n g(\xi)\cdot\iFT_n\phi(\xi)\!\in\!\Sw(\bbbr^n)$ and $|\xi|^s\!\in\!\functrep(|\xi|^s)$. Since $\phi\in\Sw_0(\bbbr^n)$ and $\iFT_n\phi\in\Sw_{00}(\bbbr^n)$, and the function $\multi(\xi)=|\xi|^s$ satisfies condition (\ref{eq1.64}), and thus $|\xi|^s\!\cdot\!\iFT_n\phi(\xi)$ in the last line of (\ref{eq2.199}) is a function in $\Sw_{00}(\bbbr^n)$ by Proposition \ref{proposition1}. Furthermore, we deduce from \cite[Theorem 2.2.14. (5)]{14classical} that
\begin{equation}\label{eq2.200}
(\ref{eq2.199})=\int_{\bbbr^n}\FT_n[|\xi|^s\!\cdot\!\iFT_n\phi(\xi)](x)\!\cdot\!g(x)dx,
\end{equation}
and $\FT_n(|\xi|^s)*\phi\in\Sw'(\bbbr^n)$ has the function representative $(\ref{eq2-117})\in\functrep(\FT_n(|\xi|^s)*\phi)\cap\Lloc$,
\begin{equation}\label{eq2-117}
x\in\bbbr^n\longmapsto\FT_n[|\xi|^s\cdot\iFT_n\phi(\xi)](x)\in\Sw_0(\bbbr^n).
\end{equation}
From \cite[Theorem 2.3.20.]{14classical} , we know that the smooth function (\ref{eq2-118}) is also in $\functrep(\FT_n(|\xi|^s)*\phi)\cap\Lloc$. Thus Lemma \ref{lemma15} (i) and the smoothness of relevant functions imply the functions (\ref{eq2-117}) and (\ref{eq2-118}) are equal,
\begin{equation}\label{eq2-118}
x\in\bbbr^n\longmapsto<\FT_n(|\xi|^s),\phi(x-\cdot)>.
\end{equation}
Now we combine (\ref{eq2-117}), (\ref{eq2-118}), (\ref{eq2-59}), and (\ref{eq2-58}) with $N=0$ to obtain the following equation for every $x\in\bbbr^n$,
\begin{align}
&\FT_n[|\xi|^s\cdot\iFT_n\phi(\xi)](x)=<\FT_n(|\xi|^s),\phi(x-\cdot)>\nonumber\\
&=\frac{\varGamma(\frac{s+n}{2})}{\pi^{\frac{s+n}{2}}}\cdot<u_{-n-s},\phi(x-\cdot)>\nonumber\\
&=\frac{\varGamma(\frac{s+n}{2})}{\varGamma(\frac{-s}{2})\pi^{s+\frac{n}{2}}}\cdot\big\{
\int_{|y|\geq 1}\phi(x-y)\cdot|y|^{-n-s}dy-\frac{\phi(x)}{s}\cdot\omega_{n-1}\nonumber\\
&\quad+\int_{|y|<1}(\phi(x-y)-\phi(x))\cdot|y|^{-n-s}dy\big\}.\label{eq2-60}
\end{align}
According to Remark \ref{remark3}, the action of the tempered distribution $\Lift_s f$ on the Schwartz function $\phi\in\Sw_0(\bbbr^n)$ is given by
\begin{equation}\label{eq2-56}
<\Lift_s f,\phi>=<f,\FT_n[|\xi|^s\iFT_n\phi(\xi)]>.
\end{equation}
Combining (\ref{eq2-117}), (\ref{eq2-60}), and (\ref{eq2-56}) together, we observe the condition that $f(x)$ is a real-valued function representative in $\functrep_{0}(f)\cap W^{1,p_0}(\bbbr^n)$ for some $p_0\in[1,\infty]$ guarantees that we can exchange the order of integration and apply proper changes of variables to obtain
\begin{align}
&<\varGamma(\frac{-s}{2})\Lift_s f,\phi>\nonumber\\
&=\varGamma(\frac{-s}{2})\cdot<f,\FT_n[|\xi|^s\iFT_n\phi(\xi)]>\nonumber\\
&=\int_{\bbbr^n}f(x)\cdot\varGamma(\frac{-s}{2})\FT_n[|\xi|^s\iFT_n\phi(\xi)](x)dx\nonumber\\
&=\int_{\bbbr^n}\frac{\varGamma(\frac{s+n}{2})}{\pi^{s+\frac{n}{2}}}\cdot\big\{
\int_{|y|\geq 1}f(x+y)\cdot|y|^{-n-s}dy-\frac{\omega_{n-1}f(x)}{s}\nonumber\\
&\quad+\int_{|y|<1}(f(x+y)-f(x))\cdot|y|^{-n-s}dy\big\}\cdot\phi(x)dx,\label{eq2-61}
\end{align}
where we notice that $\varGamma(\frac{-s}{2})$ has been canceled out in the process, and $\phi$ is an arbitrary Schwartz function in $\Sw_0(\bbbr^n)$. Equation (\ref{eq2-61}) and Minkowski's integral inequality show that $\varGamma(\frac{-s}{2})\Lift_s f\in\Sw'(\bbbr^n)$ has the function representative $(\ref{eq2-62})\in\functrep_{0}(\varGamma(\frac{-s}{2})\Lift_s f)\cap L^{p_0}(\bbbr^n)$,
\begin{equation}\label{eq2-62}
\frac{\varGamma(\frac{s+n}{2})}{\pi^{s+\frac{n}{2}}}\!\cdot\!\big\{\!\!\!
\int_{|y|\geq 1}\!\!\!\!\frac{f(x\!+\!y)}{|y|^{n+s}}dy\!-\!\frac{\omega_{n-1}f(x)}{s}\!
+\!\!\!\int_{|y|<1}\!\!\!\!\!\!\!\!\!\frac{f(x\!+\!y)\!-\!f(x)}{|y|^{n+s}}dy\big\}.
\end{equation}
Because $f$ is a real-valued function in $W^{1,p_0}(\bbbr^n)$ for some $p_0\in[1,\infty]$, both integrals in (\ref{eq2-62}) converge absolutely in the Riemann sense for almost every $x\in\bbbr^n$ and have real values. Because $\frac{s+n}{2}>0$, the definition (\ref{eq1-1}) tells us that $\varGamma(\frac{s+n}{2})$ is a well-defined positive finite real number. Therefore (\ref{eq2-62}) is a real-valued function for almost every $x\in\bbbr^n$. We have obtained two function representatives of $\varGamma(\frac{-s}{2})\Lift_s f$ in the sense of $\Sw_0'(\bbbr^n)$, i.e. the real-valued function representative $(\ref{eq2-62})\!\in\!L^{p_0}(\bbbr^n)$ and the function representative $F(x)\!\in\!L^p(\bbbr^n)$ provided by Definition \ref{definition2}. We claim that the difference between (\ref{eq2-62}) and $F(x)$ is a real-valued polynomial, and hence $F(x)$ is a real-valued function for almost every $x\in\bbbr^n$. Since $(\ref{eq2-62})\in\functrep_0(\varGamma(\frac{-s}{2})\Lift_s f)\cap L^{p_0}(\bbbr^n)$ and $F(x)\in\functrep_0(\varGamma(\frac{-s}{2})\Lift_s f)\cap L^{p}(\bbbr^n)$, the difference $[(\ref{eq2-62})-F(x)]$ is in $L^{p_0}(\bbbr^n)+L^{p}(\bbbr^n)\subseteq\Lloc$ and there exists $\tilde{f}\in\Sw'(\bbbr^n)$ so that $[(\ref{eq2-62})-F(x)]$ is in $\functrep(\tilde{f})$ and satisfies condition (\ref{eq1.101}). By Proposition \ref{proposition2}, the difference $[(\ref{eq2-62})-F(x)]$ equals a polynomial function for almost every $x\in\bbbr^n$. Let $i=\sqrt{-1}$. Assume that $F(x)=u(x)+i\!\cdot\! v(x)$ and 
\begin{equation}\label{eq2.169}
(\ref{eq2-62})\!=\!F(x)\!+\!P_1(x)\!+\!i\!\cdot\!P_2(x)\quad\text{for almost every $x\!\in\!\bbbr^n$,}
\end{equation}
where the real-valued functions $u(x)$ and $v(x)$ are the real and imaginary parts of $F(x)$, respectively, and $P_1(x)$, $P_2(x)$ are real-valued polynomials of $x\in\bbbr^n$. Since (\ref{eq2-62}) is real-valued for almost every $x\in\bbbr^n$, equation (\ref{eq2.169}) indicates $v(x)+P_2(x)=0$ for almost every $x\in\bbbr^n$. The condition $F(x)\in L^p(\bbbr^n)$ implies
		\begin{align*}
			&\int_{\bbbr^n}|P_2(x)|^p dx=\int_{\bbbr^n}|v(x)|^p dx\\
			&\leq\int_{\bbbr^n}\big(\sqrt{u(x)^2+v(x)^2}\big)^p dx=\int_{\bbbr^n}|F(x)|^p dx<\infty,
		\end{align*}
therefore $v(x)=P_2(x)=0$ for almost every $x\in\bbbr^n$. The proof of Lemma \ref{lemma11} is now complete.
\end{proof}
In the following Lemma \ref{lemma5}, we cite the famous Hardy-Littlewood-Sobolev inequality from \cite[chapter VIII, 4.2]{SteinHarmonic}, and then we provide a useful result about the operator $\Lift_s f$ immediately after.
\begin{lemma}\label{lemma5}
If $0<\gamma<n$, $1<p<q<\infty$, $\frac{1}{q}=\frac{1}{p}-\frac{n-\gamma}{n}$, and $C_{20}=C_{20}(p,q)$ is a positive finite constant, then we have
\begin{equation*}
\|f*(|y|^{-\gamma})\|_{L^q(\bbbr^n)}\leq C_{20}\cdot\|f\|_{L^p(\bbbr^n)}.
\end{equation*}
\end{lemma}
\begin{lemma}\label{lemma13}
Let $1<p<q<\infty$, $-n<s<0$, and $s=-n(\frac{1}{p}-\frac{1}{q})$. If $f\in\Sw_0'(\bbbr^n)$ has a function representative in the sense of $\Sw_0'(\bbbr^n)$ and this function representative is in $L^p(\bbbr^n)$, then $\Lift_s f=\iFT_n[|\xi|^s\cdot\FT_n f]\in\Sw'(\bbbr^n)$ and the function (\ref{eq2.174}) belongs to $\functrep_0(\Lift_s f)\cap L^q(\bbbr^n)$,
\begin{equation}\label{eq2.174}
x\in\bbbr^n\longmapsto\frac{\varGamma(\frac{s+n}{2})}{\pi^{s+\frac{n}{2}}\varGamma(\frac{-s}{2})}
\int_{\bbbr^n}f(y)\cdot|y-x|^{-n-s}dy.
\end{equation}
Furthermore, we have
\begin{equation}\label{eq2.175}
\|\Lift_s f\|_{L^q(\bbbr^n)}\lesssim\|f\|_{L^p(\bbbr^n)},
\end{equation}
and the implicit constant in (\ref{eq2.175}) depends on $n$, $p$, and $q$.
\end{lemma}
\begin{proof}[Proof of Lemma \ref{lemma13}]
Let $\phi\!\in\!\Sw_0(\bbbr^n)$. According to Remark \ref{remark3}, we have
\begin{equation}\label{eq2.176}
<\Lift_s f,\phi>=<f,\FT_n[|\xi|^s\cdot\iFT_n\phi(\xi)]>.
\end{equation}
Since $-n<s<0$, the integral in (\ref{eq2.371}) defines a tempered distribution denoted by the symbol $|\xi|^s$ so that the locally integrable function $|\xi|^s$ is in $\functrep(|\xi|^s)$ and satisfies condition (\ref{eq1.64}), 
\begin{equation}\label{eq2.371}
\int_{\bbbr^n}|\xi|^s\cdot g(\xi)d\xi\text{ for $g\in\Sw(\bbbr^n)$}.
\end{equation}
Since $\iFT_n\phi\in\Sw_{00}(\bbbr^n)$, Proposition \ref{proposition1} implies $|\xi|^s\!\cdot\!\iFT_n\phi(\xi)\!\in\!\Sw_{00}(\bbbr^n)$ and $\FT_n[|\xi|^s\!\cdot\!\iFT_n\phi(\xi)](x)\!\in\!\Sw_0(\bbbr^n)$. Since $\FT_n(|\xi|^s)$ represents the distributional Fourier transform of $|\xi|^s\in\Sw'(\bbbr^n)$, we deduce from \cite[Theorem 2.3.20.]{14classical} that $\FT_n(|\xi|^s)*\phi\in\Sw'(\bbbr^n)$ and the smooth function (\ref{eq2.202}) is in $\functrep(\FT_n(|\xi|^s)*\phi)\cap\Lloc$,
\begin{equation}\label{eq2.202}
x\in\bbbr^n\longmapsto<\FT_n(|\xi|^s),\phi(x-\cdot)>.
\end{equation}
Let $g\in\Sw(\bbbr^n)$ be an arbitrary Schwartz function and we denote $\tilde{\phi}(x)=\phi(-x)$, then we have
\begin{align}
&<\FT_n(|\xi|^s)*\phi,g>=<\FT_n(|\xi|^s),\tilde{\phi}*g>\nonumber\\
&=<|\xi|^s,\FT_n[\tilde{\phi}*g]>=<|\xi|^s,\iFT_n\phi\cdot\FT_n g>\nonumber\\
&=\int_{\bbbr^n}|\xi|^s\!\cdot\!\iFT_n\phi(\xi)\!\cdot\!\FT_n g(\xi)d\xi,\label{eq2.203}
\end{align}
where the last equation in (\ref{eq2.203}) is because $\iFT_n\phi(\xi)\!\cdot\!\FT_n g(\xi)\!\in\!\Sw(\bbbr^n)$ and $|\xi|^s\!\in\!\functrep(|\xi|^s)$. Since $|\xi|^s\!\cdot\!\iFT_n\phi(\xi)\in\Sw_{00}(\bbbr^n)$, then \cite[Theorem 2.2.14. (5)]{14classical} implies that
\begin{equation}
(\ref{eq2.203})=\int_{\bbbr^n}\FT_n[|\xi|^s\!\cdot\!\iFT_n\phi(\xi)](x)\!\cdot\!g(x)dx,
\end{equation}
and thus the function (\ref{eq2.204}) is in $\functrep(\FT_n(|\xi|^s)*\phi)\cap\Lloc$,
\begin{equation}\label{eq2.204}
x\in\bbbr^n\longmapsto\FT_n[|\xi|^s\cdot\iFT_n\phi(\xi)](x)\in\Sw_0(\bbbr^n),
\end{equation}
where $|\xi|^s$ in (\ref{eq2.204}) is a function. Therefore Lemma \ref{lemma15} (i) and the smoothness of relevant functions imply the functions (\ref{eq2.202}) and (\ref{eq2.204}) are equal for every $x\in\bbbr^n$. We recall the tempered distribution $u_z$ defined by (\ref{eq2-57}). Since $-n<s<0$, the locally integrable function 
\begin{equation}\label{eq2.178}
\xi\in\bbbr^n\longmapsto\frac{\pi^{\frac{s+n}{2}}}{\varGamma(\frac{s+n}{2})}\cdot|\xi|^s
\end{equation}
is a common function representative in
\begin{equation}\label{eq2.205}
\functrep(\frac{\pi^{\frac{s+n}{2}}}{\varGamma(\frac{s+n}{2})}\cdot|\xi|^s)\bigcap\functrep(u_s).
\end{equation}
And \cite[Theorem 2.4.6]{14classical} tells us that the distributional identity (\ref{eq2-59}) is true for $s\in(-n,0)$. Since $-n-s>-n$, equation (\ref{eq2-57}) implies $u_{-n-s}\in\Sw'(\bbbr^n)$ and the function (\ref{eq2.179}) is in $\functrep(u_{-n-s})$,
\begin{equation}\label{eq2.179}
y\in\bbbr^n\longmapsto\frac{\pi^{\frac{-s}{2}}}{\varGamma(\frac{-s}{2})}\cdot|y|^{-n-s}.
\end{equation}
We have obtained the following equation (\ref{eq2.180}) for every $x\in\bbbr^n$,
\begin{align}
&\FT_n[|\xi|^s\!\cdot\!\iFT_n\phi(\xi)](x)=<\FT_n[|\xi|^s],\phi(x-\cdot)>\nonumber\\
&=<\!\!\frac{\varGamma(\frac{s+n}{2})}{\pi^{\frac{s+n}{2}}}\FT_n u_s,\phi(x\!-\!\cdot)\!\!>
=<\!\!\frac{\varGamma(\frac{s+n}{2})}{\pi^{\frac{s+n}{2}}}u_{-n-s},\phi(x\!-\!\cdot)\!\!>\nonumber\\
&=\int_{\bbbr^n}\frac{\varGamma(\frac{s+n}{2})}{\pi^{s+\frac{n}{2}}\varGamma(\frac{-s}{2})}\cdot
|y|^{-n-s}\phi(x-y)dy\nonumber\\
&=\int_{\bbbr^n}\frac{\varGamma(\frac{s+n}{2})}{\pi^{s+\frac{n}{2}}\varGamma(\frac{-s}{2})}\cdot
|x-z|^{-n-s}\phi(z)dz.\label{eq2.180}
\end{align}
Combining (\ref{eq2.176}) and (\ref{eq2.180}) together with the condition that $f\in\Sw_0'(\bbbr^n)$ has a function representative (in the sense of $\Sw_0'(\bbbr^n)$) in $L^p(\bbbr^n)$ denoted by $f(x)$, we obtain
\begin{equation}\label{eq2.181}
<\!\Lift_s f,\phi\!\!>=\!\!\!\int_{\bbbr^n}\!\int_{\bbbr^n}\!
\frac{\varGamma(\frac{s+n}{2})}{\pi^{s+\frac{n}{2}}\varGamma(\frac{-s}{2})}\!\cdot\! f(x)\!\cdot\!|x\!-\!z|^{-n-s}\!\cdot\!\phi(z)dzdx
\end{equation}
for all $\phi\in\Sw_0(\bbbr^n)$. An application of Lemma \ref{lemma5} and H\"{o}lder's inequality with $\frac{1}{q}+\frac{1}{q'}=1$ yields that the function (\ref{eq2.182}) is in $L^q(\bbbr^n)$,
\begin{equation}\label{eq2.182}
z\in\bbbr^n\longmapsto\int_{\bbbr^n}|f(x)|\!\cdot\!|x\!-\!z|^{-n-s}dx,
\end{equation}
and the integral (\ref{eq2.183}) is finite,
\begin{equation}\label{eq2.183}
\int_{\bbbr^n}\!\int_{\bbbr^n}\!|f(x)|\!\cdot\!|x\!-\!z|^{-n-s}\!\cdot\!|\phi(z)|dxdz<\infty.
\end{equation}
We invoke \cite[Tonelli's theorem and Corollary 7 in section 20.1]{real.analysis.royden} and \cite[Theorem 11 in section 20.2]{real.analysis.royden} to deduce the nonnegative function
$$|f(x)|\!\cdot\!|x\!-\!z|^{-n-s}\!\cdot\!|\phi(z)|$$
is integrable over the product space $\bbbr^n\times\bbbr^n$ with respect to the product measure $d(x,z)$, and hence we can apply \cite[Fubini's theorem in section 20.1]{real.analysis.royden} to exchange the order of integration in (\ref{eq2.181}) and obtain that the tempered distribution $\Lift_s f$ has the function representative (\ref{eq2.174}) in the sense of $\Sw_0'(\bbbr^n)$. The fact that the function (\ref{eq2.174}) belongs to $L^q(\bbbr^n)$ and the ensuing estimate (\ref{eq2.175}) are consequences of Lemma \ref{lemma5}. The proof of Lemma \ref{lemma13} is now complete.
\end{proof}
In the next lemma, we study properties of the function defined below for $x\in\bbbr^n$ and $0\leq t<\infty$,
\begin{equation}\label{eq2.373}
\Tfunct_s(f;x,t)\!:=\!\!\!\int_0^{\infty}\!\!\!\!\!\!\Pint(f;x,t\!+\!r)\!\cdot\!r^{s-1}dr\!=\!\!\!
\int_0^{\infty}\!\!\!\!\!\!f\!*\!P_{t+r}(x)\!\cdot\!r^{s-1}dr.
\end{equation}
\begin{lemma}\label{lemma18}
Assume that $1\leq p<\infty$, $0<s<\frac{n}{p}$, and $f$ is a function in $L^p(\bbbr^n)$. The function $\Tfunct_s(f;x,t)$ given in (\ref{eq2.373}) is well-defined and finite for almost every $x\in\bbbr^n$ and every $0\leq t<\infty$. Furthermore, the function $\Tfunct_s(f;x,t)$ is smooth on $\bbbr^{n+1}_+$, and for every multi-index $\alpha=(\alpha_1,\cdots,\alpha_n,\alpha_{n+1})$, the function $\partial^{\alpha}\Tfunct_s(f;x,t)$ is a well-defined function in $L^{\infty}(\bbbr^n)$ with respect to variable $x\in\bbbr^n$ for each fixed $0<t<\infty$, and we have
\begin{equation}\label{eq2.374}
\partial^{\alpha}\Tfunct_s(f;x,t)\!=\!\!\!\int_0^{\infty}\!\!\!\!
\partial^{\alpha}\Pint(f;x,t\!+\!r)\!\cdot\!r^{s-1}dr
\end{equation}
for all $(x,t)\in\bbbr^{n+1}_+$. Moreover, we can find a sequence $\{t_k\}_{k\in\bbbn}$ of positive real numbers so that the following conditions are satisfied,
\begin{equation}\label{eq2.375}
0<t_{k+1}<t_k<\infty,\qquad\lim_{k\rightarrow\infty}t_k=0,
\end{equation}
\begin{equation}\label{eq2.376}
\lim_{k\rightarrow\infty}\!\Tfunct_s(f;x,t_k)\!=\!\Tfunct_s(f;x,0)
\text{ for almost every $x\!\in\!\bbbr^n$}.
\end{equation}
Therefore, for every real number $c\in(0,\infty)$ and for almost every $x\in\bbbr^n$, the following equation (\ref{eq2.377}) is true and both sides of equation (\ref{eq2.377}) are well-defined and finite,
\begin{equation}\label{eq2.377}
\int_0^c\!\!\partial_{n+1}\Tfunct_s(f;x,u)du\!=\!\Tfunct_s(f;x,c)\!-\!\Tfunct_s(f;x,0).
\end{equation}
\end{lemma}
\begin{proof}[Proof of Lemma \ref{lemma18}]
Step 1: In this step, we prove the function $\Tfunct_s(f;x,t)$ given in (\ref{eq2.373}) is well-defined and finite for almost every $x\in\bbbr^n$ and every $0\leq t<\infty$. Applying H\"{o}lder's inequality with $\frac{1}{p}+\frac{1}{p'}=1$, we have
\begin{equation}\label{eq2.378}
\sup_{x\in\bbbr^n}\!|\Pint(f;x,r)|\!\leq\!\|f\|_{L^p(\bbbr^n)}\!\cdot\!\|P_r\|_{L^{p'}(\bbbr^n)}
\!\lesssim\!\|f\|_{L^p(\bbbr^n)}\!\cdot\!r^{-\frac{n}{p}},
\end{equation}
where the implicit constant in (\ref{eq2.378}) depends on $n,p$, and $P_r(x)$ is given by (\ref{eq2.379}). Estimate (\ref{eq2.378}) is true for all $1\leq p<\infty$ and $0<r<\infty$. Estimate (\ref{eq2.378}) with $t+r$ in place of $r$ and the assumption that $1\leq p<\infty$, $0<s<\frac{n}{p}$, and $f\in L^p(\bbbr^n)$ imply for each fixed $0<t<\infty$,
\begin{equation}\label{eq2.380}
\|\Tfunct_s(f;\cdot,t)\|_{L^{\infty}(\bbbr^n)}\!\lesssim\!\|f\|_{L^p(\bbbr^n)}\!\cdot\!\!
\int_0^{\infty}\frac{r^{s-1}}{(t\!+\!r)^{\frac{n}{p}}}dr\!<\!\infty,
\end{equation}
with the implicit constant depending on $n,p$. Estimate (\ref{eq2.378}) also shows
\begin{equation}\label{eq2.270}
\sup_{x\in\bbbr^n}\!\int_1^{\infty}\!\!\!\!\!\!|\Pint(f;x,r)|\!\cdot\!r^{s-1}dr\!\lesssim\!
\|f\|_{L^p(\bbbr^n)}\!\cdot\!\!\!\int_1^{\infty}\!\!\!\!r^{s-\frac{n}{p}-1}dr\!<\!\infty.
\end{equation}
Furthermore, Minkowski's integral inequality and Young's inequality imply uniformly for all $0\leq t<\infty$,
\begin{align}
&\|\int_0^1\!\!\!|\Pint(f;\cdot,t+r)|\!\cdot\!r^{s-1}dr\|_{L^p(\bbbr^n)}\nonumber\\
&\leq\int_0^1\!\!\!\|f*P_{t+r}\|_{L^p(\bbbr^n)}\!\cdot\!r^{s-1}dr\nonumber\\
&\leq\!\|f\|_{L^p(\bbbr^n)}\!\cdot\!\|P\|_{L^1(\bbbr^n)}\!\cdot\!
\int_0^1 r^{s-1}dr\!<\!\infty.\label{eq2.271}
\end{align}
Estimates (\ref{eq2.270}) and (\ref{eq2.271}) implicate the function $\Tfunct_s(f;x,0)$ is well-defined and finite for almost every $x\in\bbbr^n$. Therefore the function $\Tfunct_s(f;x,t)$ is well-defined for almost every $x\!\in\!\bbbr^n$ and every $0\!\leq\!t\!<\!\infty$.\\

Step 2: In this step, we prove the function $\Tfunct_s(f;x,t)$ is smooth on $\bbbr^{n+1}_+$, and for every multi-index $\alpha=(\alpha_1,\cdots,\alpha_n,\alpha_{n+1})$, equation (\ref{eq2.374}) is true for all $(x,t)\in\bbbr^{n+1}_+$ and $\|\partial^{\alpha}\Tfunct_s(f;\cdot,t)\|_{L^{\infty}(\bbbr^n)}<\infty$ for each fixed $0<t<\infty$. First, we prove that when $|\alpha|=1$, the following equations are true for all $1\leq k\leq n$ and $(x,t)\in\bbbr^{n+1}_+$,
\begin{align}
\partial_k\Tfunct_s(f;x,t)&\!=\!\!\int_0^{\infty}\!\!\!\!\!\!
\partial_k\Pint(f;x,t\!+\!r)\!\cdot\!r^{s-1}dr,\label{eq2.191}\\
\partial_{n+1}\Tfunct_s(f;x,t)&\!=\!\!\int_0^{\infty}\!\!\!\!\!\!
\partial_{n+1}\Pint(f;x,t\!+\!r)\!\cdot\!r^{s-1}dr,\label{eq2.192}
\end{align}
and the right sides of equations (\ref{eq2.191}) and (\ref{eq2.192}) are well-defined functions in $L^{\infty}(\bbbr^n)$ with respect to variable $x\in\bbbr^n$ for each fixed $0<t<\infty$. Let $x\in\bbbr^n$ and $t\in(0,\infty)$. To prove (\ref{eq2.192}), we have
\begin{align}
&\partial_{n+1}\Tfunct_s(f;x,t)\nonumber\\
&=\!\lim_{h\rightarrow0}\!\int_0^{\infty}\!\!\frac{\Pint(f;x,t\!+\!r\!+\!h)\!-\!
\Pint(f;x,t\!+\!r)}{h}\!\cdot\!r^{s-1}dr,\label{eq2.130}
\end{align}
and we claim that the dominating function for the integrand on the right side of (\ref{eq2.130}) is given by
\begin{equation}\label{eq2.131}
U_1(r):=V_1(r;x,\frac{t}{2})\cdot r^{s-1}
\end{equation}
for every $r\in(0,\infty)$ and uniformly for all $h\in\bbbr$ with $0<|h|<\frac{t}{2}$, where
\begin{equation}\label{eq2.132}
V_1(r;x,w):=\int_{\bbbr^n}\frac{|f(x-y)|}{((w+r)^2+|y|^2)^{\frac{n+1}{2}}}dy.
\end{equation}
Since $f\in L^p(\bbbr^n)$ with $1\leq p<\infty$, we recall formula (\ref{eq1-6}) and the mean value theorem to find a real number $h^*$ between $0$ and $h$ so that
\begin{align}
&\big|\frac{1}{h}\big[\Pint(f;x,t\!+\!r\!+\!h)
\!-\!\Pint(f;x,t\!+\!r)\big]\big|\nonumber\\
&\!=\!\!\big|\partial_{n+1}\Pint(f;x,t\!+\!r\!+\!h^*)\!\big|\!\!\lesssim\!\!\!
\int_{\bbbr^n}\!\!\frac{|f(x\!-\!y)|}{(\!(t\!+\!r\!+\!h^*\!)^2\!+\!|y|^2)^{\frac{n+1}{2}}}
dy,\label{eq2.133}
\end{align}
where the implicit constant depends on $n$. When $0<r<\infty$ and $0<|h^*|<|h|<\frac{t}{2}$, we have $t+h^*\in(\frac{t}{2},\frac{3t}{2})$ and thus
\begin{equation}\label{eq2.381}
(t+r+h^*)^2+|y|^2>(\frac{t}{2}+r)^2+|y|^2.
\end{equation}
Inserting this inequality into (\ref{eq2.133}) and combining the result with the integrand on the right side of (\ref{eq2.130}), we obtain
\begin{align}
&\big|\frac{1}{h}\big[\Pint(f;x,t\!+\!r\!+\!h)
\!-\!\Pint(f;x,t\!+\!r)\big]\big|\!\cdot\!r^{s-1}\nonumber\\
&\!\lesssim\!\!\!\int_{\bbbr^n}\!\!\frac{|f(x\!-\!y)|}
{(\!(\frac{t}{2}\!+\!r\!)^2\!+\!|y|^2)^{\frac{n+1}{2}}}dy\!\cdot\!r^{s-1}\!\!=\!\!V_1(r;x,\frac{t}{2})
\!\cdot\!r^{s-1}\!\!=\!U_1(r).\label{eq2.134}
\end{align}
For all $1\leq p<\infty$ and $0<r,w<\infty$, we use H\"{o}lder's inequality with $\frac{1}{p}+\frac{1}{p'}=1$ and proper changes of variables to obtain the estimate below,
\begin{align}
&\sup_{x\in\bbbr^n}V_1(r;x,w)\!\!\leq\!\!\|f\|_{L^p(\bbbr^n)}\!\cdot\!
\|(\!(w\!+\!r)^2\!+\!|y|^2)^{-\frac{n+1}{2}}\|_{L^{p'}(\bbbr^n)}\nonumber\\
&\lesssim\!\!\|f\|_{L^p(\bbbr^n)}\!\cdot\!(w+r)^{-\frac{n}{p}-1},\label{eq2.135}
\end{align}
and the implicit constant in (\ref{eq2.135}) depends on $n,p$. Applying estimate (\ref{eq2.135}) with $w=\frac{t}{2}$ yields that
\begin{equation}\label{eq2.136}
\int_0^{\infty}U_1(r)dr\lesssim\|f\|_{L^p(\bbbr^n)}\!\cdot\!\int_0^{\infty}\frac{r^{s-1}}
{(\frac{t}{2}+r)^{\frac{n}{p}+1}}dr<\infty,
\end{equation}
since $f\in L^p(\bbbr^n)$ and $0<s<\frac{n}{p}<\frac{n}{p}+1$. Therefore, the dominating function $U_1(r)$ is integrable on the interval $(0,\infty)$ and by the dominated convergence theorem, we can bring the limit in (\ref{eq2.130}) as $h\rightarrow0$ into the integral sign and obtain equation (\ref{eq2.192}). In addition, we use formula (\ref{eq1-6}) and estimate (\ref{eq2.135}) with $w=t$ to obtain
\begin{align}
&\sup_{x\in\bbbr^n}\int_0^{\infty}\!\!\!
|\partial_{n+1}\Pint(f;x,t\!+\!r)|\!\cdot\!r^{s-1}dr\nonumber\\
&\lesssim\sup_{x\in\bbbr^n}\int_0^{\infty}\!\!\!\!\int_{\bbbr^n}\!\!\frac{|f(x\!-\!y)|}
{(\!(t\!+\!r\!)^2\!+\!|y|^2)^{\frac{n+1}{2}}}dy\!\cdot\!r^{s-1}dr\nonumber\\
&=\!\!\sup_{x\in\bbbr^n}\!\!\int_0^{\infty}\!\!\!\!\!\!V_1(r;x,t)\!\cdot\!r^{s-1}dr
\!\lesssim\!\|f\|_{L^p(\bbbr^n)}\!\cdot\!\!\!\int_0^{\infty}\!\!\!\!\!
\frac{r^{s-1}}{(t\!+\!r)^{\frac{n}{p}+1}}dr\!<\!\infty,\label{eq2.137}
\end{align}
where the implicit constants in (\ref{eq2.137}) depend on $n,p$. Equation (\ref{eq2.192}) and estimate (\ref{eq2.137}) imply that the partial derivative of the function $\Tfunct_s(f;x,t)$ with respect to variable $t$ exists and is a well-defined function in $L^{\infty}(\bbbr^n)$ for each fixed $t\in(0,\infty)$. Now we prove equation (\ref{eq2.191}). We will only prove equation (\ref{eq2.191}) for the case $k=1$, and the proof of equation (\ref{eq2.191}) when $k\in\{2,\cdots,n\}$ will be similar. We denote
\begin{equation}\label{eq2.138}
\begin{split}
&x=(x_1,x_2,\cdots,x_n)=(x_1,x'),\\
&x'=(x_2,\cdots,x_n),\quad\vec{e}_1=(1,0,\cdots,0)\in\bbbr^n\\
&E_2\!:=\!\{y\!\in\!\bbbr^n\!:\!|y_1\!-\!x_1|\!>\!t\},\quad\!\!\!\! E_3\!:=\!\{y\!\in\!\bbbr^n\!:\!|y_1\!-\!x_1|\!\leq\!t\}.
\end{split}
\end{equation}
We have
\begin{align}
&\partial_1\Tfunct_s(f;x,t)\nonumber\\
&\!=\!\lim_{h\rightarrow0}\!\int_0^{\infty}\!\!
\frac{\Pint(f;x\!+\!h\vec{e}_1,t\!+\!r)\!-\!\!\Pint(f;x,t\!+\!r)}{h}\!\cdot\!r^{s-1}dr,\label{eq2.139}
\end{align}
and we claim that the dominating function for the integrand on the right side of (\ref{eq2.139}) is
\begin{equation}\label{eq2-138}
U_2(r)+U_3(r)
\end{equation}
for every $r\in(0,\infty)$ and uniformly for all $h\in\bbbr$ with $0<|h|<\frac{t}{2}$, where $U_2(r)=V_2(r;x,t)\!\cdot\!r^{s-1}$, $U_3(r)=V_3(r;x,t)\!\cdot\!r^{s-1}$, and
\begin{align}
&V_2(r;x,t)=\int_{E_2}\frac{|f(y)|}
{\big[(t+r)^2+\frac{1}{4}|x-y|^2\big]^{\frac{n+1}{2}}}dy,\label{eq2.140}\\
&V_3(r;x,t)=\int_{E_3}\frac{|f(y)|}
{\big[(t+r)^2+|x'-y'|^2\big]^{\frac{n+1}{2}}}dy.\label{eq2.141}
\end{align}
By using formula (\ref{eq1-5}) with $k=1$ and the mean value theorem, we can find a real number $h^*$ between $0$ and $h$ so that we have
\begin{align}
&\big|\frac{1}{h}\big[\Pint(f;x\!+\!h\vec{e}_1,t\!+\!r)\!-\!\!
\Pint(f;x,t\!+\!r)\big]\big|\nonumber\\
&=\big|\partial_1\Pint(f;x\!+\!h^*\vec{e}_1,t\!+\!r)\big|\nonumber\\
&\lesssim\!\!\!\int_{\bbbr^n}\!\!\frac{|f(y)|}
{\big[(t+r)^2+|x'-y'|^2+(h^*+x_1-y_1)^2\big]^{\frac{n+1}{2}}}dy\nonumber\\
&\lesssim\!\!\!\int_{\bbbr^n}\!\!\frac{|f(y)|}
{\big[(t+r)^2+|x'-y'|^2+g(|h^*|)\big]^{\frac{n+1}{2}}}dy,\label{eq2.142}
\end{align}
where the implicit constants in (\ref{eq2.142}) depend on $n$, and $g(|h^*|)$ is the evaluation of the nonnegative quadratic function
\begin{equation}\label{eq2.143}
g(w)=w^2-2|x_1-y_1|w+|x_1-y_1|^2
\end{equation}
at $w=|h^*|$. When $0<|h|<\frac{t}{2}$ and $y\!\in\!E_2$, we have $0<|h^*|<|h|<\frac{1}{2}t<\frac{1}{2}|x_1-y_1|$. And on the closed and bounded interval $[0,\frac{1}{2}|x_1-y_1|]$, the nonnegative quadratic function $g(w)$ defined in (\ref{eq2.143}) is decreasing and achieves a nonnegative minimum value at $w=\frac{1}{2}|x_1-y_1|$, and this nonnegative minimum value is $\frac{1}{4}|x_1-y_1|^2$. We have
\begin{align}
&\int_{E_2}\!\!\frac{|f(y)|}{\big[(t+r)^2+|x'-y'|^2+g(|h^*|)\big]^{\frac{n+1}{2}}}dy\nonumber\\
&\leq\!\!\!\int_{E_2}\!\!\frac{|f(y)|}{\big[\!(t\!+\!r)^2\!+\!|x'\!-\!y'|^2\!+
\!\frac{1}{4}|x_1\!-\!y_1|^2\big]^{\frac{n+1}{2}}}dy\!\leq\!V_2(r;x,t).\label{eq2.144}
\end{align}
When $y\!\in\!E_3$, we use the nonnegativity of the function (\ref{eq2.143}) and obtain
\begin{equation}\label{eq2.145}
\int_{E_3}\!\!\frac{|f(y)|}
{\big[\!(t\!+\!r)^2\!+\!|x'\!-\!y'|^2\!+\!g(|h^*|)\big]^{\frac{n+1}{2}}}dy\!\leq\!V_3(r;x,t).
\end{equation}
Combining (\ref{eq2.142}), (\ref{eq2.144}), (\ref{eq2.145}) with the fact that $\bbbr^n=E_2\cup E_3$, we have proven the following estimate
\begin{align}
&\big|\frac{1}{h}\big[\Pint(f;x\!+\!h\vec{e}_1,t\!+\!r)\!-\!\!\Pint(f;x,t\!+\!r)\big]\big|\nonumber\\
&\lesssim V_2(r;x,t)+V_3(r;x,t)\label{eq2.146}
\end{align}
for every $r\in(0,\infty)$ and uniformly for all $h\in\bbbr$ with $0<|h|<\frac{t}{2}$, and the implicit constant in (\ref{eq2.146}) depends on $n$. To prove that we can bring the limit in (\ref{eq2.139}) as $h\rightarrow0$ into the integral sign and obtain equation (\ref{eq2.191}) for $k=1$ by the dominated convergence theorem, it suffices to prove
\begin{align}
	&\int_0^{\infty}U_2(r)+U_3(r)dr\nonumber\\
	&=\int_0^{\infty}V_2(r;x,t)r^{s-1}dr+\int_0^{\infty}V_3(r;x,t)r^{s-1}dr<\infty.\label{eq2.147}
\end{align}
We have the following pointwise estimate (\ref{eq2.148}) for every $r\in(0,\infty)$,
\begin{align}
&V_2(r;x,t)\leq\int_{\bbbr^n}\frac{|f(y)|}
{\big[(t+r)^2+\frac{1}{4}|x-y|^2\big]^{\frac{n+1}{2}}}dy\nonumber\\
&\lesssim\int_{\bbbr^n}\frac{|f(y)|}
{\big[(t+r)^2+|x-y|^2\big]^{\frac{n+1}{2}}}dy=V_1(r;x,t),\label{eq2.148}
\end{align}
where the implicit constant depends on $n$. We use estimate (\ref{eq2.137}) to obtain
\begin{equation}\label{eq2.149}
\int_0^{\infty}\!\!\!V_2(r;x,t)r^{s-1}dr\lesssim\int_0^{\infty}\!\!\!V_1(r;x,t)r^{s-1}dr<\infty.
\end{equation}
When $1<p<\infty$, we use H\"{o}lder's inequality with $\frac{1}{p}+\frac{1}{p'}=1$ to obtain the estimate below,
\begin{align}
&V_3(r;x,t)\!\!\leq\!\!\|f\|_{L^p(\bbbr^n)}\!\cdot\!\!\bigg(\!\int_{E_3}\!\!\!
[(t\!+\!r)^2\!+\!|x'\!-\!y'|^2]^{-\frac{(n+1)p'}{2}}\!dy\!\!\bigg)^{\!\!\frac{1}{p'}}\nonumber\\
&=\|f\|_{L^p(\bbbr^n)}\!\cdot\!(2t)^{\frac{1}{p'}}\!\cdot\!\bigg(\!\int_{\bbbr^{n-1}}\!\!\!\!\!\!
[(t\!+\!r)^2\!+\!|y'|^2]^{-\frac{(n+1)p'}{2}}\!dy'\!\!\bigg)^{\!\!\frac{1}{p'}}\nonumber\\
&=\|f\|_{L^p(\bbbr^n)}\!\cdot\!(2t)^{\frac{1}{p'}}\!\cdot\!(t\!+\!r)^{-\frac{n-1}{p}-2}\!\cdot\!
\bigg(\!\int_{\bbbr^{n-1}}\!\!\!\!\!\![1\!+\!|y'|^2]^{-\frac{(n+1)p'}{2}}\!
dy'\!\!\bigg)^{\!\!\frac{1}{p'}}\nonumber\\
&\lesssim\|f\|_{L^p(\bbbr^n)}\!\cdot\!(2t)^{\frac{1}{p'}}\!\cdot\!(t\!+\!r)^{-\frac{n-1}{p}-2},
\label{eq2.150}
\end{align}
and the implicit constant in (\ref{eq2.150}) depends on $n,p$. When $p=1$, we also have
\begin{align}
&V_3(r;x,t)\!\!\leq\!\!\|f\|_{L^1(E_3)}\!\cdot\!\!\sup_{z\in E_3}
\frac{1}{[(t\!+\!r)^2\!+\!|x'\!-\!z'|^2]^{\frac{n+1}{2}}}\nonumber\\
&\leq\|f\|_{L^1(\bbbr^n)}\!\cdot\!\!\sup_{\substack{z_1\in[x_1-t,x_1+t]\\z'\in\bbbr^{n-1}}}
\frac{1}{[(t\!+\!r)^2\!+\!|x'\!-\!z'|^2]^{\frac{n+1}{2}}}\nonumber\\
&=\|f\|_{L^1(\bbbr^n)}\!\cdot\!\!\sup_{z'\in\bbbr^{n-1}}
\frac{1}{[(t\!+\!r)^2\!+\!|x'\!-\!z'|^2]^{\frac{n+1}{2}}}\nonumber\\
&\leq\|f\|_{L^1(\bbbr^n)}\!\cdot\!(t+r)^{-n-1}.\label{eq2.382}
\end{align}
Combining (\ref{eq2.150}) and (\ref{eq2.382}) yields the following estimate,
\begin{equation}\label{eq2.383}
V_3(r;x,t)\lesssim\|f\|_{L^p(\bbbr^n)}\!\cdot\!(2t)^{\frac{1}{p'}}\!\cdot\!
(t\!+\!r)^{-\frac{n-1}{p}-2}
\end{equation}
for $1\leq p<\infty$, and the implicit constant in (\ref{eq2.383}) depends on $n,p$. Thus we obtain
\begin{equation}\label{eq2.151}
\int_0^{\infty}\!\!\!\!\!V_3(r;x,t)\!\cdot\!r^{s-1}dr\!\lesssim\!\|f\|_{L^p(\bbbr^n)}\!\cdot\!\!\!\int_0^{\infty}\!\!\!\frac{(2t)^{\frac{1}{p'}}\!\!\cdot\!r^{s-1}}{(t\!+\!r)^{\frac{n-1}{p}+2}}dr\!<\!\infty,
\end{equation}
for $1\leq p<\infty$, $0<s<\frac{n}{p}<\frac{n-1}{p}+2$, and $f\in L^p(\bbbr^n)$. Combining (\ref{eq2.149}) and (\ref{eq2.151}) proves (\ref{eq2.147}), and hence equation (\ref{eq2.191}) with $k=1$ is proved by invoking the dominated convergence theorem. In addition, we use formula (\ref{eq1-5}) with $k=1$ and estimate (\ref{eq2.137}) to obtain
\begin{align}
&\sup_{x\in\bbbr^n}\int_0^{\infty}\!\!\!\!\!\!|\partial_1\Pint(f;x,t\!+\!r)|\!\cdot\!r^{s-1}dr
\nonumber\\
&\lesssim\sup_{x\in\bbbr^n}\int_0^{\infty}\!\!\!\!\!\!V_1(r;x,t)\!\cdot\!r^{s-1}dr\!<\!\infty.
\label{eq2.152}
\end{align}
Therefore the partial derivative of the function $\Tfunct_s(f;x,t)$ with respect to variable $x_1$ exists and is a well-defined function in $L^{\infty}(\bbbr^n)$ for each fixed $t\in(0,\infty)$. The proof for the case $k\in\{2,\cdots,n\}$ is similar. Second, we prove that when $\alpha=(\alpha_1,\cdots,\alpha_n,\alpha_{n+1})$ and $|\alpha|=2$, the following equation is true for all $(x,t)\in\bbbr^{n+1}_+$,
\begin{equation}\label{eq2.384}
\partial^{\alpha}\Tfunct_s(f;x,t)=\int_0^{\infty}\!\!\!\!
\partial^{\alpha}\Pint(f;x,t\!+\!r)\!\cdot\!r^{s-1}dr,
\end{equation}
and the right side of equation (\ref{eq2.384}) is a well-defined function in $L^{\infty}(\bbbr^n)$ with respect to variable $x\in\bbbr^n$ for each fixed $0<t<\infty$. Let $(x,t)\in\bbbr^{n+1}_+$. To prove (\ref{eq2.384}) where $\alpha=(\alpha_1,\cdots,\alpha_n,\alpha_{n+1})$ and $\alpha_{n+1}\geq1$, we use equation (\ref{eq2.192}) and obtain
\begin{align}
&\partial^2_{n+1}\Tfunct_s(f;x,t)\nonumber\\
&=\!\lim_{h\rightarrow0}\frac{1}{h}\big[\partial_{n+1}\Tfunct_s(f;x,t+h)
-\partial_{n+1}\Tfunct_s(f;x,t)\big]\nonumber\\
&=\!\lim_{h\rightarrow0}\!\int_0^{\infty}\!\frac{\partial_{n+1}\Pint(f;x,t\!+\!r\!+\!h)\!-\!
\partial_{n+1}\Pint(f;x,t\!+\!r)}{h}\!\cdot\!r^{s-1}dr,\label{eq2.385}
\end{align}
and we claim that the dominating function for the integrand in the last line of (\ref{eq2.385}) is given by
\begin{equation}\label{eq2.386}
U_4(r):=V_4(r;x,\frac{t}{2})\cdot r^{s-1}
\end{equation}
for every $r\in(0,\infty)$ and uniformly for all $h\in\bbbr$ with $0<|h|<\frac{t}{2}$, where
\begin{equation}\label{eq2.387}
V_4(r;x,w):=\int_{\bbbr^n}\frac{|f(x-y)|}{((w+r)^2+|y|^2)^{\frac{n+2}{2}}}dy.
\end{equation}
Since $f\in L^p(\bbbr^n)$ with $1\leq p<\infty$, we recall formula (\ref{eq2.168}) and the mean value theorem to find a real number $h^*$ between $0$ and $h$ so that
\begin{align}
&\big|\frac{1}{h}\big[\partial_{n+1}\Pint(f;x,t\!+\!r\!+\!h)
\!-\!\partial_{n+1}\Pint(f;x,t\!+\!r)\big]\big|\nonumber\\
&\!=\!\!\big|\partial^2_{n+1}\Pint(f;x,t\!+\!r\!+\!h^*)\!\big|\!\!\lesssim\!\!\!
\int_{\bbbr^n}\!\!\frac{|f(x\!-\!y)|}{(\!(t\!+\!r\!+\!h^*\!)^2\!+\!|y|^2)^{\frac{n+2}{2}}}
dy,\label{eq2.388}
\end{align}
where the implicit constant depends on $n$. When $0<r<\infty$ and $0<|h^*|<|h|<\frac{t}{2}$, inequality (\ref{eq2.381}) is true. We insert (\ref{eq2.381}) into (\ref{eq2.388}) and combine the result with the integrand in the last line of (\ref{eq2.385}), then we have
\begin{align}
&\big|\frac{1}{h}\big[\partial_{n+1}\Pint(f;x,t\!+\!r\!+\!h)\!-\!
\partial_{n+1}\Pint(f;x,t\!+\!r)\big]\big|\!\cdot\!r^{s-1}\nonumber\\
&\lesssim V_4(r;x,\frac{t}{2})\!\cdot\!r^{s-1}\!\!=\!U_4(r).\label{eq2.389}
\end{align}
For all $1\leq p<\infty$ and $0<r,w<\infty$, we use H\"{o}lder's inequality with $\frac{1}{p}+\frac{1}{p'}=1$ and proper changes of variables to obtain
\begin{equation}\label{eq2.390}
\sup_{x\in\bbbr^n}V_4(r;x,w)\lesssim\|f\|_{L^p(\bbbr^n)}\!\cdot\!(w+r)^{-\frac{n}{p}-2},
\end{equation}
and the implicit constant in (\ref{eq2.390}) depends on $n,p$. Applying estimate (\ref{eq2.390}) with $w=\frac{t}{2}$ and the condition that $f\in L^p(\bbbr^n)$, $0<s<\frac{n}{p}$ yields
\begin{equation}\label{eq2.391}
\int_0^{\infty}U_4(r)dr\lesssim\|f\|_{L^p(\bbbr^n)}\!\cdot\!\int_0^{\infty}\frac{r^{s-1}}
{(\frac{t}{2}+r)^{\frac{n}{p}+2}}dr<\infty.
\end{equation}
Therefore, the dominating function $U_4(r)$ is integrable on the interval $(0,\infty)$ and by the dominated convergence theorem, we can bring the limit in (\ref{eq2.385}) as $h\rightarrow0$ into the integral sign and obtain the equation
\begin{equation}\label{eq2.392}
\partial^2_{n+1}\Tfunct_s(f;x,t)=\int_0^{\infty}\!\!\!\!
\partial^2_{n+1}\Pint(f;x,t\!+\!r)\!\cdot\!r^{s-1}dr.
\end{equation}
In addition, we use formula (\ref{eq2.168}) and estimate (\ref{eq2.390}) with $w=t$ to deduce that
\begin{align}
&\sup_{x\in\bbbr^n}\!\int_0^{\infty}\!\!\!\!\!|\partial^2_{n+1}\Pint(f;x,t\!+\!r)|\!\cdot\!r^{s-1}dr
\!\lesssim\!\sup_{x\in\bbbr^n}\!\int_0^{\infty}\!\!\!\!\!V_4(r;x,t)\!\cdot\!r^{s-1}dr\nonumber\\
&\lesssim\|f\|_{L^p(\bbbr^n)}\!\cdot\!\int_0^{\infty}\frac{r^{s-1}}
{(t+r)^{\frac{n}{p}+2}}dr<\infty,\label{eq2.393}
\end{align}
with the implicit constants depending on $n,p$. Thus $\partial^2_{n+1}\Tfunct_s(f;x,t)$ is a well-defined function in $L^{\infty}(\bbbr^n)$ for each fixed $0<t<\infty$. Next, we prove that for all $1\leq k\leq n$ and $(x,t)\in\bbbr^{n+1}_+$, we have
\begin{equation}\label{eq2.394}
\partial_k\partial_{n+1}\Tfunct_s(f;x,t)=\int_0^{\infty}\!\!\!\!
\partial_k\partial_{n+1}\Pint(f;x,t\!+\!r)\!\cdot\!r^{s-1}dr,
\end{equation}
and $\|\partial_k\partial_{n+1}\Tfunct_s(f;\cdot,t)\|_{L^{\infty}(\bbbr^n)}<\infty$ for each fixed $0<t<\infty$. We will only prove this conclusion for the case $k=1$, and the proof for the case where $2\leq k\leq n$ will be similar. We recall (\ref{eq2.138}) and use equation (\ref{eq2.192}) to obtain
\begin{align}
&\partial_1\partial_{n+1}\Tfunct_s(f;x,t)\nonumber\\
&=\!\lim_{h\rightarrow0}\frac{1}{h}\big[\partial_{n+1}\Tfunct_s(f;x+h\vec{e}_1,t)
-\partial_{n+1}\Tfunct_s(f;x,t)\big]\nonumber\\
&=\!\lim_{h\rightarrow0}\!\int_0^{\infty}\!\frac{\partial_{n+1}\Pint(f;\!x\!+\!h\vec{e}_1,\!t\!+\!r)
\!-\!\partial_{n+1}\Pint(f;\!x,\!t\!+\!r)}{h}\!\cdot\!r^{s-1}dr,\label{eq2.395}
\end{align}
and we claim that the dominating function for the integrand in the last line of (\ref{eq2.395}) is
\begin{equation}\label{eq2.396}
U_5(r)+U_6(r)
\end{equation}
for every $r\in(0,\infty)$ and uniformly for all $h\in\bbbr$ with $0<|h|<\frac{t}{2}$, where $U_5(r)=V_5(r;x,t)\!\cdot\!r^{s-1}$, $U_6(r)=V_6(r;x,t)\!\cdot\!r^{s-1}$, and
\begin{align}
&V_5(r;x,t):=\int_{E_2}\frac{|f(y)|}
{\big[(t+r)^2+\frac{1}{4}|x-y|^2\big]^{\frac{n+2}{2}}}dy,\label{eq2.397}\\
&V_6(r;x,t):=\int_{E_3}\frac{|f(y)|}
{\big[(t+r)^2+|x'-y'|^2\big]^{\frac{n+2}{2}}}dy.\label{eq2.398}
\end{align}
By using formula (\ref{eq2-55}) with $k=1$ and the mean value theorem, we can find a real number $h^*$ between $0$ and $h$ so that we have
\begin{align}
&\big|\frac{1}{h}\big[\partial_{n+1}\Pint(f;\!x\!+\!h\vec{e}_1,\!t\!+\!r)
\!-\!\partial_{n+1}\Pint(f;\!x,\!t\!+\!r)\big]\big|\nonumber\\
&=\big|\partial_1\partial_{n+1}\Pint(f;x\!+\!h^*\vec{e}_1,t\!+\!r)\big|\nonumber\\
&\lesssim\!\!\!\int_{\bbbr^n}\!\!\frac{|f(y)|}
{\big[(t+r)^2+|x'-y'|^2+(h^*+x_1-y_1)^2\big]^{\frac{n+2}{2}}}dy\nonumber\\
&\lesssim\!\!\!\int_{\bbbr^n}\!\!\frac{|f(y)|}
{\big[(t+r)^2+|x'-y'|^2+g(|h^*|)\big]^{\frac{n+2}{2}}}dy,\label{eq2.399}
\end{align}
where the implicit constants in (\ref{eq2.399}) depend on $n$, and $g(|h^*|)$ is the evaluation of the function (\ref{eq2.143}) at $w=|h^*|$. When $0<|h|<\frac{t}{2}$ and $y\!\in\!E_2$, we have $0<|h^*|<|h|<\frac{1}{2}t<\frac{1}{2}|x_1-y_1|$ and $g(|h^*|)\geq\frac{1}{4}|x_1-y_1|^2$. When $0<|h|<\frac{t}{2}$ and $y\!\in\!E_3$, we have $g(|h^*|)\geq0$. Since $\bbbr^n=E_2\cup E_3$, we can combine the above estimates with (\ref{eq2.399}) and obtain
\begin{equation}\label{eq2.400}
(\ref{eq2.399})\lesssim V_5(r;x,t)+V_6(r;x,t),
\end{equation}
where the implicit constant in (\ref{eq2.400}) depends on $n$. Combining (\ref{eq2.399}), (\ref{eq2.400}) with the integrand in the last line of (\ref{eq2.395}) proves the claim (\ref{eq2.396}). Now we prove
\begin{equation}\label{eq2.401}
\int_0^{\infty}U_5(r)+U_6(r)dr<\infty.
\end{equation}
The following pointwise estimate (\ref{eq2.402}) is true for every $0<r<\infty$,
\begin{equation}\label{eq2.402}
V_5(r;x,t)\lesssim V_4(r;x,t),
\end{equation}
with the implicit constant depending on $n$. Thus we deduce from (\ref{eq2.393}) that
\begin{align}
&\int_0^{\infty}U_5(r)dr=\int_0^{\infty}V_5(r;x,t)\cdot r^{s-1}dr\nonumber\\
&\lesssim\int_0^{\infty}V_4(r;x,t)\cdot r^{s-1}dr<\infty.\label{eq2.403}
\end{align}
We use H\"{o}lder's inequality with $\frac{1}{p}+\frac{1}{p'}=1$ and argue as in (\ref{eq2.150}) and (\ref{eq2.382}) to obtain the following estimate,
\begin{equation}\label{eq2.404}
V_6(r;x,t)\lesssim\|f\|_{L^p(\bbbr^n)}\!\cdot\!(2t)^{\frac{1}{p'}}\!\cdot\!
(t\!+\!r)^{-\frac{n-1}{p}-3}
\end{equation}
for all $1\leq p<\infty$, and the implicit constant in (\ref{eq2.404}) depends on $n,p$. Hence we have
\begin{align}
&\int_0^{\infty}\!\!\!\!U_6(r)dr=\int_0^{\infty}\!\!\!\!V_6(r;x,t)\!\cdot\!r^{s-1}dr\nonumber\\
&\lesssim\!\|f\|_{L^p(\bbbr^n)}\!\cdot\!\!\!\int_0^{\infty}\!\!\!\frac{(2t)^{\frac{1}{p'}}\!\!\cdot\!r^{s-1}}{(t\!+\!r)^{\frac{n-1}{p}+3}}dr\!<\!\infty,\label{eq2.405}
\end{align}
for $1\leq p<\infty$, $0<s<\frac{n}{p}$, and $f\in L^p(\bbbr^n)$. Combining (\ref{eq2.403}) and (\ref{eq2.405}) proves (\ref{eq2.401}), hence equation (\ref{eq2.394}) with $k=1$ is proved by invoking the dominated convergence theorem. In addition, we use formula (\ref{eq2-55}) with $k=1$ and estimate (\ref{eq2.393}) to obtain
\begin{align}
&\sup_{x\in\bbbr^n}\int_0^{\infty}\!\!\!\!\!|\partial_1\partial_{n+1}\Pint(f;x,t\!+\!r)|\!\cdot\!
r^{s-1}dr\nonumber\\
&\lesssim\sup_{x\in\bbbr^n}\int_0^{\infty}\!\!\!\!\!V_4(r;x,t)\!\cdot\!r^{s-1}dr
<\infty,\label{eq2.406}
\end{align}
with the implicit constant depending on $n$. Thus $\partial_1\partial_{n+1}\Tfunct_s(f;x,t)$ exists and is a well-defined function in $L^{\infty}(\bbbr^n)$ for each fixed $0<t<\infty$. Given a multi-index $\alpha=(\alpha_1,\cdots,\alpha_n,\alpha_{n+1})$ with $|\alpha|=2$, we can use formulae (\ref{eq2.167}), (\ref{eq2.168}), (\ref{eq2-54}), (\ref{eq2-55}), and analogous arguments to prove the equation (\ref{eq2.384}) for all $(x,t)\in\bbbr^{n+1}_+$, and $\partial^{\alpha}\Tfunct_s(f;x,t)$ is a well-defined function in $L^{\infty}(\bbbr^n)$ with respect to variable $x\in\bbbr^n$ for each fixed $0<t<\infty$. Exercising inductive computations shows that the function $\Tfunct_s(f;x,t)$ is smooth on $\bbbr^{n+1}_+$, and for every multi-index $\alpha$, the equation (\ref{eq2.374}) holds true for all $(x,t)\in\bbbr^{n+1}_+$ and the function $\partial^{\alpha}\Tfunct_s(f;x,t)$ is a well-defined function in $L^{\infty}(\bbbr^n)$ with respect to variable $x\in\bbbr^n$ for each fixed $0<t<\infty$.\\

Step 3: In this step, we prove the conclusions (\ref{eq2.375}), (\ref{eq2.376}), and (\ref{eq2.377}). Since the function $\Tfunct_s(f;x,t)$ as given in (\ref{eq2.373}) is well-defined and finite for almost every $x\in\bbbr^n$ and every $0\leq t<\infty$, we denote
\begin{align}
&T(x,t):=\Tfunct_s(f;x,t)-\Tfunct_s(f;x,0)\nonumber\\
&=\!\!\int_0^{\infty}\!\!\!\![f*P_{t+r}(x)\!-\!f*P_{r}(x)]\!\cdot\!r^{s-1}dr,\label{eq2.272}
\end{align}
\begin{align}
T^0(x,t)&:=\int_0^{1}\!\![f*P_{t+r}(x)-f*P_{r}(x)]\cdot r^{s-1}dr,\label{eq2.273}\\
T^{\infty}(x,t)&:=\int_1^{\infty}\!\!\!\![f*P_{t+r}(x)-f*P_{r}(x)]\cdot r^{s-1}dr,\label{eq2.274}
\end{align}
where $T^0(x,t)$ is the integral on the interval $r\in(0,1]$, and $T^{\infty}(x,t)$ is the integral on the interval $r\in[1,\infty)$, and we have
\begin{equation}\label{eq2.275}
|T(x,t)|\leq|T^0(x,t)|+|T^{\infty}(x,t)|.
\end{equation}
We prove (\ref{eq2.375}) and (\ref{eq2.376}) by proving that
\begin{equation}\label{eq2.276}
\lim_{t\rightarrow0}\int_K|T(x,t)|dx=0\text{ for every compact set $K\subseteq\bbbr^n$}.
\end{equation}
For every $x\in\bbbr^n$ and every fixed $r\in(0,\infty)$, we use the mean value theorem to find a real number $t^*$ so that $0<t^*<t$, and we also use formula (\ref{eq1-3-2}) and H\"{o}lder's inequality with $\frac{1}{p}\!+\!\frac{1}{p'}\!=\!1$ to obtain
\begin{align}
&\sup_{x\in\bbbr^n}\!|f\!*\!P_{t+r}(x)\!-\!f\!*\!P_{r}(x)|\!\leq\!\|f\|_{L^p(\bbbr^n)}\!\cdot\!
\|P_{t+r}-P_{r}\|_{L^{p'}(\bbbr^n)}\nonumber\\
&\lesssim\|f\|_{L^p(\bbbr^n)}\!\cdot\!\bigg(\!\int_{\bbbr^n}\!
\frac{t^{p'}}{((t^*+r)^2\!+\!|y|^2)^{\frac{(n+1)p'}{2}}}dy\!\bigg)^{\!\frac{1}{p'}}\nonumber\\
&\lesssim\!\!\|f\|_{L^p(\bbbr^n)}\!\cdot\!\bigg(\!\int_{\bbbr^n}\!\!
\frac{t^{p'}}{(r^2\!+\!|y|^2)^{\frac{(n+1)p'}{2}}}dy\!\!\bigg)^{\!\!\frac{1}{p'}}\!\!\!\!
\lesssim\!\!\|f\|_{L^p(\bbbr^n)}\!\cdot\!r^{-\frac{n}{p}-1}\!\cdot\!t,\label{eq2.277}
\end{align}
where the implicit constants in (\ref{eq2.277}) depend on $n,p$. We deduce from (\ref{eq2.277}) that
\begin{align}
&\lim_{t\rightarrow0}\|T^{\infty}(\cdot,t)\|_{L^{\infty}(\bbbr^n)}\!\!\leq\!\lim_{t\rightarrow0}\!
\int_1^{\infty}\!\!\!\!\!\!\|f\!*\!P_{t+r}\!-\!f\!*\!P_{r}\|_{L^{\infty}(\bbbr^n)}
\!\cdot\!r^{s-1}dr\nonumber\\
&\lesssim\lim_{t\rightarrow0}\|f\|_{L^p(\bbbr^n)}\!\cdot\!\int_1^{\infty}r^{s-\frac{n}{p}-2}dr\cdot t=0,\label{eq2.278}
\end{align}
since $0<s<\frac{n}{p}$ implies $\int_1^{\infty}r^{s-\frac{n}{p}-2}dr<\infty$. Therefore we have
\begin{equation}\label{eq2.279}
\lim_{t\rightarrow0}\int_K|T^{\infty}(x,t)|dx=0\text{ for every compact set $K\subseteq\bbbr^n$}.
\end{equation}
In order to prove that
\begin{equation}\label{eq2.280}
\lim_{t\rightarrow0}\int_K|T^{0}(x,t)|dx=0\text{ for every compact set $K\subseteq\bbbr^n$},
\end{equation}
we first prove that
\begin{equation}\label{eq2.281}
\lim_{t\rightarrow0}\int_{\bbbr^n}|T^{0}(x,t)|\!\cdot\!g(x)dx=0\text{ for every $g\in\Sw(\bbbr^n)$}.
\end{equation}
Let $g\in\Sw(\bbbr^n)$ and we recall (\ref{eq2.273}) to obtain
\begin{align}
&\int_{\bbbr^n}|T^{0}(x,t)\cdot g(x)|dx\nonumber\\
&\leq\!\!\!\int_{\bbbr^n}\!\int_0^{1}\!\!\!\!|f\!*\!P_{t+r}(x)\!-\!f\!*\!P_{r}(x)|
\!\cdot\!|g(x)|r^{s-1}drdx.\label{eq2.282}
\end{align}
Applying H\"{o}lder's inequality with $\frac{1}{p}\!+\!\frac{1}{p'}\!=\!1$ and $1\leq p<\infty$, Young's inequality, and the assumption that $f\in L^p(\bbbr^n)$ in a sequence yields
\begin{align}
&\int_0^{1}\!\!\!\int_{\bbbr^n}\!\!\!\!|f\!*\!P_{t+r}(x)\!-\!f\!*\!P_{r}(x)|
\!\cdot\!|g(x)|r^{s-1}dxdr\nonumber\\
&\leq\!\!\!\int_0^{1}\!\!\!\!\|f\!*\!P_{t+r}\!-\!f\!*\!P_{r}\|_{L^p(\bbbr^n)}\!\cdot\!
\|g\|_{L^{p'}(\bbbr^n)}\!\cdot\!r^{s-1}dr\nonumber\\
&\leq\!\!\!\int_0^{1}\!\!\!\!\{\|f\!*\!P_{t+r}\|_{L^p(\bbbr^n)}\!+\!
\|f\!*\!P_{r}\|_{L^p(\bbbr^n)}\}\!\cdot\!\|g\|_{L^{p'}(\bbbr^n)}\!\cdot\!r^{s-1}dr\nonumber\\
&\leq2\|f\|_{L^p(\bbbr^n)}\!\cdot\!\|P\|_{L^1(\bbbr^n)}\!\cdot\!\|g\|_{L^{p'}(\bbbr^n)}\!\cdot\!
\int_0^1 r^{s-1}dr<\infty.\label{eq2.283}
\end{align}
Estimate (\ref{eq2.283}) justifies the application of Fubini's theorem and the exchange of the order of integration in the second line of (\ref{eq2.282}), and we deduce from (\ref{eq2.282}) that
\begin{align}
&\lim_{t\rightarrow0}\int_{\bbbr^n}|T^{0}(x,t)\cdot g(x)|dx\nonumber\\
&\leq\lim_{t\rightarrow0}\int_0^{1}\!\!\!\int_{\bbbr^n}\!\!\!\!
|f\!*\!P_{t+r}(x)\!-\!f\!*\!P_{r}(x)|\!\cdot\!|g(x)|dx\!\cdot\!r^{s-1}dr\nonumber\\
&=\!\!\!\int_0^{1}\!\!\lim_{t\rightarrow0}\!\int_{\bbbr^n}\!\!\!\!
|f\!*\!P_{t+r}(x)\!-\!f\!*\!P_{r}(x)|\!\cdot\!|g(x)|dx\!\cdot\!r^{s-1}dr,\label{eq2.284}
\end{align}
where the last equation of (\ref{eq2.284}) is due to the dominated convergence theorem, since estimate (\ref{eq2.283}) also shows that uniformly for all $t\in(0,\infty)$, the dominating function for the integrand function
\begin{equation}\label{eq2.285}
r\!\in\!(0,1]\longmapsto\!\!\int_{\bbbr^n}\!\!\!\!
|f\!*\!P_{t+r}(x)\!-\!f\!*\!P_{r}(x)|\!\cdot\!|g(x)|dx\!\cdot\!r^{s-1}
\end{equation}
is given by the following
\begin{equation}\label{eq2.286}
r\!\in\!(0,1]\longmapsto2\|f\|_{L^p(\bbbr^n)}\!\cdot\!\|P\|_{L^1(\bbbr^n)}\!\cdot\!\|g\|_{L^{p'}(\bbbr^n)}\!\cdot\!r^{s-1},
\end{equation}
and this dominating function (\ref{eq2.286}) is in $L^1((0,1])$. Inserting estimate (\ref{eq2.277}) into (\ref{eq2.285}) yields that the integrand function (\ref{eq2.285}) can be estimated from above by the function (\ref{eq2.287}) pointwise in $r$,
\begin{equation}\label{eq2.287}
r\!\in\!(0,1]\longmapsto\|f\|_{L^p(\bbbr^n)}\!\cdot\!\|g\|_{L^1(\bbbr^n)}\!\cdot\!r^{s-\frac{n}{p}-2}
\!\cdot\!t.
\end{equation}
Thus the integrand function (\ref{eq2.285}) converges to $0$ as $t\rightarrow0$ for every $r\!\in\!(0,1]$. Combining this conclusion with (\ref{eq2.284}) proves (\ref{eq2.281}). As a consequence of Minkowski's integral inequality, Young's inequality, and the estimate (\ref{eq2.271}), we have
\begin{equation}\label{eq2.288}
\sup_{0<t<\infty}\!\!\|T^0(\cdot,t)\|_{L^p(\bbbr^n)}\!\leq\!2\|f\|_{L^p(\bbbr^n)}\!\cdot\!
\|P\|_{L^1(\bbbr^n)}\!\cdot\!\!\int_0^1\!\!\!r^{s-1}dr\!<\!\infty.
\end{equation}
Given a compact set $K\subseteq\bbbr^n$ and $\varepsilon>0$, we denote its $\varepsilon$-neighborhood by (\ref{eq2.211}), and we can find a function $\varphi_{\varepsilon}(x)\in\Ccinfty(\bbbr^n)\subseteq\Sw(\bbbr^n)$ so that all the conditions (\ref{eq2.212}) and (\ref{eq2.213}) are satisfied. Then H\"{o}lder's inequality with $\frac{1}{p}\!+\!\frac{1}{p'}\!=\!1$ and condition (\ref{eq2.288}) imply
\begin{align}
&\int_{K_{\varepsilon}\setminus K}|T^0(x,t)|dx
\leq\|T^0(\cdot,t)\|_{L^p(K_{\varepsilon}\setminus K)}\!\cdot\!
\Lebes^n(K_{\varepsilon}\setminus K)^{\frac{1}{p'}}\nonumber\\
&\leq\sup_{0<t<\infty}\!\!\|T^0(\cdot,t)\|_{L^p(\bbbr^n)}\!\cdot\!
\Lebes^n(K_{\varepsilon}\setminus K)^{\frac{1}{p'}}.\label{eq2.289}
\end{align}
Therefore, given $\rho>0$, there exists $\varepsilon_0>0$ so small that uniformly for all $t\in(0,\infty)$, we have
\begin{equation}\label{eq2.290}
\int_{K_{\varepsilon_0}\setminus K}|T^0(x,t)|dx<\frac{1}{2}\rho.
\end{equation}
Since $\varphi_{\varepsilon_0}\in\Sw(\bbbr^n)$ satisfies conditions (\ref{eq2.212}) and (\ref{eq2.213}), we infer from (\ref{eq2.281}) that
\begin{equation}\label{eq2.291}
\lim_{t\rightarrow0}\!\int_{\bbbr^n}\!\!\!\!\!|T^{0}(x,t)|\!\cdot\!\varphi_{\varepsilon_0}(x)dx\!=\!
\lim_{t\rightarrow0}\!\int_{K_{\varepsilon_0}}\!\!\!\!\!\!\!|T^{0}(x,t)|\!\cdot\!
\varphi_{\varepsilon_0}(x)dx\!=\!0,
\end{equation}
and thus we can find $t_0>0$ so that $0<t<t_0$ implies
\begin{equation}\label{eq2.292}
0\leq\int_{K_{\varepsilon_0}}|T^{0}(x,t)|\!\cdot\!\varphi_{\varepsilon_0}(x)dx<\frac{1}{2}\rho.
\end{equation}
We use conditions (\ref{eq2.212}) and (\ref{eq2.213}) to obtain the estimate (\ref{eq2.293}) below whenever $0<t<t_0$,
\begin{align}
&\int_K|T^0(x,t)|dx\nonumber\\
&\leq\!\!\big|\!\!\int_K\!\!\!|T^0(x,t)|dx\!-\!\!\!\int_{K_{\varepsilon_0}}\!\!\!\!\!\!\!|T^{0}(x,t)|
\!\cdot\!\varphi_{\varepsilon_0}(x)dx\big|\!+\!\!\!\int_{K_{\varepsilon_0}}\!\!\!\!\!\!\!|T^{0}(x,t)|
\!\cdot\!\varphi_{\varepsilon_0}(x)dx\nonumber\\
&\leq\int_{K_{\varepsilon_0}\setminus K}\!\!\!\!|T^{0}(x,t)|\!\cdot\!
\varphi_{\varepsilon_0}(x)dx+\frac{1}{2}\rho<\rho.\label{eq2.293}
\end{align}
Indeed, conditions (\ref{eq2.212}), (\ref{eq2.213}), and (\ref{eq2.289}) imply that 
\begin{equation}\label{eq2.294}
\int_K|T^0(x,t)|dx\!=\!\lim_{\varepsilon\rightarrow0}\int_{\bbbr^n}
|T^{0}(x,t)|\!\cdot\!\varphi_{\varepsilon}(x)dx
\end{equation}
uniformly for all $t\in(0,\infty)$, hence we can exchange the order of limit below and invoke (\ref{eq2.281}) to obtain
\begin{align}
&\lim_{t\rightarrow0}\int_K|T^0(x,t)|dx\!=\!\lim_{t\rightarrow0}
\lim_{\varepsilon\rightarrow0}\int_{\bbbr^n}|T^{0}(x,t)|\!\cdot\!\varphi_{\varepsilon}(x)dx\nonumber\\
&\!=\!\lim_{\varepsilon\rightarrow0}\lim_{t\rightarrow0}\int_{\bbbr^n}
|T^{0}(x,t)|\!\cdot\!\varphi_{\varepsilon}(x)dx\!=\!0,\label{eq2.295}
\end{align}
where $K$ is an arbitrary compact set in $\bbbr^n$. Combining (\ref{eq2.275}), (\ref{eq2.279}), and (\ref{eq2.295}) proves (\ref{eq2.276}). In particular, if $B^n(0,N)$ denotes the ball in $\bbbr^n$ centered at the origin with radius $N\in\bbbn$, then we have
\begin{equation}\label{eq2.296}
\lim_{t\rightarrow0}\int_{B^n(0,N)}|T(x,t)|dx=0,
\end{equation}
and we can find a decreasing sequence $\{t_k^{(N)}\}_{k\in\bbbn}$ of positive real numbers so that for almost every $x\in B^n(0,N)$,
\begin{equation}\label{eq2.297}
\lim_{k\rightarrow\infty}t_k^{(N)}=0\quad\text{and}\quad\lim_{k\rightarrow\infty}T(x,t_k^{(N)})=0.
\end{equation}
Therefore, we recall (\ref{eq2.272}), then the application of a classical diagonal argument by extracting a subsequence from a subsequence yields that there exists a sequence $\{t_k\}_{k\in\bbbn}$ of positive real numbers so that the following conditions are satisfied,
\begin{equation}\label{eq2.298}
0<t_{k+1}<t_k<\infty,\qquad\lim_{k\rightarrow\infty}t_k=0,
\end{equation}
\begin{align}
&\lim_{k\rightarrow\infty}\int_0^{\infty}\!\!\!\!\Pint(f;x,t_k+r)\!\cdot\!r^{s-1}dr\nonumber\\
&=\!\!\!\int_0^{\infty}\!\!\!\!\Pint(f;x,r)\!\cdot\!r^{s-1}dr
\text{ for almost every $x\!\in\!\bbbr^n$},\label{eq2.299}
\end{align}
hence conclusions (\ref{eq2.375}) and (\ref{eq2.376}) are proven. For every real number $c\in(0,\infty)$ and for almost every $x\!\in\!\bbbr^n$, we use the sequence $\{t_k\}_{k\in\bbbn}$ given by (\ref{eq2.375}) and (\ref{eq2.376}), the smoothness of the function $\Tfunct_s(f;x,t)$ on $\bbbr^{n+1}_+$, and the fundamental theorem of calculus to deduce that
\begin{align}
&\int_0^c\!\!\!\partial_{n+1}\Tfunct_s(f;x,u)du\!=\!\lim_{k\rightarrow\infty}
\int_{t_k}^c\!\!\!\partial_{n+1}\Tfunct_s(f;x,u)du\nonumber\\
&=\!\!\lim_{k\rightarrow\infty}\!\Tfunct_s(f;x,c)\!-\!\Tfunct_s(f;x,t_k)
\!=\!\Tfunct_s(f;x,c)\!-\!\Tfunct_s(f;x,0).\label{eq2.300}
\end{align}
In the right end of equation (\ref{eq2.300}), the first term is a well-defined function in $L^{\infty}(\bbbr^n)$ for each fixed $c\in(0,\infty)$ and the second term is well-defined and finite for almost every $x\in\bbbr^n$, in view of the conclusion obtained in Step 1. Therefore the conclusion (\ref{eq2.377}) is proven. The proof of Lemma \ref{lemma18} ends here.
\end{proof}
\begin{lemma}\label{lemma14}
If $1<p<\infty$, $0<s<\frac{n}{p}$, and $f\in\Lps$ has a function representative $f(x)\in\functrep_{0}(f)\cap L^{p_0}(\bbbr^n)$ for some $1\leq p_0\leq\infty$, then the Poisson integral $\Pint(f;x,t)$ is a smooth harmonic function on $\bbbr^{n+1}_+$ and belongs to $L^{\infty}(\bbbr^n)\cap L^{p_0}(\bbbr^n)$ for every $0<t<\infty$. Thus there exists a tempered distribution, denoted by the same symbol $\Pint(f;x,t)$, so that the Poisson integral $\Pint(f;x,t)$ is in $\functrep(\Pint(f;x,t))$. Furthermore, the tempered distribution $\Pint(f;x,t)$ has the smooth function representative $(\ref{eq6.3})\in\functrep_{0}(\Pint(f;x,t))\cap L^{\infty}(\bbbr^n)$ for every $t>0$,
\begin{equation}\label{eq6.3}
x\in\bbbr^n\longmapsto\frac{(2\pi)^s}{\varGamma(s)}\int_0^{\infty}\!\!\!\!\!
\Pint(\Lift_s f;x,t+r)r^{s-1}dr.
\end{equation}
Therefore, the difference of the Poisson integral $\Pint(f;x,t)$ and the function (\ref{eq6.3}) is a polynomial function for every $x\in\bbbr^n$ whose coefficients are functions of $t\in(0,\infty)$. Since this difference is also a function in $L^{\infty}(\bbbr^n)$ with respect to variable $x\in\bbbr^n$, we have
\begin{equation}\label{eq2.185}
\Pint(f;x,t)-\frac{(2\pi)^s}{\varGamma(s)}\int_0^{\infty}\!\!\!\!\!\Pint(\Lift_s f;x,t+r)r^{s-1}dr
=G(t)
\end{equation}
for all $x\in\bbbr^n$ and $t\in(0,\infty)$, where $G(t)$ is a smooth function of $t\in(0,\infty)$. Moreover, the function (\ref{eq2.188}) is smooth on $\bbbr^{n+1}_+$,
\begin{equation}\label{eq2.188}
(x,t)\in\bbbr^{n+1}_+\longmapsto\frac{(2\pi)^s}{\varGamma(s)}\int_0^{\infty}\!\!\!\!\!
\Pint(\Lift_s f;x,t+r)r^{s-1}dr,
\end{equation}
and the following equations are true for $1\leq k\leq n$ and for all $x\in\bbbr^n$ and $t\in(0,\infty)$,
\begin{align}
\partial_k\Pint(f;x,t)&\!\!=\!\frac{(2\pi)^s}{\varGamma(s)}\!\!
\int_0^{\infty}\!\!\partial_k\Pint(\Lift_s f;x,t\!+\!r)\!\cdot\!r^{s-1}dr,\label{eq2.186}\\
\partial_{n+1}\Pint(f;x,t)&\!\!=\!\frac{(2\pi)^s}{\varGamma(s)}\!\!
\int_0^{\infty}\!\!\!\!\partial_{n+1}\Pint(\Lift_s f;x,t\!+\!r)\!\cdot\!r^{s-1}dr
\!+\!G'(t),\label{eq2.187}
\end{align}
where $G'(t)$ is the first-order derivative of $G(t)$. In addition, we can find a sequence $\{t_k\}_{k\in\bbbn}$ of positive real numbers so that the following conditions are satisfied,
\begin{equation}\label{eq2.266}
0<t_{k+1}<t_k<\infty,\qquad\lim_{k\rightarrow\infty}t_k=0,
\end{equation}
\begin{align}
&\lim_{k\rightarrow\infty}\frac{(2\pi)^s}{\varGamma(s)}\!\!
\int_0^{\infty}\!\!\!\!\Pint(\Lift_s f;x,t_k+r)\!\cdot\!r^{s-1}dr\nonumber\\
&=\!\frac{(2\pi)^s}{\varGamma(s)}\!\!\int_0^{\infty}\!\!\!\!\Pint(\Lift_s f;x,r)\!\cdot\!r^{s-1}dr
\text{ for almost every $x\!\in\!\bbbr^n$}.\label{eq2.267}
\end{align}
Therefore, for every real number $c\in(0,\infty)$ and for almost every $x\in\bbbr^n$, the following equation (\ref{eq2.268}) is true and both sides of equation (\ref{eq2.268}) are well-defined and finite,
\begin{align}
&\frac{(2\pi)^s}{\varGamma(s)}\!\!\int_0^c\!\!\int_0^{\infty}\!\!\!\!
\partial_{n+1}\Pint(\Lift_s f;x,u\!+\!r)\!\cdot\!r^{s-1}drdu\nonumber\\
&=\!\frac{(2\pi)^s}{\varGamma(s)}\!\big[\!\!\int_0^{\infty}\!\!\!\!\!\!\!
\Pint(\Lift_s f;x,c\!+\!r)\!\cdot\!r^{s-1}dr\!-\!\!\!\int_0^{\infty}\!\!\!\!\!\!\!
\Pint(\Lift_s f;x,r)\!\cdot\!r^{s-1}dr\big].\label{eq2.268}
\end{align}
\end{lemma}
\begin{proof}[Proof of Lemma \ref{lemma14}]
Step 1: The fact that the function $\Pint(f;x,t)$ is smooth and harmonic on $\bbbr^{n+1}_+$ and belongs to $L^{\infty}(\bbbr^n)\cap L^{p_0}(\bbbr^n)$ is a consequence of (\ref{eq1-4}), H\"{o}lder's inequality, \cite[Theorem 1.2.12 (Young's inequality)]{14classical}, and the proof of Lemma \ref{lemma10}. Then H\"{o}lder's inequality also implies the following expression
\begin{equation}\label{eq2.189}
\bigg|\int_{\bbbr^n}\Pint(f;x,t)\!\cdot\!g(x)dx\bigg|
\end{equation}
can be estimated from above by a finite sum of Schwartz seminorms of $g$ whenever $g\in\Sw(\bbbr^n)$, and thus the integral in (\ref{eq2.189}) defines a tempered distribution denoted by the symbol $\Pint(f;x,t)$. We have the following equation (\ref{eq2.193}) for all $g\in\Sw(\bbbr^n)$,
\begin{equation}\label{eq2.193}
<\Pint(f;x,t),g>=\!\!\int_{\bbbr^n}\Pint(f;x,t)\!\cdot\!g(x)dx,
\end{equation} 
hence $\Pint(f;x,t)\in\functrep(\Pint(f;x,t))\subseteq\functrep_0(\Pint(f;x,t))$. Let $\phi\in\Sw_0(\bbbr^n)$, we use Lemma \ref{lemma16} (\ref{eq2.172}) and (\ref{eq2.244}), and Proposition \ref{proposition1} (i) to obtain the following
\begin{align}
&<\Pint(f;x,t),\phi>=<e^{-2\pi t|\xi|}\cdot\FT_n f,\iFT_n\phi>\nonumber\\
&=<\FT_n f,e^{-2\pi t|\xi|}\cdot\iFT_n\phi>,\label{eq2-126}
\end{align}
where $\FT_n f\in\Sw'(\bbbr^n)$ in (\ref{eq2-126}) is the distributional Fourier transform of the HB-extension of $f\in\Lps\subseteq\Sw_{0}'(\bbbr^n)$ according to Lemma \ref{lemma16}. Notice that since $\iFT_n\phi\in\Sw_{00}(\bbbr^n)$, $e^{-2\pi t|\xi|}\cdot\FT_n f$ in the first line of (\ref{eq2-126}) is the product of the function $e^{-2\pi t|\xi|}$ and the tempered distribution $\FT_n f$. From the discussion in Remark \ref{remark3}, we see that the distributional Fourier transform $\FT_n\Lift_s f\in\Sw'(\bbbr^n)\subseteq\Sw_{00}'(\bbbr^n)$ yields the extended tempered distribution denoted as $|\xi|^{s}\cdot\FT_n f$, and this extended tempered distribution coincides with the product of the function $|\xi|^{s}$ and the tempered distribution $\FT_n f$ when it is acting on functions in $\Sw_{00}(\bbbr^n)$. The function
\begin{equation}\label{eq2.166}
\xi\in\bbbr^n\longmapsto|\xi|^{-s}\cdot e^{-2\pi t|\xi|}
\end{equation}
satisfies condition (\ref{eq1.64}), and we deduce from Proposition \ref{proposition1} (i) that the product of the function (\ref{eq2.166}) and the tempered distribution $\FT_n\Lift_s f$ is a well-defined continuous linear functional on the subspace $\Sw_{00}(\bbbr^n)$ in the topology of $\Sw(\bbbr^n)$, and we have
		\begin{equation}\label{eq2-127}
			<|\xi|^{-s}\cdot e^{-2\pi t|\xi|}\cdot\FT_n\Lift_s f,g>=
			<\FT_n\Lift_s f,|\xi|^{-s}\cdot e^{-2\pi t|\xi|}\cdot g>,
		\end{equation}
whenever $g\!\in\!\Sw_{00}(\bbbr^n)$. Furthermore, the function $|\xi|^{-s}\!\cdot\!e^{-2\pi t|\xi|}\!\cdot\!g(\xi)$ also belongs to the subspace $\Sw_{00}(\bbbr^n)$ by Proposition \ref{proposition1}, then we have
\begin{align}
&<\!\FT_n\Lift_s f,|\xi|^{-s}\!\cdot\!e^{-2\pi t|\xi|}\!\cdot\!g\!>=
<\!|\xi|^{s}\!\cdot\!\FT_n f,|\xi|^{-s}\!\cdot\!e^{-2\pi t|\xi|}\!\cdot\!g\!>\nonumber\\
&=<\!\FT_n f,|\xi|^{s}\!\cdot\!|\xi|^{-s}\!\cdot\!e^{-2\pi t|\xi|}\!\cdot\!g\!>
=<\!\FT_n f,e^{-2\pi t|\xi|}\!\cdot\!g\!>,\label{eq2-128}
\end{align}
where the second equation of (\ref{eq2-128}) is due to (\ref{eq1.65}). Combining (\ref{eq2-126}) with equations (\ref{eq2-127}) and (\ref{eq2-128}) where $g$ is replaced by $\iFT_n\phi\in\Sw_{00}(\bbbr^n)$, we can establish the following equation for all $\phi\in\Sw_{0}(\bbbr^n)$,
\begin{equation}\label{eq2-129}
<\Pint(f;x,t),\phi>
=<|\xi|^{-s}\!\cdot\!e^{-2\pi t|\xi|}\!\cdot\!\FT_n\Lift_s f,\iFT_n\phi>.
\end{equation}
We use the gamma function identity
\begin{equation}\label{eq6.2}
|\xi|^{-s}=\frac{(2\pi)^s}{\varGamma(s)}\int_0^{\infty}e^{-2\pi|\xi|r}r^{s-1}dr,
\end{equation}
and (\ref{eq6.2}) is true for all $0<s<\infty$. Then we insert (\ref{eq6.2}) into (\ref{eq2-129}) and obtain
\begin{align}
&<\Pint(f;x,t),\phi>\nonumber\\
&=<\!\frac{(2\pi)^s}{\varGamma(s)}\!\!\int_0^{\infty}\!\!\!\!e^{-2\pi|\xi|(r+t)}r^{s-1}dr
\!\cdot\!\FT_n\Lift_s f,\iFT_n\phi\!>.\label{eq2-86}
\end{align}
We use a simple change of variable $u=2\pi|\xi|r$ and we have
\begin{equation}\label{eq2-95}
\rho(\xi):=\frac{(2\pi)^s}{\varGamma(s)}\int_0^{\infty}e^{-2\pi|\xi|(r+t)}r^{s-1}dr=|\xi|^{-s}\cdot e^{-2\pi t|\xi|}.
\end{equation}
Notice that derivatives of $\rho(\xi)$ are not continuous at $\xi=0$ and $\rho(\xi)$ satisfies the condition (\ref{eq1.64}). Thus Proposition \ref{proposition1} (i) tells us that
\begin{align}
&<\!\frac{(2\pi)^s}{\varGamma(s)}\!\!\int_0^{\infty}\!\!\!\!e^{-2\pi|\xi|(r+t)}r^{s-1}dr\!\cdot\!
\FT_n\Lift_s f,\iFT_n\phi\!>\nonumber\\
&=<\!\FT_n\Lift_s f,\frac{(2\pi)^s}{\varGamma(s)}\!\!\int_0^{\infty}\!\!\!\!e^{-2\pi|\xi|(r+t)}r^{s-1}dr
\!\cdot\!\iFT_n\phi\!>\label{eq2-94}
\end{align}
is a well-defined expression. Since $0<s<\frac{n}{p}<n$, the easily observed integrability of $\rho(\xi)$ on $\bbbr^n$ justifies the application of Fubini's theorem, and we can use formula (\ref{eq1-3}) to obtain
\begin{equation}\label{eq2-87}
\iFT_n\rho(x)=\frac{(2\pi)^s}{\varGamma(s)}\int_0^{\infty}P_{r+t}(x)r^{s-1}dr\geq0,
\end{equation}
where $P_{r+t}(x)=C_0(r+t)((r+t)^2+|x|^2)^{-\frac{n+1}{2}}$. Since $\rho(\xi)\in L^1(\bbbr^n)$, $\iFT_n\phi(\xi)\in\Sw_{00}(\bbbr^n)$, and $\rho\cdot\iFT_n\phi\in\Sw_{00}(\bbbr^n)$ by Proposition \ref{proposition1}, then we use Fubini's theorem to obtain
\begin{equation*}
\FT_n(\rho\cdot\iFT_n\phi)(x)=\FT_n\rho*\phi(x)\in\Sw_0(\bbbr^n).
\end{equation*}
From Definition \ref{definition2} and the assumption $f\in\Lps$, we deduce that the function representative in the sense of $\Sw_0'(\bbbr^n)$, of the tempered distribution $\Lift_s f$, is a function in $L^p(\bbbr^n)$ and we denote this function by $\Lift_s f(x)$. Therefore we have
\begin{align}
&(\ref{eq2-94})=<\FT_n\Lift_s f,\rho\cdot\iFT_n\phi>=<\Lift_s f,\FT_n(\rho\cdot\iFT_n\phi)>\nonumber\\
&=<\Lift_s f,\FT_n\rho*\phi>=\int_{\bbbr^n}\Lift_s f(x)\cdot\FT_n\rho*\phi(x)dx\nonumber\\
&=\int_{\bbbr^n}\int_{\bbbr^n}\Lift_s f(x)\cdot\iFT_n\rho(y-x)\cdot\phi(y)dydx,\label{eq2-98}
\end{align}
where we use the fact that $\rho(\cdot)$ is a radial function in $L^1(\bbbr^n)$ then $\FT_n\rho(\cdot)$ is a radial function in $L^{\infty}(\bbbr^n)$. We consider the convolution
\begin{align}
&\iFT_n\rho*\Lift_s f(y)=\frac{(2\pi)^s}{\varGamma(s)}\int_{\bbbr^n}\!\!\!\Lift_s f(y\!-\!x)\!\int_0^{\infty}\!\!\!\!P_{r+t}(x)r^{s-1}drdx\nonumber\\
&=\frac{(2\pi)^s}{\varGamma(s)}\int_{\bbbr^n}\int_0^{\infty}\Lift_s f(y-x)P_{r+t}(x)r^{s-1}drdx.\label{eq2-89}
\end{align}
To justify the exchange of the order of integration in (\ref{eq2-89}), we recall the function defined in (\ref{eq2.373}) and observe that
\begin{equation}\label{eq2-90}
\int_0^{\infty}\!\!\!\!\int_{\bbbr^n}\!\!\!\!|\Lift_s f(y\!-\!x)|\!\cdot\!P_{r+t}(x)
r^{s-1}dxdr\!=\!\Tfunct_s(|\Lift_s f|;y,t).
\end{equation} 
Since $1<p<\infty$, $0<s<\frac{n}{p}$, and $|\Lift_s f|(x)\in L^p(\bbbr^n)$, we deduce from Lemma \ref{lemma18} that both $\Tfunct_s(|\Lift_s f|;y,t)$ and $\Tfunct_s(\Lift_s f;y,t)$ are well-defined functions in $L^{\infty}(\bbbr^n)$ with respect to variable $y\in\bbbr^n$ for each fixed $0<t<\infty$. Therefore Fubini's theorem implicates
\begin{align}
&\iFT_n\rho*\Lift_s f(y)=(\ref{eq2-89})=\frac{(2\pi)^s}{\varGamma(s)}\Tfunct_s(\Lift_s f;y,t)
\nonumber\\
&=\frac{(2\pi)^s}{\varGamma(s)}\int_0^{\infty}\Pint(\Lift_s f;y,r+t)\cdot r^{s-1}dr,\label{eq2-91}
\end{align}
\begin{equation}\label{eq2.184}
|\iFT_n\rho|*|\Lift_s f|(y)=\frac{(2\pi)^s}{\varGamma(s)}\Tfunct_s(|\Lift_s f|;y,t),
\end{equation}
and we have
\begin{equation*}
\int_{\bbbr^n}\int_{\bbbr^n}|\Lift_s f(x)|\cdot|\iFT_n\rho(y-x)|\cdot|\phi(y)|dxdy<\infty.
\end{equation*}
We invoke \cite[Tonelli's theorem and Corollary 7 in section 20.1]{real.analysis.royden} and \cite[Theorem 11 in section 20.2]{real.analysis.royden} to deduce that the nonnegative function
$$|\Lift_s f(x)|\cdot|\iFT_n\rho(y-x)|\cdot|\phi(y)|$$
is integrable over the product space $\bbbr^n\times\bbbr^n$ with respect to the product measure $d(x,y)$, and hence we can apply \cite[Fubini's theorem in section 20.1]{real.analysis.royden} to (\ref{eq2-98}) and obtain
\begin{equation}\label{eq2-99}
(\ref{eq2-98})=\int_{\bbbr^n}\iFT_n\rho*\Lift_s f(y)\cdot\phi(y)dy.
\end{equation}
Combining (\ref{eq2-86}), (\ref{eq2-94}), (\ref{eq2-98}), (\ref{eq2-99}), and (\ref{eq2-91}) all together yields
\begin{equation}\label{eq2-100}
<\!\Pint(f;x,t),\phi\!>=\!\!\!\int_{\bbbr^n}\!\!\bigg[\!\frac{(2\pi)^s}{\varGamma(s)}\!\!\int_0^{\infty}\!\!\!\!\!\!\!\Pint(\Lift_s f;x,r\!+\!t)\!\cdot\! r^{s-1}dr\!\bigg]\!\!\cdot\!\phi(x)dx
\end{equation}
for all $\phi\in\Sw_0(\bbbr^n)$, and therefore the function (\ref{eq6.3}) is in $L^{\infty}(\bbbr^n)\cap\functrep_0(\Pint(f;x,t))$ for every $t>0$.\\

Step 2: Lemma \ref{lemma18} and equation (\ref{eq2-91}) imply that the function (\ref{eq2.188}) is smooth on $\bbbr^{n+1}_+$. Since both the Poisson integral $\Pint(f;x,t)$ and the function (\ref{eq6.3}) are continuous for every $t\in(0,\infty)$ and belong to $\functrep_0(\Pint(f;x,t))\cap L^{\infty}(\bbbr^n)$, their difference is a continuous function in $L^{\infty}(\bbbr^n)\subseteq\Lloc$ and there exists $\tilde{f}\in\Sw'(\bbbr^n)$ so that the difference
\begin{equation*}
\text{$[(\ref{eq6.3})\!-\!\Pint(f;x,t)]$ is in $\functrep(\tilde{f})$ and satisfies condition (\ref{eq1.101}).}
\end{equation*}
By Proposition \ref{proposition2}, the difference equals a continuous polynomial for almost every $x\in\bbbr^n$ whose coefficients are functions of $t\in(0,\infty)$, and the continuity of relevant functions indicates such an equality holds true for all $x\in\bbbr^n$ and $t\in(0,\infty)$. Since this difference is also a function in $L^{\infty}(\bbbr^n)$ with respect to variable $x\in\bbbr^n$, the difference of the function (\ref{eq6.3}) and the Poisson integral $\Pint(f;x,t)$ is a function of $t$ for every $x\in\bbbr^n$. Equation (\ref{eq2.185}) is proved. Since the function representative $f(x)$ is in $\functrep_{0}(f)\cap L^{p_0}(\bbbr^n)$ for some $p_0\in[1,\infty]$, formulae (\ref{eq1-5}) and (\ref{eq1-6}) and the proof of Lemma \ref{lemma10} tell us that the Poisson integral $\Pint(f;x,t)$ is smooth on $\bbbr^{n+1}_+$. Since the function $G(t)$ appearing in the equation (\ref{eq2.185}) is a function of the single variable $t$, the function $G(t)$ is smooth at every $t\in(0,\infty)$. Furthermore, we have obtained the following equations for $1\leq k\leq n$ and for all $x\in\bbbr^n$ and $t\in(0,\infty)$,
\begin{align}
\frac{\partial}{\partial x_k}\!\big[\!\Pint(f;x,t)\!\big]&\!\!=\!\frac{\partial}{\partial x_k}\!
\bigg[\!\frac{(2\pi)^s}{\varGamma(s)}\!\!\int_0^{\infty}\!\!\!\!\!\!\!
\Pint(\Lift_s f;x,t\!+\!r)r^{s-1}dr\!\bigg],\label{eq2.153}\\
\frac{\partial}{\partial t}\!\big[\!\Pint(f;x,t)\!\big]&\!\!=\!\frac{\partial}{\partial t}\!
\bigg[\!\frac{(2\pi)^s}{\varGamma(s)}\!\!\int_0^{\infty}\!\!\!\!\!\!\!
\Pint(\Lift_s f;x,t\!+\!r)r^{s-1}dr\!\bigg]\!\!+\!G'(t),\label{eq2.154}
\end{align}
where $G'(t)$ is the first-order derivative of $G(t)$. Then equations (\ref{eq2.186}) and (\ref{eq2.187}) are consequences of (\ref{eq2.153}), (\ref{eq2.154}), (\ref{eq2.191}), and (\ref{eq2.192}).\\

Step 3: By noticing that
\begin{equation*}
\frac{(2\pi)^s}{\varGamma(s)}\!\!\int_0^{\infty}\!\!\!\!\Pint(\Lift_s f;x,t_k+r)\!\cdot\!r^{s-1}dr
\!=\!\frac{(2\pi)^s}{\varGamma(s)}\Tfunct_s(\Lift_s f;x,t_k)
\end{equation*}
and
\begin{equation*}
\frac{(2\pi)^s}{\varGamma(s)}\!\!\int_0^{\infty}\!\!\!\!\Pint(\Lift_s f;x,r)\!\cdot\!r^{s-1}dr
\!=\!\frac{(2\pi)^s}{\varGamma(s)}\Tfunct_s(\Lift_s f;x,0),
\end{equation*}
we see that conclusions (\ref{eq2.266}) and (\ref{eq2.267}) are consequences of the conditions $1<p<\infty$, $0<s<\frac{n}{p}$, and $f\in\Lps$, and Lemma \ref{lemma18} (\ref{eq2.375}) and (\ref{eq2.376}). And conclusion (\ref{eq2.268}) is due to Lemma \ref{lemma18} (\ref{eq2.377}) and (\ref{eq2.374}) where we set the multi-index $\alpha=(\alpha_1,\cdots,\alpha_n,\alpha_{n+1})=(0,\cdots,0,1)$. The proof of Lemma \ref{lemma14} is now complete.
\end{proof}
\begin{remark}\label{remark2}
Lemma \ref{lemma14} has been used in the proof of Theorem \ref{theorem2} to deduce the important equations (\ref{eq6.45}) and (\ref{eq6.46}). Conclusions of Lemma \ref{lemma14} can also be obtained by assuming additionally that $f$ is an integrable function and using Fubini's theorem. By applying Lemma \ref{lemma14}, we do not assume that $f$ is an integrable function in Theorem \ref{theorem2}.
\end{remark}
\begin{remark}\label{remark4}
We rephrase \cite[Definition 2.2]{Wang2023} in the language of the theory of functional and function representatives. If $f\in\Sw'(\bbbr^n)$ has a function representative $F(x)\in\functrep(f)$, and if the distributional Fourier transform $\FT_n f$ is compactly supported in the ball $B^n(0,t)\subseteq\bbbr^n$ centered at the origin with radius $t>0$, then the associated $n$-dimensional Peetre-Fefferman-Stein maximal function of $f\in\Sw'(\bbbr^n)$ at $x\in\bbbr^n$ is defined to be
\begin{equation}\label{eq2.301}
\PFSmax_n f(x)=\esssup_{z\in\bbbr^n}\frac{|F(x-z)|}{(1+t|z|)^{\frac{n}{r}}},
\end{equation}
where we pick $r$ to be a positive number satisfying either $0<r<\min\{p,q\}$ or $0<r<p$. Then \cite[Remark 2.4]{Wang2023} is true for the rephrased definition of the Peetre-Fefferman-Stein maximal function. Furthermore, if $\varphi\in\Sw(\bbbr^n)$ and $\varphi_{\frac{1}{t}}(x)=t^n\varphi(tx)$, then \cite[Proposition 2.3.22 (11) and Theorem 2.3.20]{14classical} imply $f*\varphi_{\frac{1}{t}}\in\Sw'(\bbbr^n)$ and the distributional Fourier transform $\FT_n[f*\varphi_{\frac{1}{t}}]=\FT_n f\cdot\FT_n\varphi(\frac{\xi}{t})$ is compactly supported in the ball $B^n(0,t)$ due to the support condition of $\FT_n f\in\Sw'(\bbbr^n)$, and the smooth function (\ref{eq2.303}) is in $\functrep(f*\varphi_{\frac{1}{t}})$,
\begin{equation}\label{eq2.303}
x\!\in\!\bbbr^n\!\longmapsto<\!f,\varphi_{\frac{1}{t}}(x-\cdot)\!>=\!\!\!\int_{\bbbr^n}\!\!\!\!\!\!
F(y)\!\cdot\!\varphi_{\frac{1}{t}}(x\!-\!y)dy\!=\!F\!*\!\varphi_{\frac{1}{t}}(x),
\end{equation}
where the equations in (\ref{eq2.303}) are true by the assumption $F(x)\in\functrep(f)$ and the fact that the mapping $y\in\bbbr^n\mapsto\varphi_{\frac{1}{t}}(x-y)$ is a function in $\Sw(\bbbr^n)$ for every $x\in\bbbr^n$ and $t\in(0,\infty)$. As a consequence, the Peetre-Fefferman-Stein maximal function $\PFSmax_n(f*\varphi_{\frac{1}{t}})(x)$ is well-defined by the following expression,
\begin{equation}\label{eq2.304}
\PFSmax_n(f*\varphi_{\frac{1}{t}})(x)=\esssup_{z\in\bbbr^n}
\frac{|F*\varphi_{\frac{1}{t}}(x-z)|}{(1+t|z|)^{\frac{n}{r}}}.
\end{equation}
Then a computational process analogous to \cite[Remark 2.5]{Wang2023} yields
\begin{equation}\label{eq2.305}
\PFSmax_n(f\!*\!\varphi_{\frac{1}{t}})(x)\!\lesssim\!\PFSmax_n f(x)\text{ for all $x\!\in\!\bbbr^n$ and $0\!<\!t\!<\!\infty$},
\end{equation}
where the implicit constant in (\ref{eq2.305}) depends on $n,r,\varphi$ and does not depend on $t$. Assume now that the function representative $F(x)\in\functrep(f)$ satisfies the least additional condition that $F(x)$ is continuous on $\bbbr^n$, then $F(x)\in\Lloc$. And we fix a function $\Phi\in\Sw(\bbbr^n)$ so that
\begin{equation}\label{eq2.302}
\FT_n\Phi(\xi)=
\begin{cases}
1&\quad\text{if $|\xi|\leq1$},\\
0&\quad\text{if $|\xi|\geq2$},
\end{cases}
\end{equation}
then $\FT_n\Phi_{\frac{1}{t}}(\xi)=\FT_n\Phi(\frac{\xi}{t})=1$ if $|\xi|\leq t$, and
\begin{equation}\label{eq2.306}
spt.[1-\FT_n\Phi_{\frac{1}{t}}(\xi)]\subseteq\{\xi\in\bbbr^n:|\xi|\geq t\}.
\end{equation}
Therefore the following equation (\ref{eq2.307}) is true for all $g\in\Sw(\bbbr^n)$,
\begin{align}
&<\!f,g\!>\nonumber\\
&=<\!\FT_n f,\FT_n\Phi_{\frac{1}{t}}(\xi)\!\cdot\!\iFT_n g\!>\!+\!
<\!\FT_n f,[1\!-\!\FT_n\Phi_{\frac{1}{t}}(\xi)]\!\cdot\!\iFT_n g\!>\nonumber\\
&=<\!\FT_n f,\FT_n\Phi_{\frac{1}{t}}(\xi)\!\cdot\!\iFT_n g\!>
=<\!\FT_n\Phi_{\frac{1}{t}}(\xi)\!\cdot\!\FT_n f,\iFT_n g\!>\nonumber\\
&=<\!\iFT_n[\FT_n\Phi_{\frac{1}{t}}(\xi)\!\cdot\!\FT_n f],g\!>
=<\!f*\Phi_{\frac{1}{t}},g\!>,\label{eq2.307}
\end{align}
where we use the support condition of $\FT_n f\in\Sw'(\bbbr^n)$, the fact that the function $[1\!-\!\FT_n\Phi_{\frac{1}{t}}(\xi)]\!\cdot\!\iFT_n g(\xi)\!\in\!\Sw(\bbbr^n)$ is supported in the set (\ref{eq2.306}), \cite[Definition 2.3.15, Definition 2.3.7, and Proposition 2.3.22 (12)]{14classical} in a sequence. Hence, $f=f*\Phi_{\frac{1}{t}}$ in the sense of $\Sw'(\bbbr^n)$ and $$\functrep(f)=\functrep(f*\Phi_{\frac{1}{t}}).$$
Since $F(x)\in\functrep(f)\cap\Lloc$ and the mapping $y\in\bbbr^n\mapsto\Phi_{\frac{1}{t}}(x-y)$ is a function in $\Sw(\bbbr^n)$, \cite[Theorem 2.3.20]{14classical} yields that $f\!*\!\Phi_{\frac{1}{t}}\!\in\!\Sw'(\bbbr^n)$ has the smooth function representative $(\ref{eq2.308})\!\in\!\functrep(f\!*\!\Phi_{\frac{1}{t}})\!\cap\!\Lloc$,
\begin{equation}\label{eq2.308}
x\!\in\!\bbbr^n\!\longmapsto<\!f,\Phi_{\frac{1}{t}}(x-\cdot)\!>=\!\!\!\int_{\bbbr^n}\!\!\!\!\!\!
F(y)\!\cdot\!\Phi_{\frac{1}{t}}(x\!-\!y)dy\!=\!F\!*\!\Phi_{\frac{1}{t}}(x).
\end{equation}
Furthermore, for every multi-index $\alpha$, there exist positive finite constants $C_{\alpha}$ and $K_{\alpha}$ such that
\begin{equation}\label{eq2.309}
|\partial^{\alpha}(f\!*\!\Phi_{\frac{1}{t}})(x)|\!=\!
|\partial^{\alpha}(F\!*\!\Phi_{\frac{1}{t}})(x)|\!\leq\!C_{\alpha}\!\cdot\!(1\!+\!|x|)^{K_{\alpha}}.
\end{equation}
Lemma \ref{lemma15}, the continuity of $F(x)$, and the smoothness of $F\!*\!\Phi_{\frac{1}{t}}(x)$ imply
\begin{equation}\label{eq2.310}
F(x)=F*\Phi_{\frac{1}{t}}(x)\text{ for all $x\in\bbbr^n$},
\end{equation}
thus the continuous function $F(x)$ is smooth on all of $\bbbr^n$, and we deduce from (\ref{eq2.309}) and (\ref{eq2.310}) that for every multi-index $\alpha$ and every $x\!\in\!\bbbr^n$,
\begin{equation}\label{eq2.311}
|\partial^{\alpha}F(x)|\!\leq\!C_{\alpha}\!\cdot\!(1\!+\!|x|)^{K_{\alpha}}.
\end{equation}
Moreover, we use \cite[Definition 2.3.6]{14classical}, integration by parts, and estimate (\ref{eq2.311}) to obtain for all $g\in\Sw(\bbbr^n)$ and all multi-indices $\alpha$,
\begin{align}
&<\partial^{\alpha}f,g>=(-1)^{|\alpha|}<f,\partial^{\alpha}g>\nonumber\\
&=(-1)^{|\alpha|}\int_{\bbbr^n}F(x)\cdot\partial^{\alpha}g(x)dx
=\int_{\bbbr^n}\partial^{\alpha}F(x)\cdot g(x)dx,\label{eq2.312}
\end{align}
therefore the smooth function $\partial^{\alpha}F(x)\in\functrep(\partial^{\alpha}f)\cap\Lloc$. And \cite[Definition 2.3.16 and Proposition 2.3.22 (8)]{14classical} indicate the distributional Fourier transform $\FT_n(\partial^{\alpha}f)\!\in\!\Sw'(\bbbr^n)$ is compactly supported in the ball $B^n(0,t)$, hence the Peetre-Fefferman-Stein maximal function $\PFSmax_n(\partial^{\alpha}f)(x)$ is well-defined by the following expression,
\begin{equation}\label{eq2.313}
\PFSmax_n(\partial^{\alpha}f)(x)=\esssup_{z\in\bbbr^n}
\frac{|\partial^{\alpha}F(x-z)|}{(1+t|z|)^{\frac{n}{r}}}.
\end{equation}
Analogous to the deduction in (\ref{eq2.307}), we have
\begin{equation}\label{eq2.314}
\partial^{\alpha}f=\partial^{\alpha}f*\Phi_{\frac{1}{t}}\text{ in the sense of $\Sw'(\bbbr^n)$}.
\end{equation}
We invoke \cite[Theorem 2.3.20 and Definition 2.3.6]{14classical} and the fact that both of the functions (\ref{eq2.315}) and (\ref{eq2.316}) belong to $\Sw(\bbbr^n)$ for all $x\in\bbbr^n$ and $0<t<\infty$,
\begin{align}
&y\!\in\!\bbbr^n\!\longmapsto\!\Phi_{\frac{1}{t}}(x\!-\!y)\!=\!t^n\!\cdot\!\Phi(t(x\!-\!y)),
\label{eq2.315}\\
&y\!\in\!\bbbr^n\!\longmapsto\!t^{|\alpha|}\!\cdot\!(\partial^{\alpha}\Phi)_{\frac{1}{t}}(x\!-\!y)
\!=\!t^{n+|\alpha|}\!\cdot\!(\partial^{\alpha}\Phi)(t(x\!-\!y)),\label{eq2.316}
\end{align}
then we conclude the smooth function (\ref{eq2.317}) is in $\functrep(\partial^{\alpha}f\!*\!\Phi_{\frac{1}{t}})\!=\!\functrep(\partial^{\alpha}f)$,
\begin{align}
&x\!\in\!\bbbr^n\!\longmapsto<\!\partial^{\alpha}f,\Phi_{\frac{1}{t}}(x\!-\!\cdot)\!>
=\!(-1)^{|\alpha|}\!<\!f,\partial^{\alpha}[\Phi_{\frac{1}{t}}(x\!-\!\cdot)]\!>\nonumber\\
&=<\!f,t^{|\alpha|}\!\cdot\!(\partial^{\alpha}\Phi)_{\frac{1}{t}}(x\!-\!\cdot)\!>=\!\!\!
\int_{\bbbr^n}\!\!\!\!\!\!F(y)\!\cdot\!t^{|\alpha|}\!\cdot\!
(\partial^{\alpha}\Phi)_{\frac{1}{t}}(x\!-\!y)dy\nonumber\\
&=t^{|\alpha|}\!\cdot\!F\!*\!(\partial^{\alpha}\Phi)_{\frac{1}{t}}(x)\in\Lloc.\label{eq2.317}
\end{align}
Therefore Lemma \ref{lemma15} and the smoothness of functions $\partial^{\alpha}F(x)$ and (\ref{eq2.317}) imply for all $x\!\in\!\bbbr^n$, $t\!\in\!(0,\infty)$, and all multi-indices $\alpha$,
\begin{equation}\label{eq2.318}
\partial^{\alpha}F(x)=t^{|\alpha|}\!\cdot\!F\!*\!(\partial^{\alpha}\Phi)_{\frac{1}{t}}(x).
\end{equation}
Now we can write \cite[Lemma 2.6]{Wang2023}, i.e. \cite[Lemma 2.2.3]{14modern}, equivalently in the language of the theory of functional and function representatives as that if $0<r,t<\infty$ and $f\in\Sw'(\bbbr^n)$ has a continuous function representative $F(x)\in\functrep(f)$, and if the distributional Fourier transform $\FT_n f$ is compactly supported in the ball $B^n(0,t)\subseteq\bbbr^n$, then we have
\begin{align}
\PFSmax_n(\partial^{\alpha}f)(x)&\lesssim t\cdot\PFSmax_n f(x),\label{eq2.319}\\
\PFSmax_n f(x)&\lesssim\HLmax_n(|f|^r)(x)^{\frac{1}{r}},\label{eq2.320}
\end{align}
where $\PFSmax_n(\partial^{\alpha}f)(x)$ and $\PFSmax_n f(x)$ are given by (\ref{eq2.313}) and (\ref{eq2.301}) respectively, and $\alpha$ is a multi-index with $|\alpha|=1$, and $\HLmax_n$ is the Hardy-Littlewood maximal operator, and the implicit constants in (\ref{eq2.319}) and (\ref{eq2.320}) depend only on $n,r$ and do not depend on $t$. The proof of (\ref{eq2.319}) and (\ref{eq2.320}) is similar to \cite[the proof of Lemma 2.2.3]{14modern} where the changes can be seen by noticing that equation (\ref{eq2.310}) is the counterpart of \cite[the first equation in the proof of Lemma 2.2.3]{14modern}, and equation (\ref{eq2.318}) with $|\alpha|=1$ can be used to deduce the counterpart of \cite[the second inequality in the proof of Lemma 2.2.3]{14modern}. Since equation (\ref{eq2.310}) tells us that the continuity of the function representative $F(x)\in\functrep(f)\cap\Lloc$ implies its smoothness, \cite[Remark 2.7]{Wang2023} is true and \cite[Remark 2.8]{Wang2023} can be obtained by iterating (\ref{eq2.319}) repeatedly, and we infer from (\ref{eq2.312}) and (\ref{eq2.313}) that
\begin{equation}\label{eq2.321}
\PFSmax_n(\partial^{\alpha}f)(x)\lesssim t^{|\alpha|}\cdot\PFSmax_n f(x)\text{ for all multi-indices $\alpha$},
\end{equation}
where the implicit constant in (\ref{eq2.321}) does not depend on $t$. Furthermore, the Plancherel-Polya-Nikol'skij inequality, i.e. \cite[Lemma 2.11]{Wang2023}, can be rewritten as follows,
\begin{equation}\label{eq2.322}
\|\partial^{\alpha}f\|_{L^q(\bbbr^n)}\lesssim t^{|\alpha|+n(\frac{1}{p}-\frac{1}{q})}\cdot
\|f\|_{L^p(\bbbr^n)}
\end{equation}
whenever $0<t<\infty$, $0<p\leq q\leq\infty$, $\alpha$ is a multi-index, and $f\in\Sw'(\bbbr^n)$ has a continuous function representative $F(x)\in\functrep(f)$, and the distributional Fourier transform $\FT_n f$ is compactly supported in the ball $B^n(0,t)$. The implicit constant in (\ref{eq2.322}) depends on $\alpha,n,p,q$ and does not depend on $t$. The proof of (\ref{eq2.322}) is very much alike to \cite[the proof of Lemma 2.11]{Wang2023}, in which we use (\ref{eq2.310}), (\ref{eq2.321}), and (\ref{eq2.320}). For every $j\in\bbbz$, the function $\psi_{2^{-j}}(x)=2^{jn}\psi(2^j x)$ is in $\Sw_{0}(\bbbr^n)$, where $\psi\in\Sw_{0}(\bbbr^n)$ is the fixed function satisfying conditions (\ref{eq1-7}), (\ref{eq1-8}), and (\ref{eq1.82}). Then for every $f\in\Sw_{0}'(\bbbr^n)$, Proposition \ref{proposition1} (iv) tells us that $f*\psi_{2^{-j}}\in\Sw'(\bbbr^n)$ has the smooth function representative $(\ref{eq2.323})\in\functrep(f*\psi_{2^{-j}})$, 
\begin{equation}\label{eq2.323}
F_j:x\in\bbbr^n\longmapsto<f,\psi_{2^{-j}}(x-\cdot)>.
\end{equation}
When $g\in\Sw(\bbbr^n)$ has 
$$spt.g\subseteq\{\xi\in\bbbr^n:|\xi|<2^{j-1}\text{ or }|\xi|\geq2^{j+1}\},$$
we use formula (\ref{eq1.126}) and \cite[Definition 2.3.15]{14classical} to deduce that
\begin{align}
&<\!\FT_n(f\!*\!\psi_{2^{-j}}),g\!>=<\!\FT_n f\!\cdot\!\FT_n\psi(2^{-j}\xi),g\!>\nonumber\\
&=<\!\FT_n f,\FT_n\psi(2^{-j}\xi)\!\cdot\!g\!>=<\!\FT_n f,0\!>=\!0,\label{eq2.325}
\end{align}
thus the distributional Fourier transform $\FT_n(f*\psi_{2^{-j}})$ is compactly supported in the annulus 
$$\{\xi\in\bbbr^n:2^{j-1}\leq|\xi|<2^{j+1}\}.$$
Therefore the Peetre-Fefferman-Stein maximal function $\PFSmax_n(f*\psi_{2^{-j}})(x)$ is well-defined for every $j\in\bbbz$ by the following expression,
\begin{equation}\label{eq2.324}
\PFSmax_n(f*\psi_{2^{-j}})(x)=\esssup_{z\in\bbbr^n}
\frac{|F_j(x-z)|}{(1+2^{j+1}|z|)^{\frac{n}{r}}}.
\end{equation}
By choosing $0<r<\min\{p,q\}$, we can use (\ref{eq2.320}) and \cite[Theorem 5.6.6]{14classical} to prove that \cite[Remark 2.13]{Wang2023} holds true for $s\in\bbbr$ and $0<p,q<\infty$.
\end{remark}
We provide some supplementary discussion below about Hardy space $H^p(\bbbr^n)$ and rephrase \cite[Theorem 2.1.2]{14modern} in the language of the theory of functional and function representatives.
\begin{definition}\label{definition4}
We fix a function $\Phi\in\Sw(\bbbr^n)$ with $\int_{\bbbr^n}\Phi(x)dx=1$. For $0<p<\infty$, we define the Hardy space $H^p(\bbbr^n)$ is a subspace of $\Sw_{0}'(\bbbr^n)$, and an element $f\in\Sw_{0}'(\bbbr^n)$ is in $H^p(\bbbr^n)$ if and only if the following quasinorm is finite,
\begin{equation}\label{eq2.326}
\|f\|_{H^p(\bbbr^n)}:=\|\sup_{t>0}|f*\Phi_t|\|_{L^p(\bbbr^n)},
\end{equation}
where for every $t\in(0,\infty)$, the convolution $f*\Phi_t\in\Sw_0'(\bbbr^n)$ has a smooth function representative (\ref{eq2.334}) in the sense of $\Sw_0'(\bbbr^n)$ provided by Proposition \ref{proposition1} (ii),
\begin{equation}\label{eq2.334}
F_t:x\in\bbbr^n\longmapsto<f,\Phi_t(x-\cdot)>,
\end{equation}
and $f\!\in\!\Sw'(\bbbr^n)$ appearing in (\ref{eq2.334}) is the HB-extension of $f\!\in\!\Sw_{0}'(\bbbr^n)$.
\end{definition}
\begin{lemma}\label{lemma17}
(i) If $1<p<\infty$, then every $f\in H^p(\bbbr^n)$ has a function representative $F(x)\in\functrep_{0}(f)\cap L^p(\bbbr^n)$ and we have
\begin{equation}\label{eq2.327}
\|f\|_{L^p(\bbbr^n)}\sim\|f\|_{H^p(\bbbr^n)}=\|\sup_{t>0}|F*\Phi_t|\|_{L^p(\bbbr^n)},
\end{equation}
where the implicit constant depends on $n,p$.\\
(ii) If $p=1$, then every $f\in H^1(\bbbr^n)$ has a function representative $F(x)\in\functrep_{0}(f)\cap L^1(\bbbr^n)$ and we have
\begin{equation}\label{eq2.328}
\|f\|_{L^1(\bbbr^n)}\leq\|f\|_{H^1(\bbbr^n)}=\|\sup_{t>0}|F*\Phi_t|\|_{L^1(\bbbr^n)}.
\end{equation}
\end{lemma}
\begin{proof}[Proof of Lemma \ref{lemma17}]
Step 1: We prove Lemma \ref{lemma17} (i). We observe that the function $\Phi\in\Sw(\bbbr^n)$ fixed in Definition \ref{definition4} is not in $\Sw_{0}(\bbbr^n)$. By Proposition \ref{proposition1} (ii) equation (\ref{eq1.124}), we have $f*\Phi_t\in\Sw_{0}'(\bbbr^n)$, and
\begin{equation}\label{eq2.329}
<f,\phi>=\lim_{t\rightarrow0}<f*\Phi_t,\phi>\text{ for all $\phi\in\Sw_0(\bbbr^n)$}.
\end{equation}
And (\ref{eq1.123}) implies for each $t\in(0,\infty)$, $f*\Phi_t\in\Sw_0'(\bbbr^n)$ has the smooth function representative $(\ref{eq2.330})\in\functrep_0(f*\Phi_t)$,
\begin{equation}\label{eq2.330}
F_t:x\in\bbbr^n\longmapsto<f,\Phi_t(x-\cdot)>,
\end{equation}
where $f\in\Sw'(\bbbr^n)$ appearing in (\ref{eq2.330}) is the HB-extension of the original continuous linear functional $f\in\Sw_{0}'(\bbbr^n)$. Then the condition $f\in H^p(\bbbr^n)$ implies
\begin{align}
&\sup_{t>0}\|F_t\|_{L^p(\bbbr^n)}=\sup_{t>0}\|f*\Phi_t\|_{L^p(\bbbr^n)}\nonumber\\
&\leq\|\sup_{t>0}|f*\Phi_t|\|_{L^p(\bbbr^n)}=\|f\|_{H^p(\bbbr^n)}<\infty.\label{eq2.331}
\end{align}
Since $L^p(\bbbr^n)$ is the dual space of $L^{p'}(\bbbr^n)$ for $1<p<\infty$ and $\frac{1}{p}+\frac{1}{p'}=1$, we use the Riesz representation theorem and the Banach-Alaoglu theorem to deduce that there exists a function $F(x)\in L^p(\bbbr^n)$ and a sequence $\{t_k\}_{k\in\bbbn}$ of positive real numbers so that
\begin{equation}
0<t_{k+1}<t_k<\infty,\quad\lim_{k\rightarrow\infty}t_k=0,
\end{equation}
\begin{equation}\label{eq2.332}
\int_{\bbbr^n}\!\!\!\!\!\!F(x)\!\cdot\!g(x)dx\!=\!\!\lim_{k\rightarrow\infty}\!
\int_{\bbbr^n}\!\!\!\!\!\!F_{t_k}(x)\!\cdot\!g(x)dx
\text{ for all $g\!\in\!L^{p'}\!(\bbbr^n)$}.
\end{equation}
Combining (\ref{eq2.329}), (\ref{eq2.330}), and (\ref{eq2.332}) yields that for every $\phi\in\Sw_{0}(\bbbr^n)\subseteq\Sw(\bbbr^n)\subseteq L^{p'}\!(\bbbr^n)$, we have
\begin{align}
&<f,\phi>=\lim_{k\rightarrow\infty}<f*\Phi_{t_k},\phi>\nonumber\\
&=\!\lim_{k\rightarrow\infty}\!\int_{\bbbr^n}\!\!\!F_{t_k}(x)\!\cdot\!\phi(x)dx
\!=\!\!\int_{\bbbr^n}\!\!\!F(x)\!\cdot\!\phi(x)dx,\label{eq2.333}
\end{align}
therefore we have proven the function $F(x)\in\functrep_{0}(f)\cap L^p(\bbbr^n)$. Furthermore, we use (\ref{eq2.332}) and (\ref{eq2.331}) to obtain
\begin{align}
&\|f\|_{L^p(\bbbr^n)}=\|F\|_{L^p(\bbbr^n)}=\sup_{\|g\|_{L^{p'}\!(\bbbr^n)}\leq1}
\big|\int_{\bbbr^n}\!\!\!\!\!\!F(x)\!\cdot\!g(x)dx\big|\nonumber\\
&=\sup_{\|g\|_{L^{p'}\!(\bbbr^n)}\leq1}\lim_{k\rightarrow\infty}
\big|\int_{\bbbr^n}\!\!\!\!\!\!F_{t_k}(x)\!\cdot\!g(x)dx\big|\nonumber\\
&\leq\sup_{\|g\|_{L^{p'}\!(\bbbr^n)}\leq1}\lim_{k\rightarrow\infty}\|F_{t_k}\|_{L^p(\bbbr^n)}\cdot
\|g\|_{L^{p'}\!(\bbbr^n)}\nonumber\\
&\leq\sup_{t>0}\|F_t\|_{L^p(\bbbr^n)}\leq\|f\|_{H^p(\bbbr^n)}.\label{eq2.335}
\end{align}
Let $\tilde{\Phi}(x)=\Phi(-x)$, and we have
\begin{equation}\label{eq2.336}
\tilde{\Phi}_t*\phi(x)=\iFT_n[\FT_n\Phi(-t\xi)\cdot\FT_n\phi(\xi)](x)\in\Sw_{0}(\bbbr^n)
\end{equation}
whenever $\phi\in\Sw_{0}(\bbbr^n)$ and $t\in(0,\infty)$. We use Proposition \ref{proposition1} (ii) equation (\ref{eq1.124}) and Fubini's theorem to justify the exchange of the order of integration and obtain
\begin{align}
&<f*\Phi_t,\phi>=<f,\tilde{\Phi}_t*\phi>=\!\!\int_{\bbbr^n}\!\!\!\!F(x)\cdot
\tilde{\Phi}_t*\phi(x)dx\nonumber\\
&=\int_{\bbbr^n}F*\Phi_t(x)\cdot\phi(x)dx\text{ for all $\phi\in\Sw_{0}(\bbbr^n)$}.\label{eq2.337}
\end{align}
Since $F(x)\in\functrep_{0}(f)\cap L^p(\bbbr^n)$ and $\Phi_t(x)=t^{-n}\Phi(\frac{x}{t})\in\Sw(\bbbr^n)$, equation (\ref{eq2.337}) indicates the continuous function $F*\Phi_t(x)\in\functrep_{0}(f*\Phi_t)\cap L^p(\bbbr^n)$ for every $t\in(0,\infty)$. We also deduce from (\ref{eq2.330}) and (\ref{eq2.331}) that the smooth function $F_t(x)\in\functrep_{0}(f*\Phi_t)\cap L^p(\bbbr^n)$ for every $t\in(0,\infty)$. Hence the difference function
\begin{equation}\label{eq2.338}
F*\Phi_t(x)-F_t(x)
\end{equation}
is continuous, belongs to $L^p(\bbbr^n)$, and satisfies all the conditions of Proposition \ref{proposition2}. By Proposition \ref{proposition2} and the continuity of (\ref{eq2.338}), the function (\ref{eq2.338}) equals a continuous polynomial function for all $x\in\bbbr^n$, and the polynomial function is also in $L^p(\bbbr^n)$ and thus is identically zero. We have obtained for all $x\in\bbbr^n$ and $t\in(0,\infty)$,
\begin{equation}\label{eq2.339}
F*\Phi_t(x)=F_t(x),
\end{equation}
and the continuous function representative $F*\Phi_t(x)$ is identical to the function representative $F_t(x)$ specified by Definition \ref{definition4} (\ref{eq2.334}) and is a smooth function on $\bbbr^n$ for each $t\in(0,\infty)$. Furthermore, we have
\begin{align}
&\|f\|_{H^p(\bbbr^n)}=\|\sup_{t>0}|f*\Phi_t|\|_{L^p(\bbbr^n)}\nonumber\\
&=\|\sup_{t>0}|F_t|\|_{L^p(\bbbr^n)}=\|\sup_{t>0}|F*\Phi_t|\|_{L^p(\bbbr^n)}.\label{eq2.340}
\end{align}
The function $\Phi\in\Sw(\bbbr^n)$ has an integrable and radially decreasing majorant and the function $F(x)\in L^p(\bbbr^n)$ is in $\Lloc$, thus \cite[Corollary 2.1.12 and Theorem 2.1.6]{14classical} imply
\begin{equation}\label{eq2.341}
\|f\|_{H^p(\bbbr^n)}\!=\!\|\sup_{t>0}|F\!*\!\Phi_t|\|_{L^p(\bbbr^n)}\!\lesssim\!\|F\|_{L^p(\bbbr^n)}
\!=\!\|f\|_{L^p(\bbbr^n)},
\end{equation}
where the implicit constant in (\ref{eq2.341}) depends on $n,p$.\\

Step 2: We prove Lemma \ref{lemma17} (ii). Conclusions (\ref{eq2.329}) and (\ref{eq2.330}) are still true. Then the condition $f\in H^1(\bbbr^n)$ implies
\begin{align}
&\sup_{t>0}\|F_t\|_{L^1(\bbbr^n)}=\sup_{t>0}\|f*\Phi_t\|_{L^1(\bbbr^n)}\nonumber\\
&\leq\|\sup_{t>0}|f*\Phi_t|\|_{L^1(\bbbr^n)}=\|f\|_{H^1(\bbbr^n)}<\infty.\label{eq2.342}
\end{align}
We follow the argument given for the proof of \cite[Theorem 2.1.2]{14modern}. We denote $C_0(\bbbr^n)$ is the space of continuous functions $g$ defined on $\bbbr^n$ satisfying the condition 
\begin{equation}\label{eq2.369}
\lim_{|x|\rightarrow\infty}g(x)=0.
\end{equation}
The space $C_0(\bbbr^n)$ is equipped with the $\|\cdot\|_{L^{\infty}(\bbbr^n)}$-norm. We also denote $\borel(\bbbr^n)$ is the Borel $\sigma$-algebra on $\bbbr^n$. According to \cite[6.13 Theorem and 6.19 Theorem]{Rudin.Real.Complex.Anal}, every function $F_t(x)\in L^1(\bbbr^n)$ defines a regular complex Borel measure $\lambda_t$, which is given by the following equation,
\begin{equation}\label{eq2.343}
\lambda_t(E)\!\!=\!\!\!\int_{\bbbr^n}\!\!\!\chi_E(x)d\lambda_t(x)\!\!=\!\!\!\int_{\bbbr^n}\!\!\!
\chi_E(x)\!\cdot\!F_t(x)dx\!\!=\!\!\!\int_{E}\!\!\!F_t(x)dx
\end{equation}
for every $E\!\in\!\borel(\bbbr^n)$, and its variation is given by  
\begin{equation}\label{eq2.344}
|\lambda_t|(E)=\int_{E}|F_t(x)|dx,
\end{equation}
thus the total variation is $|\lambda_t|(\bbbr^n)=\|F_t\|_{L^1(\bbbr^n)}$. Furthermore, each $\lambda_t$ is a continuous linear functional defined on $C_0(\bbbr^n)$ by the following equation,
\begin{equation}\label{eq2.345}
<\lambda_t,g>=\!\!\!\int_{\bbbr^n}\!\!\!g(x)d\lambda_t(x)\!\!=\!\!\!\int_{\bbbr^n}\!\!\!
g(x)\!\cdot\!F_t(x)dx
\end{equation}
for every $g\!\in\!C_0(\bbbr^n)$. We deduce from (\ref{eq2.345}) and (\ref{eq2.342}) that for all $g\!\in\!C_0(\bbbr^n)$ with $\|g\|_{L^{\infty}(\bbbr^n)}\leq1$,
\begin{equation}\label{eq2.346}
\sup_{t>0}|\!<\!\lambda_t,g\!>\!|\!\leq\!\sup_{t>0}\|F_t\|_{L^1(\bbbr^n)}\!\cdot\!
\|g\|_{L^{\infty}(\bbbr^n)}\!\leq\!\|f\|_{H^1(\bbbr^n)}.
\end{equation}
By the Banach-Alaoglu theorem, i.e. \cite[3.17 Theorem]{Rudin.Funct.Anal}, we can find a sequence $\{t_k\}_{k\in\bbbn}$ of positive real numbers and a regular complex Borel measure $\lambda$, i.e. a complex Borel measure $\lambda$ defined on $\borel(\bbbr^n)$ whose variation $|\lambda|$ is a finite regular Borel measure, so that
\begin{equation}
0<t_{k+1}<t_k<\infty,\quad\lim_{k\rightarrow\infty}t_k=0,
\end{equation}
\begin{align}
&\int_{\bbbr^n}\!\!\!g(x)d\lambda(x)\!=<\!\lambda,g\!>=\!\!\lim_{k\rightarrow\infty}\!\!
<\!\lambda_{t_k},g\!>\nonumber\\
&=\!\!\lim_{k\rightarrow\infty}\!\int_{\bbbr^n}\!\!\!g(x)\!\cdot\!F_{t_k}(x)dx
\text{ for all $g\!\in\!C_0(\bbbr^n)$.}\label{eq2.347}
\end{align}
Since $f\!\in\!H^1(\bbbr^n)\!\subseteq\!\Sw_{0}'(\bbbr^n)$ and $\phi\!\in\!\Sw_{0}(\bbbr^n)\!\subseteq\!\Sw(\bbbr^n)\!\subseteq\!C_0(\bbbr^n)$, conclusions (\ref{eq2.329}), (\ref{eq2.330}), and (\ref{eq2.347}) imply
\begin{align}
&<f,\phi>=\lim_{k\rightarrow\infty}<f*\Phi_{t_k},\phi>=
\!\!\lim_{k\rightarrow\infty}\!\int_{\bbbr^n}\!\!\!F_{t_k}(x)\!\cdot\!\phi(x)dx\nonumber\\
&=\int_{\bbbr^n}\phi(x)d\lambda(x)=<\lambda,\phi>,\label{eq2.348}
\end{align}
and equation (\ref{eq2.348}) shows the continuous linear functional $f$ can be identified with the regular complex Borel measure $\lambda$ in the sense of $\Sw_{0}'(\bbbr^n)$. Now we prove that $\lambda$ is absolutely continuous with respect to the Lebesgue measure in $\bbbr^n$. Given a set $E\in\borel(\bbbr^n)$ with $\Lebes^n(E)=0$ and a number $\varepsilon>0$, we deduce from (\ref{eq2.342}) that there exists a number $\delta>0$ and an open set $U$ with $E\subseteq U\subseteq\bbbr^n$ so that
\begin{equation}\label{eq2.349}
\Lebes^n(U)<\delta\quad\text{and}\quad\int_U\sup_{t>0}|F_t(x)|dx<\varepsilon.
\end{equation}
We denote
\begin{equation}\label{eq2.350}
\begin{split}
&G(U)\!:=\!\{g\!\in\!C_0(\bbbr^n)\!:\!\text{$g$ has a compact support set}\\
&\text{$spt.g\!\subseteq\!U$ and $\|g\|_{L^{\infty}(\bbbr^n)}\!\leq\!1$}\},
\end{split}
\end{equation}
then we use (\ref{eq2.347}) to estimate the variation $|\lambda|(U)$ as below,
\begin{align}
&|\lambda|(U)=\sup_{g\in G(U)}\big|\int_{\bbbr^n}g(x)d\lambda(x)\big|\nonumber\\
&=\sup_{g\in G(U)}\lim_{k\rightarrow\infty}\big|\int_U g(x)\!\cdot\!F_{t_k}(x)dx\big|\nonumber\\
&\leq\|g\|_{L^{\infty}(\bbbr^n)}\cdot\int_U\sup_{t>0}|F_t(x)|dx<\varepsilon,\label{eq2.351}
\end{align}
and we postpone the proof of the equation in the first line of (\ref{eq2.351}) to Step 3. The arbitrariness of the number $\varepsilon>0$ implies $|\lambda|(E)=0$ and $\lambda(E)=0$, and hence the absolute continuity of $\lambda$ with respect to the Lebesgue measure in $\bbbr^n$. By the Lebesgue-Radon-Nikodym theorem, i.e. \cite[6.10 Theorem]{Rudin.Real.Complex.Anal}, there is a unique function $F(x)\in L^1(\bbbr^n)$ so that for every $E\in\borel(\bbbr^n)$,
\begin{equation}\label{eq2.352}
\lambda(E)\!=\!\!\int_{\bbbr^n}\!\!\!\chi_E(x)d\lambda(x)\!=\!\!\int_{\bbbr^n}\!\!\!\chi_E(x)
\!\cdot\!F(x)dx\!=\!\!\int_E F(x)dx.
\end{equation}
Furthermore, we have for all $g\in C_0(\bbbr^n)$,
\begin{equation}\label{eq2.353}
<\!\lambda,g\!>=\!\!\int_{\bbbr^n}g(x)d\lambda(x)\!=\!\!\int_{\bbbr^n}g(x)\!\cdot\!F(x)dx.
\end{equation}
Combining (\ref{eq2.348}) and (\ref{eq2.353}) yields that for all $\phi\in\Sw_{0}(\bbbr^n)$,
\begin{equation}\label{eq2.354}
<f,\phi>=<\lambda,\phi>=\!\!\int_{\bbbr^n}F(x)\!\cdot\!\phi(x)dx,
\end{equation}
therefore $F(x)\in\functrep_{0}(f)\cap L^1(\bbbr^n)$. Moreover, we deduce from \cite[6.13 Theorem and 6.19 Theorem]{Rudin.Real.Complex.Anal}, (\ref{eq2.352}), (\ref{eq2.353}), (\ref{eq2.347}), and (\ref{eq2.342}) that
\begin{align}
&\|f\|_{L^1(\bbbr^n)}\!=\!\|F\|_{L^1(\bbbr^n)}\!=\!|\lambda|(\bbbr^n)\!=\!\!\sup_{g\in G(\bbbr^n)}
\!\big|\!\!\int_{\bbbr^n}\!\!\!g(x)d\lambda(x)\big|\nonumber\\
&=\sup_{g\in G(\bbbr^n)}\lim_{k\rightarrow\infty}\big|\int_{\bbbr^n}g(x)\cdot F_{t_k}(x)dx\big|
\nonumber\\
&\leq\sup_{g\in G(\bbbr^n)}\lim_{k\rightarrow\infty}\int_{\bbbr^n}\!\!\!|F_{t_k}(x)|dx\!\cdot\!
\|g\|_{L^{\infty}(\bbbr^n)}\!\leq\!\|f\|_{H^1(\bbbr^n)},\label{eq2.355}
\end{align}
where $G(\bbbr^n)$ is the set defined in (\ref{eq2.350}) with $\bbbr^n$ in place of $U$. In addition, for all $\phi\in\Sw_{0}(\bbbr^n)$ and $t\in(0,\infty)$, we recall (\ref{eq2.336}) and use Proposition \ref{proposition1} (ii) equation (\ref{eq1.124}), the condition $F(x)\in\functrep_{0}(f)\cap L^1(\bbbr^n)$, and Fubini's theorem to justify the exchange of the order of integration, then we still have the equation (\ref{eq2.337}), thus the continuous function $F*\Phi_t(x)\in\functrep_{0}(f*\Phi_t)\cap L^1(\bbbr^n)$ for every $t\in(0,\infty)$. Conclusions (\ref{eq2.330}) and (\ref{eq2.342}) indicate the smooth function $F_t(x)\in\functrep_{0}(f*\Phi_t)\cap L^1(\bbbr^n)$ for every $t\in(0,\infty)$. Therefore the difference function (\ref{eq2.338}) is continuous, belongs to $L^1(\bbbr^n)$, and satisfies all the conditions of Proposition \ref{proposition2}. By Proposition \ref{proposition2} and the continuity of (\ref{eq2.338}), the function (\ref{eq2.338}) equals a continuous polynomial function for all $x\in\bbbr^n$, and the polynomial function is also in $L^1(\bbbr^n)$ and thus is identically zero. We have obtained the equation (\ref{eq2.339}) for all $x\in\bbbr^n$ and $t\in(0,\infty)$, and the continuous function representative $F*\Phi_t(x)$ is identical to the function representative $F_t(x)$ specified by Definition \ref{definition4} (\ref{eq2.334}) and is a smooth function on $\bbbr^n$ for each $t\in(0,\infty)$. Furthermore, we have
\begin{align}
&\|f\|_{H^1(\bbbr^n)}=\|\sup_{t>0}|f*\Phi_t|\|_{L^1(\bbbr^n)}\nonumber\\
&=\|\sup_{t>0}|F_t|\|_{L^1(\bbbr^n)}=\|\sup_{t>0}|F*\Phi_t|\|_{L^1(\bbbr^n)}.\label{eq2.356}
\end{align}\\

Step 3: In this step, we prove the equation in the first line of (\ref{eq2.351}). We prove that if $\lambda$ is a regular complex Borel measure defined on $\borel(\bbbr^n)$ and $U$ is an open set in $\bbbr^n$, then
\begin{equation}\label{eq2.357}
|\lambda|(U)=\sup_{g\in G(U)}\big|\int_{\bbbr^n}g(x)d\lambda(x)\big|,
\end{equation}
where $|\lambda|$ is the variation of $\lambda$ and $G(U)$ is the set defined in (\ref{eq2.350}). We denote $\tau$ is the topology on $\bbbr^n$ induced by the Euclidean norm, and $\tau_U$ is the topology on $U$ inherited from $\tau$, thus an open set in $\tau_U$ is the intersection of an open set in $\tau$ and the given open set $U\subseteq\bbbr^n$. We denote $\borel(\bbbr^n)$ is the $\sigma$-algebra of Borel sets in $\bbbr^n$ and $\borel(U)$ is the $\sigma$-algebra of Borel sets in $U$, then $\tau\subseteq\borel(\bbbr^n)$ and $\tau_U\subseteq\borel(U)$. Since both $(\bbbr^n,\tau)$ and $(U,\tau_U)$ are Hausdorff spaces whose open sets separate points, every compact set $K$ in either $(\bbbr^n,\tau)$ or $(U,\tau_U)$ is a closed set and hence is a Borel set. And we have for every Borel set $V\in\borel(U)$, there is a Borel set $W\in\borel(\bbbr^n)$ such that
\begin{equation}\label{eq2.358}
V=U\cap W.
\end{equation}
We denote $\Lambda=\lambda\big|_U$ is the restriction of the regular complex Borel measure $\lambda$ on the open set $U$, then $\Lambda$ is a complex Borel measure defined on the $\sigma$-algebra $\borel(U)$. For $V\in\borel(U)$ and $W\in\borel(\bbbr^n)$ satisfying (\ref{eq2.358}), we have
\begin{equation}\label{eq2.359}
\Lambda(V)=\lambda|_U(V)=\lambda(U\cap W).
\end{equation}
We claim that $|\Lambda|=|\lambda|\big|_U$, i.e. the variation of $\Lambda$ is the restriction of the variation of $\lambda$ on the open set $U$, and
\begin{equation}\label{eq2.360}
|\Lambda|(V)=|\lambda|(U\cap W)
\end{equation}
whenever $V\in\borel(U)$ and $W\in\borel(\bbbr^n)$ are related by (\ref{eq2.358}), therefore $|\Lambda|$ is a regular Borel measure defined on the $\sigma$-algebra $\borel(U)$ and $\Lambda$ is a regular complex Borel measure. If $V=\cup_{j\in\bbbn}V_j$ is a disjoint partition of $V\in\borel(U)$ into sets $\{V_j\}_{j\in\bbbn}\subseteq\borel(U)$, then we can write $V_j=U\cap A_j$ for some $A_j\in\borel(\bbbr^n)$ for every $j\in\bbbn$. Because $V=U\cap W$ for $W\in\borel(\bbbr^n)$, we obtain
\begin{equation*}
U\cap W=\bigcup_{j\in\bbbn}(U\cap A_j)
\end{equation*}
is a disjoint partition of $U\cap W\in\borel(\bbbr^n)$ into sets $\{U\cap A_j\}_{j\in\bbbn}\subseteq\borel(\bbbr^n)$, and thus
\begin{equation*}
\sum_{j=1}^{\infty}|\Lambda(V_j)|=\sum_{j=1}^{\infty}|\lambda(U\cap A_j)|\leq|\lambda|(U\cap W),
\end{equation*}
hence $|\Lambda|(V)\leq|\lambda|(U\cap W)$. If $U\cap W=\cup_{j\in\bbbn}W_j$ is a disjoint partition of $U\cap W\in\borel(\bbbr^n)$ into sets $\{W_j\}_{j\in\bbbn}\subseteq\borel(\bbbr^n)$, then
\begin{equation*}
V=U\cap W=U\cap(\bigcup_{j\in\bbbn}W_j)=\bigcup_{j\in\bbbn}(U\cap W_j)
\end{equation*}
is a disjoint partition of $V\in\borel(U)$ into sets $\{U\cap W_j\}_{j\in\bbbn}\subseteq\borel(U)$. Since $W_j\subseteq U\cap W\subset U$, we let $B_j=U\cap W_j=W_j$ for each $j\in\bbbn$, and we have
\begin{equation*}
\sum_{j=1}^{\infty}|\lambda(W_j)|=\sum_{j=1}^{\infty}|\lambda(U\cap W_j)|
=\sum_{j=1}^{\infty}|\Lambda(B_j)|\leq|\Lambda|(V),
\end{equation*}
thus $|\lambda|(U\cap W)\leq|\Lambda|(V)$, and therefore $|\Lambda|=|\lambda|\big|_U$ and equation (\ref{eq2.360}) is proven. Now we prove $|\Lambda|$ is a regular Borel measure defined on the $\sigma$-algebra $\borel(U)$. If $K$ is a compact set in $(U,\tau_U)$, then $K$ is a compact set in $(\bbbr^n,\tau)$, thus $K$ is closed in both topological spaces $(\bbbr^n,\tau)$ and $(U,\tau_U)$ and belongs to both $\borel(\bbbr^n)$ and $\borel(U)$, and the regularity of $|\lambda|$ and (\ref{eq2.360}) imply 
\begin{equation}\label{eq2.364}
|\Lambda|(K)=|\lambda|(U\cap K)=|\lambda|(K)<\infty. 
\end{equation}
Second, we prove for every $V\in\borel(U)$,
\begin{equation}\label{eq2.361}
|\Lambda|(V)=\inf\{|\Lambda|(E):V\subseteq E,E\in\tau_U\subseteq\borel(U)\}.
\end{equation}
Since the variation $|\Lambda|$ is a positive measure, we have $|\Lambda|(V)\leq|\Lambda|(E)$ whenever $V\subseteq E$ and $E\in\tau_U\subseteq\borel(U)$, thus $|\Lambda|(V)$ is dominated by the right side of equation (\ref{eq2.361}). To prove the reverse inequality, we notice that the condition $V\in\borel(U)$ implies there is $W\in\borel(\bbbr^n)$ so that both (\ref{eq2.358}) and (\ref{eq2.360}) are true. By the regularity of $|\lambda|$, we have
\begin{equation}\label{eq2.362}
|\lambda|(U\!\cap\!W)\!=\!\inf\{|\lambda|(F)\!:\!U\!\cap\!W\!\subseteq\! F,F\!\in\!\tau\!\subseteq\!\borel(\bbbr^n)\}.
\end{equation}
For every $F\!\in\!\tau\!\subseteq\!\borel(\bbbr^n)$ with $U\!\cap\!W\!\subseteq\!F$, we set $E\!=\!U\!\cap\!F$. Since $F$ is an open set in the topology $\tau$, $E$ is an open set in the topology $\tau_U$, and we also have $|\Lambda|(E)=|\lambda|(U\cap F)$ and $V=U\cap W\subseteq U\cap F=E$. Since the variation $|\lambda|$ is a positive measure, we obtain
\begin{align}
&|\lambda|(F)\geq|\lambda|(U\cap F)=|\Lambda|(E)\nonumber\\
&\geq\inf\{|\Lambda|(E):V\subseteq E,E\in\tau_U\subseteq\borel(U)\}.\label{eq2.363}
\end{align}
The arbitrariness of $F$ and the equation (\ref{eq2.360}) prove the reverse inequality, thus equation (\ref{eq2.361}) is proven. Third, we prove that if $E\in\tau_U\subseteq\borel(U)$, or if $E\in\borel(U)$ and $|\Lambda|(E)<\infty$, we have
\begin{equation}\label{eq2.365}
|\Lambda|(E)\!=\!\sup\{|\Lambda|(K)\!:\!K\!\subseteq\!E,\text{$K$ is a compact set in $(U,\tau_U)$}\}.
\end{equation}
Since the variation $|\Lambda|$ is a positive measure, we have $|\Lambda|(E)\geq|\Lambda|(K)$ whenever $K\subseteq E$, thus the right side of equation (\ref{eq2.365}) is dominated by $|\Lambda|(E)$. To prove the reverse inequality, the condition $E\in\borel(U)$ implies there is $F\in\borel(\bbbr^n)$ such that
\begin{equation}\label{eq2.366}
E=U\cap F\quad\text{and}\quad|\Lambda|(E)=|\lambda|(U\cap F).
\end{equation}
When $E\in\tau_U\subseteq\borel(U)$, we can choose $F\in\tau\subseteq\borel(\bbbr^n)$ so that (\ref{eq2.366}) holds true and $U\cap F$ is an open set in $\tau$, because both $U$ and $F$ are open sets in $\tau$. When $E\in\borel(U)$ and $|\Lambda|(E)<\infty$, we can choose $F\in\borel(\bbbr^n)$ so that (\ref{eq2.366}) holds true for $U\cap F\in\borel(\bbbr^n)$ and $|\lambda|(U\cap F)<\infty$. By the regularity of $|\lambda|$, we have
\begin{equation}\label{eq2.367}
\begin{split}
&|\lambda|(U\!\cap\!F)\!=\!\sup\{|\lambda|(K)\!:\!K\!\subseteq\!U\!\cap\!F\!=\!E,\\
&\text{and $K$ is a compact set in $(\bbbr^n,\tau)$}\}.
\end{split}
\end{equation}
Whenever $K$ is a compact set in $(\bbbr^n,\tau)$ and $K\!\subseteq\!U\!\cap\!F\!=\!E$, we have $K=U\cap K$. For an open covering of $U\cap K$ by sets in $\tau_U$, each and every covering open set is the intersection of the open set $U$ and an open set in $\tau$, and these open sets in $\tau$ form an open covering of $K$. Since $K$ is a compact set in $(\bbbr^n,\tau)$, there is a finite subcovering of $K$ by open sets in $\tau$. Hence we can obtain a finite subcovering of $U\cap K$ by open sets in $\tau_U$, and each element of this finite subcovering of $U\cap K$ is the intersection of the open set $U$ and an element of the finite subcovering of $K$ by sets in $\tau$. Therefore $K=U\cap K$ is a compact set in $(U,\tau_U)$. And we have by (\ref{eq2.360}),
\begin{align}
&|\lambda|(K)=|\lambda|(U\cap K)=|\Lambda|(K)\nonumber\\
&\leq\!\sup\{|\Lambda|(K)\!:\!K\!\subseteq\!E,\text{$K$ is a compact set in $(U,\tau_U)$}\}.
\label{eq2.368}
\end{align}
Combining (\ref{eq2.366}), (\ref{eq2.367}), and (\ref{eq2.368}) proves the reverse inequality
\begin{equation}
|\Lambda|(E)\!\leq\!\sup\{|\Lambda|(K)\!:\!K\!\subseteq\!E,\text{$K$ is a compact set in $(U,\tau_U)$}\},
\end{equation}
thus the equation (\ref{eq2.365}) is proven. Therefore the regularity of $|\Lambda|$ is proven and $\Lambda=\lambda\big|_U$ is a regular complex Borel measure defined on $\borel(U)$. The space $(U,\tau_U)$ is a locally compact Hausdorff space. We denote $C_0(U)$ is the space of continuous functions $g$ defined on $U$ with the property (\ref{eq2.369}), and $C_c(U)$ is the space of continuous functions $g$ defined on $U$ with a compact support set $spt.g\subseteq U$. The subspace $C_c(U)$ is dense in the space $C_0(U)$ with respect to the $\|\cdot\|_{L^{\infty}(U)}$-norm. According to \cite[6.19 Theorem]{Rudin.Real.Complex.Anal}, the mapping
\begin{equation}\label{eq2.370}
\iota:f\in C_0(U)\longmapsto\int_U f(x)d\Lambda(x)
\end{equation}
is a continuous linear functional defined on $C_0(U)$, and the norm of this continuous linear functional $\iota$ is the total variation of $\Lambda$. Since the open set $U$ belongs to both $\borel(U)$ and $\borel(\bbbr^n)$, equation (\ref{eq2.360}) indicates
\begin{align}
&|\lambda|(U)=|\Lambda|(U)=\sup_{\substack{g\in C_0(U)\\\|g\|_{L^{\infty}(U)}\leq1}}\big|
\int_U g(x)d\Lambda(x)\big|\nonumber\\
&=\sup_{\substack{g\in C_c(U)\\\|g\|_{L^{\infty}(U)}\leq1}}\big|\int_U g(x)d\Lambda(x)\big|
=\sup_{g\in G(U)}\big|\int_{\bbbr^n}g(x)d\Lambda(x)\big|,
\end{align}
where $G(U)$ is the set defined in (\ref{eq2.350}). The proof of Lemma \ref{lemma17} is complete.
\end{proof}

\section{Proof of Proposition \ref{proposition1}}\label{proof.of.proposition1}
\begin{proof}[Proof of Proposition \ref{proposition1}]
Step 1: Let $|\alpha|=\alpha_1+\alpha_2+\cdots+\alpha_n$ and $g\in\Sw_{00}(\bbbr^n)$. For a real number $L$, $[L]$ denotes the greatest integer that is less than or equal to $L$. Denote $\lambda(\xi)=\multi(\xi)\cdot g(\xi)$, then condition (\ref{eq1.64}) implies that we can find $0\leq L_1,L_2\in\bbbr$ so that $|\multi(\xi)|\lesssim|\xi|^{-L_1}$ if $0<|\xi|\leq1$, and $|\multi(\xi)|\lesssim|\xi|^{L_2}$ if $1<|\xi|<\infty$, thus we apply estimate (\ref{eq1.63}) with $\omega=(0,\cdots,0)=\vec{0}$ and $M=[L_1]+1$ to obtain
\begin{equation*}
0\leq\lim_{|\xi|\rightarrow0}|\lambda(\xi)|\lesssim\lim_{|\xi|\rightarrow0}|\xi|^{-L_1}\cdot|\xi|^{[L_1]+1}=0,
\end{equation*}
hence we can let $\lambda(0)=0$ so that $\lambda(\xi)$ is a continuous function on all of $\bbbr^n$ and has a fast decay rate when $|\xi|$ is large because of $g(\xi)\in\Sw_{00}(\bbbr^n)$ and the condition (\ref{eq1.64}) placed upon $\multi(\xi)$. We investigate the differentiability property of $\lambda(\xi)$ at $\xi=0$. For $1\leq j\leq n$, let $0\neq h\in\bbbr$ and $\vec{e_j}$ be the unit elementary vector whose $j$-th coordinate is $1$ and all the other coordinates are $0$. Then the difference quotient can be estimated from above when $|h|\rightarrow0$ as follows,
\begin{align}
&\bigg|\frac{\lambda(h\vec{e_j})-\lambda(0)}{h}\bigg|
=\frac{|\multi(h\vec{e_j})|\cdot|g(h\vec{e_j})|}{|h|}\nonumber\\
&\lesssim\frac{|h|^{-L_1}}{|h|}\cdot|g(h\vec{e_j})|
\lesssim\frac{|h|^{-L_1}}{|h|}\cdot|h|^{[L_1]+2},\label{eq1.68}
\end{align}
where in (\ref{eq1.68}), we also use estimate (\ref{eq1.63}) with $\omega=\vec{0}$ and $M=[L_1]+2$. The above estimate shows that
\begin{equation*}
\partial_j\lambda(0)=\frac{\partial}{\partial\xi_j}[\lambda(\xi)]\bigg|_{\xi=0}=\lim_{h\rightarrow0}
\frac{\lambda(h\vec{e_j})-\lambda(0)}{h}=0.
\end{equation*}
Therefore all the first order partial derivatives of $\lambda(\xi)$ at $\xi=0$ exist and equal zero. Let $\xi\neq0$ and we can use Leibniz rule to obtain
\begin{equation}\label{eq1.69}
\partial_j\lambda(\xi)=\partial_j\multi(\xi)\cdot g(\xi)+\multi(\xi)\cdot\partial_j g(\xi).
\end{equation}
Furthermore, condition (\ref{eq1.64}) implies that we can find $0\leq L_3,L_4\in\bbbr$ so that $|\partial_j\multi(\xi)|\lesssim|\xi|^{-L_3}$ if $0<|\xi|\leq1$, and $|\partial_j\multi(\xi)|\lesssim|\xi|^{L_4}$ if $1<|\xi|<\infty$, then applying estimate (\ref{eq1.63}) with $\omega=\vec{0}$ and $M=[L_3]+1$ yields
\begin{equation}\label{eq1.70}
|\partial_j\multi(\xi)\cdot g(\xi)|\lesssim|\xi|^{-L_3}\cdot|\xi|^{[L_3]+1}\quad\text{if $0<|\xi|\leq1$},
\end{equation}
and applying estimate (\ref{eq1.63}) with $\omega=\vec{e_j}$ and $M=[L_1]+1$ yields
\begin{equation}\label{eq1.71}
|\multi(\xi)\cdot\partial_j g(\xi)|\lesssim|\xi|^{-L_1}\cdot|\xi|^{[L_1]+1}\quad\text{if $0<|\xi|\leq1$}.
\end{equation}
We deduce from (\ref{eq1.69}), (\ref{eq1.70}), (\ref{eq1.71}) the following estimate
\begin{equation}\label{eq1.72}
|\partial_j\lambda(\xi)|\lesssim|\xi|^{[L_1]+1-L_1}+|\xi|^{[L_3]+1-L_3},
\end{equation}
and hence $\lim_{|\xi|\rightarrow0}\partial_j\lambda(\xi)=0=\partial_j\lambda(0)$ for $1\leq j\leq n$. Therefore all the first order partial derivatives of $\lambda(\xi)$ are continuous at $\xi=0$, and hence are continuous on all of $\bbbr^n$, and have fast decay rates when $|\xi|$ is large because of $g(\xi)\in\Sw_{00}(\bbbr^n)$ and the condition (\ref{eq1.64}) placed upon $\multi(\xi)$. We can continue these computations and find that all the partial derivatives of the function $\lambda(\xi)$ at $\xi=0$ exist and equal $0$. Moreover, applying Leibniz rule yields that a typical partial derivative of $\lambda(\xi)=\multi(\xi)\cdot g(\xi)$ is estimated pointwise from above by a finite linear combination, with coefficients depending on some fixed parameters, of terms of the following form
\begin{equation}\label{eq1.73}
|\partial^{\alpha}\multi(\xi)|\cdot|\partial^{\omega}g(\xi)|,
\end{equation}
where $\alpha$ and $\omega$ are multi-indices. Applying condition (\ref{eq1.64}) and estimate (\ref{eq1.63}) with $M=[L_1^{\alpha}]+1$ shows that $(\ref{eq1.73})\rightarrow0$ as $|\xi|\rightarrow0$. Hence all the partial derivatives of $\lambda(\xi)$ are continuous at $\xi=0$, and thus are continuous on all of $\bbbr^n$, and have fast decay rates when $|\xi|$ is large due to $g(\xi)\in\Sw_{00}(\bbbr^n)$ and condition (\ref{eq1.64}). Therefore we conclude that the function $\lambda(\xi)=\multi(\xi)\cdot g(\xi)$ is a Schwartz function in $\Sw_{00}(\bbbr^n)$.\\ 

Step 2: We prove Proposition \ref{proposition1} (i). Let $f\in\Sw_{00}'(\bbbr^n)$. The product $\multi(\xi)\cdot f$ of the function $\multi(\xi)$ and the continuous linear functional $f\in\Sw_{00}'(\bbbr^n)$, whose action on functions in $\Sw_{00}(\bbbr^n)$ is given by (\ref{eq1.65}), is a well-defined linear functional on the subspace $\Sw_{00}(\bbbr^n)$, because the right side of equation (\ref{eq1.65}) is a well-defined action of $f\in\Sw_{00}'(\bbbr^n)$ on the Schwartz function $\multi\cdot g\in\Sw_{00}(\bbbr^n)$. We show that $\multi(\xi)\cdot f$ is a continuous linear functional on the subspace $\Sw_{00}(\bbbr^n)$ in the topology of $\Sw(\bbbr^n)$. By (\ref{eq1.89}) and (\ref{eq1.107}), we can find $N\in\bbbn_0$ such that for $g\in\Sw_{00}(\bbbr^n)$, we have
\begin{equation}\label{eq1.74}
|\!\!<\!\multi(\xi)\!\cdot\! f,g\!>\!\!|\!=\!|\!\!<\!f,\multi\cdot g\!>\!\!|\!\lesssim\!\!\sum_{\substack{|\beta|\leq N\\|\gamma|\leq N}}
\sup_{\xi\in\bbbr^n}\!\big|\xi^{\beta}\partial_{\xi}^{\gamma}[\multi(\xi)\!\cdot\! g(\xi)]\big|,
\end{equation}
where $\beta$, $\gamma$ are multi-indices, and the nonnegative integer $N$ and the implicit constant in (\ref{eq1.74}) are determined by $f$. A direct application of Leibniz rule tells us that each term on the right side of (\ref{eq1.74}) can be estimated from above by a finite linear combination of terms of the following form,
\begin{equation}\label{eq1.75}
\sup_{\xi\in\bbbr^n}|\xi|^{|\beta|}\cdot\big|(\partial^{\alpha}\multi)(\xi)\big|
\cdot\big|(\partial^{\omega}g)(\xi)\big|,
\end{equation}
and the multi-indices satisfy $\alpha+\omega=\gamma$. The condition (\ref{eq1.64}) implies that we can find $0\leq L_1^{\alpha},L_2^{\alpha}\in\bbbr$ so that 
\begin{equation}\label{eq3.166}
|(\partial^{\alpha}\multi)(\xi)|\lesssim
\begin{cases}
|\xi|^{-L_1^{\alpha}}&\text{if}\quad0<|\xi|\leq1,\\
|\xi|^{L_2^{\alpha}}&\text{if}\quad1<|\xi|<\infty.
\end{cases}
\end{equation}
Applying estimate (\ref{eq1.63}) with $M_1=[\big||\beta|-L_1^{\alpha}\big|]+1>\big||\beta|-L_1^{\alpha}\big|$ yields the following
\begin{align}
&\sup_{|\xi|\leq1}|\xi|^{|\beta|}\cdot\big|(\partial^{\alpha}\multi)(\xi)\big|
\cdot\big|(\partial^{\omega}g)(\xi)\big|\nonumber\\
&\lesssim\sup_{|\xi|\leq1}|\xi|^{|\beta|-L_1^{\alpha}}\cdot\sum_{|\omega'|=M_1}\sup_{x\in\bbbr^n}|(\partial^{\omega+\omega'}g)(x)|\cdot|\xi|^{M_1}\nonumber\\
&\lesssim\sup_{|\xi|\leq1}|\xi|^{M_1-\big||\beta|-L_1^{\alpha}\big|}\cdot
\sum_{|\omega'|=M_1}\sup_{x\in\bbbr^n}|(\partial^{\omega+\omega'}g)(x)|,\label{eq1.76}
\end{align}
and (\ref{eq1.76}) can be estimated from above by a finite sum of Schwartz seminorms of function $g$. Also, we let $M_2=[|\beta|+L_2^{\alpha}]+1>|\beta|+L_2^{\alpha}$, then we have
\begin{align}
&\sup_{|\xi|>1}|\xi|^{|\beta|}\cdot\big|(\partial^{\alpha}\multi)(\xi)\big|
\cdot\big|(\partial^{\omega}g)(\xi)\big|\nonumber\\
&\lesssim\sup_{|\xi|>1}\frac{|\xi|^{|\beta|+L_2^{\alpha}}}{(1+|\xi|)^{M_2}}\cdot(1+|\xi|)^{M_2}
\cdot\big|(\partial^{\omega}g)(\xi)\big|\nonumber\\
&\lesssim\sup_{|\xi|>1}(1+|\xi|)^{M_2}\cdot\big|(\partial^{\omega}g)(\xi)\big|,\label{eq1.77}
\end{align}
and (\ref{eq1.77}) can be estimated from above by a finite sum of Schwartz seminorms of function $g$. Combining (\ref{eq1.74}), (\ref{eq1.75}), (\ref{eq1.76}), and (\ref{eq1.77}), we see that the term $|<\multi(\xi)\cdot f,g>|$ can be estimated from above by a finite sum of Schwartz seminorms of function $g$ whenever $f\in\Sw_{00}'(\bbbr^n)$ and $g\in\Sw_{00}(\bbbr^n)$, and this finite sum is determined by $\multi$ and $f$. Therefore the linear functional $\multi(\xi)\cdot f$ defined by equation (\ref{eq1.65}) is a continuous linear functional on the subspace $\Sw_{00}(\bbbr^n)$ under the topology inherited from $\Sw(\bbbr^n)$. According to the Hahn-Banach theorem, i.e. \cite[3.3 Theorem]{Rudin.Funct.Anal}, this linear functional $\multi(\xi)\cdot f$ has a continuous linear extension to a tempered distribution defined on all of $\Sw(\bbbr^n)$ in the topology of $\Sw(\bbbr^n)$, and we still denote the extended tempered distribution by $\multi(\xi)\cdot f$. The action of the extended tempered distribution $\multi(\xi)\cdot f$ on functions in $\Sw_{00}(\bbbr^n)$ is given by (\ref{eq1.65}). And the distributional inverse Fourier transform of the extended tempered distribution $\multi(\xi)\cdot f$, as given in \cite[Definition 2.3.7.]{14classical}, is well-defined by the following equation
\begin{equation}\label{eq1.78}
<\iFT_n[\multi(\xi)\cdot f],\varphi>=<\multi(\xi)\cdot f,\iFT_n\varphi>
\end{equation}
for a general Schwartz function $\varphi\in\Sw(\bbbr^n)$. Furthermore, if $\varphi\in\Sw_0(\bbbr^n)$ then $\iFT_n\varphi\in\Sw_{00}(\bbbr^n)$, and we can combine (\ref{eq1.78}) and (\ref{eq1.65}) together to deduce that
\begin{equation}\label{eq1.79}
<\!\iFT_n[\multi(\xi)\!\cdot\!f],\varphi\!>=<\!\multi(\xi)\!\cdot\!f,\iFT_n\varphi\!>
=<\!f,\multi\!\cdot\!\iFT_n\varphi\!>.
\end{equation}
We have finished the proof of Proposition \ref{proposition1} (i).\\

Step 3: We prove Proposition \ref{proposition1} (ii). If $\multi\in L^1(\bbbr^n)$ and $g\in L^1(\bbbr^n)$, then we have
\begin{equation}\label{eq3.167}
\iFT_n[\multi(\xi)\FT_n g(\xi)](x)\!=\!\!\int_{\bbbr^n}\!\int_{\bbbr^n}\!\!\!\multi(\xi)\!\cdot\!g(y)\!\cdot\!
e^{2\pi i\xi\cdot(x-y)}dyd\xi, 
\end{equation}
and
\begin{equation*}
\int_{\bbbr^n}\int_{\bbbr^n}|\multi(\xi)|\!\cdot\!|g(y)|dyd\xi<\infty.
\end{equation*}
We invoke \cite[Tonelli's theorem and Corollary 7 in section 20.1]{real.analysis.royden} and \cite[Theorem 11 in section 20.2]{real.analysis.royden} to deduce that the nonnegative function $|\multi(\xi)|\!\cdot\!|g(y)|$ is integrable over the product space $\bbbr^n\times\bbbr^n$ with respect to the product measure $d(y,\xi)$, and hence we can apply \cite[Fubini's theorem in section 20.1]{real.analysis.royden} to justify the exchange of the order of integration in (\ref{eq3.167}) and obtain
\begin{equation}\label{eq3.168}
(\ref{eq3.167})\!=\!\!\int_{\bbbr^n}\!\!\!\iFT_n\multi(x\!-\!y)\!\cdot\!g(y)dy\!=\!(\iFT_n\multi)*g(x).
\end{equation}
Equations (\ref{eq3.167}) and (\ref{eq3.168}) prove (\ref{eq1.110}). When $\varphi\in\Sw_{0}(\bbbr^n)$, we have $\FT_n\varphi\in\Sw_{00}(\bbbr^n)$ and thus $\multi(\xi)\cdot\FT_n\varphi(\xi)\in\Sw_{00}(\bbbr^n)$ and $\iFT_n[\multi(\xi)\cdot\FT_n\varphi(\xi)]\in\Sw_{0}(\bbbr^n)$. Equation (\ref{eq1.111}) is a consequence of (\ref{eq1.110}). Since $(\iFT_n\multi)*\varphi\in\Sw_0(\bbbr^n)$ and $f\in\Sw_{0}'(\bbbr^n)$, the action of the convolution $f*\FT_n\multi$ of the function $\FT_n\multi\in L^{\infty}(\bbbr^n)$ and the continuous linear functional $f\in\Sw_0'(\bbbr^n)$ on the function $\varphi\in\Sw_{0}(\bbbr^n)$, as given by equation (\ref{eq1.112}), is well-defined and makes sense. According to (\ref{eq1.89}) and (\ref{eq1.93}), we can find $N\in\bbbn_0$ so that
\begin{align}
&\big|<f*\FT_n\multi,\varphi>\big|\lesssim\sum_{\substack{|\beta|\leq N\\|\gamma|\leq N}}
\sup_{x\in\bbbr^n}\big|x^{\beta}\partial_x^{\gamma}[(\iFT_n\multi)*\varphi(x)]\big|\nonumber\\
&\lesssim\sum_{\substack{|\beta|\leq N\\|\gamma|\leq N}}\sup_{x\in\bbbr^n}
\big|x^{\beta}\partial_x^{\gamma}[\iFT_n[\multi(\xi)\cdot\FT_n\varphi(\xi)](x)]\big|,\label{eq3.169}
\end{align}
where $\beta$ and $\gamma$ are multi-indices, and $N\in\bbbn_0$ and the implicit constants in (\ref{eq3.169}) are determined by $f$. Since $\multi(\xi)\cdot\FT_n\varphi(\xi)\in\Sw_{00}(\bbbr^n)$ and $\xi^{\gamma}\cdot\multi(\xi)\cdot\FT_n\varphi(\xi)\in\Sw_{00}(\bbbr^n)$, we invoke \cite[Proposition 2.2.11 (9) and (10)]{14classical} to deduce the following formula
\begin{align}
&x^{\beta}\partial_x^{\gamma}[\iFT_n[\multi(\xi)\cdot\FT_n\varphi(\xi)](x)]\nonumber\\
&=x^{\beta}\iFT_n[(2\pi i\xi)^{\gamma}\cdot\multi(\xi)\cdot\FT_n\varphi(\xi)](x)\nonumber\\
&=(-2\pi i)^{-|\beta|}\cdot\iFT_n[\partial^{\beta}_{\xi}
[(2\pi i\xi)^{\gamma}\cdot\multi(\xi)\cdot\FT_n\varphi(\xi)]](x).\label{eq3.170}
\end{align}
Combining (\ref{eq3.169}) and (\ref{eq3.170}) and applying the Leibniz rule, we can obtain that for multi-indices $\beta$ and $\gamma$ with $|\beta|\leq N$ and $|\gamma|\leq N$, each term in the second line of (\ref{eq3.169}) can be estimated from above by a constant multiple of a finite sum of terms of the following form,
\begin{equation}\label{eq3.171}
\int_{\bbbr^n}|\partial_{\xi}^{\lambda}(\xi^{\gamma})|\!\cdot\!|\partial^{\alpha}\multi(\xi)|\!\cdot\!
|\partial^{\omega}\FT_n\varphi(\xi)|d\xi,
\end{equation}
where $\lambda,\alpha,\omega$ are multi-indices that satisfy $\lambda+\alpha+\omega=\beta$ and $\lambda\leq\gamma$, i.e. $0\leq\lambda_j\leq\gamma_j$ for all $1\leq j\leq n$, and the constant depends on $N$. We use estimates (\ref{eq3.166}) and (\ref{Sw0.condition.3}) with $M_3=[L_1^{\alpha}]+1>L_1^{\alpha}$ to deduce the following,
\begin{align}
&\int_{|\xi|\leq1}|\partial_{\xi}^{\lambda}(\xi^{\gamma})|\!\cdot\!|\partial^{\alpha}\multi(\xi)|
\!\cdot\!|\partial^{\omega}\FT_n\varphi(\xi)|d\xi\nonumber\\
&\lesssim\!\!\!\int_{|\xi|\leq1}\!\!\!\!\!\!\!\!\!|\xi|^{|\gamma|-|\lambda|}\!\cdot\!
|\xi|^{-L_1^{\alpha}}\!\cdot\!|\xi|^{M_3}d\xi\!\cdot\!\!\!\!\sum_{|\omega'|=M_3}\sup_{y\in\bbbr^n}
\!|(\partial^{\omega+\omega'}\FT_n\varphi)(y)|\nonumber\\
&\lesssim\sum_{|\omega'|=M_3}\sup_{y\in\bbbr^n}|(\partial^{\omega+\omega'}\FT_n\varphi)(y)|,
\label{eq3.172}
\end{align}
where the implicit constants in (\ref{eq3.172}) depend on $n,N,\multi$. Since the Fourier transform is a homeomorphism from $\Sw(\bbbr^n)$ onto itself, then (\ref{eq3.172}) can be estimated from above by a constant multiple of a finite sum of Schwartz seminorms of $\varphi$. We use estimate (\ref{eq3.166}) and the fact that $\partial^{\omega}\FT_n\varphi\in\Sw(\bbbr^n)$ to deduce that for $M_4=n+N+[L_2^{\alpha}]+1>n+|\gamma|-|\lambda|+L_2^{\alpha}$,
\begin{align}
&\int_{|\xi|>1}|\partial_{\xi}^{\lambda}(\xi^{\gamma})|\!\cdot\!|\partial^{\alpha}\multi(\xi)|
\!\cdot\!|\partial^{\omega}\FT_n\varphi(\xi)|d\xi\nonumber\\
&\lesssim\!\!\!\int_{|\xi|>1}\frac{|\xi|^{|\gamma|-|\lambda|}\!\cdot\!|\xi|^{L_2^{\alpha}}}
{(1+|\xi|)^{M_4}}d\xi\!\cdot\!\sup_{y\in\bbbr^n}(1+|y|)^{M_4}\!\cdot\!
|\partial^{\omega}\FT_n\varphi(y)|\nonumber\\
&\lesssim\sup_{y\in\bbbr^n}(1+|y|)^{M_4}\!\cdot\!|\partial^{\omega}\FT_n\varphi(y)|,\label{eq3.173}
\end{align}
where the implicit constants in (\ref{eq3.173}) depend on $n,N,\multi$. And (\ref{eq3.173}) can be estimated from above by a constant multiple of a finite sum of Schwartz seminorms of $\varphi$. Combining (\ref{eq3.169}), (\ref{eq3.171}), (\ref{eq3.172}), and (\ref{eq3.173}) all together, we conclude that the expression $|<f*\FT_n\multi,\varphi>|$ can be estimated from above by a constant multiple of a finite sum of Schwartz seminorms of $\varphi$, and the constant and the finite sum are both determined by $f$, $\multi$, and $n$. Therefore the convolution $f*\FT_n\multi$ defines a continuous linear functional in $\Sw_{0}'(\bbbr^n)$ by the equation (\ref{eq1.112}). According to the Hahn-Banach theorem, i.e. \cite[3.3 Theorem]{Rudin.Funct.Anal}, this linear functional $f*\FT_n\multi$ has a continuous linear extension to a tempered distribution defined on all of $\Sw(\bbbr^n)$ in the topology of $\Sw(\bbbr^n)$, and we still denote the extended tempered distribution by $f*\FT_n\multi$. The action of the extended tempered distribution $f*\FT_n\multi$ on functions in $\Sw_{0}(\bbbr^n)$ is given by (\ref{eq1.112}). And the distributional inverse Fourier transform of the extended tempered distribution $f*\FT_n\multi$, as given in \cite[Definition 2.3.7.]{14classical}, is well-defined by the following equation
\begin{equation}\label{eq3.174}
<\iFT_n[f*\FT_n\multi],g>=<f*\FT_n\multi,\iFT_n g>
\end{equation}
for a general Schwartz function $g\in\Sw(\bbbr^n)$. Furthermore, if $g\in\Sw_{00}(\bbbr^n)$ then $\iFT_n g\in\Sw_{0}(\bbbr^n)$, and we can combine (\ref{eq3.174}), (\ref{eq1.112}), and (\ref{eq1.111}) together to deduce that
\begin{align}
&<\iFT_n[f*\FT_n\multi],g>=<f*\FT_n\multi,\iFT_n g>\nonumber\\
&=<f,\iFT_n\multi*\iFT_n g>=<f,\iFT_n[\multi(\xi)\!\cdot\!g(\xi)]>.\label{eq3.175}
\end{align}
In particular, when $f\!\in\!\Sw_{0}'(\bbbr^n)$ and $g\!\in\!\Sw(\bbbr^n)$, we set $\multi(\xi)\!=\!\iFT_n g(\xi)\!\in\!\Sw(\bbbr^n)\!\subseteq\!L^1(\bbbr^n)$, then $\multi(\xi)$ satisfies the condition (\ref{eq1.64}) and $\iFT_n\multi\!=\!\tilde{g}$. As a consequence of equation (\ref{eq1.112}), the convolution $f*g$ is a continuous linear functional in $\Sw_0'(\bbbr^n)$ defined by the equation (\ref{eq1.124}). We temporarily denote the HB-extension of $f\in\Sw_{0}'(\bbbr^n)$ obtained by the Hahn-Banach theorem is $\overline{f}\in\Sw'(\bbbr^n)$. According to \cite[Theorem 2.3.20.]{14classical}, $\overline{f}*g\in\Sw'(\bbbr^n)$ has the smooth function representative (\ref{eq3.195}) in the sense of $\Sw'(\bbbr^n)$,  
\begin{equation}\label{eq3.195}
G:x\in\bbbr^n\longmapsto<\overline{f},g(x-\cdot)>.
\end{equation}
Furthermore, for every multi-index $\alpha$, there exist positive finite constants $C_{\alpha}$ and $K_{\alpha}$ such that
\begin{equation}\label{eq3.196}
|\partial^{\alpha}G(x)|\leq C_{\alpha}\cdot(1+|x|)^{K_{\alpha}}.
\end{equation}
Since $\tilde{g}*\varphi\in\Sw_{0}(\bbbr^n)$ and $\overline{f}\in\Sw'(\bbbr^n)$, \cite[Definition 2.3.13]{14classical} implies for every $\varphi\in\Sw_0(\bbbr^n)$,
\begin{equation}\label{eq3.197}
<\overline{f}*g,\varphi>=<\overline{f},\tilde{g}*\varphi>=<f,\tilde{g}*\varphi>.
\end{equation}
Equations (\ref{eq1.124}) and (\ref{eq3.197}) indicate that $f*g=\overline{f}*g$ in the sense of $\Sw_0'(\bbbr^n)$, thus the function (\ref{eq3.195}) belongs to the collection
\begin{equation}\label{eq3.198}
\functrep(\overline{f}*g)\subseteq\functrep_0(\overline{f}*g)=\functrep_0(f*g).
\end{equation}
The estimate (\ref{eq1.125}) is a consequence of (\ref{eq3.196}). Now we prove the equation (\ref{eq1.115}). If $f\in\Sw_{0}'(\bbbr^n)$, $h\in\Sw(\bbbr^n)$, $\phi\in\Sw_{0}(\bbbr^n)$, and the function $\multi(\xi)$ satisfies the conditions for Proposition \ref{proposition1} (ii), we let $\tilde{h}(x)=h(-x)$ then
\begin{equation*}
	\tilde{h}*\phi(x)=\iFT_n[\FT_n h(-\xi)\!\cdot\!\FT_n\phi(\xi)](x)\in\Sw_{0}(\bbbr^n).
\end{equation*} 
We apply \cite[Definition 2.3.13]{14classical} to the extended tempered distribution $f*\FT_n\multi$ and obtain
\begin{align}
&<(f*\FT_n\multi)*h,\phi>=<f*\FT_n\multi,\tilde{h}*\phi>\nonumber\\
&=<\!f,\iFT_n\multi\!*\!\tilde{h}\!*\!\phi\!>
=<\!f,\iFT_n[\multi(\xi)\!\cdot\!\FT_n h(-\xi)\!\cdot\!\FT_n\phi(\xi)]\!>,\label{eq3.176}
\end{align}
where equations in the second line of (\ref{eq3.176}) are due to (\ref{eq1.112}) and (\ref{eq1.111}). According to equation (\ref{eq1.124}), the convolution $f*h$ of $f\in\Sw_{0}'(\bbbr^n)$ and $h\in\Sw(\bbbr^n)$ is a well-defined continuous linear functional in $\Sw_{0}'(\bbbr^n)$. Applying (\ref{eq1.112}), (\ref{eq1.111}), and (\ref{eq1.124}) in a sequence yields the following,
\begin{align}
&<(f*h)*\FT_n\multi,\phi>=<f*h,\iFT_n\multi*\phi>\nonumber\\
&=<\!f,\tilde{h}\!*\!\iFT_n\multi\!*\!\phi\!>=
<\!f,\iFT_n[\FT_n h(-\xi)\!\cdot\!\multi(\xi)\!\cdot\!\FT_n\phi(\xi)]\!>,\label{eq3.177}
\end{align}
where the last term of (\ref{eq3.177}) is the well-defined action of $f\in\Sw_{0}'(\bbbr^n)$ on the function 
\begin{equation}\label{eq3.178}
\iFT_n[\FT_n h(-\xi)\!\cdot\!\multi(\xi)\!\cdot\!\FT_n\phi(\xi)](x)\in\Sw_{0}(\bbbr^n).
\end{equation}
Equations (\ref{eq3.176}) and (\ref{eq3.177}) justify the validity of (\ref{eq1.115}).\\

Step 4: We prove Proposition \ref{proposition1} (iii). If $f\in\Sw_{00}'(\bbbr^n)$, $\multi(\xi)\in L^1(\bbbr^n)$ satisfies condition (\ref{eq1.64}), and $\phi\in\Sw_{0}(\bbbr^n)$, then Proposition \ref{proposition1} (i) implies
\begin{equation}\label{eq3.179}
<\FT_n[\multi(\xi)\cdot f],\phi>=<f,\multi\cdot\FT_n\phi>.
\end{equation}
Let $\FT_n f\in\Sw'(\bbbr^n)\subseteq\Sw_{0}'(\bbbr^n)$ be the distributional Fourier transform of the HB-extension of $f\in\Sw_{00}'(\bbbr^n)$, then we deduce from (\ref{eq1.111}) that the function (\ref{eq3.180}) is in $\Sw_{00}(\bbbr^n)$,
\begin{equation}\label{eq3.180}
\xi\in\bbbr^n\longmapsto\FT_n[\iFT_n\multi*\phi](\xi)=\multi(\xi)\!\cdot\!\FT_n\phi(\xi).
\end{equation}
And Proposition \ref{proposition1} (ii) implies
\begin{align}
&<\FT_n f*\FT_n\multi,\phi>=<\FT_n f,\iFT_n\multi*\phi>\nonumber\\
&=<f,\FT_n[\iFT_n\multi*\phi]>=<f,\multi\cdot\FT_n\phi>.\label{eq3.181}
\end{align}
The last term of (\ref{eq3.181}) is the action of the HB-extension of $f\in\Sw_{00}'(\bbbr^n)$ on the function $(\ref{eq3.180})\in\Sw_{00}(\bbbr^n)$, and it coincides with the action of the given continuous linear functional $f\in\Sw_{00}'(\bbbr^n)$ on the function (\ref{eq3.180}). Equations (\ref{eq3.179}) and (\ref{eq3.181}) justify the validity of (\ref{eq1.116}).\\

Step 5: We prove Proposition \ref{proposition1} (iv). If $f\in\Sw_{0}'(\bbbr^n)$, $\varphi\in\Sw_0(\bbbr^n)$, $\tilde{\varphi}(x)=\varphi(-x)$, and $g\in\Sw(\bbbr^n)$, then because $\tilde{\varphi}*g\in\Sw_{0}(\bbbr^n)$, the following equation (\ref{eq3.186}) defines a linear functional on $\Sw(\bbbr^n)$, 
\begin{equation}\label{eq3.186}
<f*\varphi,g>=<f,\tilde{\varphi}*g>.
\end{equation}
By (\ref{eq1.93}) and (\ref{eq1.89}), we can find $N\in\bbbn_0$ so that
\begin{equation}\label{eq3.187}
|<f*\varphi,g>|\lesssim\rho_N(\tilde{\varphi}*g),
\end{equation}
where $N\in\bbbn_0$ and the implicit constant in (\ref{eq3.187}) are determined by $f$. For all multi-indices $\alpha$ and $\beta$ with $|\alpha|\leq N$ and $|\beta|\leq N$, we have
\begin{equation}\label{eq3.188}
\sup_{x\in\bbbr^n}|x^{\alpha}\!\cdot\!\partial^{\beta}(\tilde{\varphi}*g)(x)|
\!=\!\sup_{x\in\bbbr^n}|x^{\alpha}\!\cdot\!(\tilde{\varphi}*\partial^{\beta}g)(x)|\!\lesssim\!
\rho_N(g),
\end{equation}
and the implicit constant in (\ref{eq3.188}) is determined by $n,N,\varphi$. Inserting (\ref{eq3.188}) into (\ref{eq3.187}) and invoking \cite[Proposition 2.3.4. (b)]{14classical}, we see that the linear functional $f*\varphi$ defined by (\ref{eq3.186}) is in $\Sw'(\bbbr^n)$. We temporarily denote the HB-extension of $f\in\Sw_{0}'(\bbbr^n)$ obtained by the Hahn-Banach theorem is $\overline{f}\in\Sw'(\bbbr^n)$. According to \cite[Theorem 2.3.20.]{14classical}, $\overline{f}*\varphi\in\Sw'(\bbbr^n)$ has the smooth function representative (\ref{eq3.182}) in the sense of $\Sw'(\bbbr^n)$,  
\begin{equation}\label{eq3.182}
F:x\in\bbbr^n\longmapsto<\overline{f},\varphi(x-\cdot)>=<f,\varphi(x-\cdot)>,
\end{equation}
where the equation in (\ref{eq3.182}) holds true because for every $x\in\bbbr^n$, we have
\begin{equation}\label{eq3.183}
y\in\bbbr^n\longmapsto\varphi(x-y)\in\Sw_{0}(\bbbr^n).
\end{equation}
Furthermore, for every multi-index $\alpha$, there exist positive finite constants $C_{\alpha}$ and $K_{\alpha}$ such that
\begin{equation}\label{eq3.189}
|\partial^{\alpha}F(x)|\leq C_{\alpha}\cdot(1+|x|)^{K_{\alpha}}.
\end{equation}
Since $\tilde{\varphi}*g\in\Sw_{0}(\bbbr^n)$ and $\overline{f}\in\Sw'(\bbbr^n)$, \cite[Definition 2.3.13]{14classical} implies for every $g\in\Sw(\bbbr^n)$,
\begin{equation}\label{eq3.190}
<\overline{f}*\varphi,g>=<\overline{f},\tilde{\varphi}*g>=<f,\tilde{\varphi}*g>.
\end{equation}
Equations (\ref{eq3.186}) and (\ref{eq3.190}) indicate that $f*\varphi=\overline{f}*\varphi$ in the sense of $\Sw'(\bbbr^n)$, thus the function (\ref{eq3.182}) belongs to the collection
\begin{equation}\label{eq3.191}
\functrep(\overline{f}*\varphi)=\functrep(f*\varphi).
\end{equation}
The estimate (\ref{eq1.120}) is a consequence of (\ref{eq3.189}). Now we prove (\ref{eq1.126}). Let $g\in\Sw(\bbbr^n)$. We use \cite[Definition 2.3.7]{14classical} and (\ref{eq1.122}) to deduce
\begin{equation}\label{eq3.192}
<\FT_n[f*\varphi],g>=<f*\varphi,\FT_n g>=<f,\tilde{\varphi}*\FT_n g>.
\end{equation}
Since $\FT_n f\in\Sw'(\bbbr^n)$ is the distributional Fourier transform of the HB-extension of $f\in\Sw_{0}'(\bbbr^n)$ and since $\FT_n\varphi\in\Sw_{00}(\bbbr^n)$, we use \cite[Definition 2.3.15 and Definition 2.3.7]{14classical} to obtain
\begin{align}
&<\FT_n f\cdot\FT_n\varphi,g>=<\FT_n f,\FT_n\varphi\cdot g>\nonumber\\
&=<f,\FT_n[\FT_n\varphi\cdot g]>=<f,\tilde{\varphi}*\FT_n g>,\label{eq3.193}
\end{align}
where the last term of (\ref{eq3.193}) is the action of the HB-extension of $f\in\Sw_{0}'(\bbbr^n)$ on the function $\tilde{\varphi}*\FT_n g\in\Sw_{0}(\bbbr^n)$ and coincides with the action of the original continuous linear functional $f\in\Sw_{0}'(\bbbr^n)$ on the function $\tilde{\varphi}*\FT_n g\in\Sw_{0}(\bbbr^n)$. Equations (\ref{eq3.192}) and (\ref{eq3.193}) prove the formula (\ref{eq1.126}) is true in the sense of $\Sw'(\bbbr^n)$. The proof of Proposition \ref{proposition1} is now complete.
\end{proof}

\section{The boundedness of the Littlewood-Paley-Poisson function}\label{proof.of.theorem3}
\begin{proof}[Proof of Theorem \ref{theorem3}]
Step 1: First we notice that it suffices to prove
		\begin{equation}\label{eq3.1}
			\|(\int_0^{\infty}t^{q-sq-1}|\partial_l\Pint(f;\cdot,t)|^q dt)^{\frac{1}{q}}\|_{L^p(\bbbr^n)}
			\lesssim\|f\|_{\Fspq},
		\end{equation}
for $1\leq l\leq n+1$, $0<p,q<\infty$, $-1<s<1$, and $f\in\Fspq\subseteq\Sw_{0}'(\bbbr^n)$ has a function representative 
\begin{equation}\label{eq3.2}
f(x)\in\functrep_{0}(f)\cap L^{p_0}(\bbbr^n)\text{ for some $1\leq p_0\leq\infty$},
\end{equation}
and we will only provide a detailed proof for (\ref{eq3.1}) in the case of $l=n+1$, and the proof for (\ref{eq3.1}) in the case where $l=1,2,\cdots,n$ will be very much alike and hence is omitted. Since the function $f(x)\in L^{p_0}(\bbbr^n)$, and $\psi_{2^{-j}}(x)=2^{jn}\psi(2^j x)\in\Sw_{0}(\bbbr^n)$ for all $j\in\bbbz$, and since formula (\ref{eq1-3-2}) implies $\partial_{n+1}[P_t(x)]\in L^1(\bbbr^n)\cap L^{\infty}(\bbbr^n)$ for each fixed $t\in(0,\infty)$, Minkowski's integral inequality and formula (\ref{eq1-6}) imply the following functions in (\ref{eq3.37}), (\ref{eq3.133}), and (\ref{eq3.42}) belong to $L^{p_0}(\bbbr^n)$ for every $j\in\bbbz$ and $t\in(0,\infty)$,
\begin{equation}\label{eq3.37}
f_j(x):=f*\psi_{2^{-j}}(x),
\end{equation}
\begin{equation}\label{eq3.133}
\partial_{n+1}\Pint(f;x,t)\!=\!\partial_{n+1}[f\!*\!P_t(x)]
\!=\!f\!*\!\partial_{n+1}[P_t(\cdot)](x),
\end{equation}
\begin{align}
&\partial_{n+1}\Pint(f_j;x,t)\!=\!\partial_{n+1}[f_j\!*\!P_t(x)]
\!=\!f_j\!*\!\partial_{n+1}[P_t(\cdot)](x)\nonumber\\
&=(f\!*\!\psi_{2^{-j}})\!*\!\partial_{n+1}[P_t(\cdot)](x)
=(f\!*\!\partial_{n+1}[P_t(\cdot)])\!*\!\psi_{2^{-j}}(x),\label{eq3.42}
\end{align}
where $f$ appearing in (\ref{eq3.37}), (\ref{eq3.133}), and (\ref{eq3.42}) is the given function representative, and the second line of (\ref{eq3.42}) is due to an application of H\"{o}lder's inequality with $\frac{1}{p_0}+\frac{1}{p_0'}=1$ and Young's inequality that justifies the exchange of the order of integration. For each $j\in\bbbz$, Proposition \ref{proposition1} (iv) (\ref{eq1.119}) implies the convolution $f\!*\!\psi_{2^{-j}}$ of $f\in\Sw_{0}'(\bbbr^n)$ and $\psi_{2^{-j}}\in\Sw_{0}(\bbbr^n)$ is a tempered distribution in $\Sw'(\bbbr^n)$ and has the smooth function representative $(\ref{eq3.139})\in\functrep(f\!*\!\psi_{2^{-j}})\cap L^{p_0}(\bbbr^n)$,
\begin{align}
&x\!\in\!\bbbr^n\!\longmapsto<\!f,\psi_{2^{-j}}(x-\cdot)\!>\nonumber\\
&=\!\!\int_{\bbbr^n}\!\!\!f(y)\!\cdot\!\psi_{2^{-j}}(x-y)dy=f*\psi_{2^{-j}}(x),\label{eq3.139}
\end{align}
where the equations in (\ref{eq3.139}) are true because of the condition (\ref{eq3.2}) and the fact that the function (\ref{eq3.39}) belongs to $\Sw_{0}(\bbbr^n)$ for all $x\in\bbbr^n$ and $j\in\bbbz$,
\begin{equation}\label{eq3.39}
y\in\bbbr^n\longmapsto\psi_{2^{-j}}(x-y).
\end{equation}
Furthermore, we use Proposition \ref{proposition1} (iv) equation (\ref{eq1.122}) to deduce that
\begin{equation}\label{eq3.131}
f\!*\!\psi_{2^{-j}}\!=\!\iFT_n[\FT_n\psi(2^{-j}\xi)\FT_n f]\text{ in the sense of $\Sw'(\bbbr^n)$},
\end{equation}
where $\FT_n f$ appearing in (\ref{eq3.131}) is the distributional Fourier transform of the HB-extension of $f\in\Sw_{0}'(\bbbr^n)$. We notice that for every $t\in(0,\infty)$, the function
\begin{equation}\label{eq3.43}
\multi(\xi)=-2\pi|\xi|\cdot e^{-2\pi t|\xi|}\in L^1(\bbbr^n)
\end{equation}
is smooth on $\bbbr^n\setminus\{0\}$ and satisfies condition (\ref{eq1.64}), and 
\begin{equation}\label{eq3.55}
\partial_{n+1}[P_t(x)]=\FT_n\multi(x)=\iFT_n\multi(x), 
\end{equation}
then Proposition \ref{proposition1} (ii) tells us that the convolution $f*\partial_{n+1}[P_t(\cdot)]$ of $f\in\Fspq\subseteq\Sw_{0}'(\bbbr^n)$ and the function $\partial_{n+1}[P_t(x)]$ is a continuous linear functional in $\Sw_{0}'(\bbbr^n)$. The condition that $f\in\Fspq\subseteq\Sw_{0}'(\bbbr^n)$ has a function representative $f(x)\in\functrep_{0}(f)\cap L^{p_0}(\bbbr^n)$ implies Lemma \ref{lemma16} (\ref{eq2.238}) is applicable, thus formula (\ref{eq1-6}) implies the function
\begin{align}
&\partial_{n+1}\Pint(f;x,t)=f*\partial_{n+1}[P_t(\cdot)](x)\nonumber\\
&\in\functrep_{0}(f*\partial_{n+1}[P_t(\cdot)])\cap L^{p_0}(\bbbr^n),\label{eq3.41}
\end{align}
where $f$ appearing in the first line of (\ref{eq3.41}) is the function representative $f(x)\in\functrep_{0}(f)\cap L^{p_0}(\bbbr^n)$. By Lemma \ref{lemma8} (\ref{eq2.260}), the following equation (\ref{eq3.44}) is true in the sense of $\Sw_{0}'(\bbbr^n)$ for all $j\in\bbbz$ and $t\in(0,\infty)$,
\begin{equation}\label{eq3.44}
f_j*\partial_{n+1}[P_t(\cdot)]=(f*\partial_{n+1}[P_t(\cdot)])*\psi_{2^{-j}},
\end{equation}
where $f_j\!:=\!f\!*\!\psi_{2^{-j}}\!\in\!\Sw'(\bbbr^n)\!\subseteq\!\Sw_{0}'(\bbbr^n)$ has the smooth function representative 
\begin{equation}
(\ref{eq3.139})\!\in\!\functrep(f_j)\!\cap\!L^{p_0}(\bbbr^n)
\!\subseteq\!\functrep_0(f_j)\!\cap\!L^{p_0}(\bbbr^n).\label{eq3.45}
\end{equation}
Therefore we can apply Lemma \ref{lemma16} (\ref{eq2.238}) with $f$ being replaced by $f_j$ and deduce the function $\partial_{n+1}\Pint(f_j;x,t)$ as given in (\ref{eq3.42}) belongs to
\begin{align}
&\functrep_{0}(f_j*\partial_{n+1}[P_t(\cdot)])\cap L^{p_0}(\bbbr^n)\nonumber\\
&=\functrep_{0}((f*\partial_{n+1}[P_t(\cdot)])*\psi_{2^{-j}})\cap L^{p_0}(\bbbr^n).\label{eq3.46}
\end{align}
Indeed, a stronger conclusion can be deduced by using Proposition \ref{proposition1} (iv). Since $f*\partial_{n+1}[P_t(\cdot)]\in\Sw_{0}'(\bbbr^n)$ by Proposition \ref{proposition1} (ii) and $\psi_{2^{-j}}\in\Sw_{0}(\bbbr^n)$, Proposition \ref{proposition1} (iv), (\ref{eq3.41}), and (\ref{eq3.39}) together imply the convolution $(f\!*\!\partial_{n+1}[P_t(\cdot)])\!*\!\psi_{2^{-j}}\!\in\!\Sw'(\bbbr^n)$ has the smooth function representative $(\ref{eq3.47})\!\in\!\functrep((f*\partial_{n+1}[P_t(\cdot)])\!*\!\psi_{2^{-j}})
\!\cap\!L^{p_0}(\bbbr^n)$ for every $j\in\bbbz$ and $t\in(0,\infty)$,
\begin{align}
&x\in\bbbr^n\longmapsto<f*\partial_{n+1}[P_t(\cdot)],\psi_{2^{-j}}(x-\cdot)>\nonumber\\
&=\int_{\bbbr^n}f*\partial_{n+1}[P_t(\cdot)](y)\cdot\psi_{2^{-j}}(x-y)dy\nonumber\\
&=\partial_{n+1}\Pint(f_j;x,t),\label{eq3.47}
\end{align}
where $\partial_{n+1}\Pint(f_j;x,t)$ is the function as given in (\ref{eq3.42}), the smoothness of which is now proven. By the condition (\ref{eq1-7}) imposed on the function $\psi\in\Sw_{0}(\bbbr^n)$, the following functions (\ref{eq3.141}) and (\ref{eq3.48}) belong to $\Sw_{0}(\bbbr^n)$ for every $j\in\bbbz$,
\begin{align}
&x\in\bbbr^n\longmapsto\iFT_n[2^{-j}|\xi|\FT_n\psi(2^{-j}\xi)](x)\nonumber\\
&=2^{jn}\!\iFT_n[|\xi|\FT_n\psi(\xi)](2^j x)\!\in\!\Sw_0(\bbbr^n),\label{eq3.141}
\end{align}
\begin{align}
&x\in\bbbr^n\longmapsto\iFT_n[\frac{\FT_n\psi(2^{-j}\xi)}{2^{-j}|\xi|}](x)\nonumber\\
&=2^{jn}\!\iFT_n[\frac{\FT_n\psi(\xi)}{|\xi|}](2^j x)\!\in\!\Sw_0(\bbbr^n),\label{eq3.48}
\end{align}
and the following functions (\ref{eq3.52}) and (\ref{eq3.53}) belong to $\Sw_{0}(\bbbr^n)$ for all $x\in\bbbr^n$ and $j\in\bbbz$,
\begin{align}
&y\!\in\!\bbbr^n\!\longmapsto\!\iFT_n[2^{-j}|\xi|\FT_n\psi(2^{-j}\xi)](x\!-\!y)\!\in\!\Sw_{0}(\bbbr^n),
\label{eq3.52}\\
&y\!\in\!\bbbr^n\!\longmapsto\!\iFT_n[\frac{\FT_n\psi(2^{-j}\xi)}{2^{-j}|\xi|}](x\!-\!y)\!\in\!
\Sw_{0}(\bbbr^n).\label{eq3.53}
\end{align}
For every $j\in\bbbz$, we use the following notations for tempered distributions $f_j^{\num}$ and $f_j^{\den}$ defined below,
\begin{equation}\label{eq3.4}
f_j^{\num}:=f*\iFT_n[2^{-j}|\xi|\FT_n\psi(2^{-j}\xi)]
\end{equation}
where ``$\num$'' means the factor $2^{-j}|\xi|$ appears in the numerator, and
\begin{equation}\label{eq3.5}
f_j^{\den}:=f*\iFT_n[\frac{\FT_n\psi(2^{-j}\xi)}{2^{-j}|\xi|}]
\end{equation}
where ``$\den$'' means the factor $2^{-j}|\xi|$ appears in the denominator. Then Proposition \ref{proposition1} (iv), conditions (\ref{eq3.141}), (\ref{eq3.2}), and (\ref{eq3.52}) imply $f_j^{\num}\in\Sw'(\bbbr^n)$ has the smooth function representative $(\ref{eq3.38})\!\in\!\functrep(f_j^{\num})\!\cap\!L^{p_0}(\bbbr^n)$,
\begin{align}
&x\in\bbbr^n\longmapsto<f,\iFT_n[2^{-j}|\xi|\FT_n\psi(2^{-j}\xi)](x-\cdot)>\nonumber\\
&=\int_{\bbbr^n}f(y)\cdot\iFT_n[2^{-j}|\xi|\FT_n\psi(2^{-j}\xi)](x-y)dy\nonumber\\
&=f*\iFT_n[2^{-j}|\xi|\FT_n\psi(2^{-j}\xi)](x),\label{eq3.38}
\end{align}
where $f$ appearing in the last line of (\ref{eq3.38}) is the function representative given by (\ref{eq3.2}). And Proposition \ref{proposition1} (iv), conditions (\ref{eq3.48}), (\ref{eq3.2}), and (\ref{eq3.53}) imply $f_j^{\den}\in\Sw'(\bbbr^n)$ has the smooth function representative $(\ref{eq3.54})\!\in\!\functrep(f_j^{\den})\!\cap\!L^{p_0}(\bbbr^n)$,
\begin{align}
&x\in\bbbr^n\longmapsto<f,\iFT_n[\frac{\FT_n\psi(2^{-j}\xi)}{2^{-j}|\xi|}](x-\cdot)>\nonumber\\
&=\int_{\bbbr^n}f(y)\cdot\iFT_n[\frac{\FT_n\psi(2^{-j}\xi)}{2^{-j}|\xi|}](x-y)dy\nonumber\\
&=f*\iFT_n[\frac{\FT_n\psi(2^{-j}\xi)}{2^{-j}|\xi|}](x),\label{eq3.54}
\end{align}
where $f$ appearing in the last line of (\ref{eq3.54}) is the function representative given by (\ref{eq3.2}). Now we pick a function $\eta\in\Sw_0(\bbbr^n)$ satisfying
\begin{equation}\label{eq3.6}
spt.\FT_n\eta\subseteq\{\xi\in\bbbr^n:\frac{1}{4}\leq|\xi|<4\},
\end{equation}
and
\begin{equation}\label{eq3.7}
\FT_n\eta(\xi)=1\text{ on the annulus }\{\xi\in\bbbr^n:\frac{1}{2}\leq|\xi|<2\},
\end{equation}
then $\FT_n\eta=1$ on the support set of $\FT_n\psi$. We claim that
\begin{equation}\label{eq3.40}
f_j\!*\!\partial_{n+1}[P_t(\cdot)]\!=\!2^j f_j^{\num}\!*\!\iFT_n[-2\pi 
e^{-2\pi t|\xi|}\FT_n\eta(2^{-j}\xi)]
\end{equation}
in the sense of $\Sw_0'(\bbbr^n)$ for every $t\in(0,\infty)$ and $j\in\bbbz$, where $f_j\!=\!f\!*\!\psi_{2^{-j}}\!\in\!\Sw'(\bbbr^n)\!\subseteq\!\Sw_{0}'(\bbbr^n)$ by Proposition \ref{proposition1} (iv). Let $\phi\in\Sw_0(\bbbr^n)$ and recall the function $\multi(\xi)$ defined in (\ref{eq3.43}). Proposition \ref{proposition1} (ii) implies the convolution $f_j*\partial_{n+1}[P_t(\cdot)]\in\Sw_{0}'(\bbbr^n)$, and we use equations (\ref{eq1.112}) and (\ref{eq3.55}) to obtain
\begin{equation}\label{eq3.56}
<f_j\!*\!\partial_{n+1}[P_t(\cdot)],\phi>=<f\!*\!\psi_{2^{-j}},\partial_{n+1}[P_t(\cdot)]\!*\!\phi>.
\end{equation}
Since $\partial_{n+1}[P_t(\cdot)]\!*\!\phi\!\in\!\Sw_{0}(\bbbr^n)$ by equation (\ref{eq1.111}), we can further deduce from equation (\ref{eq1.122}) that
\begin{align}
&(\ref{eq3.56})\!=<\!f,\tilde{\psi}_{2^{-j}}\!*\!\partial_{n+1}[P_t(\cdot)]\!*\!\phi\!>\nonumber\\
&=<\!f,\iFT_n[\FT_n\psi(-2^{-j}\xi)\!\cdot\!\multi(\xi)\!\cdot\!\FT_n\phi(\xi)]\!>.\label{eq3.57}
\end{align}
The condition (\ref{eq3.6}) implies
\begin{equation}\label{eq3.58}
x\!\in\!\bbbr^n\!\longmapsto\!\iFT_n[-2\pi e^{-2\pi t|\xi|}\FT_n\eta(2^{-j}\xi)](x)
\!\in\!\Sw_{0}(\bbbr^n).
\end{equation}
Since $f_j^{\num}\in\Sw'(\bbbr^n)\subseteq\Sw_{0}'(\bbbr^n)$, we can use (\ref{eq1.122}), (\ref{eq3.58}), (\ref{eq3.4}), and (\ref{eq3.141}) to obtain the following equation for all $g\in\Sw(\bbbr^n)$,
\begin{align}
&<2^j f_j^{\num}\!*\!\iFT_n[-2\pi e^{-2\pi t|\xi|}\FT_n\eta(2^{-j}\xi)],g>\nonumber\\
&=<\!2^j f\!*\!\iFT_n[2^{-j}|\xi|\FT_n\psi(2^{-j}\xi)],
\FT_n[-2\pi e^{-2\pi t|\xi|}\FT_n\eta(2^{-j}\xi)]\!*\!g\!>\nonumber\\
&=<\!f,\iFT_n[|\xi|\FT_n\psi(-2^{-j}\xi)\!\cdot\!
-2\pi e^{-2\pi t|\xi|}\FT_n\eta(-2^{-j}\xi)\FT_n g(\xi)]\!>\nonumber\\
&=<\!f,\iFT_n[\FT_n\psi(-2^{-j}\xi)\!\cdot\!\multi(\xi)\!\cdot\!\FT_n g(\xi)]\!>,\label{eq3.59}
\end{align}
where the last equation of (\ref{eq3.59}) is true because the condition (\ref{eq3.7}) implies for all $j\in\bbbz$,
\begin{equation}\label{eq3.60}
\FT_n\eta(-2^{-j}\xi)=1\text{ whenever }\FT_n\psi(-2^{-j}\xi)\neq0,
\end{equation}
and $\multi(\xi)$ is the function defined in (\ref{eq3.43}). Combining (\ref{eq3.56}), (\ref{eq3.57}), and (\ref{eq3.59}) proves the claim (\ref{eq3.40}). The condition (\ref{eq3.6}) also implies
\begin{equation}\label{eq3.63}
x\!\in\!\bbbr^n\!\longmapsto\!\iFT_n[-2\pi|\xi|^2 e^{-2\pi t|\xi|}\FT_n\eta(2^{-j}\xi)](x)
\!\in\!\Sw_{0}(\bbbr^n).
\end{equation}
Since $f_j^{\den}\in\Sw'(\bbbr^n)\subseteq\Sw_{0}'(\bbbr^n)$, we can use (\ref{eq1.122}), (\ref{eq3.63}), (\ref{eq3.5}), and (\ref{eq3.48}) to obtain the following equation for all $g\in\Sw(\bbbr^n)$,
\begin{align}
&<2^{-j}f_j^{\den}\!*\!\iFT_n[-2\pi|\xi|^2 e^{-2\pi t|\xi|}\FT_n\eta(2^{-j}\xi)],g>\nonumber\\
&=<\!2^{-j}f\!*\!\iFT_n[\frac{\FT_n\psi(2^{-j}\xi)}{2^{-j}|\xi|}],
\FT_n[-2\pi|\xi|^2 e^{-2\pi t|\xi|}\FT_n\eta(2^{-j}\xi)]\!*\!g\!>\nonumber\\
&=<\!f,\iFT_n[\frac{\FT_n\psi(-2^{-j}\xi)}{|\xi|}\!\cdot\!
-2\pi|\xi|^2 e^{-2\pi t|\xi|}\FT_n\eta(-2^{-j}\xi)\!\cdot\!\FT_n g(\xi)]\!>\nonumber\\
&=<\!f,\iFT_n[\FT_n\psi(-2^{-j}\xi)\!\cdot\!\multi(\xi)\!\cdot\!\FT_n g(\xi)]\!>,\label{eq3.64}
\end{align}
where the last equation of (\ref{eq3.64}) is due to (\ref{eq3.60}). Combining (\ref{eq3.56}), (\ref{eq3.57}), and (\ref{eq3.64}) yields that the equation (\ref{eq3.51}) is true in the sense of $\Sw_0'(\bbbr^n)$ for every $t\!\in\!(0,\infty)$ and $j\!\in\!\bbbz$,
\begin{equation}\label{eq3.51}
f_j\!*\!\partial_{n+1}[P_t(\cdot)]\!=\!2^{-j}f_j^{\den}\!*\!\iFT_n[-2\pi|\xi|^2 
e^{-2\pi t|\xi|}\FT_n\eta(2^{-j}\xi)].
\end{equation}
By Lemma \ref{lemma8}, we have that the series in (\ref{eq2.261}) converges in the sense of $\Sw_0'(\bbbr^n)$, and we can deduce from (\ref{eq3.40}) and (\ref{eq3.51}) that for every $k\in\bbbz$ and $t\in(0,\infty)$, the following series in (\ref{eq3.123}) converges to $f*\partial_{n+1}[P_t(\cdot)]\in\Sw_{0}'(\bbbr^n)$ in the sense of $\Sw_0'(\bbbr^n)$,
\begin{align}
&f*\partial_{n+1}[P_t(\cdot)]\nonumber\\
&=\sum_{k\leq j}2^{-j}f_j^{\den}*\iFT_n[-2\pi|\xi|^2 e^{-2\pi t|\xi|}\FT_n\eta(2^{-j}\xi)]
\nonumber\\
&\quad+\sum_{j<k}2^{j}f_j^{\num}*\iFT_n[-2\pi e^{-2\pi t|\xi|}\FT_n\eta(2^{-j}\xi)].\label{eq3.123}
\end{align}
Since $f_j^{\num}\in\Sw'(\bbbr^n)$ has the smooth function representative $(\ref{eq3.38})\in\functrep(f_j^{\num})\cap L^{p_0}(\bbbr^n)$, we use Proposition \ref{proposition1} (iv) and condition (\ref{eq3.58}) to deduce the tempered distribution
\begin{equation}\label{eq3.65}
f_j^{\num}*\iFT_n[-2\pi e^{-2\pi t|\xi|}\FT_n\eta(2^{-j}\xi)]\in\Sw'(\bbbr^n)
\end{equation}
has the smooth function representative
\begin{equation}\label{eq3.66}
(\ref{eq3.67})\in\functrep(f_j^{\num}*\iFT_n[-2\pi e^{-2\pi t|\xi|}\FT_n\eta(2^{-j}\xi)])\cap L^{p_0}(\bbbr^n),
\end{equation}
where the smooth function (\ref{eq3.67}) is given by
\begin{align}
&x\in\bbbr^n\longmapsto<f_j^{\num},\iFT_n[-2\pi e^{-2\pi t|\xi|}\FT_n\eta(2^{-j}\xi)](x-\cdot)>
\nonumber\\
&=\int_{\bbbr^n}f_j^{\num}(y)\cdot\iFT_n[-2\pi e^{-2\pi t|\xi|}\FT_n\eta(2^{-j}\xi)](x-y)dy\nonumber\\
&=f_j^{\num}*\iFT_n[-2\pi e^{-2\pi t|\xi|}\FT_n\eta(2^{-j}\xi)](x),\label{eq3.67}
\end{align}
and we specify $f_j^{\num}$ appearing in the last line of (\ref{eq3.67}) is the function (\ref{eq3.38}). Since $f_j^{\den}\in\Sw'(\bbbr^n)$ has the smooth function representative $(\ref{eq3.54})\in\functrep(f_j^{\den})\cap L^{p_0}(\bbbr^n)$, we use Proposition \ref{proposition1} (iv) and condition (\ref{eq3.63}) to deduce the tempered distribution
\begin{equation}\label{eq3.68}
f_j^{\den}*\iFT_n[-2\pi|\xi|^2 e^{-2\pi t|\xi|}\FT_n\eta(2^{-j}\xi)]\in\Sw'(\bbbr^n)
\end{equation}
has the smooth function representative
\begin{equation}\label{eq3.69}
(\ref{eq3.74})\in\functrep(f_j^{\den}*\iFT_n[-2\pi|\xi|^2 e^{-2\pi t|\xi|}\FT_n\eta(2^{-j}\xi)])\cap L^{p_0}(\bbbr^n),
\end{equation}
where the smooth function (\ref{eq3.74}) is given by
\begin{align}
&x\in\bbbr^n\longmapsto<f_j^{\den},\iFT_n[-2\pi|\xi|^2 e^{-2\pi t|\xi|}\FT_n\eta(2^{-j}\xi)](x-\cdot)>
\nonumber\\
&=\int_{\bbbr^n}f_j^{\den}(y)\cdot
\iFT_n[-2\pi|\xi|^2 e^{-2\pi t|\xi|}\FT_n\eta(2^{-j}\xi)](x-y)dy\nonumber\\
&=f_j^{\den}*\iFT_n[-2\pi|\xi|^2 e^{-2\pi t|\xi|}\FT_n\eta(2^{-j}\xi)](x),\label{eq3.74}
\end{align}
and we specify $f_j^{\den}$ appearing in the last line of (\ref{eq3.74}) is the function (\ref{eq3.54}). To prove (\ref{eq3.1}) in the case where the partial derivative of the Poisson integral is with respect to $t$, we observe that equations (\ref{eq3.59}) and (\ref{eq3.64}) imply equation (\ref{eq3.127}) is true in the sense of $\Sw'(\bbbr^n)$,
\begin{align}
&2^j f_j^{\num}*\iFT_n[-2\pi e^{-2\pi t|\xi|}\FT_n\eta(2^{-j}\xi)]\nonumber\\
&=2^{-j}f_j^{\den}*\iFT_n[-2\pi|\xi|^2 e^{-2\pi t|\xi|}\FT_n\eta(2^{-j}\xi)].\label{eq3.127}
\end{align}
Since smooth functions are also in $\Lloc$, we deduce from (\ref{eq3.66}), (\ref{eq3.67}), (\ref{eq3.69}), (\ref{eq3.74}), (\ref{eq3.127}), and Lemma \ref{lemma15} (i) that equation (\ref{eq3.128}) is true for all $x\in\bbbr^n$, $t\in(0,\infty)$, and $j\in\bbbz$,
\begin{align}
&F_j(x,t):=2^j f_j^{\num}*\iFT_n[-2\pi e^{-2\pi t|\xi|}\FT_n\eta(2^{-j}\xi)](x)\nonumber\\
&=2^{-j}f_j^{\den}*\iFT_n[-2\pi|\xi|^2 e^{-2\pi t|\xi|}\FT_n\eta(2^{-j}\xi)](x).\label{eq3.128}	
\end{align}
Indeed, Lemma \ref{lemma15} (i) implies the second equation of (\ref{eq3.128}) holds true for almost every $x\in\bbbr^n$, and the smoothness of relevant functions implies the two functions in (\ref{eq3.128}) are equal for all $x\in\bbbr^n$, and we denote this common smooth function by $F_j(x,t)$. We have
\begin{align}
&F_j(x,t)\!\in\!\functrep(2^j f_j^{\num}*\iFT_n[-2\pi e^{-2\pi t|\xi|}\FT_n\eta(2^{-j}\xi)])
\cap L^{p_0}(\bbbr^n)\nonumber\\
&=\functrep(2^{-j}f_j^{\den}*\iFT_n[-2\pi|\xi|^2 e^{-2\pi t|\xi|}\FT_n\eta(2^{-j}\xi)])
\cap L^{p_0}(\bbbr^n).\label{eq3.132}
\end{align}
Given $k\!\in\!\bbbz$, we consider the series (\ref{eq3.124}) of nonnegative real numbers,
\begin{align}
&\sum_{j\in\bbbz}\!|F_j(x,t)|\!\!=\!\!\sum_{k\leq j}\!2^{-j}\big|f_j^{\den}\!*\!
\iFT_n[-2\pi|\xi|^2 e^{-2\pi t|\xi|}\FT_n\eta(2^{-j}\xi)](x)\big|\nonumber\\
&+\!\!\sum_{j<k}\!2^j\big|f_j^{\num}\!*\!\iFT_n[-2\pi e^{-2\pi t|\xi|}\FT_n\eta(2^{-j}\xi)](x)\big|.
\label{eq3.124}
\end{align}
By (\ref{eq3.128}), we have that the value of the series (\ref{eq3.124}) does not depend on $k\in\bbbz$ whenever it converges. And we begin with estimating the following term,
\begin{align}
&\int_0^{\infty}t^{q-sq-1}\bigg\{\sum_{j\in\bbbz}|F_j(x,t)|\bigg\}^q dt\nonumber\\
&\lesssim\sum_{k\in\bbbz}2^{-k(q-sq-1)}\int_{2^{-k-1}}^{2^{-k}}
\big\{\text{the series (\ref{eq3.124})}\big\}^q dt\nonumber\\
&\lesssim\sum_{k\in\bbbz}2^{-k(q-sq)}\esssup_{2^{-k-1}\leq t\leq 2^{-k}}
\big\{\text{the series (\ref{eq3.124})}\big\}^q.\label{eq3.125}
\end{align}
Furthermore, the expression (\ref{eq3.125}), into which we insert the series (\ref{eq3.124}), can be estimated from above by a constant multiple of the sum of the following two terms,
\begin{equation}\label{eq3.10}
\sum_{k\in\bbbz}2^{ksq}\!\!\!\!\esssup_{2^{-k-1}\leq t\leq 2^{-k}}\!\!\big\{
\sum_{k>j}2^{j-k}|f_j^{\num}\!*\!\iFT_n[e^{-2\pi t|\xi|}\FT_n\eta(2^{-j}\xi)](x)|\big\}^q
\end{equation}
and
\begin{equation}\label{eq3.11}
\sum_{k\in\bbbz}2^{ksq}\!\!\!\!\esssup_{2^{-k-1}\leq t\leq 2^{-k}}\!\!\big\{
\sum_{k\leq j}2^{-k-j}|f_j^{\den}\!*\!\iFT_n[|\xi|^2 e^{-2\pi t|\xi|}\FT_n\eta(2^{-j}\xi)](x)|\big\}^q,
\end{equation}
and the constant only depends on $q$.\\

Step 2: We estimate the term (\ref{eq3.10}). When $k>j$ and $2^{-k-1}\leq t\leq 2^{-k}$, we recall \cite[Definition 2.2]{Wang2023} of the Peetre-Fefferman-Stein maximal function and Remark \ref{remark4} for its rephrased version in the language of the theory of functional and function representatives. The tempered distribution $f_j^{\num}\in\Sw'(\bbbr^n)$, as defined in (\ref{eq3.4}), has the smooth function representative $(\ref{eq3.38})\!\in\!\functrep(f_j^{\num})\!\cap\!L^{p_0}(\bbbr^n)$. Furthermore, when $g\in\Sw(\bbbr^n)$ has 
$$spt.g\subseteq\{\xi\in\bbbr^n:|\xi|<2^{j-1}\text{ or }|\xi|\geq2^{j+1}\},$$
we use \cite[Definition 2.3.7]{14classical} and equation (\ref{eq1.122}) to deduce that
\begin{align}
&<\!\FT_n f_j^{\num},g\!>=<\!f_j^{\num},\FT_n g\!>\nonumber\\
&=<\!f*\iFT_n[2^{-j}|\xi|\FT_n\psi(2^{-j}\xi)],\FT_n g\!>\nonumber\\
&=<\!f,\FT_n[2^{-j}|\xi|\FT_n\psi(2^{-j}\xi)]*\FT_n g\!>\nonumber\\
&=<\!f,\FT_n[2^{-j}|\xi|\FT_n\psi(2^{-j}\xi)\!\cdot\!g(\xi)]\!>
=<\!f,\FT_n[0]\!>=0,\label{eq3.76}
\end{align}
thus the distributional Fourier transform $\FT_n f_j^{\num}$ is compactly supported in the annulus 
$$\{\xi\in\bbbr^n:2^{j-1}\leq|\xi|<2^{j+1}\}.$$
Therefore the Peetre-Fefferman-Stein maximal function $\PFSmax_n f_j^{\num}(x)$ is well-defined for every $j\in\bbbz$ by the following expression,
\begin{equation}\label{eq3.103}
\PFSmax_n f_j^{\num}(x)=\esssup_{z\in\bbbr^n}
\frac{|f_j^{\num}(x-z)|}{(1+2^{j+1}|z|)^{\frac{n}{r}}},
\end{equation}
where in the proof of this theorem, $r$ is a positive real number chosen to satisfy $0<r<\min\{p,q\}$. The function representative (\ref{eq3.38}) is smooth and satisfies the least additional continuity condition as illustrated in Remark \ref{remark4}, thus all the conclusions of \cite[Remark 2.4, Remark 2.5, Lemma 2.6, Remark 2.7, Remark 2.8, Lemma 2.11, and Remark 2.13]{Wang2023} are true for $\PFSmax_n f_j^{\num}(x)$. Then by (\ref{eq3.65}), (\ref{eq3.66}), and (\ref{eq3.67}), we have
\begin{align}
&|f_j^{\num}*\iFT_n[e^{-2\pi t|\xi|}\FT_n\eta(2^{-j}\xi)](x)|\nonumber\\
&\leq\int_{\bbbr^n}|\iFT_n[e^{-2\pi t|\xi|}\FT_n\eta(2^{-j}\xi)](z)|\cdot|f_j^{\num}(x-z)|dz\nonumber\\
&\leq\!\!\!\int_{\bbbr^n}\!\!\!\!\!|\iFT_n[e^{-2\pi t|\xi|}\FT_n\eta(\frac{t\xi}{2^j t})](z)|\!\cdot\!
(1\!+\!2^{j+1}|z|)^{\frac{n}{r}}dz\!\cdot\!\PFSmax_n f_j^{\num}(x)\nonumber\\
&=\!\!\!\int_{\bbbr^n}\!\!\!\!\!|\iFT_n[e^{-2\pi|\xi|}\FT_n\eta(\frac{\xi}{2^j t})](\frac{z}{t})|
\!\cdot\!(1\!+\!2^{j+1}t\!\cdot\!|\frac{z}{t}|)^{\frac{n}{r}}\frac{dz}{t^n}\!\cdot\!
\PFSmax_n f_j^{\num}(x)\nonumber\\
&=\!\!\!\int_{\bbbr^n}\!\!\!\!\!|\iFT_n[e^{-2\pi|\xi|}\FT_n\eta(\frac{\xi}{2^j t})](z)|\!\cdot\!
(1\!+\!2^{j+1}t|z|)^{\frac{n}{r}}dz\!\cdot\!\PFSmax_n f_j^{\num}(x),\label{eq3.12}
\end{align}
where we also use the formula $$\iFT_n[e^{-2\pi t|\xi|}\FT_n\eta(\frac{t\xi}{2^j t})](z)=t^{-n}\cdot\iFT_n[e^{-2\pi|\xi|}\FT_n\eta(\frac{\xi}{2^j t})](\frac{z}{t})$$ and an appropriate change of variable. We can split the integral in the last line of (\ref{eq3.12}) into two parts, one is an integral over the domain $\{z\in\bbbr^n:2^j t|z|\leq 1\}$ and the other is an integral over the domain $\{z\in\bbbr^n:2^j t|z|\geq 1\}$. For the first part, we estimate as follows
		\begin{align}
			&\int_{2^j t|z|\leq 1}|\iFT_n[e^{-2\pi|\xi|}\FT_n\eta(\frac{\xi}{2^j t})](z)|\cdot(1+2^{j+1}t|z|)^{\frac{n}{r}}dz\nonumber\\
			&\lesssim\int_{2^j t|z|\leq 1}\int_{\bbbr^n}e^{-2\pi|\xi|}\cdot|\FT_n\eta(\frac{\xi}{2^j t})|d\xi dz
			\nonumber\\
			&\lesssim 2^{-jn}t^{-n}\cdot\int_{2^{j-2}t\leq|\xi|<2^{j+2}t}e^{-2\pi|\xi|}d\xi
			\cdot\|\FT_n\eta\|_{L^{\infty}(\bbbr^n)}\nonumber\\
			&\lesssim e^{-\frac{\pi}{2}\cdot 2^j t}\lesssim 1,\label{eq3.13}
		\end{align}
		where the implicit constants in (\ref{eq3.13}) depend on $n,r,\eta$ and we also use the support condition of $\FT_n\eta$ given in (\ref{eq3.6}). For the second part, when $k>j$, $2^{-k-1}\leq t\leq 2^{-k}$, and $2^j t|z|\geq 1$, we can assume without loss of generality that $z=(z_1,z_2,\cdots,z_n)=(z_1,z')$, $z'=(z_2,\cdots,z_n)$, and $|z_1|=\max\{|z_1|,|z_2|,\cdots,|z_n|\}$. Therefore we have $|z_1|\cdot\sqrt{n}\geq|z|>0$ and we can use integration by parts with respect to $\xi_1$ repeatedly to obtain the following
		\begin{align}
			&|\iFT_n[e^{-2\pi|\xi|}\FT_n\eta(\frac{\xi}{2^j t})](z)|\nonumber\\
			&=|\int_{\bbbr^{n-1}}\int_{-\infty}^{\infty}\frac{\partial^M}{\partial\xi_1^M}\big(
			e^{-2\pi|\xi|}\FT_n\eta(\frac{\xi}{2^j t})\big)\cdot\frac{e^{2\pi i\xi_1 z_1}\cdot e^{2\pi i\xi'\cdot z'}}{(-2\pi iz_1)^M}d\xi_1 d\xi'|\nonumber\\
			&\lesssim|z|^{-M}\int_{\bbbr^n}\big|\frac{\partial^M}{\partial\xi_1^M}\big(
			e^{-2\pi|\xi|}\FT_n\eta(\frac{\xi}{2^j t})\big)\big|d\xi,\label{eq3.14}
		\end{align}
		where $M$ is a sufficiently large positive integer and the constant in (\ref{eq3.14}) depends on $n,M$. Combining some direct and trivial calculations with the fact that the support condition of $\FT_n\eta$ given in (\ref{eq3.6}) restricts the range of $\xi$ in the integral in the last line of (\ref{eq3.14}) to a domain where $|\xi|\sim 2^j t$, we can obtain
		\begin{equation}\label{eq3.15}
			(\ref{eq3.14})\lesssim\frac{e^{-\frac{\pi}{2}\cdot 2^j t}}{|z|^M}\cdot\sum_{l=0}^M (2^j t)^{n-l},
		\end{equation}
		where the constant in (\ref{eq3.15}) depends on $n,M,\eta$. Combining (\ref{eq3.14}) and (\ref{eq3.15}) and inserting the result into the expression below yield
		\begin{align}
			&\int_{2^j t|z|\geq 1}|\iFT_n[e^{-2\pi|\xi|}\FT_n\eta(\frac{\xi}{2^j t})](z)|\cdot(1+2^{j+1}t|z|)^{\frac{n}{r}}dz\nonumber\\
			&\lesssim\int_{2^j t|z|\geq 1}\frac{(1+2^{j+1}t|z|)^{\frac{n}{r}}}{|z|^M}dz\cdot e^{-\frac{\pi}{2}\cdot 2^j t}
			\cdot\sum_{l=0}^M (2^j t)^{n-l}\nonumber\\
			&\lesssim\int_{2^j t|z|\geq 1}\frac{(2^{j}t|z|)^{\frac{n}{r}}}{(2^j t|z|)^M}dz\cdot e^{-\frac{\pi}{2}\cdot 2^j t}\cdot\sum_{l=0}^M (2^j t)^{n+l}\nonumber\\
			&\lesssim e^{-\frac{\pi}{4}\cdot 2^{j-k}}\cdot\sum_{l=0}^M 2^{(j-k)l}
			\lesssim(M+1)e^{-\frac{\pi}{4}\cdot 2^{j-k}}\lesssim 1,\label{eq3.16}
		\end{align}
		where we pick $M$ so that $\frac{n}{r}+n<M$ and thus the integral $\int_{|z|\geq 1}|z|^{\frac{n}{r}-M}dz$ converges and we also use the conditions that $k>j$ and $2^{-k-1}\leq t\leq 2^{-k}$. The implicit constants in (\ref{eq3.16}) depend on $n,r,M,\eta$ and are independent of $k,j,t$. Combining estimates (\ref{eq3.13}) and (\ref{eq3.16}), we can conclude that the integral in the last line of (\ref{eq3.12}) can be bounded from above by a constant solely depending on the parameters $n,r,M,\eta$, therefore we have obtained the following pointwise estimate
\begin{equation}\label{eq3.17}
|f_j^{\num}*\iFT_n[e^{-2\pi t|\xi|}\FT_n\eta(2^{-j}\xi)](x)|\lesssim\PFSmax_n f_j^{\num}(x),
\end{equation}
		for $k>j$ and $2^{-k-1}\leq t\leq 2^{-k}$. Inserting (\ref{eq3.17}) into (\ref{eq3.10}) yields
		\begin{align}
			&(\ref{eq3.10})\lesssim\sum_{k\in\bbbz}2^{ksq}\big\{\sum_{k>j}2^{j-k}\PFSmax_n f_j^{\num}(x)\big\}^q\nonumber\\
			&=\sum_{k\in\bbbz}\big(\sum_{l<0}2^l\cdot 2^{ks}\PFSmax_n f_{k+l}^{\num}(x)\big)^q\nonumber\\
			&=\sum_{k\in\bbbz}\big(\sum_{l<0}2^{l(1-s)}\cdot 2^{(k+l)s}\PFSmax_n f_{k+l}^{\num}(x)\big)^q.\label{eq3.18}
		\end{align}
		If $0<q<1$, then we can have the following inequality directly and exchange the order of summations to obtain
		\begin{align}
			&(\ref{eq3.18})\leq\sum_{k\in\bbbz}\sum_{l<0}2^{l(1-s)q}\cdot 2^{(k+l)sq}\PFSmax_n f_{k+l}^{\num}(x)^q\nonumber\\
			&=\sum_{l<0}2^{l(1-s)q}\sum_{k\in\bbbz}2^{(k+l)sq}\PFSmax_n f_{k+l}^{\num}(x)^q
			\lesssim\sum_{k\in\bbbz}2^{ksq}\PFSmax_n f_{k}^{\num}(x)^q.\label{eq3.19}
		\end{align}
		If $1\leq q<\infty$, then we use the Minkowski's inequality for $\|\cdot\|_{l^q}$-norm with respect to $k\in\bbbz$ to obtain
		\begin{align}
			&(\ref{eq3.18})
			=\|\sum_{l<0}2^{l(1-s)}\cdot\{2^{(k+l)s}\PFSmax_n f_{k+l}^{\num}(x)\}_{k\in\bbbz}\|_{l^q}^q\nonumber\\
			&\leq\big(\sum_{l<0}2^{l(1-s)}\cdot(\sum_{k\in\bbbz}2^{ksq}\PFSmax_n f_{k}^{\num}(x)^q)^{\frac{1}{q}}\big)^q
			\lesssim\sum_{k\in\bbbz}2^{ksq}\PFSmax_n f_{k}^{\num}(x)^q.\label{eq3.20}
		\end{align}
		In (\ref{eq3.19}) and (\ref{eq3.20}), we use the condition $s<1$ and the constants depend on $s,q$. Combining (\ref{eq3.18}), (\ref{eq3.19}), and (\ref{eq3.20}), we have obtained
		\begin{equation}\label{eq3.21}
			(\ref{eq3.10})\lesssim\sum_{k\in\bbbz}2^{ksq}\PFSmax_n f_{k}^{\num}(x)^q
			\quad\text{for $s<1$ and $0<q<\infty$},
		\end{equation}
		and the constant only depends on $n,r,M,\eta,s,q$.\\
		
Step 3: We estimate the term (\ref{eq3.11}). When $k\leq j$ and $2^{-k-1}\leq t\leq 2^{-k}$, we recall \cite[Definition 2.2]{Wang2023} of the Peetre-Fefferman-Stein maximal function and Remark \ref{remark4} for its rephrased version in the language of the theory of functional and function representatives. The tempered distribution $f_j^{\den}\in\Sw'(\bbbr^n)$, as defined in (\ref{eq3.5}), has the smooth function representative $(\ref{eq3.54})\!\in\!\functrep(f_j^{\den})\!\cap\!L^{p_0}(\bbbr^n)$. Furthermore, when $g\in\Sw(\bbbr^n)$ has 
$$spt.g\subseteq\{\xi\in\bbbr^n:|\xi|<2^{j-1}\text{ or }|\xi|\geq2^{j+1}\},$$
we use formula (\ref{eq1.126}) and \cite[Definition 2.3.15]{14classical} to deduce that
\begin{align}
&<\!\FT_n f_j^{\den},g\!>=<\!\FT_n f\!\cdot\!\frac{\FT_n\psi(2^{-j}\xi)}{2^{-j}|\xi|},g\!>\nonumber\\
&=<\!\FT_n f,\frac{\FT_n\psi(2^{-j}\xi)}{2^{-j}|\xi|}\!\cdot\!g\!>
=<\!\FT_n f,0\!>=0,\label{eq3.104}
\end{align}
thus the distributional Fourier transform $\FT_n f_j^{\den}$ is compactly supported in the annulus 
$$\{\xi\in\bbbr^n:2^{j-1}\leq|\xi|<2^{j+1}\}.$$
Therefore the Peetre-Fefferman-Stein maximal function $\PFSmax_n f_j^{\den}(x)$ is well-defined for every $j\in\bbbz$ by the following expression,
\begin{equation}\label{eq3.106}
\PFSmax_n f_j^{\den}(x)=\esssup_{z\in\bbbr^n}
\frac{|f_j^{\den}(x-z)|}{(1+2^{j+1}|z|)^{\frac{n}{r}}},
\end{equation}
where in the proof of this theorem, $r$ is a positive real number chosen to satisfy $0<r<\min\{p,q\}$. The function representative (\ref{eq3.54}) is smooth and satisfies the least additional continuity condition as illustrated in Remark \ref{remark4}, thus all the conclusions of \cite[Remark 2.4, Remark 2.5, Lemma 2.6, Remark 2.7, Remark 2.8, Lemma 2.11, and Remark 2.13]{Wang2023} are true for $\PFSmax_n f_j^{\den}(x)$. Then by (\ref{eq3.68}), (\ref{eq3.69}), and (\ref{eq3.74}), we have
\begin{align}
&|f_j^{\den}*\iFT_n[|\xi|^2 e^{-2\pi t|\xi|}\FT_n\eta(2^{-j}\xi)](x)|\nonumber\\
&\leq\!\!\!\int_{\bbbr^n}\!\!\!\!\!|\iFT_n[|\xi|^2 e^{-2\pi t|\xi|}\FT_n\eta(2^{-j}\xi)](z)|
\!\cdot\!|f_j^{\den}(x\!-\!z)|dz\nonumber\\
&\leq\!\!\!\int_{\bbbr^n}\!\!\!\!\!|\iFT_n[|\xi|^2 e^{-2\pi t|\xi|}\FT_n\eta(2^{-j}\xi)](z)|
(1\!+\!2^{j+1}|z|)^{\frac{n}{r}}dz\!\cdot\!\PFSmax_n f_j^{\den}(x)\nonumber\\
&=\!\!\!\int_{\bbbr^n}\!\!\!\!\!t^{-2}|\iFT_n[|t\xi|^2 e^{-2\pi t|\xi|}
\FT_n\eta(\frac{t\xi}{2^j t})](z)|(1\!+\!2^{j+1}|z|)^{\frac{n}{r}}dz\!\cdot\!
\PFSmax_n f_j^{\den}(x)\nonumber\\
&=\!\!\!\int_{\bbbr^n}\!\!\!\!\!t^{-2}|\iFT_n[|\xi|^2 e^{-2\pi|\xi|}
\FT_n\eta(\frac{\xi}{2^j t})](\frac{z}{t})|(1\!+\!2^{j+1}t\!\cdot\!|\frac{z}{t}|)^{\frac{n}{r}}
\frac{dz}{t^n}\!\cdot\!\PFSmax_n f_j^{\den}(x)\nonumber\\
&=\!\!\!\int_{\bbbr^n}\!\!\!\!\!|\iFT_n[|\xi|^2 e^{-2\pi|\xi|}\!\FT_n\eta(\frac{\xi}{2^j t})](z)|
(1\!+\!2^{j+1}t|z|)^{\frac{n}{r}}dz\!\cdot\!t^{-2}\PFSmax_n f_j^{\den}(x),\label{eq3.22}
\end{align}
where we also use the formula $$\iFT_n[|t\xi|^2 e^{-2\pi t|\xi|}\FT_n\eta(\frac{t\xi}{2^j t})](z)\!=\!t^{-n}\!\cdot\!\iFT_n[|\xi|^2 e^{-2\pi|\xi|}\FT_n\eta(\frac{\xi}{2^j t})](\frac{z}{t})$$ and an appropriate change of variable. We can split the integral in the last line of (\ref{eq3.22}) into two parts, one is an integral over the domain $\{z\in\bbbr^n:2^j t|z|\leq 1\}$ and the other is an integral over the domain $\{z\in\bbbr^n:2^j t|z|\geq 1\}$. For the first part, we estimate as follows
		\begin{align}
			&\int_{2^j t|z|\leq 1}|\iFT_n[|\xi|^2 e^{-2\pi|\xi|}\FT_n\eta(\frac{\xi}{2^j t})](z)|\cdot(1+2^{j+1}t|z|)^{\frac{n}{r}}dz\nonumber\\
			&\lesssim\int_{2^j t|z|\leq 1}\int_{\bbbr^n}|\xi|^2 e^{-2\pi|\xi|}\cdot|\FT_n\eta(\frac{\xi}{2^j t})|d\xi dz
			\nonumber\\
			&\lesssim 2^{-jn}t^{-n}\cdot\int_{2^{j-2}t\leq|\xi|<2^{j+2}t}|\xi|^2 e^{-2\pi|\xi|}d\xi
			\cdot\|\FT_n\eta\|_{L^{\infty}(\bbbr^n)}\nonumber\\
			&\lesssim 2^{2(j-k)} e^{-\frac{\pi}{4}\cdot 2^{j-k}}\lesssim 1,\label{eq3.23}
		\end{align}
		where the implicit constants in (\ref{eq3.23}) depend on $n,r,\eta$ and we also use the support condition of $\FT_n\eta$ given in (\ref{eq3.6}). The last inequality of (\ref{eq3.23}) is due to the fact that $\lim_{x\rightarrow+\infty}\frac{2^{2x}}{e^{\frac{\pi}{4}\cdot 2^x}}=0$, which can be easily verified by repeatedly applying L'Hospital's rule. For the second part, when $k\leq j$, $2^{-k-1}\leq t\leq 2^{-k}$, and $2^j t|z|\geq 1$, we can assume without loss of generality that $z=(z_1,z_2,\cdots,z_n)=(z_1,z')$, $z'=(z_2,\cdots,z_n)$, and $|z_1|=\max\{|z_1|,|z_2|,\cdots,|z_n|\}$. Therefore we have $|z_1|\cdot\sqrt{n}\geq|z|>0$ and we can use integration by parts with respect to $\xi_1$ repeatedly to obtain the following
		\begin{align}
			&|\iFT_n[|\xi|^2 e^{-2\pi|\xi|}\FT_n\eta(\frac{\xi}{2^j t})](z)|\nonumber\\
			&=|\int_{\bbbr^{n-1}}\int_{-\infty}^{\infty}\frac{\partial^M}{\partial\xi_1^M}\big(|\xi|^2
			e^{-2\pi|\xi|}\FT_n\eta(\frac{\xi}{2^j t})\big)\cdot\frac{e^{2\pi i\xi_1 z_1}\cdot e^{2\pi i\xi'\cdot z'}}{(-2\pi iz_1)^M}d\xi_1 d\xi'|\nonumber\\
			&\lesssim|z|^{-M}\int_{\bbbr^n}\big|\frac{\partial^M}{\partial\xi_1^M}\big(|\xi|^2
			e^{-2\pi|\xi|}\FT_n\eta(\frac{\xi}{2^j t})\big)\big|d\xi,\label{eq3.24}
		\end{align}
		where $M$ is a sufficiently large positive integer satisfying the condition $\frac{n}{r}+n<M$ and the constant in (\ref{eq3.24}) depends on $n,M$. Combining some direct and trivial calculations with the fact that the support condition of $\FT_n\eta$ given in (\ref{eq3.6}) restricts the range of $\xi$ in the integral in the last line of (\ref{eq3.24}) to a domain where $|\xi|\sim 2^j t$, we can obtain
		\begin{equation}\label{eq3.25}
			(\ref{eq3.24})\lesssim\frac{e^{-\frac{\pi}{2}\cdot 2^j t}}{|z|^M}\cdot\sum_{l=0}^M (2^j t)^{n+2-l},
		\end{equation}
		where the constant in (\ref{eq3.25}) depends on $n,M,\eta$. Combining (\ref{eq3.24}) and (\ref{eq3.25}) and inserting the result into the expression below yield
		\begin{align}
			&\int_{2^j t|z|\geq 1}|\iFT_n[|\xi|^2 e^{-2\pi|\xi|}\FT_n\eta(\frac{\xi}{2^j t})](z)| \cdot(1+2^{j+1}t|z|)^{\frac{n}{r}}dz\nonumber\\
			&\lesssim\int_{2^j t|z|\geq 1}\frac{(1+2^{j+1}t|z|)^{\frac{n}{r}}}{|z|^M}dz\cdot e^{-\frac{\pi}{2}\cdot 2^j t}
			\cdot\sum_{l=0}^M (2^j t)^{n+2-l}\nonumber\\
			&\lesssim\int_{2^j t|z|\geq 1}\frac{(2^{j}t|z|)^{\frac{n}{r}}}{(2^j t|z|)^M}dz\cdot e^{-\frac{\pi}{2}\cdot 2^j t}\cdot\sum_{l=0}^M (2^j t)^{n+2+l}\nonumber\\
			&\lesssim(M+1)2^{(j-k)(M+2)}\cdot e^{-\frac{\pi}{4}\cdot 2^{j-k}}\lesssim 1,\label{eq3.26}
		\end{align}
where we use the conditions that $k\leq j$, $2^{-k-1}\leq t\leq 2^{-k}$, and the fact that the integral $\int_{|z|\geq 1}|z|^{\frac{n}{r}-M}dz$ converges since $\frac{n}{r}+n<M$, and the implicit constants in (\ref{eq3.26}) depend on $n,r,M,\eta$. The last inequality of (\ref{eq3.26}) is due to the fact that $\lim_{x\rightarrow+\infty}\frac{2^{(M+2)x}}{e^{\frac{\pi}{4}\cdot 2^x}}=0$, which can be verified by repeatedly applying L'Hospital's rule. Combining (\ref{eq3.23}) and (\ref{eq3.26}), we can conclude that the integral in the last line of (\ref{eq3.22}) can be bounded from above by a constant whose value only depends on the parameters $n,r,M,\eta$ and in particular, is independent of $k,j,t$. Therefore we have obtained the following pointwise estimate
\begin{equation}\label{eq3.27}
|f_j^{\den}\!*\!\iFT_n\![|\xi|^2 e^{-2\pi t|\xi|}\FT_n\eta(2^{-j}\xi)](x)|\!\lesssim\!
t^{-2}\PFSmax_n f_j^{\den}(x)\!\lesssim\!2^{2k}\PFSmax_n f_j^{\den}(x),
\end{equation}
		for $k\leq j$ and $2^{-k-1}\leq t\leq 2^{-k}$. Inserting (\ref{eq3.27}) into (\ref{eq3.11}) yields
		\begin{align}
			&(\ref{eq3.11})
			\lesssim\sum_{k\in\bbbz}2^{ksq}\big\{\sum_{k\leq j}2^{k-j}\PFSmax_n f_j^{\den}(x)\big\}^q\nonumber\\
			&=\sum_{k\in\bbbz}\big(\sum_{l\geq 0}2^{-l}\cdot 2^{ks}\PFSmax_n f_{k+l}^{\den}(x)\big)^q\nonumber\\
			&=\sum_{k\in\bbbz}\big(\sum_{l\geq 0}2^{-l(1+s)}\cdot 2^{(k+l)s}
			\PFSmax_n f_{k+l}^{\den}(x)\big)^q.\label{eq3.28}
		\end{align}
		If $0<q<1$, then we can have the following inequality directly and exchange the order of summations to obtain
		\begin{align}
			&(\ref{eq3.28})
			\leq\sum_{k\in\bbbz}\sum_{l\geq 0}2^{-l(1+s)q}\cdot 2^{(k+l)sq}\PFSmax_n f_{k+l}^{\den}(x)^q\nonumber\\
			&=\sum_{l\geq 0}2^{-l(1+s)q}\sum_{k\in\bbbz}2^{(k+l)sq}\PFSmax_n f_{k+l}^{\den}(x)^q
			\lesssim\sum_{k\in\bbbz}2^{ksq}\PFSmax_n f_{k}^{\den}(x)^q.\label{eq3.29}
		\end{align}
		If $1\leq q<\infty$, then we use the Minkowski's inequality for $\|\cdot\|_{l^q}$-norm with respect to $k\in\bbbz$ to obtain
		\begin{align}
			&(\ref{eq3.28})
			=\|\sum_{l\geq 0}2^{-l(1+s)}\cdot\{2^{(k+l)s}\PFSmax_n f_{k+l}^{\den}(x)\}_{k\in\bbbz}\|_{l^q}^q\nonumber\\
			&\leq\!\big(\sum_{l\geq 0}2^{-l(1+s)}\!\cdot\!
			(\sum_{k\in\bbbz}2^{ksq}\PFSmax_n f_{k}^{\den}(x)^q)^{\frac{1}{q}}\big)^q\!
			\lesssim\!\sum_{k\in\bbbz}2^{ksq}\PFSmax_n f_{k}^{\den}(x)^q.\label{eq3.30}
		\end{align}
		In (\ref{eq3.29}) and (\ref{eq3.30}), we use the condition $-1<s$ and the constants there depend on $s,q$. Combining (\ref{eq3.28}), (\ref{eq3.29}), and (\ref{eq3.30}), we have obtained
		\begin{equation}\label{eq3.31}
			(\ref{eq3.11})\lesssim\sum_{k\in\bbbz}2^{ksq}\PFSmax_n f_{k}^{\den}(x)^q
			\quad\text{for $-1<s$ and $0<q<\infty$},
		\end{equation}
		and the constant only depends on $n,r,M,\eta,s,q$.\\

Step 4: Now we combine the expression (\ref{eq3.125}), into which we insert the series (\ref{eq3.124}), with the estimates (\ref{eq3.10}), (\ref{eq3.11}), (\ref{eq3.21}), (\ref{eq3.31}) all together, raise both sides of the resulting inequality to the power of $\frac{1}{q}$, and put the subsequent result into $\|\cdot\|_{L^p(\bbbr^n)}$-quasinorm, then we can obtain the estimate below
\begin{align}
&\|\big[\text{the expression (\ref{eq3.125})}\big]^{\frac{1}{q}}\|_{L^p(\bbbr^n)}\nonumber\\
&\lesssim\!\|(\sum_{k\in\bbbz}2^{ksq}\PFSmax_n f_{k}^{\num}(x)^q)^{\frac{1}{q}}\|_{L^p(\bbbr^n)}\!+\!
\|(\sum_{k\in\bbbz}2^{ksq}\PFSmax_n f_{k}^{\den}(x)^q)^{\frac{1}{q}}\|_{L^p(\bbbr^n)}\nonumber\\
&\lesssim\|\{2^{ks}f_k^{\num}\}_{k\in\bbbz}\|_{L^p(l^q)}
+\|\{2^{ks}f_k^{\den}\}_{k\in\bbbz}\|_{L^p(l^q)},\label{eq3.32}
\end{align}
where the constants in (\ref{eq3.32}) depend on $n,r,M,\eta,s,p,q$, and the last inequality in (\ref{eq3.32}) is due to \cite[Remark 2.13]{Wang2023} since we choose $0<r<\min\{p,q\}$ and the $n$-dimensional distributional Fourier transforms of $f_k^{\num}$ and $f_k^{\den}$ are both compactly supported in the annulus $\{\xi\in\bbbr^n:2^{k-1}\leq|\xi|<2^{k+1}\}$ for each $k\in\bbbz$. Recall that $f_k^{\num}\in\Sw'(\bbbr^n)$, as defined in (\ref{eq3.4}), has the smooth function representative $(\ref{eq3.38})\in\functrep(f_k^{\num})\cap L^{p_0}(\bbbr^n)$, and the conditions (\ref{eq3.6}) and (\ref{eq3.7}) imposed on the function $\eta\in\Sw_{0}(\bbbr^n)$, then we have
\begin{align}
&f_k^{\num}(x)=f*\iFT_n[2^{-k}|\xi|\FT_n\psi(2^{-k}\xi)](x)\nonumber\\
&=f*\iFT_n[2^{-k}|\xi|\FT_n\eta(2^{-k}\xi)\FT_n\psi(2^{-k}\xi)](x)\nonumber\\
&=f_k*\iFT_n[2^{-k}|\xi|\FT_n\eta(2^{-k}\xi)](x),\label{eq3.121}
\end{align}
where $f$ appearing in the first two lines of (\ref{eq3.121}) is the function representative given by (\ref{eq3.2}), and $f_k(x)=f*\psi_{2^{-k}}(x)$ appearing in the last line of (\ref{eq3.121}) is the function defined in (\ref{eq3.37}). According to (\ref{eq3.139}) and Proposition \ref{proposition1} (iv), $f_k\!=\!f\!*\!\psi_{2^{-k}}\!\in\!\Sw'(\bbbr^n)$ has the smooth function representative $f_k(x)\in\functrep(f_k)\cap L^{p_0}(\bbbr^n)$, and equation (\ref{eq2.325}) indicates the distributional Fourier transform $\FT_n f_k$ is compactly supported in the annulus 
$\{\xi\in\bbbr^n:2^{k-1}\leq|\xi|<2^{k+1}\}$, therefore \cite[Definition 2.2]{Wang2023} and Remark \ref{remark4} tell us that the Peetre-Fefferman-Stein maximal function $\PFSmax_n f_k(x)$ is well-defined for every $k\in\bbbz$ by the following expression,
\begin{equation}\label{eq3.122}
\PFSmax_n f_k(x)\!=\!\esssup_{z\in\bbbr^n}\!\frac{|f_k(x\!-\!z)|}{(1\!+\!2^{k+1}|z|)^{\frac{n}{r}}}
\!=\!\esssup_{z\in\bbbr^n}\!\frac{|f*\psi_{2^{-k}}(x\!-\!z)|}{(1\!+\!2^{k+1}|z|)^{\frac{n}{r}}}.
\end{equation}
Furthermore, we also deduce from (\ref{eq3.121}) that
\begin{align}
&|f_k^{\num}(x)|=|f_k*\iFT_n[2^{-k}|\xi|\FT_n\eta(2^{-k}\xi)](x)|\nonumber\\
&\leq\int_{\bbbr^n}|\iFT_n[2^{-k}|\xi|\FT_n\eta(2^{-k}\xi)](z)|\cdot|f_k(x-z)|dz\nonumber\\
&\leq\!\!\!\int_{\bbbr^n}\!\!\!\!\!2^{kn}|\iFT_n[|\xi|\FT_n\eta(\xi)](2^k z)|\!\cdot\!
(1\!+\!2^{k+1}|z|)^{\frac{n}{r}}dz\PFSmax_n f_k(x)\nonumber\\
&=\!\!\!\int_{\bbbr^n}\!\!\!\!\!|\iFT_n[|\xi|\FT_n\eta(\xi)](z)|\!\cdot\!(1\!+\!2|z|)^{\frac{n}{r}}dz
\PFSmax_n f_k(x)\!\lesssim\!\PFSmax_n f_k(x),\label{eq3.33}
\end{align}
where in (\ref{eq3.33}), we also use the formula $$\iFT_n[2^{-k}|\xi|\FT_n\eta(2^{-k}\xi)](z)\!=\!2^{kn}\iFT_n[|\xi|\FT_n\eta(\xi)](2^k z),$$
and the integral in the last line of (\ref{eq3.33}) converges to a positive finite constant whose value depends on $n,r,\eta$ since $\iFT_n[|\xi|\FT_n\eta(\xi)](z)\in\Sw(\bbbr^n)$. Recall that $f_k^{\den}\in\Sw'(\bbbr^n)$, as defined in (\ref{eq3.5}), has the smooth function representative $(\ref{eq3.54})\in\functrep(f_k^{\den})\cap L^{p_0}(\bbbr^n)$, and the conditions (\ref{eq3.6}) and (\ref{eq3.7}) imposed on the function $\eta\in\Sw_{0}(\bbbr^n)$, then we have
\begin{align}
&f_k^{\den}(x)=f*\iFT_n[\frac{\FT_n\psi(2^{-k}\xi)}{2^{-k}|\xi|}](x)\nonumber\\
&=f*\iFT_n[\frac{\FT_n\eta(2^{-k}\xi)\FT_n\psi(2^{-k}\xi)}{2^{-k}|\xi|}](x)\nonumber\\
&=f_k*\iFT_n[\frac{\FT_n\eta(2^{-k}\xi)}{2^{-k}|\xi|}](x),\label{eq3.126}
\end{align}
where $f$ appearing in the first two lines of (\ref{eq3.126}) is the function representative given by (\ref{eq3.2}), and $f_k(x)=f*\psi_{2^{-k}}(x)$ appearing in the last line of (\ref{eq3.126}) is the function defined in (\ref{eq3.37}). Furthermore, we use the method given in (\ref{eq3.33}) and the fact that the function $\iFT_n[\frac{\FT_n\eta(\xi)}{|\xi|}](z)\in\Sw(\bbbr^n)$ to deduce that for every $k\in\bbbz$ and $x\in\bbbr^n$,
\begin{equation}\label{eq3.34}
|f_k^{\den}(x)|\lesssim\PFSmax_n f_k(x),
\end{equation}
where the implicit constant in (\ref{eq3.34}) depends only on $n,r,\eta$ and is independent of $k$ and $x$. Inserting (\ref{eq3.33}) and (\ref{eq3.34}) into (\ref{eq3.32}) and invoking \cite[Remark 2.13]{Wang2023} again, we can obtain
		\begin{align}
			&\|\big[\text{the expression (\ref{eq3.125})}\big]^{\frac{1}{q}}\|_{L^p(\bbbr^n)}\nonumber\\
			&\lesssim\|\{2^{ks}\PFSmax_n f_k\}_{k\in\bbbz}\|_{L^p(l^q)}\lesssim\|f\|_{\Fspq}<\infty,\label{eq3.35}
		\end{align}
where $0<p,q<\infty$, $-1<s<1$ and $f\in\Fspq$ has a function representative $f(x)\in\functrep_0(f)\cap L^{p_0}(\bbbr^n)$ for some $p_0\in[1,\infty]$. Inequality (\ref{eq3.35}) implies that the series (\ref{eq3.124}) converges for almost every $x\in\bbbr^n$ and almost every $t\in(0,\infty)$. Recall (\ref{eq3.128}). We denote
\begin{equation}\label{eq3.156}
S_N(x,t):=\sum_{|j|\leq N}F_{j}(x,t),\quad H(x,t):=\sum_{j\in\bbbz}F_{j}(x,t),
\end{equation}
\begin{equation}\label{eq3.157}
\overline{S}_N(x,t):=\sum_{|j|\leq N}|F_{j}(x,t)|,\quad
\overline{H}(x,t):=\sum_{j\in\bbbz}|F_{j}(x,t)|,
\end{equation}
then both of the following limits (\ref{eq3.134}) and (\ref{eq3.135}) exist and are finite for almost every $x\in\bbbr^n$ and almost every $t\in(0,\infty)$,
\begin{align}
&\lim_{N\rightarrow\infty}S_N(x,t)=H(x,t),\label{eq3.134}\\
&\lim_{N\rightarrow\infty}\overline{S}_N(x,t)=\overline{H}(x,t).\label{eq3.135}
\end{align}
From equation (\ref{eq3.123}), we see that for every $k\in\bbbz$ and $t\in(0,\infty)$, the partial sum (\ref{eq3.151}) converges to $f*\partial_{n+1}[P_t(\cdot)]\in\Sw_{0}'(\bbbr^n)$ in the sense of $\Sw_0'(\bbbr^n)$ as $N\rightarrow\infty$,
\begin{align}
&\sum_{k\leq j\leq N}2^{-j}f_j^{\den}*\iFT_n[-2\pi|\xi|^2 e^{-2\pi t|\xi|}\FT_n\eta(2^{-j}\xi)]
\nonumber\\
&+\sum_{-N\leq j<k}2^{j}f_j^{\num}*\iFT_n[-2\pi e^{-2\pi t|\xi|}\FT_n\eta(2^{-j}\xi)].\label{eq3.151}
\end{align}
We make the following claim.
\begin{equation}\label{claim1}
\begin{split}
&\text{If $f\in L^{p_0}(\bbbr^n)$ and $1\leq p_0<\infty$, the pointwise}\\
&\text{limit $H(x,t)$ equals the function (\ref{eq3.133}) for al-}\\
&\text{-most every $x\!\in\!\bbbr^n$ and almost every $t\!\in\!(0,\infty)$.}\\
&\text{If $f\!\in\!L^{p_0}(\bbbr^n)$ and $p_0\!=\!\infty$, then $H(x,t)$ belongs}\\
&\text{to $\functrep_0(\partial_{n+1}\Pint(f;x,t))$ for almost every $t\!\in\!\!(0,\infty)$.}
\end{split}
\end{equation}
We will assume the claim (\ref{claim1}) is true for now and postpone its proof to Step 5 of the current section. Furthermore, we have justified the inequality
\begin{align}
&\|(\int_0^{\infty}t^{q-sq-1}|\partial_{n+1}\Pint(f;\cdot,t)|^q dt)^{\frac{1}{q}}\|_{L^p(\bbbr^n)}\nonumber\\
&\lesssim\|\big[\text{the expression (\ref{eq3.125})}\big]^{\frac{1}{q}}\|_{L^p(\bbbr^n)},\label{eq3.36}
\end{align}
and the implicit constant depends on $s,q$. In (\ref{claim1}) and (\ref{eq3.36}), $\partial_{n+1}\Pint(f;$ $x,t)$ is the function (\ref{eq3.133}) when the function representative $f(x)$ given by (\ref{eq3.2}) is in $L^{p_0}(\bbbr^n)$ and $1\leq p_0<\infty$, and $\partial_{n+1}\Pint(f;x,t)$ is the tempered distribution defined by the integral
\begin{equation}\label{eq3.136}
\int_{\bbbr^n}\partial_{n+1}\Pint(f;x,t)\cdot g(x)dx\text{ for $g\in\Sw(\bbbr^n)$}
\end{equation}
so that the function $(\ref{eq3.133})\in\functrep(\partial_{n+1}\Pint(f;x,t))\subseteq\functrep_0(\partial_{n+1}\Pint(f;x,t))$ when the function representative $f(x)$ given by (\ref{eq3.2}) is in $L^{p_0}(\bbbr^n)$ and $p_0=\infty$. Inequalities (\ref{eq3.35}) and (\ref{eq3.36}) together complete the proof of (\ref{eq3.1}) when $l=n+1$. The proof of (\ref{eq3.1}) where $1\leq l\leq n$ can be obtained in a similar way.\\
		
Step 5: In this step, we prove the claim (\ref{claim1}). Let $t\in(0,\infty)$. For every $j\in\bbbz$, we denote
\begin{equation}\label{eq3.73}
\multi_j(\xi)=\FT_n\psi(2^{-j}\xi)\cdot-2\pi|\xi|e^{-2\pi t|\xi|}\in\Sw_{00}(\bbbr^n),
\end{equation}
then $\multi_j(\xi)$ is compactly supported in the annulus $\{\xi\in\bbbr^n:2^{j-1}\leq|\xi|<2^{j+1}\}$, and both of the following functions (\ref{eq3.137}) and (\ref{eq3.138}) belong to $\Sw_{0}(\bbbr^n)$ for every $x\in\bbbr^n$ and $j\in\bbbz$,
\begin{align}
&y\in\bbbr^n\longmapsto\iFT_n\multi_j(y)\in\Sw_{0}(\bbbr^n),\label{eq3.137}\\
&y\in\bbbr^n\longmapsto\iFT_n\multi_j(x-y)\in\Sw_{0}(\bbbr^n).\label{eq3.138}
\end{align}
Then Proposition \ref{proposition1} (iv), conditions (\ref{eq3.2}), (\ref{eq3.137}), and (\ref{eq3.138}) all together indicate that for the given $f\in\Fspq\subseteq\Sw_{0}'(\bbbr^n)$, the convolution $f*\iFT_n\multi_j\in\Sw'(\bbbr^n)$ has the smooth function representative $(\ref{eq3.153})\in\functrep(f*\iFT_n\multi_j)\cap L^{p_0}(\bbbr^n)$,
\begin{align}
&x\in\bbbr^n\longmapsto<f,\iFT_n\multi_j(x-\cdot)>\nonumber\\
&=\!\!\int_{\bbbr^n}\!\!\!\!f(y)\cdot\iFT_n\multi_j(x-y)dy=f*\iFT_n\multi_j(x).\label{eq3.153}
\end{align}
Equations (\ref{eq3.59}), (\ref{eq3.64}), (\ref{eq3.127}) and Proposition \ref{proposition1} (iv) equation (\ref{eq1.122}) imply the following equation (\ref{eq3.152}) is true in the sense of $\Sw'(\bbbr^n)$,
\begin{align}
&f*\iFT_n\multi_j\nonumber\\
&\!=2^{j}f_j^{\num}*\iFT_n[-2\pi e^{-2\pi t|\xi|}\FT_n\eta(2^{-j}\xi)]\nonumber\\
&\!=2^{-j}f_j^{\den}*\iFT_n[-2\pi |\xi|^2 e^{-2\pi t|\xi|}\FT_n\eta(2^{-j}\xi)],\label{eq3.152}
\end{align}
where $f\in\Fspq\subseteq\Sw_{0}'(\bbbr^n)$ in the first line of (\ref{eq3.152}). Therefore we infer from (\ref{eq3.132}), (\ref{eq3.153}), (\ref{eq3.152}), Lemma \ref{lemma15} (i), and the smoothness of relevant functions that for all $x\in\bbbr^n$, $t\in(0,\infty)$, and $j\in\bbbz$,
\begin{equation}\label{eq3.140}
F_j(x,t)=f*\iFT_n\multi_j(x),
\end{equation}
where $F_j(x,t)$ is the smooth function defined in (\ref{eq3.128}) and $f$ appearing on the right side of (\ref{eq3.140}) is the function representative given in (\ref{eq3.2}). Let $\phi\in\Sw_0(\bbbr^n)$, then by (\ref{eq3.151}) and (\ref{eq3.132}), we have
\begin{equation}\label{eq3.62}
<\!f*\partial_{n+1}[P_t(\cdot)],\phi\!>
=\!\!\lim_{N\rightarrow\infty}\int_{\bbbr^n}\sum_{|j|\leq N}F_{j}(x,t)\!\cdot\!\phi(x)dx.
\end{equation}
Now we recall (\ref{eq3.156}), (\ref{eq3.157}), (\ref{eq3.134}), and (\ref{eq3.135}). Then the dominating function for the integrand on the right side of (\ref{eq3.62}) is given by
\begin{equation*}
\bigg|\sum_{|j|\leq N}F_{j}(x,t)\cdot\phi(x)\bigg|\leq\overline{H}(x,t)\cdot|\phi(x)|
\end{equation*}
for every $x\in\bbbr^n$ and $t\in(0,\infty)$, and uniformly for all $N\geq0$. Our goal is to prove that
\begin{equation}\label{eq3.72}
\int_{\bbbr^n}\overline{H}(x,t)\cdot|\phi(x)|dx<\infty
\end{equation}
if $f(x)$ is the function representative given in (\ref{eq3.2}). Then we can apply the dominated convergence theorem and pass the limit under the integral sign in (\ref{eq3.62}) to obtain that
\begin{equation}\label{eq3.70}
<f*\partial_{n+1}[P_t(\cdot)],\phi>=\int_{\bbbr^n}H(x,t)\cdot\phi(x)dx
\end{equation}
for almost every $t\in(0,\infty)$, whenever $\phi\in\Sw_0(\bbbr^n)$, and hence justify the aforementioned claim (\ref{claim1}).\\

If $f(x)$ is the function representative given in (\ref{eq3.2}) and $p_0=1$, then we use (\ref{eq3.140}) and \cite[Theorem 1.2.10. (Minkowski's inequality)]{14classical} to obtain the following estimate for almost every $t\in(0,\infty)$,
\begin{align}
&\|\overline{H}(\cdot,t)-\overline{S}_N(\cdot,t)\|_{L^{1}(\bbbr^n)}\leq\sum_{|j|>N}\|F_{j}(\cdot,t)\|_{L^{1}(\bbbr^n)}\nonumber\\
&\!=\!\!\sum_{|j|>N}\!\!\|f\!*\!\iFT_n\multi_j\|_{L^{1}(\bbbr^n)}\!\leq\!\!
\sum_{|j|>N}\!\!\|\iFT_n\multi_j\|_{L^{1}(\bbbr^n)}\!\cdot\!\|f\|_{L^1(\bbbr^n)}.\label{eq3.75}
\end{align}
To estimate each term in the last summation of (\ref{eq3.75}), we pick a positive finite real number $\delta_j$ for every $j\in\bbbz$ and the value of $\delta_j$ will be determined later, then we can use (\ref{eq3.73}) to write
\begin{align}
&\|\iFT_n\multi_j\|_{L^{1}(\bbbr^n)}\nonumber\\
&=\int_{|x|<\delta_j}\big|\iFT_n[-2\pi|\xi|e^{-2\pi t|\xi|}\cdot\FT_n\psi(2^{-j}\xi)](x)\big|dx\nonumber\\
&\quad+\int_{|x|\geq\delta_j}\big|\iFT_n[-2\pi|\xi|e^{-2\pi t|\xi|}\cdot\FT_n\psi(2^{-j}\xi)](x)\big|dx.\label{eq3.107}
\end{align}
We can use the support condition (\ref{eq1-7}) of $\FT_n\psi$ and estimate the first term on the right side of (\ref{eq3.107}) as follows,
\begin{align}
	&\int_{|x|<\delta_j}\big|\iFT_n[-2\pi|\xi|e^{-2\pi t|\xi|}\cdot\FT_n\psi(2^{-j}\xi)](x)\big|dx\nonumber\\
	&\leq2\pi\nu_n\delta_j^n\int_{\bbbr^n}|\xi|e^{-2\pi t|\xi|}\cdot|\FT_n\psi(2^{-j}\xi)|d\xi\nonumber\\
	&\lesssim\delta_j^n\cdot2^{j(n+1)}\cdot e^{-\pi t\cdot2^j},\label{eq3.108}
\end{align}
where $\nu_n$ is the volume of the unit ball in $\bbbr^n$, and the implicit constant in (\ref{eq3.108}) depends on $n,\psi$ and is independent of $j\in\bbbz$. To estimate the second term on the right side of (\ref{eq3.107}) when $|x|\geq\delta_j$, we can assume without loss of generality that $x=(x_1,x_2,\cdots,x_n)$ and $|x_1|=\max\{|x_1|,|x_2|,\cdots,|x_n|\}$, then we have $0<\delta_j\leq|x|\leq\sqrt{n}|x_1|$. Furthermore, we can use the integration by parts formula with respect to $\xi_1$ repeatedly for $n+1$ times to obtain the following equation
\begin{align}
	&\iFT_n[-2\pi|\xi|e^{-2\pi t|\xi|}\cdot\FT_n\psi(2^{-j}\xi)](x)\nonumber\\
	&\!=\!\!\!\int_{\bbbr^n}\!\frac{\partial^{n+1}}{\partial\xi_1^{n+1}}\!\bigg[\!\!
	-\!2\pi|\xi|e^{-2\pi t|\xi|}\!\cdot\!\FT_n\psi(2^{-j}\xi)\!\bigg]\!\!\cdot\!\frac{e^{2\pi i\xi\cdot x}}{(-2\pi ix_1)^{n+1}}d\xi,\label{eq3.109}
\end{align}
and hence we have
\begin{align}
	&\int_{|x|\geq\delta_j}\big|\iFT_n[-2\pi|\xi|e^{-2\pi t|\xi|}\cdot\FT_n\psi(2^{-j}\xi)](x)\big|dx\nonumber\\
	&\lesssim\int_{|x|\geq\delta_j}\frac{1}{|x|^{n+1}}dx\cdot
	\int_{\bbbr^n}\big|\partial^{\alpha}[|\xi|e^{-2\pi t|\xi|}
	\cdot\FT_n\psi(2^{-j}\xi)]\big|d\xi\nonumber\\
	&\lesssim\frac{1}{\delta_j}\int_{\bbbr^n}\big|\partial^{\alpha}[|\xi|e^{-2\pi t|\xi|}
	\cdot\FT_n\psi(2^{-j}\xi)]\big|d\xi,\label{eq3.110}
\end{align}
where the implicit constants in (\ref{eq3.110}) depend on $n$, and $\alpha$ is the multi-index
\begin{equation}\label{eq3.111}
	\alpha=(n+1,0,\cdots,0),\qquad|\alpha|=n+1.
\end{equation}
Now we let $\beta_1$, $\beta_2$, $\gamma_1$, $\gamma_2$, $\omega_1$ be multi-indices satisfying the following conditions
\begin{align}
	&\alpha=\beta_1+\gamma_1+\omega_1,\qquad|\omega_1|>0,\label{eq3.112}\\
	&\alpha=\beta_2+\gamma_2,\label{eq3.113}
\end{align}
then we can use Leibniz rule to estimate the integral over $\bbbr^n$ in the last line of (\ref{eq3.110}) from above by a constant multiple of the sum of terms of the following forms,
\begin{equation}\label{eq3.114}
	\int_{\bbbr^n}\!\!\!2^{-j|\beta_1|}|(\partial^{\beta_1}\FT_n\psi)(2^{-j}\xi)|\!\cdot\!|\partial^{\gamma_1}(|\xi|)|\!\cdot\!|\partial^{\omega_1}(e^{-2\pi t|\xi|})|d\xi
\end{equation}
and
\begin{equation}\label{eq3.115}
	\int_{\bbbr^n}\!\!\!2^{-j|\beta_2|}|(\partial^{\beta_2}\FT_n\psi)(2^{-j}\xi)|\!\cdot\!|\partial^{\gamma_2}(|\xi|)|\!\cdot\!e^{-2\pi t|\xi|}d\xi,
\end{equation}
and the constant depends on $n$. By direct computations, we have the formulae
\begin{align}
&|\partial^{\gamma'}(|\xi|)|\lesssim|\xi|^{1-|\gamma'|},\label{eq1-47}\\
&|\partial^{\gamma''}(e^{-2\pi t|\xi|})|\lesssim e^{-2\pi t|\xi|}\cdot\sum_{l=0}^{|\gamma''|-1} \frac{1}{|\xi|^l},\label{eq1-48}
\end{align}
where $\gamma'$ and $\gamma''$ are multi-indices, and the implicit constant in (\ref{eq1-48}) depends on $t$. We use the support condition (\ref{eq1-7}) of $\FT_n\psi$ and formulae (\ref{eq1-47}), (\ref{eq1-48}) to estimate (\ref{eq3.114}) from above by the following
\begin{align}
	&(\ref{eq3.114})
	\lesssim2^{jn}\cdot2^{j(1-|\beta_1|-|\gamma_1|)}\cdot\sum_{l_1=0}^{|\omega_1|-1}2^{-jl_1}e^{-\pi t\cdot2^j}
	\nonumber\\
	&=\sum_{l_1=0}^{|\omega_1|-1}2^{j(n+1-|\beta_1|-|\gamma_1|-l_1)}e^{-\pi t\cdot2^j},\label{eq3.116}
\end{align}
where the implicit constant depends on $n,t,\psi$. Due to the conditions (\ref{eq3.111}) and (\ref{eq3.112}), we have
\begin{equation}\label{eq3.117}
	1\leq n+1-|\beta_1|-|\gamma_1|-l_1\leq n+1,
\end{equation}
and hence we obtain
\begin{equation}\label{eq3.118}
	(\ref{eq3.114})\lesssim
	\begin{cases}
		2^{j(n+1)}\cdot e^{-\pi t\cdot2^j} &\text{if}\quad j\geq0,\\
		2^{j}\cdot e^{-\pi t\cdot2^j} &\text{if}\quad j<0,
	\end{cases}
\end{equation}
where the implicit constant depends on $n,t,\psi$. We use the support condition (\ref{eq1-7}) of $\FT_n\psi$ and formula (\ref{eq1-47}) to estimate (\ref{eq3.115}) from above by the following
\begin{equation}\label{eq3.119}
	(\ref{eq3.115})\lesssim2^{j(n+1-|\beta_2|-|\gamma_2|)}\cdot e^{-\pi t\cdot2^j}=e^{-\pi t\cdot2^j},
\end{equation}
where the implicit constant depends on $n,\psi$, and we also use the equation $n+1=|\beta_2|+|\gamma_2|$ due to the conditions (\ref{eq3.111}) and (\ref{eq3.113}), and estimate (\ref{eq3.119}) is true for all $j\in\bbbz$. We combine (\ref{eq3.107}), (\ref{eq3.108}), (\ref{eq3.110}), (\ref{eq3.114}), (\ref{eq3.115}), (\ref{eq3.118}), (\ref{eq3.119}) all together and obtain the following estimate
\begin{align}
&\|\iFT_n\multi_j\|_{L^{1}(\bbbr^n)}\nonumber\\
&\!\lesssim\!\!
\begin{cases}
\!\delta_j^n\!\cdot\!2^{j(n+1)}\!\cdot\!e^{-\pi t\cdot2^j}\!\!+\!\frac{2^{j(n+1)}}{\delta_j}\!\cdot\! e^{-\pi t\cdot2^j}\!\!+\!\frac{1}{\delta_j}\!\cdot\!e^{-\pi t\cdot2^j} &\text{if}\quad j\!\geq\!0,\\
\!\delta_j^n\!\cdot\!2^{j(n+1)}\!\cdot\!e^{-\pi t\cdot2^j}\!\!+\!\frac{2^{j}}{\delta_j}\!\cdot\!
e^{-\pi t\cdot2^j}\!\!+\!\frac{1}{\delta_j}\!\cdot\!e^{-\pi t\cdot2^j} &\text{if}\quad j\!<\!0,
\end{cases}
\label{eq3.120}
\end{align}
and the implicit constant in (\ref{eq3.120}) depends on $n,t,\psi$ and is independent of $j\in\bbbz$. We want to make sure that all of the following $6$ series converge,
\begin{align*}
	&\sum_{j\geq1}\delta_j^n\cdot2^{j(n+1)}\cdot e^{-\pi t\cdot2^j},\quad
	\sum_{j\geq1}\frac{2^{j(n+1)}}{\delta_j}\cdot e^{-\pi t\cdot2^j},\quad
	\sum_{j\geq1}\frac{1}{\delta_j}\cdot e^{-\pi t\cdot2^j},\\
	&\sum_{j\leq-1}\delta_j^n\cdot2^{j(n+1)}\cdot e^{-\pi t\cdot2^j},\quad
	\sum_{j\leq-1}\frac{2^{j}}{\delta_j}\cdot e^{-\pi t\cdot2^j},\quad 
	\sum_{j\leq-1}\frac{1}{\delta_j}\cdot e^{-\pi t\cdot2^j},
\end{align*}
therefore choosing $\delta_j=2^{-j}$ will suffice. Estimate (\ref{eq3.120}) with $\delta_j=2^{-j}$ indicates $\overline{S}_N(x,t)\in L^{1}(\bbbr^n)$ for each $N\in\bbbn_0$ and $t\in(0,\infty)$. Inserting estimate (\ref{eq3.120}) with $\delta_j=2^{-j}$ into (\ref{eq3.75}) yields that the sequence $\{\overline{S}_N(x,t)\}_{N\in\bbbn_0}$ of partial sums converges to the function $\overline{H}(x,t)$ in $\|\cdot\|_{L^{1}(\bbbr^n)}$-norm as $N\rightarrow\infty$, and by the completeness of the space $L^{1}(\bbbr^n)$, we have $\overline{H}(x,t)\in L^{1}(\bbbr^n)$ and (\ref{eq3.72}) is proven in the case where $f\in L^{1}(\bbbr^n)$. In addition, the inequality $|H(x,t)|\leq\overline{H}(x,t)$ implies $H(x,t)\in L^{1}(\bbbr^n)$. The assumption (\ref{eq3.2}) with $p_0=1$ and (\ref{eq1-6}) imply the function $(\ref{eq3.133})\in L^1(\bbbr^n)$. We deduce from equation (\ref{eq3.70}) and Lemma \ref{lemma16} (\ref{eq2.238}) that both of the functions (\ref{eq3.133}) and $H(x,t)$ belong to $\functrep_{0}(f*\partial_{n+1}[P_t(\cdot)])\cap L^1(\bbbr^n)$, thus the following equation is true for all $\phi\in\Sw_0(\bbbr^n)$, and for almost every $t\in(0,\infty)$,
\begin{equation}\label{eq3.129}
\int_{\bbbr^n}\partial_{n+1}\Pint(f;x,t)\cdot\phi(x)dx=\int_{\bbbr^n}H(x,t)\cdot\phi(x)dx.
\end{equation}
Hence the difference
\begin{equation}\label{eq3.130}
\partial_{n+1}\Pint(f;x,t)-H(x,t)
\end{equation}
satisfies condition (\ref{eq1.101}) and is in $L^{1}(\bbbr^n)$. There exists $\tilde{f}\in\Sw'(\bbbr^n)$ so that the difference function (\ref{eq3.130}) is in $\functrep(\tilde{f})$. Therefore, by Proposition \ref{proposition2}, the difference (\ref{eq3.130}) equals a polynomial for almost every $x\in\bbbr^n$ whose coefficients are functions of $t\in(0,\infty)$. Since the difference (\ref{eq3.130}) also belongs to $L^{1}(\bbbr^n)$ with respect to variable $x$, the polynomial is identically zero, and $(\ref{eq3.133})=\partial_{n+1}\Pint(f;x,t)=H(x,t)$ for almost every $x\in\bbbr^n$ and almost every $t\!\in\!(0,\infty)$, where $f$ appearing in $\partial_{n+1}\Pint(f;x,t)$ is the function representative $f(x)\in\functrep_{0}(f)\cap L^1(\bbbr^n)$.\\

If $f(x)$ is the function representative given in (\ref{eq3.2}) and $1<p_0<\infty$, then we recall (\ref{eq3.140}) and (\ref{eq3.73}) and we have for almost every $t\!\in\!(0,\infty)$,
\begin{align}
&\|\overline{H}(\cdot,t)-\overline{S}_N(\cdot,t)\|_{L^{p_0}(\bbbr^n)}\leq\sum_{|j|>N}\|F_{j}(\cdot,t)\|_{L^{p_0}(\bbbr^n)}\nonumber\\
&=\sum_{|j|>N}\|f*\iFT_n\multi_j\|_{L^{p_0}(\bbbr^n)}.\label{eq3.77}
\end{align}
To estimate each term in the summation of (\ref{eq3.77}), we use the Mihlin-H\"{o}rmander multiplier theorem, i.e. \cite[Theorem 6.2.7.]{14classical}, and for every $j\in\bbbz$, we will find $\|\multi_j\|_{L^{\infty}(\bbbr^n)}$ and a positive finite constant $A_j$ so that
\begin{equation}\label{eq3.78}
\sup_{0<R<\infty}R^{2|\alpha|-n}\int_{\frac{R}{2}<|\xi|<2R}|\partial_{\xi}^{\alpha}\multi_j(\xi)|^2 d\xi\leq A_j^2<\infty
\end{equation}
for all multi-indices $\alpha$ with $|\alpha|\leq[\frac{n}{2}]+1$ and $[\frac{n}{2}]$ is the greatest integer less than or equal to $\frac{n}{2}$. According to \cite[Theorem 6.2.7 and Deﬁnition 2.5.11]{14classical}, we define the operator $W_j(g)$ for every $j\in\bbbz$ and $g\in\Sw(\bbbr^n)$ as below,
\begin{equation}\label{eq3.159}
W_j(g)(x):=\iFT_n[\multi_j(\xi)\FT_n g(\xi)](x),
\end{equation}
then the estimate (\ref{eq3.79}) is true for all $g\in\Sw(\bbbr^n)$,
\begin{equation}\label{eq3.79}
\|W_j(g)\|_{L^{p_0}(\bbbr^n)}\lesssim\big(A_j+\|\multi_j\|_{L^{\infty}(\bbbr^n)}\big)\cdot
\|g\|_{L^{p_0}(\bbbr^n)},
\end{equation}
where the implicit constant in (\ref{eq3.79}) depends on $n$ and $p_0$. Since $\Sw(\bbbr^n)$ is dense in $L^{p_0}(\bbbr^n)$, we can find a sequence $\{g_u\}_{u\in\bbbn}$ of Schwartz functions so that
\begin{equation}\label{eq3.160}
\lim_{u\rightarrow\infty}\|f-g_u\|_{L^{p_0}(\bbbr^n)}=0.
\end{equation}
By the linearity of the operator $W_j$, we have
\begin{equation}\label{eq3.161}
\|W_j(g_u)\!-\!W_j(g_v)\|_{L^{p_0}(\bbbr^n)}\!\!\lesssim\!\!\big(A_j\!+\!\|\multi_j\|_{L^{\infty}(\bbbr^n)}\big)\!\cdot\!\|g_u\!-\!g_v\|_{L^{p_0}(\bbbr^n)},
\end{equation}
thus the sequence $\{W_j(g_u)\}_{u\in\bbbn}$ is a Cauchy sequence in $L^{p_0}(\bbbr^n)$ and converges to a function in $L^{p_0}(\bbbr^n)$. We denote this function by $W_j(f)$, then we have
\begin{equation}\label{eq3.162}
\lim_{u\rightarrow\infty}\|W_j(f)-W_j(g_u)\|_{L^{p_0}(\bbbr^n)}=0,
\end{equation}
and the estimate (\ref{eq3.79}) is true with $f\in L^{p_0}(\bbbr^n)$ in place of $g\in\Sw(\bbbr^n)$. Since $\multi_j\in\Sw_{00}(\bbbr^n)$ and $g_u\in\Sw(\bbbr^n)$ for every $u\in\bbbn$, we can write $W_j(g_u)(x)$ equivalently as follows,
\begin{equation}\label{eq3.163}
W_j(g_u)(x)\!=\!\iFT_n[\multi_j(\xi)\FT_n g_u(\xi)](x)\!=\!g_u\!*\!\iFT_n\multi_j(x).
\end{equation}
Then \cite[Theorem 1.2.12.\! (Young's inequality)]{14classical} implies
\begin{align}
&\lim_{u\rightarrow\infty}\|f*\iFT_n\multi_j-W_j(g_u)\|_{L^{p_0}(\bbbr^n)}\nonumber\\
&\leq\lim_{u\rightarrow\infty}\|\iFT_n\multi_j\|_{L^1(\bbbr^n)}
\cdot\|f-g_u\|_{L^{p_0}(\bbbr^n)}=0,\label{eq3.164}
\end{align}
because $\multi_j\in\Sw_{00}(\bbbr^n)$ implies $\iFT_n\multi_j\in\Sw_{0}(\bbbr^n)\subseteq L^1(\bbbr^n)$. Combining (\ref{eq3.162}) and (\ref{eq3.164}) with the triangle inequality and the estimate (\ref{eq3.79}) in which we replace $g\in\Sw(\bbbr^n)$ by $f\in L^{p_0}(\bbbr^n)$, we have obtained the following inequality for every $j\in\bbbz$,
\begin{align}
&\|f*\iFT_n\multi_j\|_{L^{p_0}(\bbbr^n)}=\|W_j(f)\|_{L^{p_0}(\bbbr^n)}\nonumber\\
&\!\lesssim\big(A_j+\|\multi_j\|_{L^{\infty}(\bbbr^n)}\big)\cdot\|f\|_{L^{p_0}(\bbbr^n)},
\label{eq3.165}
\end{align}
where the implicit constant in (\ref{eq3.165}) depends on $n$ and $p_0$, and $f$ appearing in (\ref{eq3.165}) is the function representative given by (\ref{eq3.2}). We will use this inequality (\ref{eq3.165}), and now we estimate the quantities $\|\multi_j\|_{L^{\infty}(\bbbr^n)}$ and $A_j$. By using the defining expression (\ref{eq3.73}) and the support condition (\ref{eq1-7}) of $\FT_n\psi$, we can obtain
		\begin{equation}\label{eq3.80}
			\|\multi_j\|_{L^{\infty}(\bbbr^n)}\lesssim2^je^{-\pi t\cdot2^j},
		\end{equation}
and the implicit constant in (\ref{eq3.80}) depends on $\psi$. To estimate $A_j$, we let $\alpha$, $\beta_1$, $\beta_2$, $\beta_3$, $\beta_4$, $\gamma_1$, $\gamma_2$, $\gamma_3$, $\gamma_4$, $\omega_1$, $\omega_3$ all be multi-indices satisfying the conditions that
		\begin{align}
			&|\alpha|\leq[\frac{n}{2}]+1,\qquad|\omega_1|>0,\qquad|\omega_3|>0,\label{eq3.81}\\
			&\alpha=\beta_1+\gamma_1+\omega_1=\beta_3+\gamma_3+\omega_3,\label{eq3.82}\\
			&\alpha=\beta_2+\gamma_2=\beta_4+\gamma_4.\label{eq3.83}
		\end{align}
Then we can use the Leibniz rule to estimate $|\partial_{\xi}^{\alpha}\multi_j(\xi)|$ from above by a constant multiple of the sum of the following terms,
\begin{equation}\label{eq3.84}
\sum_{\alpha=\beta_1+\gamma_1+\omega_1}\!\!\!\!2^{-j|\beta_1|}\big|(\partial^{\beta_1}\FT_n\psi)(2^{-j}\xi)\big|\!\cdot\!\big|\partial^{\gamma_1}(|\xi|)\big|\!\cdot\!\big|\partial^{\omega_1}(e^{-2\pi t|\xi|})\big|
\end{equation}
		and
		\begin{equation}\label{eq3.85}
			\sum_{\alpha=\beta_2+\gamma_2}2^{-j|\beta_2|}\big|(\partial^{\beta_2}\FT_n\psi)(2^{-j}\xi)\big|\cdot
			\big|\partial^{\gamma_2}(|\xi|)\big|\cdot e^{-2\pi t|\xi|}.
		\end{equation}
The above estimate is still true if we replace $\beta_1,\gamma_1,\omega_1$ in (\ref{eq3.84}) by $\beta_3,\gamma_3,\omega_3$, and $\beta_2,\gamma_2$ in (\ref{eq3.85}) by $\beta_4,\gamma_4$. Hence we can estimate $$|\partial_{\xi}^{\alpha}\multi_j(\xi)|^2$$ from above by a constant multiple of the sum of terms of the following forms,
		\begin{align}
			&2^{-j|\beta_1|-j|\beta_3|}\big|(\partial^{\beta_1}\FT_n\psi)(2^{-j}\xi)\big|\cdot
			\big|\partial^{\gamma_1}(|\xi|)\big|\cdot\big|\partial^{\omega_1}(e^{-2\pi t|\xi|})\big|\nonumber\\
			&\cdot\big|(\partial^{\beta_3}\FT_n\psi)(2^{-j}\xi)\big|\cdot
			\big|\partial^{\gamma_3}(|\xi|)\big|\cdot\big|\partial^{\omega_3}(e^{-2\pi t|\xi|})\big|,\label{eq3.86}
		\end{align}
		and
		\begin{align}
			&2^{-j|\beta_1|-j|\beta_4|}\big|(\partial^{\beta_1}\FT_n\psi)(2^{-j}\xi)\big|\cdot
			\big|\partial^{\gamma_1}(|\xi|)\big|\cdot\big|\partial^{\omega_1}(e^{-2\pi t|\xi|})\big|\nonumber\\
			&\cdot\big|(\partial^{\beta_4}\FT_n\psi)(2^{-j}\xi)\big|\cdot
			\big|\partial^{\gamma_4}(|\xi|)\big|\cdot e^{-2\pi t|\xi|},\label{eq3.87}
		\end{align}
		and
		\begin{align}
			&2^{-j|\beta_3|-j|\beta_2|}\big|(\partial^{\beta_3}\FT_n\psi)(2^{-j}\xi)\big|\cdot
			\big|\partial^{\gamma_3}(|\xi|)\big|\cdot\big|\partial^{\omega_3}(e^{-2\pi t|\xi|})\big|\nonumber\\
			&\cdot\big|(\partial^{\beta_2}\FT_n\psi)(2^{-j}\xi)\big|\cdot
			\big|\partial^{\gamma_2}(|\xi|)\big|\cdot e^{-2\pi t|\xi|},\label{eq3.88}
		\end{align}
		and
		\begin{align}
			&2^{-j|\beta_2|-j|\beta_4|}\big|(\partial^{\beta_2}\FT_n\psi)(2^{-j}\xi)\big|\cdot
			\big|\partial^{\gamma_2}(|\xi|)\big|\nonumber\\
			&\cdot\big|(\partial^{\beta_4}\FT_n\psi)(2^{-j}\xi)\big|\cdot
			\big|\partial^{\gamma_4}(|\xi|)\big|\cdot e^{-4\pi t|\xi|},\label{eq3.89}
		\end{align}
where the constant depends on $\alpha$. Due to the support condition (\ref{eq1-7}) of $\FT_n\psi$, each of the above expressions (\ref{eq3.86}), (\ref{eq3.87}), (\ref{eq3.88}), (\ref{eq3.89}) is supported in the annulus
		\begin{equation}\label{eq3.90}
			\{\xi\in\bbbr^n:2^{j-1}\leq|\xi|\leq2^{j+1}\},
		\end{equation}
		and hence if $R<2^{j-2}$ or $R>2^{j+2}$ in (\ref{eq3.78}), the annulus
		\begin{equation}\label{eq3.91}
			\{\xi\in\bbbr^n:\frac{R}{2}<|\xi|<2R\}
		\end{equation}
		will not intersect the annulus (\ref{eq3.90}). Therefore after putting expressions (\ref{eq3.86}), (\ref{eq3.87}), (\ref{eq3.88}), (\ref{eq3.89}) into (\ref{eq3.78}), the range of $R$ in (\ref{eq3.78}) is restricted to the closed and bounded interval $[2^{j-2},2^{j+2}]$. Inserting (\ref{eq3.86}) into (\ref{eq3.78}) and invoking formulae (\ref{eq1-47}), (\ref{eq1-48}) yields the following estimate
		\begin{align}
			&\sup_{2^{j-2}\leq R\leq2^{j+2}}R^{2|\alpha|-n}\int_{\frac{R}{2}<|\xi|<2R}\big[\text{expression (\ref{eq3.86})}\big]d\xi\nonumber\\
			&\leq\sup_{2^{j-2}\leq R\leq2^{j+2}}R^{2|\alpha|-n}\int_{2^{j-1}\leq|\xi|\leq2^{j+1}}\big[\text{expression (\ref{eq3.86})}\big]d\xi\nonumber\\
			&\lesssim2^{j(2|\alpha|-n)}\cdot2^{jn}\cdot2^{-j(|\beta_1|+|\beta_3|)}\cdot2^{j(2-|\gamma_1|-|\gamma_3|)}
			\nonumber\\
			&\quad\cdot\sum_{l_1=0}^{|\omega_1|-1}2^{-jl_1}\cdot\sum_{l_3=0}^{|\omega_3|-1}2^{-jl_3}
			\cdot e^{-2\pi t\cdot2^j}\nonumber\\
			&=\sum_{l_1=0}^{|\omega_1|-1}\sum_{l_3=0}^{|\omega_3|-1}2^{j(2+2|\alpha|-|\beta_1|-|\beta_3|-|\gamma_1|
				-|\gamma_3|-l_1-l_3)}\cdot e^{-2\pi t\cdot2^j},\label{eq3.92}
		\end{align}
		where the implicit constant in (\ref{eq3.92}) depends on $\alpha,n,t,\psi$ and is independent of $j\in\bbbz$. Conditions (\ref{eq3.81}) and (\ref{eq3.82}) imply
\begin{equation}\label{eq3.93}
2\!\leq\!2|\alpha|\!-\!|\beta_1|\!-\!|\beta_3|\!-\!|\gamma_1|\!-\!|\gamma_3|\!-\!l_1\!-\!l_3
\!\leq\!2|\alpha|\!\leq\!n\!+\!2.
\end{equation}
		Therefore we have obtained the following estimate
		\begin{align}
			&\sup_{2^{j-2}\leq R\leq2^{j+2}}R^{2|\alpha|-n}\int_{\frac{R}{2}<|\xi|<2R}\big[\text{expression (\ref{eq3.86})}\big]d\xi\nonumber\\
			&\lesssim
			\begin{cases}
				2^{j(n+4)}\cdot e^{-2\pi t\cdot2^j} &\text{if}\quad j\geq0,\\
				2^{4j}\cdot e^{-2\pi t\cdot2^j} &\text{if}\quad j<0,
			\end{cases}
			\label{eq3.94}
		\end{align}
		where the implicit constant depends on $\alpha,n,t,\psi$ and is independent of $j\in\bbbz$. Inserting (\ref{eq3.89}) into (\ref{eq3.78}) and invoking formula (\ref{eq1-47}) yields the following estimate
\begin{align}
&\sup_{2^{j-2}\leq R\leq2^{j+2}}R^{2|\alpha|-n}\int_{\frac{R}{2}<|\xi|<2R}\big[\text{expression (\ref{eq3.89})}\big]d\xi\nonumber\\
&\leq\sup_{2^{j-2}\leq R\leq2^{j+2}}R^{2|\alpha|-n}\int_{2^{j-1}\leq|\xi|\leq2^{j+1}}
\big[\text{expression (\ref{eq3.89})}\big]d\xi\nonumber\\
&\lesssim\!2^{j(2|\alpha|-n)}\!\cdot\!2^{jn}\!\cdot\!2^{-j(|\beta_2|+|\beta_4|)}\!\cdot\!
2^{j(2-|\gamma_2|-|\gamma_4|)}\!\cdot\!e^{-2\pi t\cdot2^j}\nonumber\\
&=2^{2j}\!\cdot\!e^{-2\pi t\cdot2^j},\label{eq3.95}
\end{align}
		where we use the condition (\ref{eq3.83}), and the implicit constant in (\ref{eq3.95}) depends on $\alpha,n,\psi$ and is independent of $j\in\bbbz$. Estimate (\ref{eq3.95}) is true for all $j\in\bbbz$. Inserting (\ref{eq3.87}) into (\ref{eq3.78}) and invoking formulae (\ref{eq1-47}) and (\ref{eq1-48}) yields the following estimate
\begin{align}
&\sup_{2^{j-2}\leq R\leq2^{j+2}}R^{2|\alpha|-n}\int_{\frac{R}{2}<|\xi|<2R}\big[\text{expression (\ref{eq3.87})}\big]d\xi\nonumber\\
&\leq\sup_{2^{j-2}\leq R\leq2^{j+2}}R^{2|\alpha|-n}\int_{2^{j-1}\leq|\xi|\leq2^{j+1}}\big[\text{expression (\ref{eq3.87})}\big]d\xi\nonumber\\
&\lesssim\!2^{j(2|\alpha|-n)}\!\cdot\!2^{jn}\!\cdot\!2^{-j(|\beta_1|+|\beta_4|)}\!\cdot\!
2^{j(2-|\gamma_1|-|\gamma_4|)}
\!\cdot\!\!\sum_{l_1=0}^{|\omega_1|-1}2^{-jl_1}\!\cdot\!e^{-2\pi t\cdot2^j}\nonumber\\
&=\sum_{l_1=0}^{|\omega_1|-1} 2^{j(2+2|\alpha|-|\beta_1|-|\beta_4|-|\gamma_1|
-|\gamma_4|-l_1)}\cdot e^{-2\pi t\cdot2^j},\label{eq3.96}
\end{align}
		where the implicit constant in (\ref{eq3.96}) depends on $\alpha,n,t,\psi$ and is independent of $j\in\bbbz$. Conditions (\ref{eq3.81}), (\ref{eq3.82}), and (\ref{eq3.83}) imply
\begin{equation}\label{eq3.97}
1\!\leq\!2|\alpha|\!-\!|\beta_1|\!-\!|\beta_4|\!-\!|\gamma_1|\!-\!|\gamma_4|\!-\!l_1\!\leq\!2|\alpha|
\!\leq\!n\!+\!2.
\end{equation}
		Therefore we have obtained the following estimate
		\begin{align}
			&\sup_{2^{j-2}\leq R\leq2^{j+2}}R^{2|\alpha|-n}\int_{\frac{R}{2}<|\xi|<2R}\big[\text{expression (\ref{eq3.87})}\big]d\xi\nonumber\\
			&\lesssim
			\begin{cases}
				2^{j(n+4)}\cdot e^{-2\pi t\cdot2^j} &\text{if}\quad j\geq0,\\
				2^{3j}\cdot e^{-2\pi t\cdot2^j} &\text{if}\quad j<0,
			\end{cases}
			\label{eq3.98}
		\end{align}
		where the implicit constant depends on $\alpha,n,t,\psi$ and is independent of $j\in\bbbz$. Since (\ref{eq3.88}) can be obtained from (\ref{eq3.87}) by replacing $\beta_1,\gamma_1,\omega_1$ in (\ref{eq3.87}) by $\beta_3,\gamma_3,\omega_3$ and replacing $\beta_4,\gamma_4$ in (\ref{eq3.87}) by $\beta_2,\gamma_2$, we can insert (\ref{eq3.88}) into (\ref{eq3.78}) and use an argument that is similar to the above to obtain the following estimate
		\begin{align}
			&\sup_{2^{j-2}\leq R\leq2^{j+2}}R^{2|\alpha|-n}\int_{\frac{R}{2}<|\xi|<2R}\big[\text{expression (\ref{eq3.88})}\big]d\xi\nonumber\\
			&\lesssim
			\begin{cases}
				2^{j(n+4)}\cdot e^{-2\pi t\cdot2^j} &\text{if}\quad j\geq0,\\
				2^{3j}\cdot e^{-2\pi t\cdot2^j} &\text{if}\quad j<0,
			\end{cases}
			\label{eq3.99}
		\end{align}
		and the implicit constant in (\ref{eq3.99}) depends on $\alpha,n,t,\psi$ and is independent of $j\in\bbbz$. We can combine (\ref{eq3.94}), (\ref{eq3.95}), (\ref{eq3.98}), (\ref{eq3.99}) with (\ref{eq3.78}) and notice that $|\alpha|\leq[\frac{n}{2}]+1$ in (\ref{eq3.78}), then we can estimate $A_j$ from above as follows,
		\begin{equation}\label{eq3.100}
			A_j\!\lesssim\!\!
			\begin{cases}
				2^{j(\frac{n}{2}+2)}\cdot e^{-\pi t\cdot2^j}+2^{j}\cdot e^{-\pi t\cdot2^j} 
				&\text{if }j\!\geq\!0,\\
				2^{2j}\cdot e^{-\pi t\cdot2^j}+2^{j}\cdot e^{-\pi t\cdot2^j}+2^{\frac{3j}{2}}\cdot e^{-\pi t\cdot2^j} &\text{if }j\!<\!0,
			\end{cases}
		\end{equation}
		where the implicit constant depends on $n,t,\psi$. We use the inequalities
		\begin{align*}
			&2^j\leq2^{j(\frac{n}{2}+2)}\qquad\text{if $j\geq0$},\\
			&2^{2j}\leq2^j,\qquad2^{\frac{3j}{2}}\leq2^j,\qquad\text{if $j<0$},
		\end{align*}
		and recall the estimate (\ref{eq3.80}), then we can obtain the following estimate
		\begin{equation}\label{eq3.101}
			A_j+\|\multi_j\|_{L^{\infty}(\bbbr^n)}\lesssim
			\begin{cases}
				2^{j(\frac{n}{2}+2)}\cdot e^{-\pi t\cdot2^j} &\text{if }j\geq0,\\
				2^{j}\cdot e^{-\pi t\cdot2^j} &\text{if }j<0,
			\end{cases}
		\end{equation}
where the implicit constant in (\ref{eq3.101}) depends on $n,t,\psi$ and is independent of $j\in\bbbz$. Now we can combine (\ref{eq3.165}) and (\ref{eq3.101}), and insert the result into (\ref{eq3.77}) to obtain that for almost every $t\!\in\!(0,\infty)$, we have
\begin{align*}
&\|\overline{H}(\cdot,t)\!-\!\overline{S}_N(\cdot,t)\|_{L^{p_0}(\bbbr^n)}\!\leq\!\!
\sum_{|j|>N}\!\|f*\iFT_n\multi_j\|_{L^{p_0}(\bbbr^n)}\nonumber\\
&\lesssim\!\!\sum_{j<-N}\!2^{j}\!\cdot\!e^{-\pi t\cdot2^j}\!\cdot\!\|f\|_{L^{p_0}(\bbbr^n)}
\!+\!\sum_{j>N}\!2^{j(\frac{n}{2}+2)}\!\cdot\!e^{-\pi t\cdot2^j}\!\cdot\!\|f\|_{L^{p_0}(\bbbr^n)}.
\end{align*}
Both of the following series in (\ref{eq3.102}) converge,
\begin{equation}\label{eq3.102}
\sum_{j\leq-1}2^{j}\cdot e^{-\pi t\cdot2^j}\qquad\text{and}\qquad
\sum_{j\geq1}2^{j(\frac{n}{2}+2)}\cdot e^{-\pi t\cdot2^j},
\end{equation}
and the convergence of the former series is because $e^{-\pi t\cdot2^j}\leq1$ for all $j\leq-1$, and the convergence of the latter series is justified by using the ratio test. Estimates (\ref{eq3.165}) and (\ref{eq3.101}) also show that $\overline{S}_N(x,t)\in L^{p_0}(\bbbr^n)$ for each $N\in\bbbn_0$ and $t\in(0,\infty)$. Hence the sequence $\{\overline{S}_N(x,t)\}_{N\in\bbbn_0}$ of partial sums converges to the function $\overline{H}(x,t)$ in $\|\cdot\|_{L^{p_0}(\bbbr^n)}$-norm, and by the completeness of the space $L^{p_0}(\bbbr^n)$, we have $\overline{H}(x,t)\in L^{p_0}(\bbbr^n)$ and (\ref{eq3.72}) is proven by applying the H\"{o}lder's inequality with $\frac{1}{p_0}+\frac{1}{p_0'}=1$ in the case where $f\in L^{p_0}(\bbbr^n)$ for some $p_0\in(1,\infty)$. In addition, the inequality $|H(x,t)|\leq\overline{H}(x,t)$ implies $H(x,t)\in L^{p_0}(\bbbr^n)$. The assumption (\ref{eq3.2}) with $1<p_0<\infty$, formula (\ref{eq1-6}), and \cite[Theorem 1.2.12. (Young's inequality)]{14classical} imply the function $(\ref{eq3.133})\in L^{p_0}(\bbbr^n)$. We deduce from equation (\ref{eq3.70}) and Lemma \ref{lemma16} (\ref{eq2.238}) that both of the functions (\ref{eq3.133}) and $H(x,t)$ belong to $\functrep_{0}(f*\partial_{n+1}[P_t(\cdot)])\cap L^{p_0}(\bbbr^n)$, thus equation (\ref{eq3.129}) is true for all $\phi\in\Sw_0(\bbbr^n)$, and for almost every $t\!\in\!(0,\infty)$. And the difference (\ref{eq3.130}) satisfies condition (\ref{eq1.101}) and is in $L^{p_0}(\bbbr^n)\subseteq\Lloc$. There exists $\tilde{f}\in\Sw'(\bbbr^n)$ so that the difference function (\ref{eq3.130}) is in $\functrep(\tilde{f})$. Therefore, by Proposition \ref{proposition2}, the difference (\ref{eq3.130}) equals a polynomial for almost every $x\in\bbbr^n$ whose coefficients are functions of $t\in(0,\infty)$. Since the difference (\ref{eq3.130}) also belongs to $L^{p_0}(\bbbr^n)$ with respect to variable $x$, the polynomial is identically zero, and $(\ref{eq3.133})=\partial_{n+1}\Pint(f;x,t)=H(x,t)$ for almost every $x\in\bbbr^n$ and almost every $t\!\in\!(0,\infty)$, where $f$ appearing in $\partial_{n+1}\Pint(f;x,t)$ is the function representative $f(x)\in\functrep_{0}(f)\cap L^{p_0}(\bbbr^n)$.\\
		
If $f(x)$ is the function representative given in (\ref{eq3.2}) and $p_0=\infty$, then we recall (\ref{eq3.157}) and (\ref{eq3.140}) to obtain that for almost every $t\!\in\!(0,\infty)$,
\begin{align}
&\|\overline{H}(\cdot,t)-\overline{S}_N(\cdot,t)\|_{L^{\infty}(\bbbr^n)}\leq\sum_{|j|>N}\|F_{j}(\cdot,t)\|_{L^{\infty}(\bbbr^n)}\nonumber\\
&\!\leq\sum_{|j|>N}\|\iFT_n\multi_j\|_{L^{1}(\bbbr^n)}\cdot\|f\|_{L^{\infty}(\bbbr^n)}.\label{eq3.105}
\end{align}
Estimate (\ref{eq3.120}) with $\delta_j=2^{-j}$ indicates $\overline{S}_N(x,t)\in L^{\infty}(\bbbr^n)$ for each $N\in\bbbn_0$ and $t\in(0,\infty)$. Inserting estimate (\ref{eq3.120}) with $\delta_j=2^{-j}$ into (\ref{eq3.105}) yields that the sequence $\{\overline{S}_N(x,t)\}_{N\in\bbbn_0}$ of partial sums converges to the function $\overline{H}(x,t)$ in $\|\cdot\|_{L^{\infty}(\bbbr^n)}$-norm as $N\rightarrow\infty$, and by the completeness of the space $L^{\infty}(\bbbr^n)$, we have $\overline{H}(x,t)\in L^{\infty}(\bbbr^n)$ and (\ref{eq3.72}) is proven in the case where $f\in L^{\infty}(\bbbr^n)$. In addition, the inequality $|H(x,t)|\leq\overline{H}(x,t)$ implies $H(x,t)\in L^{\infty}(\bbbr^n)$. The assumption (\ref{eq3.2}) with $p_0=\infty$ and (\ref{eq1-6}) imply the function $(\ref{eq3.133})\in L^{\infty}(\bbbr^n)$. We deduce from equation (\ref{eq3.70}) and Lemma \ref{lemma16} (\ref{eq2.238}) that both of the functions (\ref{eq3.133}) and $H(x,t)$ belong to $\functrep_{0}(f*\partial_{n+1}[P_t(\cdot)])\cap L^{\infty}(\bbbr^n)$, thus equation (\ref{eq3.129}) is true for all $\phi\in\Sw_0(\bbbr^n)$, and for almost every $t\!\in\!(0,\infty)$. And the difference (\ref{eq3.130}) satisfies condition (\ref{eq1.101}) and is in $L^{\infty}(\bbbr^n)\subseteq\Lloc$. There exists $\tilde{f}\in\Sw'(\bbbr^n)$ so that the difference function (\ref{eq3.130}) is in $\functrep(\tilde{f})$. Therefore, by Proposition \ref{proposition2}, the difference (\ref{eq3.130}) equals a polynomial for almost every $x\in\bbbr^n$ whose coefficients are functions of $t\in(0,\infty)$. Since the difference (\ref{eq3.130}) also belongs to $L^{\infty}(\bbbr^n)$ with respect to variable $x$, the difference (\ref{eq3.130}) is almost everywhere equal to a function of variable $t$. From (\ref{eq3.70}) and Lemma \ref{lemma16} (\ref{eq2.245}), we deduce that
\begin{equation}\label{eq3.142}
H(x,t)\in\functrep_{0}(f*\partial_{n+1}[P_t(\cdot)])=\functrep_{0}(\partial_{n+1}\Pint(f;x,t)),
\end{equation}
where $\partial_{n+1}\Pint(f;x,t)$ appearing in (\ref{eq3.142}) is the tempered distribution defined by the integral (\ref{eq1.121}). The proof of Theorem \ref{theorem3} is complete.
\end{proof}

\section{Proofs of Corollary \ref{corollary1} and Corollary \ref{corollary2}}\label{proofs.of.corollaries.1&2}
\begin{proof}[Proof of Corollary \ref{corollary1}]
We can use the formula (\ref{eq1-14}) to deduce that $$\|f\|_{\dot{F}^0_{p,q}(\bbbr^n)}\leq\|f\|_{\dot{F}^0_{p,2}(\bbbr^n)}$$ when $2\leq q<\infty$, and Corollary \ref{corollary1} is proved if we apply Theorem \ref{theorem3} with $s=0$. Furthermore, we can use Lemma \ref{lemma12} (i) and (ii) to deduce that
	\begin{equation}\label{eq1-23}
		\|\LPg_{0,q}(f)\|_{L^p(\bbbr^n)}\lesssim\|f\|_{L^p(\bbbr^n)},
	\end{equation}
if $1<p<\infty$, $2\leq q<\infty$, and $f\in\dot{F}^0_{p,2}(\bbbr^n)$. We can also use \cite[estimate (2.2.25) of Theorem 2.2.9]{14modern} to deduce that
\begin{equation}\label{eq1-24}
\|\LPg_{0,q}(f)\|_{L^p(\bbbr^n)}\lesssim\|f\|_{H^p(\bbbr^n)},
\end{equation}
if $0<p<1$, $2\leq q<\infty$, and $f\in H^p(\bbbr^n)$ is a tempered distribution and has a function representative $F(x)\in\functrep(f)\cap L^1(\bbbr^n)$, and $\|\cdot\|_{H^p(\bbbr^n)}$ represents the Hardy quasinorm. If $p=1$, $2\leq q<\infty$, and $f\in H^1(\bbbr^n)$, then we have
\begin{equation}\label{eq1-24-1}
\|\LPg_{0,q}(f)\|_{L^1(\bbbr^n)}\lesssim\|f\|_{H^1(\bbbr^n)}.
\end{equation}
In view of Definition \ref{definition4} and Lemma \ref{lemma17}, when $1\leq p<\infty$, every $f\in H^p(\bbbr^n)\subseteq\Sw_{0}'(\bbbr^n)$ has a function representative $F(x)\in\functrep_{0}(f)\cap L^p(\bbbr^n)$ and we have
\begin{equation}\label{eq5.24}
\|f\|_{H^p(\bbbr^n)}\!=\!\|\sup_{t>0}|f\!*\!\Phi_t|\|_{L^p(\bbbr^n)}\!=\!
\|\sup_{t>0}|F\!*\!\Phi_t|\|_{L^p(\bbbr^n)}.
\end{equation}
Under the conditions for (\ref{eq1-24}), \cite[Theorem 2.3.20]{14classical} implies $f*\Phi_t\in\Sw'(\bbbr^n)$ has the smooth function representative $(\ref{eq5.25})\!\in\!\functrep(f*\Phi_t)\cap L^1(\bbbr^n)$, which is also the function representative specified by Definition \ref{definition4} and belongs to $L^p(\bbbr^n)$ for $0<p<1$,
\begin{equation}\label{eq5.25}
x\!\in\!\bbbr^n\!\longmapsto<\!f,\Phi_t(x\!-\cdot)\!>=\!\!\!\int_{\bbbr^n}\!\!\!\!\!F(y)\!\cdot\!\Phi_t(x\!-\!y)dy\!=\!F\!*\!\Phi_t(x),
\end{equation}
thus equation (\ref{eq5.24}) is still true. Because of (\ref{eq3.39}) and \cite[Theorem 2.3.20]{14classical}, $f*\psi_{2^{-j}}\in\Sw'(\bbbr^n)$ has the smooth function representative $(\ref{eq5.26})\in\functrep(f*\psi_{2^{-j}})\cap L^1(\bbbr^n)$,
\begin{equation}\label{eq5.26}
x\!\in\!\bbbr^n\!\mapsto<\!\!f,\psi_{2^{-j}}(x\!-\cdot)\!\!>=\!\!\!\int_{\bbbr^n}\!\!\!\!\!F(y)\!\cdot\!\psi_{2^{-j}}(x\!-\!y)dy\!=\!F\!*\!\psi_{2^{-j}}(x).
\end{equation}
We argue as in the proof of \cite[Theorem 2.2.9]{14modern}. Since $\{\Phi_t\}_{t>0}$ is an approximate identity and since $F*\psi_{2^{-j}}(x)\in L^1(\bbbr^n)$ for all $j\in\bbbz$, we have for any $M\in\bbbn$,
\begin{equation}\label{eq5.27}
\bigg|\!\sum_{|j|\leq M}\!\!r_j(w)\!\cdot\!F\!*\!\psi_{2^{-j}}(x)\bigg|\!\!\leq\!\sup_{t>0}\!
\bigg|\Phi_t\!*\!\!\bigg(\!\sum_{|j|\leq M}\!\!r_j(w)\!\cdot\!F\!*\!\psi_{2^{-j}}\!\!\bigg)(x)\bigg|,
\end{equation}
where $r_j{}'s$ are the Rademacher functions introduced in \cite[Appendix C]{14classical}. Since $F(x)\in L^1(\bbbr^n)$, \cite[Theorem 6.1.2 (6.1.5)]{14classical} tells us that for almost every $x\in\bbbr^n$,
\begin{equation}\label{eq5.28}
\bigg(\sum_{j\in\bbbz}|F*\psi_{2^{-j}}(x)|^2\bigg)^{\frac{1}{2}}<\infty,
\end{equation}
thus \cite[Appendix C.2 Khintchine's inequality]{14classical} is applicable. And we have
\begin{align}
&\int_{\bbbr^n}\bigg(\sum_{|j|\leq M}|F*\psi_{2^{-j}}(x)|^2\bigg)^{\frac{p}{2}}dx\nonumber\\
&\lesssim\int_{\bbbr^n}\int_0^1\bigg|\!\sum_{|j|\leq M}\!\!r_j(w)\!\cdot\!
F\!*\!\psi_{2^{-j}}(x)\bigg|^p dwdx\nonumber\\
&=\int_0^1\int_{\bbbr^n}\bigg|\!\sum_{|j|\leq M}\!\!r_j(w)\!\cdot\!
F\!*\!\psi_{2^{-j}}(x)\bigg|^p dxdw\nonumber\\
&\leq\int_0^1\int_{\bbbr^n}\!\sup_{t>0}\!\bigg|\Phi_t\!*\!\!\bigg(\!\sum_{|j|\leq M}\!\!r_j(w)\!\cdot\!
F\!*\!\psi_{2^{-j}}\!\!\bigg)(x)\bigg|^p dxdw\nonumber\\
&\lesssim\int_0^1\|\sup_{t>0}|F*\Phi_t|\|_{L^p(\bbbr^n)}^p dw=\|f\|_{H^p(\bbbr^n)}^p,\label{eq5.29}
\end{align}
where the inequality in the penultimate line of (\ref{eq5.29}) is due to (\ref{eq5.27}), and the inequality in the last line of (\ref{eq5.29}) is due to \cite[Theorem 2.1.14]{14modern}, and the equation in the last line of (\ref{eq5.29}) is by (\ref{eq5.24}). Under the conditions for (\ref{eq1-24-1}), we use Lemma \ref{lemma17} (ii) to obtain a function representative $F(x)\in\functrep_{0}(f)\cap L^1(\bbbr^n)$ and
\begin{equation}\label{eq5.30}
\|f\|_{H^1(\bbbr^n)}=\|\sup_{t>0}|F*\Phi_t|\|_{L^1(\bbbr^n)}.
\end{equation}
And (\ref{eq5.26}) is still true. The argument given for (\ref{eq5.27}), (\ref{eq5.28}), (\ref{eq5.29}) only uses the condition $F(x)\in L^1(\bbbr^n)$, and thus we obtain when $p=1$ and $M\in\bbbn$,
\begin{align}
&\int_{\bbbr^n}\bigg(\sum_{|j|\leq M}|F*\psi_{2^{-j}}(x)|^2\bigg)^{\frac{1}{2}}dx\nonumber\\
&\lesssim\int_0^1\|\sup_{t>0}|F*\Phi_t|\|_{L^1(\bbbr^n)}dw=\|f\|_{H^1(\bbbr^n)},\label{eq5.31}
\end{align}
where the last equation in (\ref{eq5.31}) is due to (\ref{eq5.30}). Notice that when $0<p<1$ and $f\in H^p(\bbbr^n)\subseteq\Sw_{0}'(\bbbr^n)$ has a function representative $F(x)\in\functrep_0(f)\cap L^1(\bbbr^n)$, the equation (\ref{eq5.24}) is unknown, since the function $\Phi\in\Sw(\bbbr^n)$ given in Definition \ref{definition4} is not in $\Sw_{0}(\bbbr^n)$ and since the continuous function representative $F*\Phi_t(x)\in\functrep_{0}(f*\Phi_t)\cap L^1(\bbbr^n)$ that can be obtained by the argument given in (\ref{eq2.337}) may not be identical to the smooth function representative $F_t(x)\in\functrep_{0}(f*\Phi_t)\cap L^p(\bbbr^n)$ specified by Definition \ref{definition4} (\ref{eq2.334}) and satisfying estimate (\ref{eq1.125}). Therefore inequality (\ref{eq1-24}) is unknown under the weaker condition that $0<p<1$, $2\leq q<\infty$, and $f\in H^p(\bbbr^n)$ has a function representative $F(x)\in\functrep_0(f)\cap L^1(\bbbr^n)$. In \cite[Corollary of Theorem 3]{Stein1958}, E. M. Stein proves that the classical Littlewood-Paley $g$-function, which can be recovered from (\ref{eq1-19}) by setting $s=0,q=2$, is a bounded operator from $L^1(\bbbr^n)$ into $L^{1,\infty}(\bbbr^n)$. In \cite[Theorem 3.1]{JL2018}, H. Jiao and H. Lin prove that the classical Littlewood-Paley $g$-function defined on a non-homogeneous metric measure space is bounded from $L^1(\mu)$ into $L^{1,\infty}(\mu)$ assuming the boundedness on $L^2(\mu)$, where $\mu$ is a Borel measure satisfying the upper doubling condition. In \cite[Remark 5.1 (ii), Theorem 5.1, and Theorem 5.2]{JL2018}, the authors also prove that the classical Littlewood-Paley $g$-function defined on a non-homogeneous metric measure space is bounded from the atomic Hardy space $\tilde{H}^1(\mu)$ into $L^1(\mu)$, and is bounded from the molecular Hardy space $\tilde{H}^{p,q,\gamma,\epsilon_1}_{mb,\rho}(\mu)$ into $L^p(\mu)$ for some $p\in(0,1)$, assuming the boundedness on $L^2(\mu)$. The boundedness property of the generalized homogeneous Littlewood–Paley $g$-function from Hardy space $H^1(\mu)$ into $L^1(\mu)$ is also studied in \cite{FZ2016} by X. Fu and J. M. Zhao. Moreover, when we let $2\leq p=q<\infty$ and $f\in\dot{F}^0_{q,2}(\bbbr^n)$, the inequality (\ref{eq1-23}) yields
\begin{equation}\label{eq1-25}
\int_{\bbbr^n}\int_0^{\infty}t^{q-1}|\nabla_{n+1}\Pint(f;x,t)|^q dtdx\lesssim\int_{\bbbr^n}|f(x)|^q dx,
\end{equation}
where $f$ appearing on both sides of (\ref{eq1-25}) is the function representative $f(x)\in\functrep_{0}(f)\cap L^q(\bbbr^n)$. When $2\leq p=q<\infty$ and $f(x)$ is a function in $L^q(\bbbr^n)$, the integral in (\ref{eq5.32}) defines a tempered distribution denoted by the symbol $f$,
\begin{equation}\label{eq5.32}
\int_{\bbbr^n}f(x)\!\cdot\!g(x)dx\text{ for $g\in\Sw(\bbbr^n)$},
\end{equation}
and we have $f(x)\in\functrep(f)\subseteq\functrep_{0}(f)$, thus Lemma \ref{lemma12} (i) implies $f\in\dot{F}^0_{q,2}(\bbbr^n)$, hence inequality (\ref{eq1-25}) holds true. This inequality has been used in the proofs of other important results in this paper. Theorem \ref{theorem3} is also used in the significant estimate (\ref{eq4.26}) in the proof of Theorem \ref{theorem4}. The proof of Corollary \ref{corollary1} is complete.
\end{proof}
\begin{proof}[Proof of Corollary \ref{corollary2}]
	Let $f\in L^q(\bbbr^n)$ be a function. By using the defining expression (\ref{eq1-26}) and the condition $1<\lambda<\infty$, we have
	\begin{align}
		&\|\LPG_{\lambda,q}(f)\|_{L^q(\bbbr^n)}^q\nonumber\\
		&=\!\!\!\int_0^{\infty}\!\!\!\!\int_{\bbbr^n}\!\int_{\bbbr^n}\!\!\!
		\big(\frac{1}{1\!+\!\frac{|x-y|}{t}}\big)^{\lambda n}\frac{dx}{t^n}\cdot t^{q-1}
		|\nabla_{n+1}\Pint(f;y,t)|^q dydt\nonumber\\
		&\lesssim\int_0^{\infty}\!\!\int_{\bbbr^n}t^{q-1}|\nabla_{n+1}\Pint(f;y,t)|^q dydt
		\lesssim\|f\|_{L^q(\bbbr^n)}^q,\label{eq1-28}
	\end{align}
	where the constants in (\ref{eq1-28}) depend on $n$, $\lambda$, and some fixed parameters due to the application of estimate (\ref{eq1-25}). If we pick a positive constant $A$ satisfying $\frac{1}{q'}<A<s+\frac{1}{q'}$ for $\frac{1}{q}+\frac{1}{q'}=1$, then we can use H\"{o}lder's inequality to obtain
	\begin{align}
		&(\int_0^{2|x-y|}|\partial_{n+1}\Pint(f;y,t)|t^A\cdot t^{s-A}dt)^q\nonumber\\
		&\lesssim\int_0^{2|x-y|}|\partial_{n+1}\Pint(f;y,t)|^q t^{qA}dt\cdot|x-y|^{q(1+s-A)-1},\label{eq1-29}
	\end{align}
	and the constant depends on $q,s,A$. Now we use the defining expression (\ref{eq1-27}), inequality (\ref{eq1-29}), Fubini's theorem, and estimate (\ref{eq1-25}) in a sequence to obtain
	\begin{align}
		&\|\Remain_{s,q}(f)\|_{L^q(\bbbr^n)}^q\nonumber\\
		&=\!\!\int_{\bbbr^n}\!\int_{\bbbr^n}\!\frac{1}{|x\!-\!y|^{n+sq}}\!\cdot\!(\!\int_0^{2|x-y|}\!\!\!\!\!\!\!\!\!\!|\partial_{n+1}\Pint(f;y,t)|t^s dt)^q dydx\nonumber\\
		&\lesssim\!\!\int_{\bbbr^n}\!\int_{\bbbr^n}\!\!\!|x\!-\!y|^{q-qA-n-1}\!\cdot\!\!
		\int_0^{2|x-y|}\!\!\!\!\!\!\!\!\!\!|\partial_{n+1}\Pint(f;y,t)|^q t^{qA}dtdydx\nonumber\\
		&=\!\!\int_{\bbbr^n}\!\int_0^{\infty}\!\!\!\int_{\frac{t}{2}\leq|x-y|}\!\!\!\!\!\!\!\!\!\!|x\!-\!y|^{q-qA-n-1}dx\!\cdot\!|\partial_{n+1}\Pint(f;y,t)|^q t^{qA}dtdy\nonumber\\
		&\lesssim\int_{\bbbr^n}\int_0^{\infty}t^{q-1}|\partial_{n+1}\Pint(f;y,t)|^q dtdy
		\lesssim\|f\|_{L^q(\bbbr^n)}^q,\label{eq1-30}
	\end{align}
	where the constants in (\ref{eq1-30}) depend on $n,s,q,A$, and some other fixed parameters due to the application of estimate (\ref{eq1-25}).
\end{proof}

\section{The weak type $(p,p)$ boundedness of the $\LPG_{\lambda,q}$-function}\label{proof.of.theorem4}
\begin{proof}[Proof of Theorem \ref{theorem4}]
		Step 1: Assume $f$ is a function in $L^p(\bbbr^n)$. If $$\esssup_{x\in\bbbr^n}|f(x)|\leq\alpha<\infty,$$ then since $2\leq q<\infty$ and $\lambda>1$, we can use Corollary \ref{corollary2} to obtain the following inequality
		\begin{align}
			&\Lebes^n(\{x\in\bbbr^n:\LPG_{\lambda,q}(f)(x)>\alpha\})\nonumber\\
			&\leq\alpha^{-q}\|\LPG_{\lambda,q}(f)\|_{L^q(\bbbr^n)}^q\nonumber\\
			&\lesssim\alpha^{-q}\int_{\bbbr^n}|f(x)|^{q-p}\cdot|f(x)|^p dx\leq\alpha^{-p}\|f\|_{L^p(\bbbr^n)}^p,
			\label{eq4.1}
		\end{align}
		and hence inequality (\ref{eq1-31}) is proved in this case. If $$0<\alpha<\esssup_{x\in\bbbr^n}|f(x)|,$$ we apply Lemma \ref{lemma3} to the function $f\in L^p(\bbbr^n)$ and the set $\Omega=\{x\in\bbbr^n:M(|f|^p)(x)>\alpha^p\}$, and obtain a collection of Whitney cubes $\{I_j\}_j$ and two functions, the ``good'' function $g$ and the ``bad'' function $b$, where $\{I_j\}_j$, $g$ and $b$ satisfy all the conclusions of Lemma \ref{lemma3}. Then by Lemma \ref{lemma3} (v), we have the pointwise estimate $\LPG_{\lambda,q}(f)(x)\lesssim\LPG_{\lambda,q}(g)(x)+\LPG_{\lambda,q}(b)(x)$ and hence the following inequality
		\begin{align*}
			&\Lebes^n(\{x\in\bbbr^n:\LPG_{\lambda,q}(f)(x)>\alpha\})\\
			&\leq\Lebes^n(\{x\in\bbbr^n:\frac{\alpha}{2}\lesssim\LPG_{\lambda,q}(g)(x)\})
			+\Lebes^n(\{x\in\bbbr^n:\frac{\alpha}{2}\lesssim\LPG_{\lambda,q}(b)(x)\}),
		\end{align*}
		where the implicit constants in the above inequalities only depend on $q$. Therefore to prove (\ref{eq1-31}), it suffices to prove the following two inequalities,
		\begin{align}
			&\Lebes^n(\{x\in\bbbr^n:\frac{\alpha}{2}\lesssim\LPG_{\lambda,q}(g)(x)\})
			\lesssim\alpha^{-p}\|f\|_{L^p(\bbbr^n)}^p,\label{eq4.2}\\
			&\Lebes^n(\{x\in\bbbr^n:\frac{\alpha}{2}\lesssim\LPG_{\lambda,q}(b)(x)\})
			\lesssim\alpha^{-p}\|f\|_{L^p(\bbbr^n)}^p,\label{eq4.3}
		\end{align}
		where the values of implicit constants only depend on some fixed parameters. By Lemma \ref{lemma3} (viii), we have 
		\begin{equation}\label{eq4.4}
			\|g\|_{L^q(\bbbr^n)}^q\lesssim\alpha^{q-p}\|g\|_{L^p(\bbbr^n)}^p
			\lesssim\alpha^{q-p}\|f\|_{L^p(\bbbr^n)}^p,
		\end{equation}
		where the constants in (\ref{eq4.4}) depend on $n,p,q$. Inequality (\ref{eq4.4}) shows that $g$ is a function in $L^q(\bbbr^n)$, and hence we can apply Corollary \ref{corollary2} to obtain
		\begin{align}
			&\Lebes^n(\{x\in\bbbr^n:\frac{\alpha}{2}\lesssim\LPG_{\lambda,q}(g)(x)\})\nonumber\\
			&\lesssim(\frac{2}{\alpha})^q\|\LPG_{\lambda,q}(g)\|_{L^q(\bbbr^n)}^q
			\lesssim\alpha^{-q}\|g\|_{L^q(\bbbr^n)}^q\lesssim\alpha^{-p}\|f\|_{L^p(\bbbr^n)}^p,\label{eq4.6}
		\end{align}
		where the constants in (\ref{eq4.6}) depend on $n,p,q$ and the implicit constant coming from Corollary \ref{corollary2}. Inequality (\ref{eq4.2}) is proved. Now we prove inequality (\ref{eq4.3}). We consider the following equation
		\begin{align}
			&\Lebes^n(\{x\in\bbbr^n:\frac{\alpha}{2}\lesssim\LPG_{\lambda,q}(b)(x)\})\nonumber\\
			&=\Lebes^n(\{x\in\bbbr^n:\alpha^q\lesssim\LPG_{\lambda,q}(b)(x)^q\}),\label{eq4.7}
		\end{align}
		where the implicit constants depend on $q$. We recall the relations between two Whitney cubes given in (\ref{near}) and (\ref{far}), and we also use the decomposition
		\begin{equation}\label{eq4.8}
			\bbbr^n=\Omega^c\bigcup\Omega=\Omega^c\bigcup\bigg(\bigcup_m I_m\bigg).
		\end{equation}
		By Lemma \ref{lemma3} (v), the function $b$ is supported in $\Omega$ and if $\chi_{I_j}$ denotes the characteristic function of the cube $I_j$ and if $b_j=b\cdot\chi_{I_j}$, then we have the equation
		\begin{equation}\label{eq4.9}
			b=\sum_j b_j=\bnearIm+\bfarIm,
		\end{equation}
		and for each $m$, we denote
		\begin{equation}\label{eq4.10}
			\bnearIm=\sum_{\substack{j\\I_j\text{ near }I_m}}b_j\quad\text{and}\quad
			\bfarIm=\sum_{\substack{j\\I_j\text{ far from }I_m}}b_j,
		\end{equation}
		where ``$(n.m.)$'' comes from the phrase ``near $I_m$'' while ``$(f.m.)$'' comes from the phrase ``far from $I_m$''. By using the defining expression (\ref{eq1-26}), we can estimate $\LPG_{\lambda,q}(b)(x)^q$ from above by a constant multiple of the sum of the following $(2n+3)$ terms,
\begin{align}
A_k(x)\!\!&:=\!\!\!\int_0^{\infty}\!\!\!\!\int_{\Omega^c}\!\!\!\big(\frac{t}{t\!+\!|x\!-\!y|}\big)^{\lambda n}\!\!\!\cdot\!t^{q-n-1}|\partial_k\Pint(b;y,t)|^q dydt,\label{eq4.11}\\
B_k(x)\!\!&:=\!\!\!\sum_m\!\!\!\int_0^{\infty}\!\!\!\!\int_{I_m}\!\!\!\!\big(\frac{t}{t\!+\!|x\!-\!y|}\big)^{\lambda n}\!\!\!\cdot\!t^{q-n-1}|\partial_k\Pint(\bfarIm;y,t)|^q dydt,\label{eq4.12}\\
D(x)\!\!&:=\!\!\!\sum_m\!\!\!\int_0^{\infty}\!\!\!\!\int_{I_m}\!\!\!\!\big(\frac{t}{t\!+\!|x\!-\!y|}\big)^{\lambda n}\!\!\!\cdot\!t^{q-n-1}|\nabla_{n+1}\Pint(\bnearIm;y,t)|^q dydt,\label{eq4.13}
\end{align}
		where $1\leq k\leq n+1$ in (\ref{eq4.11}) and (\ref{eq4.12}), and the constant depends on $n,q$. Combining the above pointwise estimate with (\ref{eq4.7}), we can obtain
		\begin{equation}\label{eq4.14}
			\alpha^q\leq C_1\cdot(\sum_{k=1}^{n+1}A_k(x)+\sum_{k=1}^{n+1}B_k(x)+D(x)),
		\end{equation}
		where $C_1=C_1(n,q)$ and $x$ belongs to the set $\{x\!\in\!\bbbr^n\!:\!\alpha^q\!\lesssim\!\LPG_{\lambda,q}(b)(x)^q\}$. Inequality (\ref{eq4.14}) says that at least one of the terms $D(x)$, $A_k(x)$ and $B_k(x)$ for $1\leq k\leq n+1$ is greater than or equal to $\frac{\alpha^q}{C_1(2n+3)}$, therefore we can estimate (\ref{eq4.7}) from above by the sum of the following,
		\begin{align}
			&\sum_{k=1}^{n+1}\Lebes^n(\{x\in\bbbr^n:\alpha^q\lesssim A_k(x)\}),\label{eq4.15}\\
			&\sum_{k=1}^{n+1}\Lebes^n(\{x\in\bbbr^n:\alpha^q\lesssim B_k(x)\}),\label{eq4.16}\\
			&\Lebes^n(\{x\in\bbbr^n:\alpha^q\lesssim D(x)\}),\label{eq4.17}
		\end{align}
		and the implicit constants in (\ref{eq4.15}), (\ref{eq4.16}), (\ref{eq4.17}) depend only on $n,q$. Therefore to prove (\ref{eq4.3}), it suffices to prove that each and every one of (\ref{eq4.15}), (\ref{eq4.16}), (\ref{eq4.17}) is dominated by a constant multiple of $\alpha^{-p}\|f\|_{L^p(\bbbr^n)}^p$. We use the following notations for $1\leq k\leq n+1$,
\begin{align}
\bar{A}_k(x)\!\!&:=\!\!\!\int_0^{\infty}\!\!\!\!\int_{\Omega^c}\!\!\!\big(\frac{t}{t\!+\!|x\!-\!y|}\big)^{\lambda n}\!\!\!\cdot\!t^{-n}|\partial_k\Pint(b;y,t)|dydt,\label{eq4.18}\\
\bar{B}_k(x)\!\!&:=\!\!\!\sum_m\!\!\!\int_0^{\infty}\!\!\!\!\int_{I_m}\!\!\!\!\big(\frac{t}
{t\!+\!|x\!-\!y|}\big)^{\lambda n}\!\!\!\cdot\!t^{-n}|\partial_k\Pint(\bfarIm;y,t)|dydt.\label{eq4.19}
\end{align}\\
		
		Step 2: We estimate the term (\ref{eq4.15}). From (\ref{eq4.11}) and (\ref{eq4.18}), we can apply Lemma \ref{lemma6} (1) to the term $|\partial_k\Pint(b;y,t)|^{q-1}$ and obtain $$\alpha^q\lesssim A_k(x)\lesssim\alpha^{q-1}\bar{A}_k(x),$$ where $x$ belongs to the set $\{x\in\bbbr^n:\alpha^q\lesssim A_k(x)\}$ and $1\leq k\leq n+1$, and the constants depend on $n,p,q$. Therefore invoking Chebyshev's inequality, Fubini's theorem, Lemma \ref{lemma7} (\ref{eq2-13}), Lemma \ref{lemma3} (vi) and (i) in a sequence yields the following estimate
		\begin{align}
			&(\ref{eq4.15})\lesssim\sum_{k=1}^{n+1}\Lebes^n(\{x\in\bbbr^n:\alpha\lesssim\bar{A}_k(x)\})\nonumber\\
			&\lesssim\sum_{k=1}^{n+1}\frac{1}{\alpha}\int_0^{\infty}\!\!\!\!\int_{\Omega^c}\!\!\int_{\bbbr^n}
			\big(\frac{t}{t+|x-y|}\big)^{\lambda n}\frac{dx}{t^n}\cdot |\partial_k\Pint(b;y,t)|dydt\nonumber\\
			&\lesssim\sum_{k=1}^{n+1}\frac{1}{\alpha}\int_0^{\infty}\!\!\!\!\int_{\Omega^c}|\partial_k\Pint(b;y,t)|dydt
			\nonumber\\
			&\lesssim\frac{1}{\alpha}\cdot\sum_j\int_{I_j}|b(z)|dz\lesssim\sum_j\Lebes^n(I_j)
			\lesssim\alpha^{-p}\|f\|_{L^p(\bbbr^n)}^p,\label{eq4.20}
		\end{align}
		where the constants in (\ref{eq4.20}) depend on $n,p,q,\lambda$, and we also use the fact that for $1<p<q$ and $\lambda=\frac{q}{p}>1$,
		\begin{equation}\label{eq4.42}
			\int_{\bbbr^n}(1+|x|)^{-n\lambda}dx<\infty.
		\end{equation}
		
		Step 3: We estimate the term (\ref{eq4.16}). From (\ref{eq4.12}) and (\ref{eq4.19}), we can apply Lemma \ref{lemma6} (2) to the term $|\partial_k\Pint(\bfarIm;y,t)|^{q-1}$ and obtain $$\alpha^q\lesssim B_k(x)\lesssim\alpha^{q-1}\bar{B}_k(x),$$ where $x$ belongs to the set $\{x\in\bbbr^n:\alpha^q\lesssim B_k(x)\}$ and $1\leq k\leq n+1$, and the constants depend on $n,p,q$. Therefore invoking Chebyshev's inequality, Fubini's theorem, Lemma \ref{lemma7} (\ref{eq2-14}), Lemma \ref{lemma3} (vi) and (i) in a sequence yields the following estimate
\begin{align}
&(\ref{eq4.16})\lesssim\sum_{k=1}^{n+1}\Lebes^n(\{x\in\bbbr^n:\alpha\lesssim\bar{B}_k(x)\})\nonumber\\
&\lesssim\sum_{k=1}^{n+1}\!\frac{1}{\alpha}\!\sum_m\!\!\int_0^{\infty}\!\!\!\!\int_{I_m}\!\!\int_{\bbbr^n}\!\!\!\!\big(\frac{t}{t\!+\!|x\!-\!y|}\big)^{\lambda n}\frac{dx}{t^n}
|\partial_k\Pint(\bfarIm;y,t)|dydt\nonumber\\
&\lesssim\sum_{k=1}^{n+1}\frac{1}{\alpha}\sum_m\int_0^{\infty}\!\!\!\!\int_{I_m}|\partial_k\Pint(\bfarIm;y,t)|dydt\nonumber\\
&\lesssim\frac{1}{\alpha}\cdot\sum_j\int_{I_j}|b(z)|dz\lesssim\sum_j\Lebes^n(I_j)
\lesssim\alpha^{-p}\|f\|_{L^p(\bbbr^n)}^p,\label{eq4.21}
\end{align}
		where the constants in (\ref{eq4.21}) depend on $n,p,q,\lambda$, and we also use the fact (\ref{eq4.42}).\\
		
		Step 4: Finally, we prove that (\ref{eq4.17}) can be estimated from above by a constant multiple of $\alpha^{-p}\|f\|_{L^p(\bbbr^n)}^p$. By Lemma \ref{lemma3} (i) and Chebyshev's inequality, we have
		\begin{align}
			&(\ref{eq4.17})
			\!=\!\Lebes^n(\{x\in\Omega:\alpha^q\lesssim D(x)\})\!+\!\Lebes^n(\{x\notin\Omega:\alpha^q\lesssim D(x)\})\nonumber\\
			&\!\lesssim\!\Lebes^n(\Omega)\!+\!\alpha^{-q}\!\!\!\int_{\Omega^c}\!\!\!D(x)dx\!
			\lesssim\!\alpha^{-p}\|f\|_{L^p(\bbbr^n)}^p\!+\!\alpha^{-q}\!\!\!\int_{\Omega^c}\!\!\!D(x)dx,\label{eq4.22}
		\end{align}
		where the constants in (\ref{eq4.22}) depend on $n,q$. If $I_m$ is a Whitney cube and $x\in\Omega^c$, $y\in I_m$, then we can use Lemma \ref{lemma1} (b) to obtain $$|x-y|\geq\dist(I_m,\Omega^c)\geq\sqrt{n}\cdot l(I_m).$$ Therefore we recall the defining expression (\ref{eq4.13}) and use Fubini's theorem to estimate the term $\alpha^{-q}\int_{\Omega^c}D(x)dx$ from above by the following expression
		\begin{align}
			&\alpha^{-q}\sum_m\int_0^{\infty}\!\!\!\!\int_{I_m}\!\!\int_{\Omega^c}
			\frac{t^{\lambda n+q-n-1}}{(t+|x-y|)^{\lambda n}}dx\cdot|\nabla_{n+1}\Pint(\bnearIm;y,t)|^q dydt\nonumber\\
			&\lesssim\alpha^{-q}\sum_m\int_0^{\infty}\!\!\!\!\int_{I_m}\!\!\int_{|z|\geq\sqrt{n}l(I_m)}
			\big(\frac{t}{t+|z|}\big)^{\lambda n}dz\nonumber\\
			&\quad\cdot t^{q-n-1}|\nabla_{n+1}\Pint(\bnearIm;y,t)|^q dydt\nonumber\\
			&\lesssim\alpha^{-q}\sum_m\int_0^{\infty}\!\!\!\!\int_{I_m}\!\!\int_{|z|\geq\frac{\sqrt{n}l(I_m)}{t}}
			\frac{t^{q-1}}{|z|^{\lambda n}}dz\cdot|\nabla_{n+1}\Pint(\bnearIm;y,t)|^q dydt\nonumber\\
			&\lesssim\alpha^{-q}\sum_m\Lebes^n(I_m)^{1-\lambda}\!\!\int_0^{\infty}\!\!\!\!\int_{I_m}\!\!
			t^{\lambda n-n+q-1}\big|\nabla_{n+1}\Pint(\bnearIm;y,t)\big|^q dydt,\label{eq4.23}
		\end{align}
		where the constants in (\ref{eq4.23}) depend on $n,\lambda$, and we also use the formula $\Lebes^n(I_m)=l(I_m)^n$. For every $m$, the iterated integral in the last line of (\ref{eq4.23}) can be estimated from above by the following expression
		\begin{equation}\label{eq4.24}
			\int_{\bbbr^n}\int_0^{\infty}t^{q+qn(\frac{\lambda}{q}-\frac{1}{q})-1}
			\big|\nabla_{n+1}\Pint(\bnearIm;y,t)\big|^q dtdy,
		\end{equation}
		and we can identify (\ref{eq4.24}) as $$\|\LPg_{-n(\frac{\lambda}{q}-\frac{1}{q}),q}(\bnearIm)\|_{L^q(\bbbr^n)}^q,$$ where $\LPg_{-n(\frac{\lambda}{q}-\frac{1}{q}),q}(\bnearIm)(\cdot)$ is the Littlewood-Paley-Poisson function defined in (\ref{eq1-19}). From our assumption $\lambda=\frac{q}{p}$ and $0<n(\frac{1}{p}-\frac{1}{q})<1$, we see that
		\begin{equation}\label{eq4.25}
		-n\leq-1<\kappa:=-n(\frac{\lambda}{q}-\frac{1}{q})=-n(\frac{1}{p}-\frac{1}{q})<0.
		\end{equation}
Assume for now that $\bnearIm(x)$ is a function in $L^p(\bbbr^n)$ for $1<p<q<\infty$ and defines a tempered distribution by the following integral
\begin{equation}\label{eq4.34}
\int_{\bbbr^n}\bnearIm(x)\cdot g(x)dx\text{ for $g\in\Sw(\bbbr^n)$},
\end{equation} 
and we denote this tempered distribution by the symbol $\bnearIm$ and we also assume $\bnearIm\in\Sw'(\bbbr^n)$ is an element in the homogeneous Triebel-Lizorkin space $\dot{F}^{\kappa}_{q,q}(\bbbr^n)$. These assumptions will be verified later, then all the conditions of Theorem \ref{theorem3} are satisfied and we can apply Theorem \ref{theorem3} to obtain
		\begin{align}
			&(\ref{eq4.24})=\|\LPg_{\kappa,q}(\bnearIm)\|_{L^q(\bbbr^n)}^q\nonumber\\
			&\lesssim\|\bnearIm\|_{\dot{F}^{\kappa}_{q,q}(\bbbr^n)}^q\sim
			\|\Lift_{\kappa}(\bnearIm)\|_{\dot{F}^{0}_{q,q}(\bbbr^n)}^q,\label{eq4.26}
		\end{align}
where the last step is due to the well-known lifting property introduced in \cite[section 5.2.3]{1983functionspaces} by H. Triebel, and $\Lift_{\kappa}(\bnearIm)\!=\!\iFT_n[|\xi|^{\kappa}\cdot\FT_n\bnearIm]$. We deduce from Lemma \ref{lemma13} that $\Lift_{\kappa}(\bnearIm)\in\Sw'(\bbbr^n)$ and the function (\ref{eq4.43}) is in $\functrep_0(\Lift_{\kappa}(\bnearIm))\cap L^q(\bbbr^n)$,
\begin{equation}\label{eq4.43}
x\in\bbbr^n\longmapsto
\frac{\varGamma(\frac{\kappa+n}{2})}{\pi^{\kappa+\frac{n}{2}}\varGamma(\frac{-\kappa}{2})}
\int_{\bbbr^n}\bnearIm(y)\cdot|y-x|^{-n-\kappa}dy.
\end{equation}
Furthermore, we have
\begin{equation}\label{eq4.30}
\|\Lift_{\kappa}(\bnearIm)\|_{L^q(\bbbr^n)}\lesssim\|\bnearIm\|_{L^p(\bbbr^n)},
\end{equation}
and the implicit constant in (\ref{eq4.30}) depends on $n$, $p$, and $q$. Since $2\leq q<\infty$, we can use formula (\ref{eq1-14}) and Lemma \ref{lemma12} (i) to deduce
\begin{equation}\label{eq4.27}
\|\Lift_{\kappa}(\bnearIm)\|_{\dot{F}^{0}_{q,q}(\bbbr^n)}\!\leq\!
\|\Lift_{\kappa}(\bnearIm)\|_{\dot{F}^{0}_{q,2}(\bbbr^n)}\!\lesssim\!
\|\Lift_{\kappa}(\bnearIm)\|_{L^q(\bbbr^n)}.
\end{equation}
Now we recall (\ref{eq4.10}) and (\ref{near}). Since Whitney cubes have disjoint interiors and since Lemma \ref{lemma2} (2) implies that the total volume of all the Whitney cube $I_j$'s, which satisfy the condition ``$I_j$ near $I_m$'' for some fixed $m$, is bounded by the volume of the ball $B^n(c(I_m),11(\sqrt{n}+1)^2 l(I_m))\subseteq\bbbr^n$, we have
\begin{equation}\label{eq4.31}
\sum_{\substack{j\\I_j\text{ near }I_m}}\Lebes^n(I_j)\lesssim l(I_m)^n\sim\Lebes^n(I_m),
\end{equation}
and the implicit constant depends only on $n$. We can apply Lemma \ref{lemma3} (vi) to obtain
\begin{equation}\label{eq4.32}
\|\bnearIm\|_{L^p(\bbbr^n)}^p\!\!=\!\!\!\!\sum_{\substack{j\\I_j\text{ near }I_m}}\!\!\!\!\!\int_{I_j}\!\!|b(z)|^p dz\!\lesssim\!
\alpha^p\!\!\!\!\sum_{\substack{j\\I_j\text{ near }I_m}}\!\!\!\!\!\Lebes^n(I_j)\!\lesssim\!\alpha^p\Lebes^n(I_m),
\end{equation}
where the constants depend on $n,p$. Estimate (\ref{eq4.32}) shows $\bnearIm(x)$ is a function in $L^p(\bbbr^n)$ for $1<p<q<\infty$. Combining (\ref{eq4.26}), (\ref{eq4.27}), (\ref{eq4.30}), (\ref{eq4.32}) and the condition $\lambda=\frac{q}{p}$ yields
\begin{equation}\label{eq4.33}
\|\bnearIm\|_{\dot{F}^{\kappa}_{q,q}(\bbbr^n)}^q\lesssim\|\bnearIm\|_{L^p(\bbbr^n)}^q
\lesssim\alpha^q\Lebes^n(I_m)^{\lambda},
\end{equation}
and since $\alpha^q\Lebes^n(I_m)^{\lambda}$ is controlled by $\alpha^q\Lebes^n(\Omega)^{\lambda}\lesssim\|f\|_{L^p(\bbbr^n)}^q<\infty$, the estimate (\ref{eq4.33}) also shows that $\bnearIm\in\Sw'(\bbbr^n)$ is an element of $\dot{F}^{\kappa}_{q,q}(\bbbr^n)$ and hence justifies the application of Theorem \ref{theorem3} in (\ref{eq4.26}). Combining (\ref{eq4.22}), (\ref{eq4.23}), (\ref{eq4.24}), (\ref{eq4.26}), (\ref{eq4.33}) all together and invoking Lemma \ref{lemma3} (i), we can conclude that (\ref{eq4.17}) can be estimated from above by a constant multiple of $\alpha^{-p}\|f\|_{L^p(\bbbr^n)}^p$. Inequalities (\ref{eq4.20}), (\ref{eq4.21}) and the above conclusion complete the proof of Theorem \ref{theorem4}.
\end{proof}

\section{The weak type $(p,p)$ boundedness of the $\Remain_{s,q}$-function}\label{proof.of.theorem5}
\begin{proof}[Proof of Theorem \ref{theorem5}]
Step 1: Assume $f$ is a function in $L^p(\bbbr^n)$. If $$\esssup_{x\in\bbbr^n}|f(x)|\leq\alpha<\infty,$$ then since $2\leq q<\infty$ and $0<s<1$, we can use Corollary \ref{corollary2} and the same method as given in (\ref{eq4.1}), with $\Remain_{s,q}(f)$ in place of $\LPG_{\lambda,q}(f)$, to obtain inequality (\ref{eq1-32}). If 
\begin{equation}\label{eq5.23}
0<\alpha<\esssup_{x\in\bbbr^n}|f(x)|,
\end{equation}
we apply Lemma \ref{lemma3} to the function $f\in L^p(\bbbr^n)$ and the set $\Omega=\{x\in\bbbr^n:M(|f|^p)(x)>\alpha^p\}$, and obtain a collection of Whitney cubes $\{I_j\}_j$ and two functions, the ``good'' function $g$ and the ``bad'' function $b$, where $\{I_j\}_j$, $g$ and $b$ satisfy all the conclusions of Lemma \ref{lemma3}. Then by Lemma \ref{lemma3} (v), we have the following inequality
\begin{align*}
&\Lebes^n(\{x\in\bbbr^n:\Remain_{s,q}(f)(x)>\alpha\})\\
&\!\leq\!\Lebes^n(\{x\!\in\!\bbbr^n\!:\!\frac{\alpha}{2}\!\lesssim\!\Remain_{s,q}(g)(x)\})\!+\!
\Lebes^n(\{x\!\in\!\bbbr^n\!:\!\frac{\alpha}{2}\!\lesssim\!\Remain_{s,q}(b)(x)\}),
\end{align*}
		where the implicit constants only depend on $q$. Therefore to prove (\ref{eq1-32}), it suffices to prove the following two inequalities,
		\begin{align}
			&\Lebes^n(\{x\in\bbbr^n:\frac{\alpha}{2}\lesssim\Remain_{s,q}(g)(x)\})
			\lesssim\alpha^{-p}\|f\|_{L^p(\bbbr^n)}^p,\label{eq5.1}\\
			&\Lebes^n(\{x\in\bbbr^n:\frac{\alpha}{2}\lesssim\Remain_{s,q}(b)(x)\})
			\lesssim\alpha^{-p}\|f\|_{L^p(\bbbr^n)}^p,\label{eq5.2}
		\end{align}
		where the values of implicit constants only depend on some fixed parameters. By Lemma \ref{lemma3} (viii), we see that when $\alpha$ satisfies (\ref{eq5.23}) and $f\in L^p(\bbbr^n)$, inequality (\ref{eq4.4}) is still true and hence $g$ is a function in $L^q(\bbbr^n)$, therefore we can apply Corollary \ref{corollary2} to obtain
		\begin{align}
			&\Lebes^n(\{x\in\bbbr^n:\frac{\alpha}{2}\lesssim\Remain_{s,q}(g)(x)\})\nonumber\\
			&\lesssim(\frac{2}{\alpha})^q\|\Remain_{s,q}(g)\|_{L^q(\bbbr^n)}^q
			\lesssim\alpha^{-q}\|g\|_{L^q(\bbbr^n)}^q\lesssim\alpha^{-p}\|f\|_{L^p(\bbbr^n)}^p,\label{eq5.3}
		\end{align}
		where the constants in (\ref{eq5.3}) depend on $n,p,q$ and the implicit constant coming from Corollary \ref{corollary2}. Inequality (\ref{eq5.1}) is proved. Now we prove inequality (\ref{eq5.2}). We consider the following equation
		\begin{align}
			&\Lebes^n(\{x\in\bbbr^n:\frac{\alpha}{2}\lesssim\Remain_{s,q}(b)(x)\})\nonumber\\
			&=\Lebes^n(\{x\in\bbbr^n:\alpha^q\lesssim\Remain_{s,q}(b)(x)^q\}),\label{eq5.4}
		\end{align}
where the implicit constants depend on $q$. We recall the relations between two Whitney cubes given in (\ref{near}) and (\ref{far}), and we also use the decomposition (\ref{eq4.8}). If $\chi_{I_j}$ denotes the characteristic function of the cube $I_j$ and if $b_j=b\cdot\chi_{I_j}$, then we can continue using the equation (\ref{eq4.9}) and $\bnearIm$, $\bfarIm$ have the same meaning as in (\ref{eq4.10}). By using the defining expression (\ref{eq1-27}), we can estimate $\Remain_{s,q}(b)(x)^q$ from above by a constant multiple of the sum of the following three terms,
\begin{align}
&\Remain_{s,q}^1(b)(x)\!\!:=\!\!\int_{\Omega^c}\!\!\!\!\!|x\!-\!y|^{-n-sq}(\!\int_0^{2|x-y|}\!\!\!\!\!\!\!\!\!\!\!\!\!|\partial_{n+1}\Pint(b;y,t)|t^s dt)^q dy,\label{eq5.5}\\
&\Remain_{s,q}^2(b)(x)\!\!:=\!\!\sum_m\!\!\int_{I_m}\!\!\!\!\!|x\!-\!y|^{-n-sq}(\!\int_0^{2|x-y|}\!\!\!\!\!\!\!\!\!\!\!\!\!|\partial_{n+1}\Pint(\bfarIm;y,t)|t^s dt)^q dy,\label{eq5.6}\\
&\Remain_{s,q}^3(b)(x)\!\!:=\!\!\sum_m\!\!\int_{I_m}\!\!\!\!\!|x\!-\!y|^{-n-sq}(\!\int_0^{2|x-y|}\!\!\!\!\!\!\!\!\!\!\!\!\!|\partial_{n+1}\Pint(\bnearIm;y,t)|t^s dt)^q dy,\label{eq5.7}
\end{align}
and the constant depends on $q$. Therefore we can estimate (\ref{eq5.4}) from above by the sum of the following,
		\begin{align}
			&\Lebes^n(\{x\in\bbbr^n:\alpha^q\lesssim\Remain_{s,q}^1(b)(x)\}),\label{eq5.8}\\
			&\Lebes^n(\{x\in\bbbr^n:\alpha^q\lesssim\Remain_{s,q}^2(b)(x)\}),\label{eq5.9}\\
			&\Lebes^n(\{x\in\bbbr^n:\alpha^q\lesssim\Remain_{s,q}^3(b)(x)\}),\label{eq5.10}
		\end{align}
		and the implicit constants in (\ref{eq5.8}), (\ref{eq5.9}), (\ref{eq5.10}) depend only on $q$. Therefore to prove (\ref{eq5.2}), it suffices to prove that each and every one of (\ref{eq5.8}), (\ref{eq5.9}), (\ref{eq5.10}) is dominated by a constant multiple of $\alpha^{-p}\|f\|_{L^p(\bbbr^n)}^p$. We use the following notations in the proof below,
\begin{align}
&\bar{\Remain}_{s,q}^1(b)(x)\!\!:=\!\!\int_{\Omega^c}\!\frac{1}{|x\!-\!y|^{n+s}}\!\!\int_0^{2|x-y|}\!\!\!\!\!\!\!\!\!\!\!\!\!|\partial_{n+1}\Pint(b;y,t)|t^s dtdy,\label{eq5.11}\\
&\bar{\Remain}_{s,q}^2(b)(x)\!\!:=\!\!\sum_m\!\!\int_{I_m}\!\frac{1}{|x\!-\!y|^{n+s}}\!\!\int_0^{2|x-y|}\!\!\!\!\!\!\!\!\!\!\!\!\!|\partial_{n+1}\Pint(\bfarIm;y,t)|t^s dtdy.\label{eq5.12}
\end{align}\\
		
Step 2: To estimate (\ref{eq5.8}), we use Lemma \ref{lemma6} (1) for $k=n+1$ and $y\in\Omega^c$ to obtain
\begin{align}
&(\int_0^{2|x-y|}\!\!\!\!\!\!\!|\partial_{n+1}\Pint(b;y,t)|t^s dt)^{q-1}\nonumber\\
&\lesssim(\int_0^{2|x-y|}\!\!\!\!\!\!\!\alpha t^{s-1}dt)^{q-1}\lesssim\alpha^{q-1}|x-y|^{s(q-1)},\label{eq5.13}
\end{align} 
		and the constants depend on $n,p,q$. Inserting (\ref{eq5.13}) into (\ref{eq5.5}) yields 
		$$\alpha^q\lesssim\Remain_{s,q}^1(b)(x)\lesssim\alpha^{q-1}\bar{\Remain}_{s,q}^1(b)(x),$$ where $x$ belongs to the set $\{x\in\bbbr^n:\alpha^q\lesssim\Remain_{s,q}^1(b)(x)\}$. Therefore invoking Chebyshev's inequality, Fubini's theorem, Lemma \ref{lemma7} (\ref{eq2-13}) for $k=n+1$, Lemma \ref{lemma3} (vi) and (i) in a sequence yields the following estimate
		\begin{align}
			&(\ref{eq5.8})\!\lesssim\!\Lebes^n(\{x\in\bbbr^n\!:\!\alpha\!\lesssim\!\bar{\Remain}_{s,q}^1(b)(x)\})
			\!\lesssim\!\alpha^{-1}\!\!\!\int_{\bbbr^n}\!\!\!\bar{\Remain}_{s,q}^1(b)(x)dx\nonumber\\
			&=\alpha^{-1}\!\!\!\int_0^{\infty}\!\!\!\int_{\Omega^c}\int_{\frac{t}{2}\leq|x-y|}\!\frac{1}{|x-y|^{n+s}}dx
			\big|\partial_{n+1}\Pint(b;y,t)\big|t^s dydt\nonumber\\
			&\lesssim\alpha^{-1}\int_0^{\infty}\int_{\Omega^c}t^{-s}\cdot
			\big|\partial_{n+1}\Pint(b;y,t)\big|t^s dydt\nonumber\\
			&\lesssim\alpha^{-1}\sum_j\int_{I_j}|b(z)|dz\lesssim\sum_j\Lebes^n(I_j)
			\lesssim\alpha^{-p}\|f\|_{L^p(\bbbr^n)}^p,\label{eq5.14}
		\end{align}
		where the implicit constants in (\ref{eq5.14}) depend on $n,p,q$.\\
		
Step 3: To estimate (\ref{eq5.9}), we use Lemma \ref{lemma6} (2) for $k=n+1$ and $y\in I_m$ for some Whitney cube $I_m$ to obtain
\begin{align}
&(\int_0^{2|x-y|}\!\!\!\!\!\!\!|\partial_{n+1}\Pint(\bfarIm;y,t)|t^s dt)^{q-1}\nonumber\\
&\lesssim(\int_0^{2|x-y|}\!\!\!\!\!\!\!\alpha t^{s-1}dt)^{q-1}\lesssim\alpha^{q-1}|x-y|^{s(q-1)},\label{eq5.15}
\end{align} 
and the constants depend on $n,p,q$. Inserting (\ref{eq5.15}) into (\ref{eq5.6}) yields 
$$\alpha^q\lesssim\Remain_{s,q}^2(b)(x)\lesssim\alpha^{q-1}\bar{\Remain}_{s,q}^2(b)(x),$$ where $x$ belongs to the set $\{x\in\bbbr^n:\alpha^q\lesssim\Remain_{s,q}^2(b)(x)\}$. Therefore invoking Chebyshev's inequality, Fubini's theorem, Lemma \ref{lemma7} (\ref{eq2-14}) for $k=n+1$, Lemma \ref{lemma3} (vi) and (i) in a sequence yields the following estimate
		\begin{align}
			&(\ref{eq5.9})\!\lesssim\!\Lebes^n(\{x\in\bbbr^n\!:\!\alpha\!\lesssim\!\bar{\Remain}_{s,q}^2(b)(x)\})\!
			\lesssim\!\alpha^{-1}\!\!\!\int_{\bbbr^n}\!\!\!\bar{\Remain}_{s,q}^2(b)(x)dx\nonumber\\
			&=\alpha^{-1}\!\sum_m\!\!\int_0^{\infty}\!\!\!\!\int_{I_m}\!\int_{\frac{t}{2}\leq|x-y|}\!\!\!\!\!\!\!\!\!\!
			|x-y|^{-n-s}dx\big|\partial_{n+1}\Pint(\bfarIm;y,t)\big|t^s dydt\nonumber\\
			&\lesssim\alpha^{-1}\sum_m\int_0^{\infty}\int_{I_m}t^{-s}\cdot|\partial_{n+1}\Pint(\bfarIm;y,t)|t^s dydt
			\nonumber\\
			&\lesssim\alpha^{-1}\sum_j\int_{I_j}|b(z)|dz\lesssim\sum_j\Lebes^n(I_j)
			\lesssim\alpha^{-p}\|f\|_{L^p(\bbbr^n)}^p,\label{eq5.16}
		\end{align}
		where the implicit constants in (\ref{eq5.16}) depend on $n,p,q$.\\
		
		Step 4: Finally we estimate the term (\ref{eq5.10}). By Lemma \ref{lemma3} (i) and Chebyshev's inequality, we have
		\begin{align}
			&(\ref{eq5.10})=\Lebes^n(\{x\in\Omega:\alpha^q\lesssim\Remain_{s,q}^3(b)(x)\})\nonumber\\
			&+\Lebes^n(\{x\notin\Omega:\alpha^q\lesssim \Remain_{s,q}^3(b)(x)\})\nonumber\\
			&\lesssim\Lebes^n(\Omega)+\alpha^{-q}\int_{\Omega^c}\Remain_{s,q}^3(b)(x)dx\nonumber\\
			&\lesssim\alpha^{-p}\|f\|_{L^p(\bbbr^n)}^p+\alpha^{-q}\int_{\Omega^c}\Remain_{s,q}^3(b)(x)dx,\label{eq5.17}
		\end{align}
		where the constants in (\ref{eq5.17}) depend on $n,q$. Now we estimate the term $\alpha^{-q}\int_{\Omega^c}\Remain_{s,q}^3(b)(x)dx$. Recall (\ref{eq1-6}), which is applicable due to (\ref{eq4.32}), then we have
		\begin{align}
			&\int_0^{2|x-y|}|\partial_{n+1}\Pint(\bnearIm;y,t)|t^s dt\nonumber\\
			&\!\lesssim\!\!\!\int_{\bbbr^n}\!\!\!|\bnearIm(y-z)|\!\!\int_0^{\infty}\!\!\!\!\frac{t^s}{(t^2+|z|^2)^{\frac{n+1}{2}}}\!+\!\frac{t^{2+s}}{(t^2+|z|^2)^{\frac{n+3}{2}}}dtdz,\label{eq5.18}
		\end{align}
		where the constant depends on $n$. We apply the change of variable $u=\frac{t}{|z|}$ in the integral with respect to $t$ and get
		\begin{align}
			&\int_0^{\infty}\frac{t^s}{(t^2+|z|^2)^{\frac{n+1}{2}}}+\frac{t^{2+s}}{(t^2+|z|^2)^{\frac{n+3}{2}}}dt
			\nonumber\\
			&=|z|^{s-n}\cdot\int_0^{\infty}\frac{u^s}{(1+u^2)^{\frac{n+1}{2}}}+\frac{u^{2+s}}{(1+u^2)^{\frac{n+3}{2}}}du
			\lesssim|z|^{s-n},\label{eq5.19}
		\end{align}
		where the last integral with respect to $u$ in (\ref{eq5.19}) converges since $0<s<1\leq n$, and the constant in (\ref{eq5.19}) depends on $n,p,q$. We put (\ref{eq5.19}) into (\ref{eq5.18}), insert the resulting estimate into the defining expression (\ref{eq5.7}) of $\Remain_{s,q}^3(b)(x)$ and apply Fubini's theorem, then we can estimate the term $\alpha^{-q}\int_{\Omega^c}\Remain_{s,q}^3(b)(x)dx$ from above by a constant multiple of the following expression
		\begin{equation}\label{eq5.20}
			\alpha^{-q}\sum_m\!\!\int_{I_m}\int_{\Omega^c}\!\!|x-y|^{-n-sq}dx\cdot
			\big(\!\!\int_{\bbbr^n}\!\!|\bnearIm(y-z)|\cdot|z|^{s-n}dz\big)^q dy,
		\end{equation}
		and the constant depends only on $n,p,q$. By Lemma \ref{lemma1} (b), when $y\in I_m$ and $x\in\Omega^c$, we have $|x-y|\geq\dist(I_m,\Omega^c)\geq\sqrt{n}\cdot l(I_m)$ and thus
		\begin{equation}\label{eq5.21}
			\int_{\Omega^c}\!\!|x-y|^{-n-sq}dx\leq\int_{|x|\geq\sqrt{n}\cdot l(I_m)}\!\!|x|^{-n-sq}dx\lesssim
			\Lebes^n(I_m)^{-\frac{sq}{n}},
		\end{equation}
		where we also use the formula $\Lebes^n(I_m)=l(I_m)^n$. Inserting (\ref{eq5.21}) into (\ref{eq5.20}) and invoking the famous Hardy-Littlewood-Sobolev inequality, i.e. Lemma \ref{lemma5}, yield
		\begin{align}
			&(\ref{eq5.20})\!\lesssim\!\alpha^{-q}\sum_m\Lebes^n(I_m)^{-\frac{sq}{n}}\!\!\!\int_{I_m}\!\!
			\big(\!\!\int_{\bbbr^n}\!\!\!|\bnearIm(y-z)|\!\cdot\!|z|^{s-n}dz\big)^q dy\nonumber\\
			&\lesssim\alpha^{-q}\sum_m\Lebes^n(I_m)^{-\frac{sq}{n}}\int_{\bbbr^n}
			\big(\!\!\int_{\bbbr^n}\!\!|\bnearIm(y-z)|\!\cdot\!|z|^{s-n}dz\big)^q dy\nonumber\\
			&\lesssim\alpha^{-q}\sum_m\Lebes^n(I_m)^{-\frac{sq}{n}}\|\bnearIm\|_{L^p(\bbbr^n)}^q
			\lesssim\alpha^{-p}\|f\|_{L^p(\bbbr^n)}^p,\label{eq5.22}
		\end{align}
		where we also use the condition that $1<p<q<\infty$ and $0<s=n(\frac{1}{p}-\frac{1}{q})<1$, and the last inequality in (\ref{eq5.22}) is due to estimate (\ref{eq4.32}) and Lemma \ref{lemma3} (i). The implicit constants in (\ref{eq5.22}) depend on $n,p,q$. Combining (\ref{eq5.17}), (\ref{eq5.20}), and (\ref{eq5.22}), we find that the term (\ref{eq5.10}) can be estimated from above by a constant multiple of $\alpha^{-p}\|f\|_{L^p(\bbbr^n)}^p$, and the constant only depends on $n,p,q$. The proof of Theorem \ref{theorem5} is now complete.
\end{proof}

\section{The weak type boundedness of the $\Diff_{s,q}$-function in $\dot{L}^p_s(\bbbr^n)$}\label{proof.of.theorem2}
\begin{proof}[Proof of Theorem \ref{theorem2}]
Step 1: It is sufficient to prove that
\begin{equation}\label{eq6.0}
\Lebes^n(\{x\in\bbbr^n:\Diff_{s,q}f(x)>\alpha\})\lesssim\alpha^{-p}\|f\|_{\Lps}^p
\end{equation}
for $0<\alpha<\infty$. To provide a pointwise estimate for $\Diff_{s,q}f(x)$, we follow the method given in \cite[Appendix A]{Piero2019} and the hint in \cite[chapter V, 6.12]{Stein1971}. By the almost everywhere pointwise convergence property of Poisson integrals (cf. \cite[chapter I, Theorem 1.25]{SteinWeissFourier}) and the fact that the real-valued function representative $f(x)\in\functrep_{0}(f)\cap W^{1,p_0}(\bbbr^n)$ for some $p_0\in[1,\infty]$ implies $f(x)\in L^{p_0}(\bbbr^n)$, we have for almost every $x\in\bbbr^n$,
\begin{equation*}
\lim_{t\rightarrow0}\Pint(f;x,t)=\lim_{t\rightarrow0}f*P_t(x)=f(x).
\end{equation*} 
By Lemma \ref{lemma14} (\ref{eq2.266}) and (\ref{eq2.267}), we can find a sequence $\{t_k\}_{k\in\bbbn}$ of positive real numbers and denote $C_1=\frac{(2\pi)^s}{\varGamma(s)}$ so that
\begin{equation*}
0<t_{k+1}<t_k<\infty,\qquad\lim_{k\rightarrow\infty}t_k=0,
\end{equation*}
\begin{equation*}
\lim_{k\rightarrow\infty}\!C_1\!\cdot\!\!\!\int_0^{\infty}\!\!\!\!\!\!\Pint(\Lift_s f;x,t_k\!+\!r)
\!\cdot\!r^{s-1}dr\!=\!C_1\!\cdot\!\!\!\int_0^{\infty}\!\!\!\!\!\!\Pint(\Lift_s f;x,r)
\!\cdot\!r^{s-1}dr
\end{equation*}
for almost every $x\!\in\!\bbbr^n$. By the fundamental theorem of calculus, we can denote $\bar{y}=\frac{y}{|y|}$ and rewrite $f(x+y)-f(x)$ identically as
\begin{align*}
&\lim_{k\rightarrow\infty}\big\{\!\Pint(f;x\!+\!y,t_k)\!-\!\Pint(f;x\!+\!y,|y|)\\
&+\!\Pint(f;x\!+\!y,|y|)\!-\!\Pint(f;x,|y|)\!+\!\Pint(f;x,|y|)\!-\!\Pint(f;x,t_k)\!\big\},
\end{align*}
and the above expression is equivalent to
\begin{align}
&-\!\!\int_0^{|y|}\!\!\!\partial_{n+1}\Pint(f;x\!+\!y,u)du\!+\!\!\int_0^{|y|}\!\!\!\nabla_n\Pint(f;x\!+\!u\bar{y},|y|)\!\cdot\!\bar{y}du\nonumber\\
&+\!\!\int_0^{|y|}\!\!\!\partial_{n+1}\Pint(f;x,u)du.\label{eq6.1}
\end{align}
Notice that we use the almost everywhere pointwise convergence property of the Poisson integral $\Pint(f;x,t)$ to deduce (\ref{eq6.1}), and this property is a property of the Poisson integral $\Pint(f;x,t)$ as a function rather than a tempered distribution, therefore the expressions $\Pint(f;x+y,u)$, $\Pint(f;x+u\bar{y},|y|)$, and $\Pint(f;x,u)$ appearing in (\ref{eq6.1}) are Poisson integrals considered as functions. Since $1<p<\infty$, $0<s=n(\frac{1}{p}-\frac{1}{q})<\frac{n}{p}$, and the assumption of Theorem \ref{theorem2} implies $f\in\Lps$ has a real-valued function representative $f(x)\in\functrep_{0}(f)\cap L^{p_0}(\bbbr^n)$ for some $p_0\in[1,\infty]$, we denote $\bar{y}_k=\frac{y_k}{|y|}$ is the $k$-th coordinate of $\bar{y}$, and then we use Lemma \ref{lemma14} (\ref{eq2.186}) and (\ref{eq2.187}) to deduce that
\begin{align}
&\int_0^{|y|}\!\!\!\!\!\!\nabla_n\Pint(f;x\!+\!u\bar{y},|y|)\!\cdot\!\bar{y}du\!=\!\!\!
\int_0^{|y|}\!\sum_{k=1}^n\!\partial_k\Pint(f;x\!+\!u\bar{y},|y|)\!\cdot\!\bar{y}_k du\nonumber\\
&=\!C_1\!\cdot\!\!\!\int_0^{|y|}\!\!\!\int_0^{\infty}\!\!\!\!\big[\nabla_n\Pint(\Lift_s f;x\!+\!u\bar{y},|y|\!+\!r)\!\cdot\!\bar{y}\big]r^{s-1}drdu,\label{eq6.45}
\end{align}
and
\begin{align}
&-\!\!\!\int_0^{|y|}\!\!\!\!\!\!\partial_{n+1}\Pint(f;x\!+\!y,u)du\!+\!\!\!\int_0^{|y|}\!\!\!\!\!\!
\partial_{n+1}\Pint(f;x,u)du\nonumber\\
&=\!\!\!\int_0^{|y|}\!\!\!\!\!\!-\partial_{n+1}\Pint(f;x\!+\!y,u)\!+\!\partial_{n+1}\Pint(f;x,u)du
\nonumber\\
&=\!\!\!\int_0^{|y|}\!\!\!\!\!\!-C_1\!\cdot\!\!\!\int_0^{\infty}\!\!\!\!\!\!\partial_{n+1}\Pint(\Lift_s f;x\!+\!y,u\!+\!r)r^{s-1}dr\!-\!G'(u)\nonumber\\
&\quad+\!C_1\!\cdot\!\!\!\int_0^{\infty}\!\!\!\!\!\!\partial_{n+1}\Pint(\Lift_s f;x,u\!+\!r)r^{s-1}dr\!+\!G'(u)du\nonumber\\
&=\!-C_1\!\!\!\int_0^{|y|}\!\!\!\!\int_0^{\infty}\!\!\!\!\!\!\partial_{n+1}\Pint(\Lift_s f;x\!+\!y,u\!+\!r)r^{s-1}drdu\nonumber\\
&\quad+\!C_1\!\!\!\int_0^{|y|}\!\!\!\!\int_0^{\infty}\!\!\!\!\!\!\partial_{n+1}\Pint(\Lift_s f;x,u\!+\!r)r^{s-1}drdu,\label{eq6.46}
\end{align}
where the last equation of (\ref{eq6.46}) is due to the fact that the conclusion (\ref{eq2.268}) of Lemma \ref{lemma14} implies both of the following integrals (\ref{eq6.47}) and (\ref{eq6.48}) exist and are finite for almost every $x\in\bbbr^n$ and $y\in\bbbr^n$,
\begin{align}
&\int_0^{|y|}\!\!\!\!\int_0^{\infty}\!\!\!\!\!\!\partial_{n+1}\Pint(\Lift_s f;x\!+\!y,u\!+\!r)r^{s-1}drdu,\label{eq6.47}\\
&\int_0^{|y|}\!\!\!\!\int_0^{\infty}\!\!\!\!\!\!\partial_{n+1}\Pint(\Lift_s f;x,u\!+\!r)r^{s-1}drdu.
\label{eq6.48}
\end{align}
Under the given conditions, we can use Lemma \ref{lemma11} to identify $\varGamma(\frac{-s}{2})\Lift_s f$ with a real-valued function in $L^p(\bbbr^n)$. And by Lemma \ref{lemma10}, we know $|\nabla_{n+1}\Pint(\varGamma(\frac{-s}{2})\Lift_s f;x,t)|^q$ is subharmonic on $\bbbr^{n+1}_+$ for $2\leq q<\infty$. We will use this subharmonicity and hence from now on, we work with the function $\varGamma(\frac{-s}{2})f(x)$. Combining (\ref{eq6.1}), (\ref{eq6.45}), and (\ref{eq6.46}) together with an appropriate change of variable, and multiplying the result by $\varGamma(\frac{-s}{2})$, we can estimate the term $|\varGamma(\frac{-s}{2})f(x+y)-\varGamma(\frac{-s}{2})f(x)|$ from above by the constant $C_1$ times the sum of the following three terms,
\begin{equation}\label{eq6.7}
\int_0^{|y|}\!\!\!\!\int_u^{\infty}\!\!\!\!\!|\partial_{n+1}\Pint(\varGamma(\frac{-s}{2})\Lift_s f;x\!+\!y,v)|\!\cdot\!(v\!-\!u)^{s-1}dvdu,
\end{equation}
\begin{equation}\label{eq6.8}
\int_0^{|y|}\!\!\!\!\int_u^{\infty}\!\!\!\!\!|\nabla_n\Pint(\varGamma(\frac{-s}{2})\Lift_s f;x\!+\!u\bar{y},|y|\!+\!v\!-\!u)|\!\cdot\!(v\!-\!u)^{s-1}dvdu,
\end{equation}
\begin{equation}
\int_0^{|y|}\!\!\!\!\int_u^{\infty}\!\!\!\!\!|\partial_{n+1}\Pint(\varGamma(\frac{-s}{2})\Lift_s f;x,v)|\!\cdot\!(v\!-\!u)^{s-1}dvdu.\label{eq6.9}
\end{equation}
We observe that when $0\leq u\leq|y|$, $2|y|\leq v$ and $u\leq v$, the three points $(x+y,v)$, $(x+u\bar{y},|y|+v-u)$, $(x,v)$ belong to the common cone $\gamma_{\frac{1}{2}}(x)\subseteq\bbbr^{n+1}_+$ whose vertex is $(x,0)$ for some $x\in\bbbr^n$, and whose aperture is $\frac{1}{2}$. Therefore we can split the integrals in (\ref{eq6.7}), (\ref{eq6.8}), (\ref{eq6.9}), and rearrange the summation, then we can obtain
		\begin{align}
			&|\varGamma(\frac{-s}{2})f(x+y)-\varGamma(\frac{-s}{2})f(x)|\nonumber\\
			&\leq(\ref{eq6.7})+(\ref{eq6.8})+(\ref{eq6.9})\nonumber\\
			&=C_1\cdot\big\{F_1(x,y)+F_2(x,y)+F_3(x,y)+F_4(x,y)\big\},\label{eq6.10}
		\end{align}
where
\begin{equation}\label{eq6.11}
F_1(x,y)\!\!=\!\!\int_0^{|y|}\!\!\!\!\int_u^{2|y|}\!\!\!\!\!\!\!|\partial_{n+1}\Pint(\varGamma(\frac{-s}{2})\Lift_s f;x\!+\!y,v)|\frac{dvdu}{(v\!-\!u)^{1-s}},
\end{equation}
\begin{align}
&F_2(x,y)\nonumber\\
&\!=\!\!\int_0^{|y|}\!\!\!\!\int_u^{2|y|}\!\!\!\!\!\!\!|\nabla_n\Pint(\varGamma(\frac{-s}{2})\Lift_s f;x\!+\!u\bar{y},|y|\!+\!v\!-\!u)|\frac{dvdu}{(v\!-\!u)^{1-s}},\label{eq6.12}
\end{align}
\begin{equation}\label{eq6.13}
F_3(x,y)\!\!=\!\!\int_0^{|y|}\!\!\!\!\int_u^{2|y|}\!\!\!\!\!\!\!|\partial_{n+1}\Pint(\varGamma(\frac{-s}{2})\Lift_s f;x,v)|\frac{dvdu}{(v\!-\!u)^{1-s}},
\end{equation}
and
		\begin{align}
			&F_4(x,y)\nonumber\\
			&\!=\int_0^{|y|}\!\!\!\!\int_{2|y|}^{\infty}\big\{|\partial_{n+1}\Pint(\varGamma(\frac{-s}{2})\Lift_s f;x+y,v)|\nonumber\\
			&\quad+|\nabla_n\Pint(\varGamma(\frac{-s}{2})\Lift_s f;x+u\bar{y},|y|+v-u)|\nonumber\\
			&\quad+|\partial_{n+1}\Pint(\varGamma(\frac{-s}{2})\Lift_s f;x,v)|\big\}
			\cdot\frac{dvdu}{(v-u)^{1-s}}.\label{eq6.14}
		\end{align}
Furthermore, we recall the defining expression (\ref{eq1-10}) and obtain
\begin{align}
&\Diff_{s,q}f(x)^q=\frac{1}{|\varGamma(\frac{-s}{2})|^q}\cdot\int_{\bbbr^n}
\frac{|\varGamma(\frac{-s}{2})f(x+y)-\varGamma(\frac{-s}{2})f(x)|^q}{|y|^{n+sq}}dy\nonumber\\
&\lesssim\sum_{k=1}^4\int_{\bbbr^n}\frac{F_k(x,y)^q}{|y|^{n+sq}}dy\quad
\text{for almost every $x\in\bbbr^n$,}\label{eq6.15}
\end{align}
where the implicit constant is real-valued and depends on $q,s$. We will estimate each term on the right side of inequality (\ref{eq6.15}).\\
		
		Step 2: We begin with the term $\int_{\bbbr^n}\frac{F_4(x,y)^q}{|y|^{n+sq}}dy$. Observe that the three points $(x+y,v)$, $(x+u\bar{y},|y|+v-u)$, $(x,v)$ in the iterated integral of (\ref{eq6.14}) belong to the set
		\begin{equation}\label{eq6.16}
			\{(\xi,\eta)\in\bbbr^{n+1}_+:\xi\in\bbbr^n,\eta>0,|\xi-x|\leq|y|\leq\frac{1}{2}\eta\}.
		\end{equation}
		If we pick a positive real number $A$ so that $0<s<A<1$, then the condition $0\leq u\leq|y|\leq\frac{1}{2}v$ implies $v^{s-1}\leq(v-u)^{s-1}\leq 2^{1-s}v^{s-1}$, and thus
		\begin{equation}\label{eq6.17}
			\int_0^{|y|}\!\!\!\!\int_{2|y|}^{\infty}\frac{(v-u)^{s-1}}{v^A}dvdu\lesssim
			\int_0^{|y|}\!\!\!\!\int_{2|y|}^{\infty}v^{s-A-1}dvdu\lesssim|y|^{1+s-A},
		\end{equation}
and the constants in (\ref{eq6.17}) depend on $s,A$. First, we have
\begin{align}
&\int_0^{|y|}\!\!\!\!\int_{2|y|}^{\infty}\!\!\!\!\!|\partial_{n+1}\Pint(\varGamma(\frac{-s}{2})\Lift_s f;x\!+\!y,v)|v^A\!\cdot\!\frac{(v\!-\!u)^{s-1}}{v^A}dvdu\nonumber\\
&\leq\!\!\sup_{|\xi-x|\leq|y|\leq\frac{1}{2}\eta}\!\!\!\!|\partial_{n+1}\Pint(\varGamma(\frac{-s}{2})\Lift_s f;\xi,\eta)|\eta^A\!\cdot\!\!\!\int_0^{|y|}\!\!\!\!\int_{2|y|}^{\infty}
\!\!\frac{(v\!-\!u)^{s-1}}{v^A}dvdu\nonumber\\
&\lesssim\!\!\sup_{|\xi-x|\leq|y|\leq\frac{1}{2}\eta}\!\!\!\!|\partial_{n+1}\Pint(\varGamma(\frac{-s}{2})\Lift_s f;\xi,\eta)|\eta^A\!\cdot\!|y|^{1+s-A}.\label{eq6.18}
\end{align}
Second, we have
\begin{align}
&\int_0^{|y|}\!\!\!\!\int_{2|y|}^{\infty}\!\!\!\!\!|\nabla_n\Pint(\varGamma(\frac{-s}{2})\Lift_s f;x\!+\!u\bar{y},|y|\!+\!v\!-\!u)|(|y|\!+\!v\!-\!u)^A\nonumber\\
&\quad\cdot\frac{(v\!-\!u)^{s-1}}{(|y|\!+\!v\!-\!u)^A}dvdu\nonumber\\
&\leq\sup_{|\xi-x|\leq|y|\leq\frac{1}{2}\eta}|\nabla_n\Pint(\varGamma(\frac{-s}{2})\Lift_s f;\xi,\eta)|\eta^A
\nonumber\\
&\quad\cdot\int_0^{|y|}\!\!\!\!\int_{2|y|}^{\infty}\frac{(v-u)^{s-1}}{(|y|+v-u)^A}dvdu\nonumber\\
&\lesssim\!\!\sup_{|\xi-x|\leq|y|\leq\frac{1}{2}\eta}\!\!\!\!|\nabla_n\Pint(\varGamma(\frac{-s}{2})\Lift_s f;\xi,\eta)|\eta^A\!\cdot\!\!\!\int_0^{|y|}\!\!\!\!
\int_{2|y|}^{\infty}\!\!\!\!v^{s-A-1}dvdu\nonumber\\
&\lesssim\!\!\sup_{|\xi-x|\leq|y|\leq\frac{1}{2}\eta}\!\!\!\!|\nabla_n\Pint(\varGamma(\frac{-s}{2})\Lift_s f;\xi,\eta)|\eta^A\cdot|y|^{1+s-A},\label{eq6.19}
\end{align}
and the constants in (\ref{eq6.19}) depend on $s,A$. Third, we have
\begin{align}
&\int_0^{|y|}\!\!\!\!\int_{2|y|}^{\infty}\!\!\!\!\!|\partial_{n+1}\Pint(\varGamma(\frac{-s}{2})\Lift_s f;x,v)|v^A\!\cdot\!\frac{(v\!-\!u)^{s-1}}{v^A}dvdu\nonumber\\
&\leq\!\!\sup_{|\xi-x|\leq|y|\leq\frac{1}{2}\eta}\!\!\!\!|\partial_{n+1}\Pint(\varGamma(\frac{-s}{2})\Lift_s f;\xi,\eta)|\eta^A\!\cdot\!\!\!\int_0^{|y|}\!\!\!\!
\int_{2|y|}^{\infty}\!\!\frac{(v\!-\!u)^{s-1}}{v^A}dvdu\nonumber\\
&\lesssim\!\!\sup_{|\xi-x|\leq|y|\leq\frac{1}{2}\eta}\!\!\!\!|\partial_{n+1}\Pint(\varGamma(\frac{-s}{2})\Lift_s f;\xi,\eta)|\eta^A\!\cdot\!|y|^{1+s-A}.\label{eq6.20}
\end{align}
Inserting (\ref{eq6.18}), (\ref{eq6.19}), (\ref{eq6.20}) into (\ref{eq6.14}) yields
\begin{equation}\label{eq6.21}
F_4(x,y)^q\!\!
\lesssim\!\!\sup_{|\xi-x|\leq|y|\leq\frac{1}{2}\eta}\!\!\!\!|\nabla_{n+1}\Pint(\varGamma(\frac{-s}{2})\Lift_s f;\xi,\eta)|^q\eta^{qA}\!\cdot\!|y|^{q(1+s-A)},
\end{equation}
where the constant in (\ref{eq6.21}) is real-valued and depends on $q,s,A$. Now we apply Lemma \ref{lemma11} and Lemma \ref{lemma10}, then $\varGamma(\frac{-s}{2})\Lift_s f$ can be identified with a real-valued function in $L^p(\bbbr^n)$ and hence 
\begin{equation}\label{eq6.42}
|\nabla_{n+1}\Pint(\varGamma(\frac{-s}{2})\Lift_s f;\xi,\eta)|^q
\end{equation}
is subharmonic on $\bbbr^{n+1}_+$ when $2\leq q<\infty$. We consider the $(n+1)$-dimensional ball $B^{n+1}((\xi,\eta),\frac{1}{2}\eta)\subseteq\bbbr^{n+1}_+$ centered at $(\xi,\eta)$ with radius $\frac{1}{2}\eta$, where $\xi\in\bbbr^n$ and $\eta>0$, and we have
		\begin{align}
			&|\nabla_{n+1}\Pint(\varGamma(\frac{-s}{2})\Lift_s f;\xi,\eta)|^q\nonumber\\
			&\leq\mvint_{B^{n+1}((\xi,\eta),\frac{1}{2}\eta)}|\nabla_{n+1}\Pint(\varGamma(\frac{-s}{2})\Lift_s f;z,t)|^q d(z,t),\label{eq6.22}
		\end{align}
where the right side of this inequality is a mean value integral over the aforementioned $(n+1)$-dimensional ball. We also observe that if $(\xi,\eta)$ belongs to the set (\ref{eq6.16}) and $(z,t)$ belongs to the ball $B^{n+1}((\xi,\eta),\frac{1}{2}\eta)$ for $z\in\bbbr^n$ and $t>0$, then we have $|\xi-x|\leq|y|\leq\frac{1}{2}\eta$, $|z-\xi|\leq\frac{1}{2}\eta$, $\frac{1}{2}\eta\leq t\leq\frac{3}{2}\eta$, and these inequalities indicate that $|y|\leq\frac{1}{2}\eta\leq t$, $\frac{2}{3}t\leq\eta\leq 2t$, and $|z-x|\leq|z-\xi|+|\xi-x|\leq\eta\leq 2t$. Therefore we combine the above indications with (\ref{eq6.21}), (\ref{eq6.22}) and apply Fubini's theorem, and then we can estimate the term $\int_{\bbbr^n}\frac{F_4(x,y)^q}{|y|^{n+sq}}dy$ from above by a constant multiple of the following,
\begin{align}
&\int_{\bbbr^n}|y|^{q-qA-n}\sup_{|\xi-x|\leq|y|\leq\frac{1}{2}\eta}\frac{\eta^{qA-n-1}}{\nu_{n+1}}\nonumber\\
&\quad\cdot\int_{B^{n+1}((\xi,\eta),\frac{1}{2}\eta)}\!\!\!\!
|\nabla_{n+1}\Pint(\varGamma(\frac{-s}{2})\Lift_s f;z,t)|^q d(z,t)dy\nonumber\\
&\lesssim\!\!\int_{\bbbr^n}\!\!\!\!|y|^{q-qA-n}\sup_{|\xi-x|\leq|y|\leq\frac{1}{2}\eta}
\int_{B^{n+1}((\xi,\eta),\frac{1}{2}\eta)}\!\!\!\!\!\!\!\!\!\!\!\!\!\!\!\!\!\!\!\!\!\!\!
|\nabla_{n+1}\Pint(\varGamma(\frac{-s}{2})\Lift_s f;z,t)|^q\nonumber\\
&\quad\cdot t^{qA-n-1}d(z,t)dy\nonumber\\
&\lesssim\!\!\int_{\bbbr^n}\!\!\!\!|y|^{q-qA-n}\!\!\int_{|y|}^{\infty}\!\!\!\int_{|z-x|\leq 2t}\!\!\!\!\!\!\!\!\!\!\!\!\!\!\!\!
|\nabla_{n+1}\Pint(\varGamma(\frac{-s}{2})\Lift_s f;z,t)|^q\!\cdot\!t^{qA-n-1}dzdtdy\nonumber\\
&=\!\!\int_{0}^{\infty}\!\!\!\!\int_{|y|\leq t}\!\!\!\!\!\!\!\!\!|y|^{q-qA-n}dy\!\!\int_{|z-x|\leq 2t}\!\!\!\!\!\!\!\!\!\!\!\!\!\!\!\!
|\nabla_{n+1}\Pint(\varGamma(\frac{-s}{2})\Lift_s f;z,t)|^q t^{qA-n-1}dzdt\nonumber\\
&\lesssim\int_{0}^{\infty}\!\!\!\!\int_{|z-x|\leq 2t}\!\!\!\!\!\!\!\!1\cdot
|\nabla_{n+1}\Pint(\varGamma(\frac{-s}{2})\Lift_s f;z,t)|^q t^{q-n-1}dzdt,\label{eq6.23}
\end{align}
where the constants in (\ref{eq6.23}) depend on $n,q,s,A$, and $0<s<A<1$, and $\nu_{n+1}$ denotes the volume of the unit ball in $\bbbr^{n+1}$. Notice that when $|z-x|\leq 2t$, we have
\begin{equation}\label{eq6.24}
1\leq 3^{\lambda n}\cdot\big(\frac{t}{t+|z-x|}\big)^{\lambda n}\quad\text{for any }\lambda>0.
\end{equation}
Inserting (\ref{eq6.24}) into (\ref{eq6.23}) yields the following pointwise estimate
\begin{equation}\label{eq6.25}
\int_{\bbbr^n}\frac{F_4(x,y)^q}{|y|^{n+sq}}dy\lesssim\LPG_{\lambda,q}(\varGamma(\frac{-s}{2})\Lift_s f)(x)^q
\end{equation}
for almost every $x\in\bbbr^n$, where the implicit constant depends on $n,q,s,A,\lambda$, and $\LPG_{\lambda,q}(\varGamma(\frac{-s}{2})\Lift_s f)(x)$ is the generalized Littlewood-Paley $\LPG_{\lambda,q}$-function defined in (\ref{eq1-26}).\\
		
Step 3: We estimate the term $\int_{\bbbr^n}\frac{F_3(x,y)^q}{|y|^{n+sq}}dy$. We notice that when $0\leq v\leq 2|y|$, the point $(x,v)$ is in the set
\begin{equation}\label{eq6.26}
\{(\xi,\eta)\in\bbbr^{n+1}_+:\xi\in\bbbr^n,\eta>0,|\xi-x|\leq\eta\leq 2|y|\},
\end{equation}
then we can exchange the order of integration in the defining expression (\ref{eq6.13}) of $F_3(x,y)$ and obtain
\begin{align}
&F_3(x,y)^q\nonumber\\
&\lesssim\!(\!\int_0^{2|y|}\!\!\!\!\!|\partial_{n+1}\Pint(\varGamma(\frac{-s}{2})\Lift_s f;x,v)|v^B\!\cdot\!v^{s-B}dv)^q
\nonumber\\
&\leq\!\!\sup_{|\xi-x|\leq\eta\leq 2|y|}\!\!\!\!|\nabla_{n+1}\Pint(\varGamma(\frac{-s}{2})\Lift_s f;\xi,\eta)|^q\eta^{Bq}\!\cdot\!(\int_0^{2|y|}\!\!\!\! v^{s-B}dv)^q\nonumber\\
&\lesssim\!\!\sup_{|\xi-x|\leq\eta\leq 2|y|}\!\!\!\!|\nabla_{n+1}\Pint(\varGamma(\frac{-s}{2})\Lift_s f;\xi,\eta)|^q\eta^{Bq}\!\cdot\!|y|^{q(1+s-B)},\label{eq6.27}
\end{align}
where the constants in (\ref{eq6.27}) depend on $s,q,B$, and $B$ is a positive real number that satisfies $1<B<1+s$. We use the subharmonicity of (\ref{eq6.42}) given by Lemma \ref{lemma11} and Lemma \ref{lemma10}, and consider the $(n+1)$-dimensional ball $B^{n+1}((\xi,\eta),\frac{1}{2}\eta)\subseteq\bbbr^{n+1}_+$, then we still have the inequality (\ref{eq6.22}). Observe that if $(\xi,\eta)$ is in the set (\ref{eq6.26}) and $(z,t)$ belongs to the ball $B^{n+1}((\xi,\eta),\frac{1}{2}\eta)$ for $z\in\bbbr^n$ and $t>0$, we have $|\xi-x|\leq\eta\leq2|y|$, $|z-\xi|\leq\frac{1}{2}\eta$, $|t-\eta|\leq\frac{1}{2}\eta$, and these inequalities indicate that $\frac{2}{3}t\leq\eta\leq2t$, $t\leq3|y|$, $|z-x|\leq|z-\xi|+|\xi-x|\leq\frac{3}{2}\eta\leq3t$. Therefore we can combine the above indications with (\ref{eq6.27}), (\ref{eq6.22}), and apply Fubini's theorem, and then we can estimate the term $\int_{\bbbr^n}\frac{F_3(x,y)^q}{|y|^{n+sq}}dy$ from above by a constant multiple of the following,
\begin{align}
&\int_{\bbbr^n}|y|^{q-qB-n}\sup_{|\xi-x|\leq\eta\leq2|y|}\frac{\eta^{qB-n-1}}{\nu_{n+1}}\nonumber\\
&\quad\cdot\!\int_{B^{n+1}((\xi,\eta),\frac{1}{2}\eta)}\!\!\!\!\!\!\!\!\!\!\!\!\!\!\!\!\!\!\!\!\!\!\!\!
|\nabla_{n+1}\Pint(\varGamma(\frac{-s}{2})\Lift_s f;z,t)|^q d(z,t)dy\nonumber\\
&\lesssim\!\!\int_{\bbbr^n}\!\!\!\!|y|^{q-qB-n}\sup_{|\xi-x|\leq\eta\leq2|y|}
\int_{B^{n+1}((\xi,\eta),\frac{1}{2}\eta)}\!\!\!\!\!\!\!\!\!\!\!\!\!\!\!\!\!\!\!\!\!\!\!\!
|\nabla_{n+1}\Pint(\varGamma(\frac{-s}{2})\Lift_s f;z,t)|^q\nonumber\\
&\quad\cdot t^{qB-n-1}d(z,t)dy\nonumber\\
&\lesssim\!\!\!\int_{\bbbr^n}\!\!\!\!|y|^{q-qB-n}\!\!\int_{0}^{3|y|}\!\!\!\!\int_{|z-x|\leq3t}\!\!\!\!\!\!\!\!\!\!\!\!\!\!\!
|\nabla_{n+1}\Pint(\varGamma(\frac{-s}{2})\Lift_s f;z,t)|^q\!\cdot\!t^{qB-n-1}dzdtdy\nonumber\\
&=\!\!\!\int_{0}^{\infty}\!\!\!\!\int_{\frac{t}{3}\leq|y|}\!\!\!\!\!\!\!\!|y|^{q-qB-n}dy\!\!\int_{|z-x|\leq3t}\!\!\!\!\!\!\!\!\!\!\!\!\!\!\!
|\nabla_{n+1}\Pint(\varGamma(\frac{-s}{2})\Lift_s f;z,t)|^q t^{qB-n-1}dzdt\nonumber\\
&\lesssim\int_{0}^{\infty}\!\!\!\!\int_{|z-x|\leq3t}\!\!\!\!\!\!\!\!1\cdot
|\nabla_{n+1}\Pint(\varGamma(\frac{-s}{2})\Lift_s f;z,t)|^q t^{q-n-1}dzdt,\label{eq6.28}
\end{align}
where the constants in (\ref{eq6.28}) depend on $n,q,s,B$, and $1<B<1+s$, and $\nu_{n+1}$ is the volume of the unit ball in $\bbbr^{n+1}$. Notice that when $|z-x|\leq3t$, we have
		\begin{equation}\label{eq6.29}
			1\leq 4^{\lambda n}\cdot\big(\frac{t}{t+|z-x|}\big)^{\lambda n}\quad\text{for any }\lambda>0.
		\end{equation}
		Inserting (\ref{eq6.29}) into (\ref{eq6.28}) yields the following pointwise estimate
		\begin{equation}\label{eq6.30}
			\int_{\bbbr^n}\frac{F_3(x,y)^q}{|y|^{n+sq}}dy\lesssim\LPG_{\lambda,q}(\varGamma(\frac{-s}{2})\Lift_s f)(x)^q
		\end{equation}
for almost every $x\in\bbbr^n$, where the implicit constant depends on $n,q,s,B,\lambda$.\\
		
Step 4: We estimate the term $\int_{\bbbr^n}\frac{F_2(x,y)^q}{|y|^{n+sq}}dy$. Observe that when $0\leq u\leq|y|$ and $u\leq v\leq2|y|$, and $\bar{y}=\frac{y}{|y|}$, the point $(x+u\bar{y},|y|+v-u)$ belongs to the set
		\begin{equation}\label{eq6.31}
			\{(\xi,\eta)\in\bbbr^{n+1}_+:\xi\in\bbbr^n,\eta>0,|\xi-x|\leq|y|\leq\eta\}.
		\end{equation}
		Furthermore, we continue using the positive number $A$ satisfying $0<s<A<1$ and we have
\begin{equation}\label{eq6.32}
\int_0^{|y|}\!\!\!\!\int_u^{2|y|}\!\!\!\!\!\frac{(v\!-\!u)^{s-1}}{(|y|\!+\!v\!-\!u)^A}dvdu\!\!\leq\!\!\!\int_0^{|y|}\!\!\!\!\int_u^{2|y|}\!\!\frac{(v\!-\!u)^{s-1}}{|y|^A}dvdu\!\!\lesssim\!\!|y|^{1+s-A}.
\end{equation}
		Now we recall the defining expression (\ref{eq6.12}), multiply the integrand of (\ref{eq6.12}) by the factor $(|y|+v-u)^A$ and divide it by the same factor, and then we apply the above observation and the estimate (\ref{eq6.32}) to obtain the following
		\begin{align}
			&F_2(x,y)^q\leq\sup_{|\xi-x|\leq|y|\leq\eta}\!\!\!\!|\nabla_{n}\Pint(\varGamma(\frac{-s}{2})\Lift_s f;\xi,\eta)|^q\eta^{Aq}\nonumber\\
			&\cdot(\int_0^{|y|}\!\!\!\!\int_u^{2|y|}\frac{(v-u)^{s-1}}{(|y|+v-u)^A}dvdu)^q
			\nonumber\\
			&\lesssim\sup_{|\xi-x|\leq|y|\leq\eta}\!\!\!\!|\nabla_{n+1}\Pint(\varGamma(\frac{-s}{2})\Lift_s f;\xi,\eta)|^q\eta^{Aq}\cdot|y|^{q(1+s-A)},\label{eq6.33}
		\end{align}
and the constant depends on $n,p,q$. We use the subharmonicity of (\ref{eq6.42}) over the $(n+1)$-dimensional ball $B^{n+1}((\xi,\eta),\frac{1}{2}\eta)\subseteq\bbbr^{n+1}_+$ where $\xi\in\bbbr^n$, $\eta>0$, and $2\leq q<\infty$, and this subharmonicity is due to Lemma \ref{lemma11} and Lemma \ref{lemma10}, then we still have the inequality (\ref{eq6.22}). When $(\xi,\eta)$ belongs to the set (\ref{eq6.31}) and $(z,t)$ belongs to the ball $B^{n+1}((\xi,\eta),\frac{1}{2}\eta)$ for $z\in\bbbr^n$ and $t>0$, we have $|\xi-x|\leq|y|\leq\eta$, $|z-\xi|\leq\frac{1}{2}\eta$, $\frac{1}{2}\eta\leq t\leq\frac{3}{2}\eta$, and these inequalities indicate that $\frac{2}{3}t\leq\eta\leq2t$, $|y|\leq2t$, and $|z-x|\leq|z-\xi|+|\xi-x|\leq\frac{3}{2}\eta\leq3t$. Therefore we can combine the above indications with (\ref{eq6.33}), (\ref{eq6.22}), and apply Fubini's theorem, and then we can estimate the term $\int_{\bbbr^n}\frac{F_2(x,y)^q}{|y|^{n+sq}}dy$ from above by a constant multiple of the following,
\begin{align}
&\int_{\bbbr^n}|y|^{q-qA-n}\sup_{|\xi-x|\leq|y|\leq\eta}\frac{\eta^{qA-n-1}}{\nu_{n+1}}\nonumber\\
&\quad\cdot\!\!\int_{B^{n+1}((\xi,\eta),\frac{1}{2}\eta)}\!\!\!\!\!\!\!\!\!\!\!\!\!\!\!\!\!\!\!\!\!\!\!\!|\nabla_{n+1}\Pint(\varGamma(\frac{-s}{2})\Lift_s f;z,t)|^q d(z,t)dy\nonumber\\
&\lesssim\!\!\!\int_{\bbbr^n}\!\!\!\!|y|^{q-qA-n}\sup_{|\xi-x|\leq|y|\leq\eta}
\int_{B^{n+1}((\xi,\eta),\frac{1}{2}\eta)}\!\!\!\!\!\!\!\!\!\!\!\!\!\!\!\!\!\!\!\!\!\!\!\!
|\nabla_{n+1}\Pint(\varGamma(\frac{-s}{2})\Lift_s f;z,t)|^q\nonumber\\
&\quad\cdot t^{qA-n-1}d(z,t)dy\nonumber\\
&\lesssim\!\!\!\int_{\bbbr^n}\!\!\!\!|y|^{q-qA-n}\!\!\int_{\frac{|y|}{2}}^{\infty}\!\!
\int_{|z-x|\leq3t}\!\!\!\!\!\!\!\!\!\!\!\!\!\!\!\!
|\nabla_{n+1}\Pint(\varGamma(\frac{-s}{2})\Lift_s f;z,t)|^q\!\cdot\!t^{qA-n-1}dzdtdy\nonumber\\
&=\!\!\!\int_{0}^{\infty}\!\!\!\!\int_{|y|\leq2t}\!\!\!\!\!\!\!\!\!\!|y|^{q-qA-n}dy\!\!\int_{|z-x|\leq3t}\!\!\!\!\!\!\!\!\!\!\!\!\!\!\!\!
|\nabla_{n+1}\Pint(\varGamma(\frac{-s}{2})\Lift_s f;z,t)|^q\!\cdot\!t^{qA-n-1}dzdt\nonumber\\
&\lesssim\int_{0}^{\infty}\!\!\!\!\int_{|z-x|\leq3t}\!\!\!\!\!\!\!\!1\cdot
|\nabla_{n+1}\Pint(\varGamma(\frac{-s}{2})\Lift_s f;z,t)|^q t^{q-n-1}dzdt,\label{eq6.34}
\end{align}
		where the constants in (\ref{eq6.34}) depend on $n,q,s,A$, and $0<s<A<1$, and $\nu_{n+1}$ denotes the volume of the unit ball in $\bbbr^{n+1}$. Inserting (\ref{eq6.29}) into (\ref{eq6.34}) yields the following pointwise estimate
		\begin{equation}\label{eq6.35}
			\int_{\bbbr^n}\frac{F_2(x,y)^q}{|y|^{n+sq}}dy\lesssim\LPG_{\lambda,q}(\varGamma(\frac{-s}{2})\Lift_s f)(x)^q
		\end{equation}
for almost every $x\in\bbbr^n$, where the implicit constant depends on $n,q,s,A,\lambda$, and $\LPG_{\lambda,q}(\varGamma(\frac{-s}{2})\Lift_s f)(x)$ is the generalized Littlewood-Paley $\LPG_{\lambda,q}$-function defined in (\ref{eq1-26}).\\
		
Step 5: We estimate the term $\int_{\bbbr^n}\frac{F_1(x,y)^q}{|y|^{n+sq}}dy$. We can exchange the order of integration in the defining expression (\ref{eq6.11}) and obtain
\begin{equation}\label{eq6.36}
F_1(x,y)\lesssim\int_0^{2|y|}\!\!\!\!\!|\partial_{n+1}\Pint(\varGamma(\frac{-s}{2})\Lift_s f;x+y,v)|v^s dv,
\end{equation} 
where the constant in (\ref{eq6.36}) depends on $0<s<1$. Therefore with a proper change of variables, we have obtained the pointwise estimate for almost every $x\in\bbbr^n$,
\begin{equation}\label{eq6.37}
\int_{\bbbr^n}\frac{F_1(x,y)^q}{|y|^{n+sq}}dy\lesssim\Remain_{s,q}(\varGamma(\frac{-s}{2})\Lift_s f)(x)^q,
\end{equation}
where the constant in (\ref{eq6.37}) depends on $s,q$, and $\Remain_{s,q}(\varGamma(\frac{-s}{2})\Lift_s f)(x)$ is the $\Remain_{s,q}$-function given by (\ref{eq1-27}).\\
		
Step 6: From inequalities (\ref{eq6.15}), (\ref{eq6.25}), (\ref{eq6.30}), (\ref{eq6.35}), (\ref{eq6.37}), we can deduce the pointwise estimate
\begin{equation}\label{eq6.38}
\Diff_{s,q}f(x)\!\lesssim\!\LPG_{\lambda,q}(\varGamma(\frac{-s}{2})\Lift_s f)(x)
\!+\!\Remain_{s,q}(\varGamma(\frac{-s}{2})\Lift_s f)(x)
\end{equation}
for almost every $x\in\bbbr^n$, and the implicit constant in (\ref{eq6.38}) only depends on some fixed parameters. Furthermore, we have obtained
		\begin{align}
			&\Lebes^n(\{x\in\bbbr^n:\Diff_{s,q}f(x)>\alpha\})\nonumber\\
			&\leq\Lebes^n(\{x\in\bbbr^n:\frac{\alpha}{2}\lesssim\LPG_{\lambda,q}(\varGamma(\frac{-s}{2})\Lift_s f)(x)\}) 
			\nonumber\\
			&\quad+\Lebes^n(\{x\in\bbbr^n:\frac{\alpha}{2}\lesssim\Remain_{s,q}(\varGamma(\frac{-s}{2})\Lift_s f)(x)\}),
			\label{eq6.39}
		\end{align}
for any positive finite real number $\alpha$. The assumption that $f\in\Lps$ has a real-valued function representative $f(x)\in\functrep_{0}(f)\cap W^{1,p_0}(\bbbr^n)$ for some $1\leq p_0\leq\infty$ and Lemma \ref{lemma11} indicate that the tempered distribution $\varGamma(\frac{-s}{2})\Lift_s f$ has a real-valued function representative (in the sense of $\Sw_0'(\bbbr^n)$) in $L^p(\bbbr^n)$. Therefore Theorem \ref{theorem4} and Theorem \ref{theorem5} are applicable. The first term on the right side of (\ref{eq6.39}) can be estimated from above by a constant multiple of $\alpha^{-p}\|\varGamma(\frac{-s}{2})\Lift_s f\|_{L^p(\bbbr^n)}^p\sim\alpha^{-p}\|f\|_{\Lps}^p$ due to Theorem \ref{theorem4} and Definition \ref{definition2}. The second term on the right side of (\ref{eq6.39}) can be estimated from above by a constant multiple of $\alpha^{-p}\|\varGamma(\frac{-s}{2})\Lift_s f\|_{L^p(\bbbr^n)}^p\sim\alpha^{-p}\|f\|_{\Lps}^p$ due to Theorem \ref{theorem5} and Definition \ref{definition2}. The proof of Theorem \ref{theorem2} is now complete.
\end{proof}

\section{Proofs of Corollary \ref{corollary3} and Corollary \ref{corollary4}}\label{proofs.of.corollaries.3&4}
\begin{proof}[Proof of Corollary \ref{corollary3}]
Under the given conditions, we can pick a real number $\bar{p}$ so that $2\leq q<\infty$,
\begin{equation}\label{eq7.1}
\text{$1<\bar{p}<p<q$, and $n(\frac{1}{p}-\frac{1}{q})<s=n(\frac{1}{\bar{p}}-\frac{1}{q})<1$.}
\end{equation}
Since $f\in\Sw_0(\bbbr^n)$ and the function $\multi(\xi)=|\xi|^s$ satisfies condition (\ref{eq1.64}), then Proposition \ref{proposition1} indicates $\FT_n f(\xi)\in\Sw_{00}(\bbbr^n)$, $|\xi|^s\cdot\FT_n f(\xi)\in\Sw_{00}(\bbbr^n)$, thus $\Lift_s f(x)\!=\!\iFT_n[|\xi|^s\!\cdot\!\FT_n f(\xi)](x)\!\in\!\Sw_0(\bbbr^n)$, and hence
\begin{equation}\label{eq7.2}
\Lift_s f(x)\in L^{\bar{p}}(\bbbr^n)\cap L^{p}(\bbbr^n)\cap L^{q}(\bbbr^n),
\end{equation}
therefore $f$ is a real-valued function that belongs to
\begin{equation}\label{eq7.3}
\dot{L}^{\bar{p}}_s(\bbbr^n)\cap\dot{L}^{p}_s(\bbbr^n)\cap\dot{L}^{q}_s(\bbbr^n)\cap W^{1,p_0}(\bbbr^n)
\end{equation}
for each $p_0\in[1,\infty]$, and inequality (\ref{eq6.38}) with $\lambda=\frac{q}{\bar{p}}$ is true for the parameters $\bar{p}$, $q$, $s$ satisfying condition (\ref{eq7.1}). This is because exchanging the order of integration in Step 2, Step 3, Step 4, and Step 5 in section \ref{proof.of.theorem2} in order to obtain the estimates (\ref{eq6.25}), (\ref{eq6.30}), (\ref{eq6.35}), and (\ref{eq6.37}), requires that $\LPG_{\lambda,q}(\varGamma(\frac{-s}{2})\Lift_s f)(x)$ and $\Remain_{s,q}(\varGamma(\frac{-s}{2})\Lift_s f)(x)$ are finite for almost every $x\in\bbbr^n$, which is a consequence of Theorem \ref{theorem4} and Theorem \ref{theorem5} and the fact that $\varGamma(\frac{-s}{2})\Lift_s f(x)\in L^{\bar{p}}(\bbbr^n)$ is real-valued according to Lemma \ref{lemma11}. Indeed, we set $p_0=\bar{p}$ in (\ref{eq7.3}) and let $\phi\in\Sw_{0}(\bbbr^n)$, then $|\xi|^s\cdot\iFT_n\phi(\xi)\in\Sw_{00}(\bbbr^n)$ and we invoke \cite[Theorem 2.2.14 (5)]{14classical} to obtain that
\begin{align}
&\int_{\bbbr^n}\varGamma(\frac{-s}{2})\Lift_s f(x)\cdot\phi(x)dx\nonumber\\
&=\int_{\bbbr^n}\varGamma(\frac{-s}{2})f(x)\cdot\FT_n[|\xi|^s\cdot\iFT_n\phi(\xi)](x)dx\nonumber\\
&=\int_{\bbbr^n}\big[\text{the expression (\ref{eq2-62})}\big]\cdot\phi(x)dx,\label{eq7.11}
\end{align}
where we use (\ref{eq2-60}), and the third and fourth equations of (\ref{eq2-61}) are applicable since the assumption that $f\in\Sw_{0}(\bbbr^n)$ is a real-valued function guarantees the exchange of the order of integration. Since both $\varGamma(\frac{-s}{2})\Lift_s f(x)$ and the expression (\ref{eq2-62}) belong to $L^{\bar{p}}(\bbbr^n)$, their difference
\begin{equation}\label{eq7.12}
\big\{\varGamma(\frac{-s}{2})\Lift_s f(x)-\text{the expression (\ref{eq2-62})}\big\}
\end{equation} 
is in $L^{\bar{p}}(\bbbr^n)\subseteq\Lloc$, satisfies condition (\ref{eq1.101}), and belongs to $\functrep(\tilde{f})$ for some $\tilde{f}\in\Sw'(\bbbr^n)$, thus Proposition \ref{proposition2} implies the difference (\ref{eq7.12}) equals a polynomial for almost every $x\in\bbbr^n$. Since this polynomial is also in $L^{\bar{p}}(\bbbr^n)$, the polynomial is identically zero and we have for almost every $x\in\bbbr^n$,
\begin{equation}\label{eq7.13}
\varGamma(\frac{-s}{2})\Lift_s f(x)=\text{the expression (\ref{eq2-62})}.
\end{equation}
The expression (\ref{eq2-62}) is continuous in $x\in\bbbr^n$ and $\varGamma(\frac{-s}{2})\Lift_s f(x)\in\Sw_{0}(\bbbr^n)$ when $f\in\Sw_{0}(\bbbr^n)$, hence equation (\ref{eq7.13}) is true for all $x\in\bbbr^n$. The assumption of Corollary \ref{corollary3} now implies the expression (\ref{eq2-62}) is real-valued for all $x\in\bbbr^n$, and so is the function $\varGamma(\frac{-s}{2})\Lift_s f(x)$. Applying the $\|\cdot\|_{L^p(\bbbr^n)}$-norm to both sides of inequality (\ref{eq6.38}) yields the estimate
\begin{equation*}
\|\Diff_{s,q}f\|_{L^p(\bbbr^n)}\!\lesssim\!\|\LPG_{\lambda,q}(\varGamma(\frac{-s}{2})\Lift_s f)\|_{L^p(\bbbr^n)}\!\!+\!\|\Remain_{s,q}(\varGamma(\frac{-s}{2})\Lift_s f)\|_{L^p(\bbbr^n)}.
\end{equation*}
Now it remains to prove that the operators $\LPG_{\lambda,q}$ and $\Remain_{s,q}$, as defined in (\ref{eq1-26}) and (\ref{eq1-27}), are bounded from $L^p(\bbbr^n)$ to $L^p(\bbbr^n)$. By (\ref{eq7.1}) and Corollary \ref{corollary2}, we have $2\leq q<\infty$, $1<\lambda=\frac{q}{\bar{p}}<\infty$, and $0<s=n(\frac{1}{\bar{p}}-\frac{1}{q})<1$, thus the operators $\LPG_{\lambda,q}$ and $\Remain_{s,q}$ are bounded from $L^q(\bbbr^n)$ to $L^q(\bbbr^n)$, and hence are bounded from $L^q(\bbbr^n)$ to $L^{q,\infty}(\bbbr^n)$. The parameters $\bar{p}$, $q$, $s$ with $\lambda=\frac{q}{\bar{p}}$ satisfy condition (\ref{eq7.1}), then they also satisfy conditions of Theorem \ref{theorem4} and Theorem \ref{theorem5}, therefore the operators $\LPG_{\lambda,q}$ and $\Remain_{s,q}$ are bounded from $L^{\bar{p}}(\bbbr^n)$ to $L^{\bar{p},\infty}(\bbbr^n)$. Applying the Marcinkiewicz interpolation theorem, i.e. \cite[Theorem 1.3.2.]{14classical}, yields the $L^p(\bbbr^n)$-boundedness property of the operators $\LPG_{\lambda,q}$ and $\Remain_{s,q}$. We recall Definition \ref{definition2} and complete the proof of Corollary \ref{corollary3}.
\end{proof}
\begin{proof}[Proof of Corollary \ref{corollary4}]
Let $f\in\Sw_0(\bbbr^n)$ be a nonzero Schwartz function, then we have $\|f\|_{\Fspq}<\infty$ and $\|f\|_{\Lps}<\infty$. Furthermore, there exist positive finite constant numbers $C_1$ and $R$ and a subset $E\subseteq\bbbr^n$ so that all of the following conditions are met,
\begin{equation}\label{eq7.4}
E\subseteq\Omega:=\{z\in\bbbr^n:|z|\leq R\},
\end{equation}
\begin{equation}\label{eq7.5}
0<\Lebes^n(E)<\infty,
\end{equation}
\begin{equation}\label{eq7.6}
x\in\Omega^c:=\bbbr^n\setminus\Omega\quad\Rightarrow\quad|f(x)|<\frac{1}{2}C_1,
\end{equation}
\begin{equation}\label{eq7.7}
y\in E\quad\Rightarrow\quad|f(y)|>C_1.
\end{equation}
When $0<p<q<\infty$ and $-\infty<s\leq n(\frac{1}{p}-\frac{1}{q})$, we consider points $x\in\bbbr^n$ and $y\in E$ satisfying the condition $|x|>2R$. By conditions (\ref{eq7.4}), (\ref{eq7.6}), and (\ref{eq7.7}), we have $|y|\leq R<\frac{1}{2}|x|$, $\frac{1}{2}|x|\leq|y-x|\leq\frac{3}{2}|x|$, and
\begin{equation*}
|f(y)-f(x)|\geq|f(y)|-|f(x)|>\frac{1}{2}C_1.
\end{equation*}
Therefore, we can obtain
\begin{align}
&\|\Diff_{s,q}f\|_{L^p(\bbbr^n)}^p=\int_{\bbbr^n}\bigg(\int_{\bbbr^n}\frac{|f(y)-f(x)|^q}
{|y-x|^{n+sq}}dy\bigg)^{\frac{p}{q}}dx\nonumber\\
&\geq\int_{|x|>2R}\bigg(\int_{E}\frac{|f(y)-f(x)|^q}{|y-x|^{n+sq}}dy\bigg)^{\frac{p}{q}}dx\nonumber\\
&\geq\int_{|x|>2R}\bigg(\int_{E}\frac{C_2}{|x|^{n+sq}}dy\bigg)^{\frac{p}{q}}dx\nonumber\\
&=\int_{|x|>2R}\frac{[\Lebes^n(E)\!\cdot\!C_2]^{\frac{p}{q}}}
{|x|^{\frac{np}{q}+sp}}dx=\infty,\label{eq7.9}
\end{align}
since $\frac{np}{q}+sp\leq\frac{np}{q}+np(\frac{1}{p}-\frac{1}{q})=n$, and $C_2=C_2(n,s,q,C_1)$. Thus (\ref{eq1.88}) is proven. When $0<p,q<\infty$ and $-\infty<s\leq0$, we consider points $x\in E$ and $y\in\bbbr^n$ satisfying the condition $|y|>2R$. By conditions (\ref{eq7.4}), (\ref{eq7.6}), and (\ref{eq7.7}), we have $|x|\leq R<\frac{1}{2}|y|$, $\frac{1}{2}|y|\leq|y-x|\leq\frac{3}{2}|y|$, and
\begin{equation*}
|f(y)-f(x)|\geq|f(x)|-|f(y)|>C_1-\frac{1}{2}C_1=\frac{1}{2}C_1.
\end{equation*}
Hence, we have
\begin{align}
&\|\Diff_{s,q}f\|_{L^{p,\infty}(\bbbr^n)}^p\!\!=\!\sup_{\alpha>0}\alpha^p\!\cdot\!
\Lebes^n(\{x\!\in\!\bbbr^n\!:\!\Diff_{s,q}f(x)\!>\!\alpha\})\nonumber\\
&=\!\sup_{\alpha>0}\alpha^p\!\cdot\!\Lebes^n(\{x\!\in\!\bbbr^n\!:\!\Diff_{s,q}f(x)^q\!>\!\alpha^q\})
\nonumber\\
&\geq\!\sup_{\alpha>0}\alpha^p\!\cdot\!\Lebes^n(\{x\!\in\! E\!:\!\int_{\bbbr^n}
\frac{|f(y)-f(x)|^q}{|y-x|^{n+sq}}dy\!>\!\alpha^q\}).\label{eq7.8}
\end{align}
However, for all $x\in E$, we have
\begin{align}
&\int_{\bbbr^n}\frac{|f(y)-f(x)|^q}{|y-x|^{n+sq}}dy
\geq\int_{|y|>2R}\frac{|f(y)-f(x)|^q}{|y-x|^{n+sq}}dy\nonumber\\
&\geq\int_{|y|>2R}\frac{C_3}{|y|^{n+sq}}dy=\infty,\label{eq7.10}
\end{align}
since $n+sq\leq n$, and $C_3=C_3(n,s,q,C_1)$. Thus for all $0<\alpha<\infty$, we have
\begin{equation}\label{eq7.14}
\Lebes^n(\{x\!\in\! E\!:\!\!\int_{\bbbr^n}\!\!
\frac{|f(y)\!-\!f(x)|^q}{|y\!-\!x|^{n+sq}}dy\!>\!\alpha^q\})\!=\!\Lebes^n(E).
\end{equation}
Combining (\ref{eq7.8}), (\ref{eq7.14}) with the condition (\ref{eq7.5}) proves (\ref{eq1.133}). The proof of Corollary \ref{corollary4} is complete.
\end{proof}

	\section*{Acknowledgments}
	The author would like to thank the anonymous reviewer for the review of this article.
	
	\section*{Statements and Declarations}
	The author has no competing interests to declare that are relevant to the content of this article.
	
	
	
	
	
	
	
	


\begin{thebibliography}{99}
		\bibitem{AQCP2018}
		\newblock H. Al-Qassem and L. Cheng and Y. Pan, 
		\newblock On generalized Littlewood–Paley functions, 
		\newblock \emph{Collectanea Mathematica}, \textbf{69(2)} (2018), 297-314.
		
		\bibitem{AQCP2021}
		\newblock Hussain Al-Qassem and Leslie Cheng and Yibiao Pan, 
		\newblock Generalized Littlewood-Paley functions on product spaces, 
		\newblock \emph{Turkish Journal of Mathematics}, \textbf{45(1)} (2021), 319-345.
		
		\bibitem{BM2022} 
		\newblock Cristina Benea and Camil Muscalu, 
		\newblock Multiple vector-valued, mixed-norm estimates for Littlewood–Paley square functions, 
		\newblock \emph{Publicacions Matemàtiques}, \textbf{66(2)} (2022), 631-681.
		
		\bibitem{BCP1962} 
		\newblock A. Benedek and A. P. Calderón and R. Panzone, 
		\newblock Convolution operators on Banach space valued functions, 
		\newblock \emph{Proceedings of the National Academy of Sciences}, \textbf{48(3)} (1962), 356-365.
		\bibitem{Bourgain1989} 
		\newblock Jean Bourgain,
		\newblock On the behavior of the constant in the Littlewood-Paley inequality, 
		\newblock \emph{Geometric Aspects of Functional Analysis}, Lecture Notes in Mathematics, Volume 1376, 202-208, Springer, Berlin, Heidelberg, 1989.
		
		\bibitem{CHIFRRSY}
		\newblock Mingming Cao and Mahdi Hormozi and Gonzalo Ibañez-Firnkorn and Israel P. Rivera-Ríos and Zengyan Si and Kôzô Yabuta,
		\newblock Weak and strong type estimates for the multilinear Littlewood–Paley operators, 
		\newblock \emph{Journal of Fourier Analysis and Applications}, \textbf{27(4)} (2021), 62.
		
		\bibitem{CLX2017}
		\newblock Mingming Cao and Kangwei Li and Qingying Xue, 
		\newblock A characterization of two-weight norm inequality for Littlewood-Paley $g_{\lambda}^{*}$-function, 
		\newblock \emph{The Journal of Geometric Analysis}, \textbf{28(2)} (2018), 842-865.
		
		\bibitem{CX2021}
		\newblock Mingming Cao and Qingying Xue, 
		\newblock Multilinear Littlewood-Paley-Stein operators on non-homogeneous spaces, 
		\newblock \emph{The Journal of Geometric Analysis}, \textbf{31(9)} (2021), 9295-9337.
		
		\bibitem{CLYY2017}
		\newblock Der-Chen Chang and Jun Liu and Dachun Yang and Wen Yuan, 
		\newblock Littlewood-Paley characterizations of Haj\l{}asz-Sobolev and Triebel-Lizorkin spaces via averages on balls, 
		\newblock \emph{Potential Analysis}, \textbf{46(2)} (2017), 227-259.
		
		\bibitem{CWYZ2020}
		\newblock Der-Chen Chang and Songbai Wang and Dachun Yang and Yangyang Zhang, 
		\newblock Littlewood-Paley characterizations of Hardy-type spaces associated with ball quasi-Banach function spaces, 
		\newblock \emph{Complex Analysis and Operator Theory}, \textbf{14(3)} (2020), 40.
		
		\bibitem{Cheng2007}
		\newblock Leslie C. Cheng, 
		\newblock On Littlewood-Paley Functions, 
		\newblock \emph{Proceedings of the American Mathematical Society}, \textbf{135(10)} (2007), 3241-3247.
		
		\bibitem{Piero2019} 
		\newblock Piero D'Ancona, 
		\newblock A short proof of commutator estimates, 
		\newblock \emph{Journal of Fourier Analysis and Applications}, \textbf{25(3)} (2019), 1134-1146.
		
		\bibitem{DFP2000} 
		\newblock Yong Ding and Dashan Fan and Yibiao Pan, 
		\newblock On Littlewood-Paley functions and singular integrals, 
		\newblock \emph{Hokkaido Mathematical Journal}, \textbf{29(3)} (2000), 537-552.
		
		\bibitem{Javier2013}
		\newblock Javier Duoandikoetxea,
		\newblock Sharp $L^p$ boundedness for a class of square functions, 
		\newblock \emph{Revista Matemática Complutense}, \textbf{26(2)} (2013), 535-548.
		
		\bibitem{FS2002} 
		\newblock Dashan Fan and Shuichi Sato, 
		\newblock Remarks on Littlewood-Paley functions and singular integrals, 
		\newblock \emph{Journal of the Mathematical Society of Japan}, \textbf{54(3)} (2002), 565-585.
		
		\bibitem{FZ2021} 
		\newblock Dashan Fan and Fayou Zhao, 
		\newblock Littlewood-Paley functions and Triebel-Lizorkin spaces, Besov spaces, 
		\newblock \emph{Analysis in Theory and Applications}, \textbf{37(3)} (2021), 267-288.
		
		\bibitem{fefferman1970}
		\newblock Charles Fefferman,
		\newblock Inequalities for strongly singular convolution operators, 
		\newblock \emph{Acta Mathematica}, \textbf{124} (1970), 9-36.
		
		\bibitem{FS1982}
		\newblock Robert Fefferman and Elias M. Stein, 
		\newblock Singular integrals on product spaces, 
		\newblock \emph{Advances in Mathematics}, \textbf{45(2)} (1982), 117-143.
		
		\bibitem{LX2018}
		\newblock Feng Liu and Qingying Xue, 
		\newblock Characterizations of the multiple Littlewood–Paley operators on product domains, 
		\newblock \emph{Publicationes Mathematicae Debrecen}, \textbf{92/3-4} (2018), 419-439.
		
		\bibitem{FZ2016}
		\newblock Xing Fu and Ji Man Zhao,
		\newblock Endpoint estimates of generalized homogeneous Littlewood–Paley g-functions over non-homogeneous metric measure spaces, 
		\newblock \emph{Acta Mathematica Sinica, English Series}, \textbf{32(9)} (2016), 1035-1074.
		
		\bibitem{GM2022} 
		\newblock Arash Ghorbanalizadeh and Monire Mikaeili Nia,
		\newblock Strong- and weak-type estimate for Littlewood–Paley operators associated with Laplace–Bessel differential operator, 
		\newblock \emph{Banach Journal of Mathematical Analysis}, \textbf{16(2)} (2022), 21.
		
		\bibitem{GM2023} 
		\newblock Arash Ghorbanalizadeh and Monire Mikaeili Nia,
		\newblock Weak-type estimate for the Littlewood-Paley operators associated with Laplace-Bessel differential operator, 
		\newblock \emph{The Journal of Geometric Analysis}, \textbf{33(11)} (2023), 353.
		\bibitem{14classical} 
		\newblock Loukas Grafakos, 
		\newblock \emph{Classical Fourier Analysis}, 
		\newblock Graduate Texts in Mathematics, Volume 249, 3$^{rd}$ edition, Springer, New York, 2014.
		
		\bibitem{14modern} 
		\newblock Loukas Grafakos, 
		\newblock \emph{Modern Fourier Analysis}, 
		\newblock Graduate Texts in Mathematics, Volume 250, 3$^{rd}$ edition, Springer, New York, 2014.
		
		\bibitem{HX2019} 
		\newblock Sha He and Qingying Xue,
		\newblock Parametrized multilinear Littlewood-Paley operators on Hardy spaces, 
		\newblock \emph{Taiwanese Journal of Mathematics}, \textbf{23(1)} (2019), 87-101.
		
		\bibitem{HYY2016}
		\newblock Ziyi He and Dachun Yang and Wen Yuan,
		\newblock Littlewood-Paley characterizations of second-order Sobolev spaces via averages on balls, 
		\newblock \emph{Canadian Mathematical Bulletin}, \textbf{59(1)} (2016), 104-118.
		
		\bibitem{HYY2018} 
		\newblock Ziyi He and Dachun Yang and Wen Yuan, 
		\newblock Littlewood-Paley characterizations of higher-order Sobolev spaces via averages on balls, 
		\newblock \emph{Mathematische Nachrichten}, \textbf{291(2-3)} (2018), 284-325.
		
		\bibitem{HouWu2019}
		\newblock Xianming Hou and Huoxiong Wu,
		\newblock Limiting weak-type behaviors for certain Littlewood-Paley functions, 
		\newblock \emph{Acta Mathematica Scientia}, \textbf{39(1)} (2019), 11-25.
		
		\bibitem{HLYY2019}
		\newblock Long Huang and Jun Liu and Dachun Yang and Wen Yuan, 
		\newblock Atomic and Littlewood–Paley characterizations of anisotropic mixed-norm Hardy spaces and their applications, 
		\newblock \emph{The Journal of Geometric Analysis}, \textbf{29(3)} (2019), 1991-2067.
		
		\bibitem{JYYZ2022} 
		\newblock Hongchao Jia and Dachun Yang and Wen Yuan and Yangyang Zhang,
		\newblock Estimates for Littlewood–Paley operators on ball Campanato-type function spaces, 
		\newblock \emph{Results in Mathematics}, \textbf{78} (2022), 37.
		
		\bibitem{JL2018} 
		\newblock Huaye Jiao and Haibo Lin,
		\newblock Boundedness of Littlewood-Paley g-functions on non-homogeneous metric measure spaces, 
		\newblock \emph{New York Journal of Mathematics}, \textbf{24} (2018), 815-847.
		
		\bibitem{Lerner2019}
		\newblock Andrei K. Lerner, 
		\newblock Quantitative weighted estimates for the Littlewood–Paley square function and Marcinkiewicz multipliers, 
		\newblock \emph{Mathematical Research Letters}, \textbf{26(2)} (2019), 537-556.
		
		\bibitem{LW2014} 
		\newblock Feng Liu and Huoxiong Wu,
		\newblock On the $L^2$ boundedness for the multiple Littlewood–Paley functions with rough kernels, 
		\newblock \emph{Journal of Mathematical Analysis and Applications}, \textbf{410(1)} (2014), 403-410.
		
		\bibitem{LuTao2017} 
		\newblock Guanghui Lu and Shuangping Tao,
		\newblock Commutators of Littlewood-Paley $g_{\kappa}^*$-functions on non-homogeneous metric measure spaces, 
		\newblock \emph{Open Mathematics}, \textbf{15(1)} (2017), 1283-1299.
		
		\bibitem{MNY2010}  
		\newblock Yan Meng and Eiichi Nakai and Dachun Yang, 
		\newblock Estimates for Lusin-area and Littlewood-Paley $g_{\lambda}^*$ functions over spaces of homogeneous type, 
		\newblock \emph{Nonlinear Analysis: Theory, Methods \& Applications}, \textbf{72(5)} (2010), 2721-2736.
		
		\bibitem{MY2008}  
		\newblock Yan Meng and Dachun Yang,
		\newblock Estimates for Littlewood–Paley operators in $BMO(\mathbb{R}^n)$, 
		\newblock \emph{Journal of Mathematical Analysis and Applications}, \textbf{346(1)} (2008), 30-38.
		
		\bibitem{MW1974} 
		\newblock Benjamin Muckenhoupt and Richard L. Wheeden, 
		\newblock Norm inequalities for the Littlewood-Paley function $g_{\lambda}^*$, 
		\newblock \emph{Transactions of the American Mathematical Society}, \textbf{191} (1974), 95-111.
		
		\bibitem{Prats19} 
		\newblock Martí Prats, 
		\newblock Measuring Triebel-Lizorkin fractional smoothness on domains in terms of first-order differences, 
		\newblock \emph{Journal of the London Mathematical Society}, \textbf{100(2)} (2019), 692-716.
		
		\bibitem{real.analysis.royden} 
		\newblock Halsey L. Royden and Patrick M. Fitzpatrick, 
		\newblock \emph{Real analysis}, 
		\newblock 4$^{th}$ edition, Pearson Education, Inc., Prentice Hall 2010.
		
		\bibitem{Rudin.Real.Complex.Anal}
		\newblock Walter Rudin,
		\newblock \emph{Real and complex analysis},
		\newblock McGraw-Hill Series in Higher Mathematics, 3$^{rd}$ edition, McGraw-Hill international editions, 1987.
		
		\bibitem{Rudin.Funct.Anal}
		\newblock Walter Rudin,
		\newblock \emph{Functional analysis},
		\newblock International Series in Pure and Applied Mathematics, 2$^{nd}$ edition, McGraw-Hill Science/Engineering/Math, 1991.
		
		\bibitem{Sato2008} 
		\newblock Shuichi Sato, 
		\newblock Estimates for Littlewood-Paley functions and extrapolation, 
		\newblock \emph{Integral Equations and Operator Theory}, \textbf{62(3)} (2008), 429-440.
		
		\bibitem{seeger1989note} 
		\newblock Andreas Seeger, 
		\newblock A note on Triebel-Lizorkin spaces, 
		\newblock \emph{Banach Center Publications}, \textbf{22(1)} (1989), 391-400.
		
		\bibitem{SXY2014}
		\newblock Shaoguang Shi and Qingying Xue and Kôzô Yabuta, 
		\newblock On the boundedness of multilinear Littlewood-Paley $g_{\lambda}^*$ function, 
		\newblock \emph{Journal de Mathématiques Pures et Appliquées}, \textbf{101(3)} (2014), 394-413.
		
		\bibitem{Stein1958}
		\newblock Elias M. Stein, 
		\newblock On the functions of Littlewood-Paley, Lusin, and Marcinkiewicz, 
		\newblock \emph{Transactions of the American Mathematical Society}, \textbf{88(2)} (1958), 430-466.
		
		\bibitem{Stein1961} 
		\newblock Elias M. Stein, 
		\newblock The characterization of functions arising as potentials, 
		\newblock \emph{Bulletin of the American Mathematical Society}, \textbf{67(1)} (1961), 102-104.
		
		\bibitem{Stein1961-2} 
		\newblock Elias M. Stein, 
		\newblock On some functions of Littlewood-Paley and Zygmund, 
		\newblock \emph{Bulletin of the American Mathematical Society}, \textbf{67(1)} (1961), 99-101.
		
		\bibitem{Stein1971} 
		\newblock Elias M. Stein,
		\newblock \emph{Singular Integrals and Differentiability Properties of Functions}, 
		\newblock (PMS-30), Volume 30, Princeton University Press, Princeton, 1970.
		
		\bibitem{SteinHarmonic} 
		\newblock Elias M. Stein,
		\newblock \emph{Harmonic analysis: real-variable methods, orthogonality, and oscillatory integrals}, 
		\newblock (PMS-43), Princeton University Press, Princeton, 1993.
		
		\bibitem{SteinWeissFourier} 
		\newblock Elias M. Stein and Guido Weiss,
		\newblock \emph{Introduction to Fourier Analysis on Euclidean Spaces}, 
		\newblock (PMS-32), Princeton University Press, Princeton, 1971.
		
		\bibitem{TW2014} 
		\newblock Shuangping Tao and Lijuan Wang,
		\newblock Boundedness of Littlewood-Paley operators and their commutators on Herz-Morrey spaces with variable exponent,
		\newblock \emph{Journal of Inequalities and Applications}, \textbf{2014(1)} (2014), 227.
		
		\bibitem{1983functionspaces} 
		\newblock Hans Triebel, 
		\newblock \emph{Theory of Function Spaces}, 
		\newblock Monographs in Mathematics, Springer Basel, 1983.
		
		\bibitem{1992functionspaces}
		\newblock Hans Triebel, 
		\newblock \emph{Theory of Function Spaces II}, 
		\newblock Monographs in Mathematics, Springer Basel, 1992.
		
		\bibitem{Wang2023} 
		\newblock Lifeng Wang,
		\newblock Inequalities in homogeneous Triebel-Lizorkin and Besov-Lipschitz spaces, 
		\newblock \emph{Communications on Pure and Applied Analysis}, \textbf{22(4)} (2023), 1318-1393.
	\end{thebibliography}
\end{document}